\documentclass[letterpaper,11pt,oneside,reqno]{amsart}
\usepackage{bm}
\usepackage{mathrsfs}
\usepackage{amsfonts,amsmath, amssymb,amsthm,amscd}
\usepackage[height=9.6in,width=5.95in]{geometry}
\usepackage[mpexclude,DIV13]{typearea}
\usepackage{verbatim}
\usepackage{hyperref}
\usepackage{graphicx}
\usepackage[latin1]{inputenc}
\usepackage{latexsym}
\usepackage{lscape}
\usepackage{epsfig}
\usepackage{subfigure}
\include{bibtex}

\usepackage{setspace}
\makeatletter
\def\printnotation{{%
\def\indexname{Index of notation}
\begin{theindex}
\@input{\jobname.ntn}
\end{theindex}
}}
\makeatother

\makeglossary

\begin{document}
\bibliographystyle{alpha}
\newcommand{\cn}[1]{\overline{#1}}
\newcommand{\e}[0]{\epsilon}

\newcommand{\Pfree}[5]{\ensuremath{\mathbb{P}_{\textrm{\tiny{free}}}^{#1,#2,#3,#4,#5}}}
\newcommand{\PfreeShort}{\ensuremath{\mathbb{P}_{\textrm{\tiny{free}}}}}
\newcommand{\PfreeExp}[5]{\ensuremath{\mathbb{E}_{\textrm{\tiny{free}}}^{#1,#2,#3,#4,#5}}}
\newcommand{\PfreeExpShort}{\ensuremath{\mathbb{E}_{\textrm{\tiny{free}}}}}

\newcommand{\PH}[8]{\ensuremath{\mathbb{P}^{#1,#2,#3,#4,#5,#6,#7}_{#8}}}
\newcommand{\PHShort}[1]{\ensuremath{\mathbb{P}_{#1}}}
\newcommand{\PHExp}[8]{\ensuremath{\mathbb{E}^{#1,#2,#3,#4,#5,#6,#7}_{#8}}}
\newcommand{\PHExpShort}[1]{\ensuremath{\mathbb{E}_{#1}}}
%%%%

%AAAAAAA
\newcommand{\WH}[8]{\ensuremath{\mathbb{P}^{#1,#2,#3,#4,#5,#6,#7}_{#8}}}
\newcommand{\Wfree}[5]{\ensuremath{\mathbb{P}^{#1,#2,#3,#4,#5}}}
\newcommand{\WHShort}[3]{\ensuremath{\mathbb{P}^{#1,#2}_{#3}}}
\newcommand{\WHShortCouple}[2]{\ensuremath{\mathbb{P}^{#1}_{#2}}}

\newcommand{\walk}[3]{\ensuremath{X^{#1,#2}_{#3}}}
\newcommand{\walkupdated}[3]{\ensuremath{\tilde{X}^{#1,#2}_{#3}}}
\newcommand{\walkfull}[2]{\ensuremath{X^{#1,#2}}}
\newcommand{\walkfullupdated}[2]{\ensuremath{\tilde{X}^{#1,#2}}}
%AAAAAAA

%\newcommand{\D}[8]{\ensuremath{D^{#1,#2,#3,#4,#5,#6,#7}_{#8}}}
%\newcommand{\DShort}[1]{\ensuremath{D_{#1}}}
\newcommand{\partfunc}[8]{\ensuremath{Z^{#1,#2,#3,#4,#5,#6,#7}_{#8}}}
\newcommand{\partfuncShort}[1]{\ensuremath{Z_{#1}}}
\newcommand{\bolt}[8]{\ensuremath{W^{#1,#2,#3,#4,#5,#6,#7}_{#8}}}
\newcommand{\boltShort}[1]{\ensuremath{W_{#1}}}
\newcommand{\boltNew}{\ensuremath{W}}
\newcommand{\QTLH}{\ensuremath{\mathfrak{Q}}}
\newcommand{\QTLHgen}{\ensuremath{\mathfrak{L}}}

\newcommand{\whitenoise}{\ensuremath{\mathscr{\dot{W}}}}

\newcommand{\Rstar}{\ensuremath{\hat{R}}}
\newcommand{\EE}{\ensuremath{\mathbb{E}}}
\newcommand{\PP}{\ensuremath{\mathbb{P}}}
\newcommand{\PPBM}{\ensuremath{\mathbb{P}^{\mathrm{BM}}}}
\newcommand{\PPBB}{\ensuremath{\mathbb{P}^{\mathrm{BB}}}}
\newcommand{\var}{\textrm{var}}
\newcommand{\N}{\ensuremath{\mathbb{N}}}
\newcommand{\R}{\ensuremath{\mathbb{R}}}
\newcommand{\C}{\ensuremath{\mathbb{C}}}
\newcommand{\Z}{\ensuremath{\mathbb{Z}}}
\newcommand{\Q}{\ensuremath{\mathbb{Q}}}
\newcommand{\T}{\ensuremath{\mathbb{T}}}
\newcommand{\E}[0]{\mathbb{E}}
\newcommand{\OO}[0]{\Omega}
\newcommand{\dd}{{\rm d}}
\newcommand{\F}[0]{\mathfrak{F}}
\def \Ai {{\rm Ai}}
\newcommand{\G}[0]{\mathfrak{G}}
\newcommand{\ta}[0]{\theta}
\newcommand{\w}[0]{\omega}
\newcommand{\ra}[0]{\rightarrow}
\newcommand{\vectoro}{\overline}
\newcommand{\crairy}{\mathcal{CA}}
\newcommand{\AP}{\mathfrak{a}}
\newcommand{\BP}{\mathfrak{b}}
\newcommand{\Fext}{\mathcal{F}_{{\rm ext}}}
\newtheorem{theorem}{Theorem}[section]
\newtheorem{partialtheorem}{Partial Theorem}[section]
\newtheorem{conj}[theorem]{Conjecture}
\newtheorem{lemma}[theorem]{Lemma}
\newtheorem{proposition}[theorem]{Proposition}
\newtheorem{corollary}[theorem]{Corollary}
\newtheorem{claim}[theorem]{Claim}
\newtheorem{experiment}[theorem]{Experimental Result}

%KPZ equation
\newcommand{\HHKPZ}{\mathcal{H}} %%%%%%%%%%%%%%%
\newcommand{\HKPZ}[2]{\HHKPZ(#1,#2)}
\newcommand{\HKPZshort}{\HHKPZ}
\newcommand{\HKPZt}[2]{\HHKPZ^{(#1)}(#1,#2)}
\newcommand{\HKPZtzero}[2]{\HHKPZ^{(#1)}_0(#2)}
\newcommand{\tildeHKPZt}[2]{\tilde{\HHKPZ}^{(#1)}(#1,#2)}
\newcommand{\tildeHKPZtzero}[2]{\tilde{\HHKPZ}^{(t)}(#1,#2)}
%KPZ narrow wedge
\newcommand{\HKPZnw}[2]{\HHKPZ^{\mathrm{nw}}(#1,#2)}
\newcommand{\HKPZnwshort}{\HHKPZ^{\mathrm{nw}}}
%KPZ stationary
\newcommand{\HKPZeq}[2]{\HHKPZ^{\mathrm{stat}}(#1,#2)}
\newcommand{\HKPZeqshort}{\HHKPZ^{\mathrm{stat}}}
%KPZ initial data
\newcommand{\Hzero}[1]{\HHKPZ_0(#1)}
\newcommand{\Hzerotilde}[1]{\tilde{\HHKPZ}_0(#1)}
\newcommand{\Hzeroshort}{\HHKPZ_0}
\newcommand{\Hzerot}[2]{\HHKPZ_0^{(#1)}(#2)}
%KPZ line ensemble
\newcommand{\HKPZline}[2]{\HHKPZ^{#1}_{#2}}  %\HKPZline{t}{n}
\newcommand{\HKPZlinet}[1]{\HHKPZ^{#1}}  %\HKPZlinet{t}
\newcommand{\HSDline}[3]{\HHKPZ^{#1,#2}_{#3}}  %\HKPZline{t}{N}{n}
\newcommand{\HSDlinet}[2]{\HHKPZ^{#1,#2}}  %\HKPZline{t}{N}
%KPZ fixed point line ensemble
\newcommand{\HHKPZFP}{\mathfrak{H}} %%%%%%%%%%%%%%%
\newcommand{\HKPZFPline}[2]{\HHKPZFP^{#1}_{#2}}  %\HKPZFPline{t}{n}
\newcommand{\HKPZFPlinet}[1]{\HHKPZFP^{#1}}  %\HKPZFPlinet{t}

\newcommand{\HKPZFPlinetnw}[1]{\HHKPZFP^{\mathrm{nw},#1}}  %\HKPZFPlinetnw{t}

%SD fixed point line ensemble
\newcommand{\HSDFPline}[3]{\HHKPZFP^{#1,#2}_{#3}}  %\HKSDFPline{t}{N}{n}
\newcommand{\HSDFPlinet}[2]{\HHKPZFP^{#1,#2}}  %\HKSDFPlinet{t}{N}
%Crossover Airy2 process
\newcommand{\AAiry}{\mathcal{A}} %%%%%%%%%%%%%%%
\newcommand{\Airyt}[2]{\AAiry^{#1}(#2)}
%OConnell partition function line ensemble
\newcommand{\ZZOCon}{Z} %%%%%%%%%%%%%%%
\newcommand{\OConZ}[2]{\ZZOCon^{#1}_{#2}} %\OConZ{N}{n}
\newcommand{\OConZshort}[1]{\ZZOCon^{#1}} %\OConshort{N}
%OConnell free energy line ensemble
\newcommand{\XXOCon}{X} %%%%%%%%%%%%%%%
\newcommand{\OConX}[2]{\XXOCon^{#1}_{#2}} %\OConX{N}{n}
\newcommand{\OConXshort}[1]{\XXOCon^{#1}} %\OConXshort{N}
%Zero temp OConnell free energy line ensemble
\newcommand{\MMOCon}{M} %%%%%%%%%%%%%%%
\newcommand{\OConM}[2]{\MMOCon^{#1}_{#2}} %\OConM{N}{n}
\newcommand{\OConMshort}[1]{\MMOCon^{#1}} %\OConMshort{N}
%OConnell-Warren
\newcommand{\ZZOConWar}{\mathcal{Z}} %%%%%%%%%%%%%%%
\newcommand{\ZOConWar}[3]{\ZZOConWar_{#1}(#2,#3)}  % \ZOConWar{n}{t}{x}
\newcommand{\ZOConWarshort}{\ZZOConWar}
%OConnell intermediate disorder rescaled partition function line ensemble
\newcommand{\ZZSD}{\mathcal{Z}} %%%%%%%%%%%%%%%
\newcommand{\ZSD}[4]{\ZZSD^{#3,#1}_{#2}(#4)}   %\ZSD{N}{n}{t}{x}
\newcommand{\ZSDshort}[1]{\ZZSD} %\ZSDshort{N}
%SHE
\newcommand{\ZZSHE}{\mathcal{Z}} %%%%%%%%%%%%%%%
\newcommand{\ZSHE}[2]{\ZZSHE(#1,#2)}
\newcommand{\ZSHEn}[3]{\ZZSHE_{#3}(#1,#2)}
\newcommand{\ZSHEshort}{\ZZSHE}
%SHE narrow wedge
\newcommand{\ZSHEnw}[2]{\ZZSHE^{\mathrm{nw}}(#1,#2)}
\newcommand{\ZSHEnwshort}{\ZZSHE^{\mathrm{nw}}}
%Polymer endpoint
\newcommand{\PPendpoint}{\mathcal{P}^{t}_{\whitenoise}}

\newcommand{\Hyp}{\mathbf{Hyp}}
\newcommand{\RN}{\mathrm{RN}}
\newcommand{\Ham}{\ensuremath{\mathbf{H}}}

\newcommand{\hfixedpt}{\mathfrak{H}}
\newcommand{\afixedpt}{\mathcal{A}}

\newcommand{\hfixedptone}{\mathfrak{H1}}
\newcommand{\hfixedpttwo}{\mathfrak{H2}}

\newcommand{\maxone}{\mathcal{M}}

\newcommand{\Alan}{{\bf \mathsf{Alan}}}
\newcommand{\tzero}{1}

\newcommand{\Rkle}{\ensuremath{\mathbb{R}^k_{>}}}
\newcommand{\Rklezero}{\ensuremath{\mathbb{R}^k_{>0}}}

\def\todo#1{\marginpar{\raggedright\footnotesize #1}}
\def\change#1{{\color{green}\todo{change}#1}}
\def\note#1{\textup{\textsf{\color{blue}(#1)}}}

\theoremstyle{definition}
\newtheorem{rem}[theorem]{Remark}

\theoremstyle{definition}
\newtheorem{com}[theorem]{Comment}

\theoremstyle{definition}
\newtheorem{definition}[theorem]{Definition}

\theoremstyle{definition}
\newtheorem{definitions}[theorem]{Definitions}

\theoremstyle{definition}
\newtheorem{conjecture}[theorem]{Conjecture}

%\title[Locally Brownian KPZ]{The Hopf-Cole solution of the narrow-wedge Kardar-Parisi-Zhang equation is locally Brownian}
\title[KPZ line ensemble]{KPZ line ensemble}

\author[I. Corwin]{Ivan Corwin}
\address{I. Corwin, Columbia University,
Department of Mathematics,
2990 Broadway,
New York, NY 10027, USA,
and Clay Mathematics Institute, 10 Memorial Blvd. Suite 902, Providence, RI 02903, USA,
and Massachusetts Institute of Technology,
Department of Mathematics,
77 Massachusetts Avenue, Cambridge, MA 02139-4307, USA}
\email{ivan.corwin@gmail.com}

\author[A. Hammond]{Alan Hammond}
\address{A. Hammond\\
  Department of Statistics\\
  University of Oxford\\
  1 South Parks Road\\
  Oxford, U.K.}
\email{hammond@stats.ox.ac.uk}

\begin{abstract}
For each $t\geq 1$ we construct an $\N$-indexed ensemble of random continuous curves with three properties:
\begin{enumerate}
\item The lowest indexed curve is distributed as the time $t$ Hopf-Cole solution to the Kardar-Parisi-Zhang (KPZ) stochastic PDE with narrow wedge initial data;
\item The entire ensemble satisfies a resampling invariance which we call the {\it $\Ham$-Brownian Gibbs property} (with $\Ham(x)=e^{x}$);
\item Increments of the lowest indexed curve, when centered by $-t/24$ and scaled down vertically by $t^{1/3}$ and horizontally by $t^{2/3}$, remain uniformly absolutely continuous (i.e. have tight Radon-Nikodym derivatives) with respect to Brownian bridges as time $t$ goes to infinity.
\end{enumerate}
This construction uses as inputs the diffusion that O'Connell discovered \cite{OCon} in relation to the O'Connell-Yor semi-discrete Brownian polymer, the convergence result of Nica \cite{Nica} of the lowest indexed curve of that diffusion to the solution of the KPZ equation with narrow wedge initial data, and the one-point distribution formula proved by
Amir-Corwin-Quastel~\cite{ACQ} for the solution of the KPZ equation with narrow wedge initial data.

We provide four main applications of this construction:
\begin{enumerate}
\item Uniform (as $t$ goes to infinity) Brownian absolute continuity of the time $t$ solution to the KPZ equation with narrow wedge initial data, even when scaled vertically by $t^{1/3}$ and horizontally  by $t^{2/3}$;
\item Universality of the $t^{1/3}$ one-point (vertical) fluctuation scale for the solution of the KPZ equation with general initial data;
\item Concentration in the $t^{2/3}$ scale for the endpoint of the continuum directed random polymer;
\item Exponential upper and lower tail bounds for the solution at fixed time of the  KPZ equation with general initial data.
\end{enumerate}
%
%We construct a {\it KPZ$_t$ line ensemble} -- an $\N$-indexed collection of random continuous curves which satisfies a resampling invariance called the {\it $\Ham$-Brownian Gibbs property} (with $\Ham(x)=e^{x}$) and whose lowest indexed curve is distributed as the time $t$ Hopf-Cole solution to the Kardar-Parisi-Zhang (KPZ) stochastic PDE with narrow-wedge initial data. We prove four main applications of this construction:
%\begin{enumerate}
%\item Uniform (in $t$) Brownian absolute continuity of the fixed time narrow-wedge initial data KPZ equation, even after fluctuation scaling of order $t^{1/3}$ and spatial scaling of order $t^{2/3}$;
%\item Universality of the $t^{1/3}$ one-point fluctuation scale for general initial data KPZ equation;
%\item Concentration in the $t^{2/3}$ scale for the endpoint of the continuum directed random polymer;
%\item Exponential upper and lower tail bounds for the fixed time general initial data KPZ equation.
%\end{enumerate}
\end{abstract}

\maketitle
\setcounter{tocdepth}{2}
\newpage
\tableofcontents
\hypersetup{linktocpage}
\newpage

%\begin{itemize}
%\item don't forget to put in the new key prop proof picture.
%~\item Reorganize the introduction so as to have a brief intro to KPZ, then the four applications. Then a sketchy statement of the main theorem, a sketch of how it is proved (including a clear list of the inputs... including the ACQ results noting that we don't yet know multipoint information) emphasizing our use of Gibbs property, monotonicity and Brownian estimates as the only other tools.
%~\item Prove the one-third exponent
%~\item Prove the same limit result of remark 1.9
%\item prove the density result
%~\item state the BQS corollary for stationary initial data
%~\item prove the transveral exponent result
%\item prove the tail exponent result
%~\item prove the BCFV corollary
%\item Review the figures and captions (including adding the result numbering to the flow chart and fixing up the overview line labels)
%~\item rewrite the outline
%~\item rerun the glossary
%~\item check the bibliography for usage and form
%\end{itemize}

\section{Introduction and applications}

We start by introducing the KPZ equation and then state four theorems about: (1) the locally Brownian nature of the narrow wedge initial data solution; (2) the $t^{1/3}$ scale of the general initial data solution; (3) the $t^{2/3}$ transversal scale for the continuum directed random polymer endpoint; and (4) exponential upper and lower tail bounds for the general initial data solution. We then describe Theorem \ref{mainthm}, the main result of this paper, from which these four theorems follow. This theorem establishes the existence of a family of structures called {\it KPZ$_t$ line ensembles} which are related to the narrow wedge initial data KPZ equation, which enjoy a certain resampling invariance, and which behave in a uniformly controllable manner over all $t\in [1,\infty)$. We describe the three inputs used in constructing these structures, the three main tools used in conjunction with these inputs, and briefly outline the steps of the construction.

The existence of the structures constructed in Theorem \ref{mainthm} is not a priori clear and in no small part relies on the recently discovered \cite{ACQ,OCon} {\it integrability} or {\it exact solvability} associated with the KPZ equation and the O'Connell-Yor semi-discrete polymer model (which in a suitable limit converges to the KPZ equation). It is an ongoing challenge to take limits of the integrable structure of the O'Connell-Yor semi-discrete polymer and this has presently only been achieved at the level of one-point distributions \cite{BorCor,BCF}. Here we construct a limit of the O'Connell-Yor semi-discrete polymer model in a far richer sense. We do so by reinterpreting the integrable structure of that model in purely probabilistic terms by means of a line ensemble with an $\Ham$-Brownian Gibbs property -- a sort of spatial Markov property in which the ensemble is invariant under resampling any given curve's increments according to the Brownian bridge measure, reweighted by an energetic interaction (determined by the Hamiltonian $\Ham$) with adjacently indexed curves. Using this probabilistic perspective, we show tightness of the line ensemble associated with the O'Connell-Yor semi-discrete polymer model and, by extracting subsequential limits, we construct KPZ$_t$ line ensembles and show that they enjoy an $\Ham$-Brownian Gibbs property as well.

A key obstacle in the study of the KPZ equation has been the lack of information about its two-point distribution, or more generally its regularity. This can be contrasted to the study of the Airy$_2$ process in which multipoint information is readily available via its determinantal structure. Theorem \ref{mainthm} shows the regularity of the KPZ equation under $t^{1/3}$ and $t^{2/3}$ scaling and readily leads to proofs of several natural and longstanding problems (four applications in total) concerning the KPZ equation and its long time scaling behavior. These applications extend far beyond those results directly accessible via integrability techniques and strongly rely upon the probabilistic perspective of the $\Ham$-Brownian Gibbs property which is central to this work.

The $\Ham$-Brownian Gibbs property generalizes (or softens) the non-intersection Brownian Gibbs property which arises in the study of tiling models, Dyson Brownian motion, non-intersecting Brownian bridges, the totally asymmetric simple exclusion process, last passage percolation with exponential or geometric weights, and the polynuclear growth model (see references in the introduction of \cite{CH}, where this non-intersecting Brownian Gibbs property was studied at length).

\subsection{The Kardar-Parisi-Zhang equation}

The Kardar-Parisi-Zhang (KPZ) stochastic partial differential equation (SPDE) \cite{KPZ} is written formally as
%\footnote{The general form of the equation is $\partial_t\HKPZ{t}{x} = \nu\partial_x^2 \HKPZ{t}{x} + \tfrac{1}{2}\lambda (\partial_x\HKPZ{t}{x})^2 + \sqrt{D}\whitenoise$ where $|\nu|,|\lambda|,D$ are positive parameters. A change of variables and rescaling reduces the system to (\ref{KPZ}).}
\begin{equation}\label{KPZ}
\partial_t\HKPZ{t}{x} = \tfrac{1}{2}\partial_x^2 \HKPZ{t}{x} + \tfrac{1}{2}\big(\partial_x\HKPZ{t}{x}\big)^2 + \whitenoise,
\end{equation}
where $\whitenoise$ is space-time Gaussian white noise (see \cite{CorwinReview} or \cite{ACQ} for mathematical background).

This equation is a central example within a large class of randomly growing one-dimensional interfaces, particle systems and directed polymers (as demonstrated by extensive numerical results, some physical experiments and limited mathematical proofs -- see \cite{CorwinReview,SaSpReview} and references therein). It has been understood since the work of \cite{BC,BG,ACQ,Hairer} that the following definition is the relevant notion for a solution to the KPZ equation.

\begin{definition}\label{SHEdef}
The {\it Hopf-Cole solution to the Kardar-Parisi-Zhang equation} is defined as\glossary{$\HKPZ{t}{x}$, Hopf-Cole solution to the Kardar-Parisi-Zhang equation}
$$\HKPZ{t}{x}:=\log \ZSHE{t}{x}$$
where $\ZSHE{t}{x}$ \glossary{$\ZSHE{t}{x}$, Solution to the multiplicative stochastic heat equation} is the solution to the {\it multiplicative stochastic heat equation}
%\footnote{In some literature (e.g. \cite{ACQ}) the sign of the non-linearity in (\ref{KPZ}) is reversed, in which case $\HKPZ{t}{x}=-\log \ZSHE{t}{x}$.}
\begin{equation}\label{SHE}
\partial_t \ZSHE{t}{x} = \tfrac{1}{2}\partial_x^2 \ZSHE{t}{x} +  \whitenoise(t,x) \ZSHE{t}{x},
\end{equation}
where $\whitenoise(t,x)$ \glossary{$\whitenoise$, Space-time Gaussian white noise} is space-time Gaussian white noise. We will use $\PP$ and $\EE$ to represent the probability and expectation operators associated with $\whitenoise$. When initial data is random, this randomness will also be included in $\PP$ and $\EE$. The equation (\ref{SHE}) should be understood in its integrated form and is well-posed for a large class of initial data (see Section 2.2.2 of the review~\cite{CorwinReview}). When discussing the KPZ equation we will always be referring to the Hopf-Cole solution. For KPZ initial data $\Hzeroshort:\R\to \R$,\glossary{$\Hzeroshort$, KPZ initial data} the solution is defined by starting the stochastic heat equation with initial data $\ZSHE{0}{x} = \exp\left\{\Hzero{x}\right\}$.

The {\it narrow wedge} initial data is not defined in terms of any $\Hzeroshort$ but corresponds with setting $\ZSHE{0}{x}$ equal to  $\delta_{x=0}$, \glossary{$\delta_{x=0}$, Dirac delta function at $0$} a Dirac delta function at $0$. We write $\ZSHEnw{t}{x}$\glossary{$\ZSHEnw{t}{x}$, Stochastic heat equation with Dirac delta function initial data} and $\HKPZnw{t}{x}$\glossary{$\HKPZnw{t}{x}$, KPZ equation with narrow wedge initial data} to denote respectively $\ZSHE{t}{x}$ and $\HKPZ{t}{x}$ with narrow wedge initial data (see \cite{ACQ,BG} for examples of how this initial data arises from the weakly asymmetric simple exclusion process).
%\footnote{Narrow-wedge initial data arises in the weakly asymmetric scaling limit of the corner growth model (or the ASEP height function) when the initial growth geometry is wedge-like (or the particle configuration step-like). See \cite{ACQ,BG,CQ}.}.
The scaled solution to the narrow wedge initial data KPZ equation is written as $\HKPZFPlinetnw{t}(x)$, \glossary{$\HKPZFPlinetnw{t}(x)$,  scaled solution to narrow wedge initial data KPZ equation} and defined by
\begin{equation}\label{crossairy2}
\HKPZnw{t}{x} = -\frac{t}{24} + t^{1/3}\HKPZFPlinetnw{t}\big(t^{-2/3}x\big) \, . %\log(\sqrt{2\pi t})
\end{equation}
It is believed that under this $t^{1/3}$ vertical and $t^{2/3}$ horizontal scaling the KPZ equation (and all other processes in the KPZ universality class) should scale to the same space-time process, called the {\it KPZ fixed point} \cite{CQ2}.

The narrow wedge initial data multiplicative stochastic heat equation also describes the evolution of the partition function for a point-to-point continuum directed random polymer \cite{AKQ2} and thus $\HKPZnwshort$ can be interpreted as the quenched free energy (see Section \ref{s.oconwar}). Define the {\it point-to-line quenched continuum directed random polymer endpoint} as the random variable $X$ with density
\begin{equation}\label{PPendpointdef}
\PPendpoint\big(X\in \dd x\big) := \frac{\ZSHEnw{t}{x}\dd x}{\int_{-\infty}^{\infty}\ZSHEnw{t}{y}\dd y}.
\end{equation}\glossary{$\PPendpoint$, Point-to-line quenched continuum directed random polymer endpoint distribution}
This measure is defined for almost every $\whitenoise$.
\end{definition}

\subsection{Uniform Brownian absolute continuity of the KPZ equation}\label{s.unifbwn}

The {\it stochastic Burgers equation} with conservative noise is defined \cite{BQS} as the spatial derivative of the KPZ equation $u(t,x) := \partial_x \HKPZ{t}{x}$. If $u(0,\cdot)=dB(\cdot)$ then, at a later time $t$, $u(t,\cdot)=dB'(\cdot)$ where $dB$ and $dB'$ are (correlated) one-dimensional Gaussian white noises. Thus when $\Hzero{\cdot}=B(\cdot)$, at a later time $t$, $\HKPZ{t}{\cdot}-\HKPZ{t}{0}$ has the distribution of Brownian motion. This initial data is called {\it stationary} and the associated solution to the KPZ equation is denoted by $\HKPZeq{t}{x}$ \glossary{$\HKPZeq{t}{x}$, KPZ equation with stationary (two-sided Brownian) initial data}.

It is believed that running the KPZ equation for any arbitrary positive time $t$ on any initial data will yield a solution which is locally Brownian. The meaning of locally Brownian is a matter of interpretation. Quastel-Remenik~\cite{QR} proved that the difference $\HKPZnw{t}{x} - \left(\HKPZeq{t}{x}-\HKPZeq{t}{0}\right)$ between the narrow wedge and stationary KPZ equation solutions (coupled to the same $\whitenoise$) is a finite variation process in $x$. Hairer~\cite{Hairer} proved that for a large class of nice initial data (not including narrow wedge though), the KPZ equation on the periodic spatial domain $[0,1]$ yields solutions with H\"{o}lder continuity $1/2-$, and, when subtracting off the stationary solution, the H\"{o}lder exponent improves to $3/2-$.

Our first application of our main result in this paper, Theorem \ref{mainthm}, is that the solution to the KPZ equation with narrow wedge initial data is locally Brownian in the sense that its spatial increments are absolutely continuous with respect to Brownian bridges. In fact, we can show a stronger result that the spatial increments of the time $t$ scaled (by $t^{1/3}$ vertically and $t^{2/3}$ horizontally) KPZ equation are absolutely continuous with respect to Brownian bridge with a Radon-Nikodym derivative which is tight as $t\to \infty$. We can also show that for vertical scaling of $t^{\nu/3}$ and horizontal scaling of $t^{2\nu/3}$, for any $\nu<1$, the increments of the KPZ equation with narrow wedge initial data converge to Brownian bridge. The Brownian absolute continuity (and the fact that it remains uniformly controlled under the scaling in large time $t$) is the primary innovation of Theorem~\ref{abscontThm} in comparison with earlier work on the locally Brownian nature of the KPZ equation.

It may also be possible to prove analogous results for a few other types of KPZ initial data (such as those mentioned in Section \ref{s.otherdata}). It is presently
unclear whether the anticipated locally Brownian nature of solutions for completely general initial data can be proved in the manner of this paper.

\medskip

\begin{theorem}\label{abscontThm}
We have the following:
\begin{enumerate}
\item For all $t>0$, $x\in \R$ and $\delta>0$, the measure on continuous functions mapping $[0,\delta]\to \R$ given by
\begin{equation*}
y\mapsto \HKPZFPlinetnw{t}(y + x) - \HKPZFPlinetnw{t}(x)
\end{equation*}
is absolutely continuous with respect to standard Brownian motion on $[0,\delta]$, and the Radon-Nikodym derivative is tight as $t$ varies in $[1,\infty)$ (with $x\in \R, \delta>0$ being kept fixed). %For any scaling parameter $\lambda_t>0$ such that $t^{-2/3} \lambda_t\to 0$ as $t\to\infty$, the measure on functions mapping $[0,\delta]\to \R$ given by
%\begin{equation*}
%y\mapsto \lambda_t^{-1/2}\Big(\HKPZnw{t}{\lambda_t(y + x)} - \HKPZnw{t}{\lambda_t x}\Big)
%\end{equation*}
%converges in distribution to standard Brownian motion on $[0,\delta]$.
\item For all $t>0$, $x\in \R$ and $\delta>0$, the measure on continuous functions mapping $[0,\delta]\to \R$ given by
\begin{equation}\label{eq.measufunction}
y\mapsto \HKPZFPlinetnw{t}(y + x) - \left(\frac{\delta - y}{\delta}\,\HKPZFPlinetnw{t}(x) + \frac{y}{\delta}\, \HKPZFPlinetnw{t}(x+\delta) \right)
\end{equation}
is absolutely continuous with respect to standard Brownian bridge on $[0,\delta]$, and the Radon-Nikodym derivative is tight as $t$ varies in $[1,\infty)$ (with $x\in \R, \delta>0$ being kept fixed).
\end{enumerate}
\end{theorem}
\begin{proof}
The results (1) follow those of (2) in an easy manner (as in \cite[Proposition 4.1]{CH}). For (2), the absolute continuity and Radon-Nikodym derivative tightness for $\HKPZFPlinetnw{t}(\cdot)$ is just a restatement of the third property of the KPZ line ensemble constructed in Theorem \ref{mainthm}.
\end{proof}
\begin{rem}
One might conjecture from the above result that for any scaling parameter $\lambda_t>0$ such that $t^{-2/3} \lambda_t\to 0$ as $t\to\infty$, the measure on functions mapping $[0,\delta]\to \R$ given by
\begin{equation*}
y \mapsto \lambda_t^{-1/2}\bigg(\HKPZnw{t}{\lambda_t(y + x)} - \Big(\frac{\delta - y}{\delta}\,\HKPZnw{t}{\lambda_t x} + \frac{y}{\delta}\, \HKPZnw{t}{\lambda_t(x+\delta)}\Big)\bigg)
\end{equation*}
converges in distribution to standard Brownian bridge on $[0,\delta]$. We do not provide a proof of this here since it does not seem to be such an immediate consequence.
\end{rem}

\subsection{Order $t^{1/3}$ fluctuations for the general initial data KPZ equation}\label{s.orderthird}
Drawing on 1977 dynamical renormalization group work of Forster-Nelson-Stephens \cite{FNS}, in 1986 Kardar-Parisi-Zhang \cite{KPZ} predicted that the equation which now bears their name would display non-trivial large $t$ fluctuations when scaled horizontally by $t^{2/3}$ and vertically by $t^{1/3}$. Non-trivial can be interpreted in various ways, and the first rigorous mathematical confirmation of the $t^{1/3}$ aspect of this prediction was due to Bal\'{a}zs-Quastel-Sepp\"{a}l\"{a}inen in 2009 \cite{BQS} who proved that the stationary initial data KPZ equation $\HKPZeq{t}{x}$ has one-point variance of order $t^{2/3}$ (i.e. $t^{1/3}$ fluctuations). The work of Amir-Corwin-Quastel in 2010 \cite{ACQ} computed the one-point distribution for the narrow wedge initial data KPZ equation and proved that, after $t^{1/3}$ scaling, the distribution converges (as $t\to \infty$) to the $F_{\textrm{GUE}}$ Tracy-Widom distribution. This finite $t$ distribution was discovered independently and in parallel in non-rigorous
work of Sasamoto-Spohn \cite{SaSp}, Dotsenko \cite{Dot}, and Calabrese-Le Doussal-Rosso \cite{CDR}.
Other rigorous work \cite{CQ,BorCor,BCF,BCFV}
 has led to
analogous exact distribution formulas for a few other {\it special} types of initial data (including stationary).

We now state the first result which proves the prediction of Kardar-Parisi-Zhang for a very wide class of initial data. In fact, we can allow the initial data to scale in a $(t^{1/3}, t^{2/3})$ manner, and still prove that the one-point fluctuations are of order $t^{1/3}$. Before stating this result, we define a class of functions which are at least sometimes not too negative, and which display at most quadratic growth with coefficient strictly less than $1/2$. This type of growth condition seems to be necessary for the existence of solutions to the KPZ equation since otherwise for the stochastic heat equation, the decay of the Gaussian heat kernel is overwhelmed by the growth of the initial data. See \cite[Section 3.2]{CorwinReview} for some results regarding the existence of solutions to the KPZ equation.

\begin{definition}\label{d.hyp}
For $C,\delta,\kappa,M>0$ we say that a function $f:\R\to \R\cup \{-\infty\}$ satisfies hypothesis $\Hyp(C,\delta,\kappa,M)$ \glossary{$\Hyp(C,\delta,\kappa,M)$, Hypothesis on KPZ initial data}if
\begin{itemize}
% \item  $f(x) \leq C$ for all $x \in [-M,M]$;
 \item  $f(x) \leq C+ (1 - \kappa)x^2/2$ for all $x \in \R$;
 \item  $\textrm{Leb} \left\{ x \in [-M,M]: f(x) \geq - C  \right\} \geq \delta$ where $\textrm{Leb}$ denotes Lebesgue measure.
\end{itemize}
\end{definition}

In the following theorem we consider the solution to the KPZ equation at time $t$. As $t$ gets large, it is nature to consider initial data which varies in the $(t^{1/3},t^{2/3}$)-scale. Thus, to enable this we allow the initial data to depend on $t$ as well. This makes for a slightly awkward notation whereby in $\HKPZt{t}{x}$, the variable $t$ represents both the time of the KPZ equation as well as the parameter indexing the initial data.

\begin{theorem}\label{t.univonethird}
Fix any $C,\delta,\kappa,M>0$ and consider a collection of functions $f^{(t)}:\R\to \R\cup\{-\infty\}$ which satisfy hypothesis $\Hyp(C,\delta,\kappa,M)$ for all $t\geq 1$. Let $\HKPZt{t}{x}$ \glossary{$\HKPZt{t}{x}$ Time $t$ KPZ equation solution with initial data indexed by $t$} represent the solution to the KPZ equation when started from initial data $\Hzerot{t}{x}=t^{1/3} f^{(t)}(t^{-2/3} x)$. Then the following holds.
\begin{enumerate}
\item For all $\e>0$ there exists a constant $C_1=C_1(\e,C,\delta,\kappa,M)$ such that, for all $t\geq 1$,
\begin{equation*}
\PP \left(\Bigg|\frac{\HKPZt{t}{0} +\frac{t}{24}}{t^{1/3}}\Bigg| \leq C_1 \right) > 1 -  \e \, .
\end{equation*}
\item Consider a second collection of functions $\tilde{f}^{(t)}:\R\to \R\cup\{-\infty\}$ which satisfy hypothesis $\Hyp(C,\delta,\kappa,M)$ for all $t\geq 1$, and let $\tildeHKPZt{t}{x}$ be the solution to the KPZ equation when started from initial data $t^{1/3} \tilde{f}^{(t)}(t^{-2/3} x)$. If for all compact $I\subset \R$ and all $\e>0$,
$$
\lim_{t\to \infty}\PP\bigg(\sup_{x\in I} \big\vert f^{(t)}(x)-\tilde{f}^{(t)}(x)\big\vert>\e\bigg)=0,
$$
then
$$
\frac{\HKPZt{t}{0} -\tildeHKPZt{t}{0}}{t^{1/3}}
$$
converges to zero in probability as $t\to \infty$.
\item For all $\e>0$ there exists a constant $C_2=C_2(\e,C,\delta,\kappa,M)$ such that, for all $y\in \R$, $\eta>0$ and $t\geq 1$,
\begin{equation*}
\PP \left(\frac{\HKPZt{t}{0} +\frac{t}{24}}{t^{1/3}} \in (y,y+\eta)\right) \leq C_2 \eta + \e.
\end{equation*}
\end{enumerate}
\end{theorem}

%\begin{theorem}\label{t.univonethird} %OLD THEOREM
%Fix any $M,C,\delta>0$ and consider a collection of functions $f^{(t)}:\R\to \R\cup\{-\infty\}$ which satisfy hypothesis $\Hyp(M,C,\delta)$ for all $t\geq 1$. Let $\HKPZt{t}{x}= t^{1/3}f^{(t)}(t^{-2/3} x)$ \glossary{$\HKPZt{t}{x}$} represent the solution to the KPZ equation when started from initial data $\Hzerot{t}{x}=t^{1/3} f^{(t)}(t^{-2/3} x)$. Then for all $\e>0$ there exists a constant $C_1=C_1(\e,M,C,\delta)$ such that for all $t\geq 1$,
%\begin{equation}
%\PP \left(\Bigg|\frac{\HKPZt{t}{0} +\tfrac{t}{24}}{t^{1/3}}\Bigg| \leq C_1 \right) > 1 -  \e \, .
%\end{equation}
%Also, for all $\e>0$ there exists a constant $C_2=C_2(\e,M,C,\delta)$ \note{check the dependencies} such that for all $x\in \R$, $\eta>0$ and $t\geq 1$,
%\begin{equation}
%\PP \left(\frac{\HKPZt{t}{0} +\tfrac{t}{24}}{t^{1/3}} \in (x,x+\eta)\right) \leq C_2 \eta + \e.
%\end{equation}
%\end{theorem}

The first part of this theorem shows that the random variable $\tfrac{\HKPZt{t}{0} +\frac{t}{24}}{t^{1/3}}$ is tight as $t$ grows whereas the third part of the theorem shows that it does not go to zero in probability (in fact, it almost shows that the random variable has a density in this limit). The second part demonstrates how if initial data is close in the correct scale, then solutions will also be close in that scale.

This theorem is proved in Section \ref{s.univonethird} as an application of Theorem \ref{mainthm} and the input of information about the narrow wedge initial data KPZ equation one-point distribution recorded in Proposition \ref{ACQprop}.

\begin{corollary}
Consider the following five cases of KPZ initial data:
\begin{enumerate}
\item {\it Flat}: $\Hzero{x} = 0$ for all $x\in \R$;
\item {\it Stationary}: $\Hzero{x} = B(x)$ with $B(x)$ a two-sided Brownian motion with $B(0)=0$;
\item {\it Half flat / half stationary}:  $\Hzero{x} = 0$ for all $x>0$ and $B(x)$ for all $x\leq 0$ (with $B$ a one-sided Brownian motion);
\item {\it Half flat / half narrow wedge}: $\Hzero{x} = 0$ for all $x>0$ and $-\infty$ for all $x\leq 0$;
\item {\it Half stationary / half narrow wedge}: $\Hzero{x} = B(x)$ for all $x>0$ and $-\infty$ for all $x\leq 0$ (with $B$ a one-sided Brownian motion);
\end{enumerate}
In each of these cases, for all $\e>0$ there exists a constant $C_1$ such that for all $t\geq 1$,
\begin{equation*}
\PP \Bigg(\bigg|\frac{\HKPZ{t}{0} +\frac{t}{24}}{t^{1/3}}\bigg| \leq C_1 \Bigg) > 1 -  \e \, ,
\end{equation*}
and another constant $C_2$ such that for all $y\in \R$, $\eta>0$ and $t\geq 1$,
\begin{equation*}
\PP \Bigg(\frac{\HKPZ{t}{0} +\frac{t}{24}}{t^{1/3}} \in (y,y+\eta)\Bigg) \leq C_2 \eta + \e.
\end{equation*}
\end{corollary}
\begin{proof}
For flat as well as half flat / half narrow wedge initial data this follows immediately by applying Theorem \ref{t.univonethird}. When the initial data involves Brownian motion, it will not always satisfy $\Hyp(C,\delta,\kappa,M)$; however, by virtue of Lemma \ref{l.wegwew}, for $\kappa,\delta,M>0$ fixed and for any $\e>0$, by taking $C$ large enough we can be sure that the initial data satisfies $\Hyp(C,\delta,\kappa,M)$ with probability at least $1-\e/2$. On this event, we can apply Theorem \ref{t.univonethird} with $\e/2$. Combining these two $\e/2$ terms yields the desired result.
\end{proof}

\begin{rem}
The analogous narrow wedge result is not stated above. This is for two reasons. The first is that this result is, in fact, an input to the proof of Theorem~\ref{t.univonethird}, so to call it a corollary as well would be circular. The result follows from Proposition~\ref{ACQprop}. The second is that, in the manner that Theorem~\ref{t.univonethird} is stated, the result does not immediately apply to purely atomic measure initial data (for the stochastic heat equation). This difficulty should be easily remedied, but we do not pursue such a more general statement here. Let us also note that the choice of studying the fluctuations at $x=0$ is arbitrary. The same result holds for general $x$ as can be proved in the same manner, or as follows by studying a suitably modified initial data.
\end{rem}

%\begin{rem} \note{Consider actually proving this}
%We expect (but do not pursue) that the following result should be provable via relatively minor modifications of our proof of this theorem. Consider  a second function $\tilde{f}^{(t)}:\R\to \R\cup\{-\infty\}$ which also satisfies hypothesis $\Hyp(M,C,\delta)$ for all $t\geq 1$. If $f^{(t)}(\cdot)-\tilde{f}^{(t)}(\cdot)$ converges uniformly to zero on every compact interval as $t\to \infty$ then letting $\tildeHKPZt{t}{x}$ be the KPZ equation solution associated with $\tilde{f}^{(t)}$, the random variable
%$$
%\frac{\HKPZt{t}{0} +\frac{t}{24}}{t^{1/3}} -  \frac{\tildeHKPZt{t}{0} +\frac{t}{24}}{t^{1/3}}
%$$
%converges to zero in probability as $t\to \infty$. This should also be easily provable when $f^{(t)}$ and $\tilde{f}^{(t)}$ are random.
%\end{rem}

\begin{rem}
If $f^{(t)}(x)$ has a limit as $t\to \infty$ then it is conjectured in \cite{CQ2} that the one-point centered and scaled fluctuations considered above should converge to a random variable described via a variational problem involving the Airy process (see Section \ref{s.airy} and also \cite[Section 1.4]{RemQuasReview}). A result in this vein is proved for the totally asymmetric exclusion process (TASEP) in \cite{CLW}.
\end{rem}

We can prove a variant of Theorem \ref{t.univonethird}(2) in which the compact interval $I$ is replaced by all of $\R$ but whose conclusion is valid for the full spatial process rather than for the distribution at just one-point. We state this variant below and provide the proof since it is quite simple and independent of the construction or properties of our KPZ$_t$ line ensembles. The proof relies only on two facts:  (1) the KPZ equation evolution is attractive (meaning that it maintains the height ordering of initial data); and (2) the KPZ equation preserves a global height shift. For TASEP, a similar idea is explained in \cite{BFS} in a remark after the statement of Theorem 2.

\begin{proposition}\label{p.feeefe}
Consider two collections of functions $f^{(t)}, \tilde{f}^{(t)}:\R\to \R\cup\{-\infty\}$ and let $\HKPZt{t}{x}$ and $\tildeHKPZt{t}{x}$ be the solutions to the KPZ equation started from respective initial data $t^{1/3} \tilde{f}^{(t)}(t^{-2/3} x)$ and $t^{1/3} \tilde{f}^{(t)}(t^{-2/3} x)$. If for all $\e>0$,
$$
\sup_{x\in \R} \big\vert f^{(t)}(x)-\tilde{f}^{(t)}(x)\big\vert
$$
converges to zero in probability as $t\to \infty$, then so too will
$$
\sup_{x\in \R} \bigg\vert \frac{\HKPZt{t}{x}-\tildeHKPZt{t}{x}}{t^{1/3}}\bigg\vert.
$$
\end{proposition}
\begin{proof}
Letting
$$M^{(t)} = \sup_{x\in \R} \big\vert f^{(t)}(x)-\tilde{f}^{(t)}(x)\big\vert,$$
if follows that
$$
\frac{\tildeHKPZtzero{0}{\cdot}}{t^{1/3}} - M^{(t)} \leq \frac{\tildeHKPZtzero{0}{\cdot}}{t^{1/3}} \leq \frac{\tildeHKPZtzero{0}{\cdot}}{t^{1/3}} + M^{(t)}.
$$
The KPZ equation is attractive, in the sense that if $\Hzero{x}\leq \Hzerotilde{x}$ for all $x\in \R$, then running the KPZ equation (with the same noise) from each of these initial data results in solutions which likewise satisfy $\HKPZt{t}{x}\leq \tildeHKPZt{t}{x}$ for all $t\geq 0$ and $x\in \R$. This follows from M\"{u}ller's comparison principle \cite[Theorem 3.1]{M} for the stochastic heat equation, and can also be seen as a consequence of the fact that the height function of the weakly asymmetric exclusion process (which is an attractive particle system) converges to the KPZ equation \cite{BG,ACQ,CQ}, or that directed polymers with boundary conditions converge to the KPZ equation \cite{AKQ2,MRQ,Nica}. It is also that case that running the KPZ equation from initial data $\Hzero{\cdot}$ and $\Hzero{\cdot}+M$ for a constant $M$ results in solutions $\HKPZ{t}{\cdot}$ and $\HKPZ{t}{\cdot} + M$. From these two facts it follows immediately that
$$
y\mapsto \frac{\tildeHKPZt{t}{y}}{t^{1/3}} - M^{(t)} \leq \frac{\HKPZt{t}{y}}{t^{1/3}} \leq \frac{\tildeHKPZt{t}{y}}{t^{1/3}} + M^{(t)}.
$$
Since $M^{(t)}$ converges to zero in probability, this implies the result of the proposition. In other words, running the KPZ equation will not increase the supremum norm of the difference between two choices of initial data.
\end{proof}
\begin{rem}
Assume that there is a constant $M_t$ such that $\sup_{x\in \R} \big\vert f^{(t)}(x)-\tilde{f}^{(t)}(x) \big\vert \leq M_t$ and $t^{-1/3}M_t \to 0$ as $t\to \infty$. Then by the same reasoning as in the proof of Proposition \ref{p.feeefe} we find that (assuming the finiteness of these moments)
$$
t^{-2/3} \Big(\var\big(\HKPZt{t}{x}\big) - \var\big(\tildeHKPZt{t}{x}\big)\Big) \to 0
$$
as $t\to \infty$. The only case for which the $t^{2/3}$ scaling for the variance of the solution to the KPZ equation is known is for stationary initial data \cite{BQS}. This observation then implies that for any bounded perturbation of the stationary initial data, the variance also scales like $t^{2/3}$ (and moreover that the difference in the two scaled variances goes to zero). This provides an alternative proof of Theorem 1.6 from the recently posted paper \cite{Greg} of Moreno Flores-Sepp\"{a}l\"{a}inen-Valko.
\end{rem}

\subsection{Order $t^{2/3}$ fluctuations for the continuum directed random polymer endpoint}\label{transec}
One way of interpreting the $t^{2/3}$ prediction of Kardar-Parisi-Zhang is to  predict that the endpoint of the continuum directed random polymer is non-trivially concentrated on the scale $t^{2/3}$.\footnote{In fact, it is predicted that for almost every realization of $\whitenoise$, $\PPendpoint$ should concentrate as $t\to \infty$ on a single point (which should correspond to the limiting argmax of $\HKPZFPlinet{t}(x)$). Varying $\whitenoise$, this localization point should likewise vary in the $t^{2/3}$ scale.} We prove this here.

\begin{theorem}\label{app2}
Let $X$ be a random variable distributed according to the continuum directed random polymer endpoint distribution $\PPendpoint$ given in (\ref{PPendpointdef}) and recall that $\PP$ represents the probability measure associated to $\whitenoise$.
\begin{enumerate}
\item {\it Localization in the scale $t^{2/3}$}: for all $\e>0$, there exists $C>0$ such that, for all $t\geq 1$,
\begin{equation*}
\PP\Bigg(\PPendpoint\bigg(\frac{|X|}{t^{2/3}} <C \bigg) \geq 1-\e\Bigg) \geq 1-\e.
\end{equation*}
\item {\it Delocalization in the scale $t^{2/3}$}:  for all $\e>0$ and $x\in \R$, there exists $h>0$ such that, for all $t\geq 1$,
\begin{equation*}
\PP\Bigg(\PPendpoint\bigg(\bigg|\frac{X}{t^{2/3}} -x\bigg| >h \bigg) \geq 1-\e\Bigg) \geq 1-\e.
\end{equation*}
\end{enumerate}
\end{theorem}

This theorem is proved in Section \ref{s.proofapp2} as an application of Theorem \ref{mainthm} and the input of information about the narrow wedge initial data KPZ equation one-point distribution recorded in Proposition \ref{ACQprop}.

%\begin{rem}
%Though this theorem shows that the  $t^{2/3}$ scaled polymer endpoint does not concentrate in the vicinity of the origin, the proof is easily modified to show the same result with $\frac{|X|}{t^{2/3}}$ replaced more generally with $|t^{-2/3} X - x|$ for any $x\in \R$.
%\end{rem}

\begin{rem}
There is a {\em bona fide} continuum directed polymer whose endpoint is distributed according to $\PPendpoint$ \cite{AKQ3}. The above result is the first rigorous demonstration of the $2/3$ transversal exponent for the continuum directed random polymer (or KPZ equation) itself. There are a few other models in the KPZ universality class for which this exponent has been previously demonstrated. For Poissonian last passage percolation, Johansson \cite{KJtransversal} proved the analog of the $t^{2/3}$  scaling prediction by utilizing estimates coming from Riemann-Hilbert asymptotic analysis of the exact formulas available for that model (see further recent developments of this in \cite[Theorem 2.1]{FerNej}). Johansson's work in concert with more recent work of \cite{BFP} enables on to demonstrate a similar result for exponential random variable last passage percolation. More recently, Sepp\"{a}l\"{a}inen \cite{S} proved this scaling exponent for a discrete directed polymer with log-gamma distributed weights with special ``stationary'' boundary weights; (without the boundary weights, \cite{S} provides a $t^{2/3}$ upper bound as in Theorem \ref{app2}(1), but loses the corresponding lower bound as in Theorem \ref{app2}(2)). Sepp\"{a}l\"{a}inen-
Valko \cite{SeppValko} have similar results for the O'Connell-Yor semi-discrete Brownian polymer.

It is expected that, as $t\to \infty$, the endpoint distribution $\PPendpoint$ should converge to a universal limit which has been determined through the analysis of some of the other models in the KPZ universality class (see the review \cite{RemQuasReview} and references therein).
%It is expected (see the review \cite{RemQuasReview}) that as $t\to \infty$, the endpoint distribution $\PPendpoint$ should converge to a universal limit. This limit has been computed exactly for a certain discrete zero temperature directed polymer model (see \cite[Section 4.2]{RemQuasReview} and references therein). While we can not prove convergence, the methods used in our proof of Theorem \ref{app2} should (in conjunction with the tail bounds of Theorem \ref{ACQprop}) provide exponential tail bounds on the endpoint distribution. These bounds, however, will need to be uniform in $t$ (as those of Theorem \ref{ACQprop} are only uniform for the upper tail). We do not pursue these bounds here, but remark that the method is similar to that used in \cite[Proposition 4.4]{CH} to establish tail bounds on the (conjecturally) universal limiting endpoint distribution.
\end{rem}

\subsection{Tail bounds for the general initial data KPZ equation}\label{s.tailbounds}
Our final application is to prove exponential upper and lower tail bounds for the general initial data KPZ equation. As an input we use known results summarized in Proposition \ref{ACQprop} for the tails of the narrow wedge initial data KPZ equation which are due to Moreno Flores (lower tail) and Corwin-Quastel (upper tail). To our knowledge, the only other previously known general initial data tail bound for the KPZ equation is for uniformly bounded initial data (i.e. bounded everywhere by a constant), where \cite{Davar} proves an upper tail exponential bound (by studying high moments of the stochastic heat equation). The next result applies to that setting and also provides an exponential lower tail bound.

\begin{theorem}\label{t.tailbounds}
Fix $t\geq 1$ and $C,\delta,\kappa,M>0$ so that $\kappa>1-t^{-1}$ (i.e., so that $f$ satisfying $\Hyp(C,\delta,\kappa,M)$ will be such that $f(x)<C+x^2/2t$). Then there exist constants $c_1=c_1(t,C,\delta,\kappa,M)>0$, $c_2=c_2(t,C,\delta,\kappa,M)>0$, $c_3 = c_3(t,C,\delta,M)>0$ and $c_3 = c_3(t,C,\delta,M)>0$ such that, for all $s_0\geq 0$ and $s\geq s_0+1$, if $\Hzero{x}$ satisfies hypothesis $\Hyp(C+t^{1/3} s_0 ,\delta,\kappa,M)$ then the time $t$ KPZ equation solution with initial data $\Hzero{x}$ satisfies
$$
\PP\big( \HKPZ{t}{0} <-c_3 s\big) \leq c_1 e^{-c_2 (s- s_0)^2},
\qquad \textrm{and} \qquad
\PP\big( \HKPZ{t}{0} >c_4 + s\big) \leq c_1 e^{-c_2 (s-s_0)}.
$$
\end{theorem}
This theorem is proved in Section \ref{s.tailboundsproof} as an application of Theorem \ref{mainthm} and the input of information about the narrow wedge initial data KPZ equation one-point distribution recorded in Proposition \ref{ACQprop}. The reason for inclusion of the $s_0$ in the statement of Theorem \ref{t.tailbounds} is due to our desire to prove the next corollary which shows how we can extend the tail bounds to random initial data (here one particular choice utilized in \cite{BCFV} is addressed).

\begin{corollary}\label{forBCFV}
Consider the KPZ equation with initial data $\Hzero{x} = B(x) - \beta x \mathbf{1}_{x\leq 0} + b x \mathbf{1}_{x\geq 0}$ where $b,\beta\in \R$, and $B(\cdot)$ is a two-sided Brownian motion with $B(0)=0$. Then for all $t>0$ there exist constants $c'_1=c'_1(t,b,\beta)$, $c'_2=c'_2(t,b,\beta)$, $c'_3=c'_3(t,b,\beta)$ and $c'_4=c'_4(t,b,\beta)$ such that, for all $s\geq 1$,
\begin{equation*}
\PP\big( \HKPZ{t}{0} <-c'_3s\big) \leq c'_1 e^{-c'_2 s^{3/2}}
\qquad\textrm{and}\qquad
\PP\big( \HKPZ{t}{0} >c'_4 +  s\big) \leq c'_1 e^{-c'_2 s}.
\end{equation*}
\end{corollary}
\begin{proof}
We may apply Lemma \ref{l.wegwew} to prove that for any $b,\beta\in \R$ there exist constants $C,\delta,\kappa,M>0$ (with $\kappa>1-t^{-1}$) and $\tilde{c}_1,\tilde{c}_2>0$ such that, for all $s\geq 1$,
$\Hzero{x}$ satisfies $\Hyp(C+ t^{1/3} s/2 ,\delta,\kappa,M)$ with probability at least $\tilde{c}_1 e^{-\tilde{c_2} s^{3/2}}$. When this occurs we may apply Theorem \ref{t.tailbounds} with $s_0=s/2$. Call the constants returned by Theorem~\ref{t.tailbounds} $c_1=c_1(t,C,\delta,\kappa,M)>0$, $c_2=c_2(t,C,\delta,\kappa,M)>0$, $c_3=c_3(t,C,\delta,\kappa,M)>0$ and $c_4=c_4(t,C,\delta,\kappa,M)>0$. This implies that
$$
\PP\big( \HKPZ{t}{0} <-c_3 s\big) \leq c_1 e^{-c_2(s-s/2)^2} +  \tilde{c}_1 e^{-\tilde{c_2} s^{3/2}},\quad \textrm{and}\quad \PP\big( \HKPZ{t}{0} >c_4+ s\big) \leq c_1 e^{-c_2 (s-s/2)}+ \tilde{c}_1 e^{-\tilde{c_2} s^{3/2}}.
$$
These bounds clearly translate into those claimed by the corollary by choosing $c'_1$ and $c'_2$ accordingly along with setting  $c'_3=c_3$ and $c'_4=c_4$.
\end{proof}

\begin{rem}
The results of Theorem \ref{t.tailbounds} and Corollary \ref{forBCFV} apply for any fixed time $t\geq 1$ (Note that the lower bound of $1$ could be replaced as in Remark \ref{r.abrbdsd} by any fixed $t_0>0$). The reason why these results are stated for fixed $t$, as opposed to being stated as unform results as $t\geq 1$ varies, is that the narrow wedge initial data input given in Proposition \ref{ACQprop}(3) is only stated (and known) for fixed time $t$. If a uniform in $t$ result were known instead, then this would translate into uniform in $t$ results for  Theorem \ref{t.tailbounds} and Corollary \ref{forBCFV}. In fact, Proposition \ref{ACQprop}(2) gives an exponential upper tail bound for the scaled narrow wedge initial data KPZ equation, which is valid uniformly over $t\geq 1$. As such, one might hope to parlay that into a similar upper bound for general initial data. However, the proof which we have found for the upper tail bound relies on a lower tail bound. Therefore, in the absence of a similarly uniform (over $t\geq 1$) lower tail bound for the scaled KPZ equation, we are not able to provide uniform upper or lower tail bounds for general initial data.
\end{rem}
\subsection{KPZ line ensemble construction}\label{s.kpzlovers}

Theorems \ref{abscontThm}, \ref{t.univonethird}, \ref{app2} and \ref{t.tailbounds} arise here as applications of the main technical result of this paper which we now recount. The detailed statement of this result is given as Theorem \ref{mainthm}, and the relevant concepts (such as line ensembles and the $\Ham$-Brownian Gibbs property) are defined in Section \ref{s.lineensemblesandh}. After we have described this result, we will discuss the three inputs and the three tools used in proving it, as well as sketch the construction. An overview of all of this is illustrated in Figure \ref{seqncmptfig}.

\begin{figure}
\centering\epsfig{file=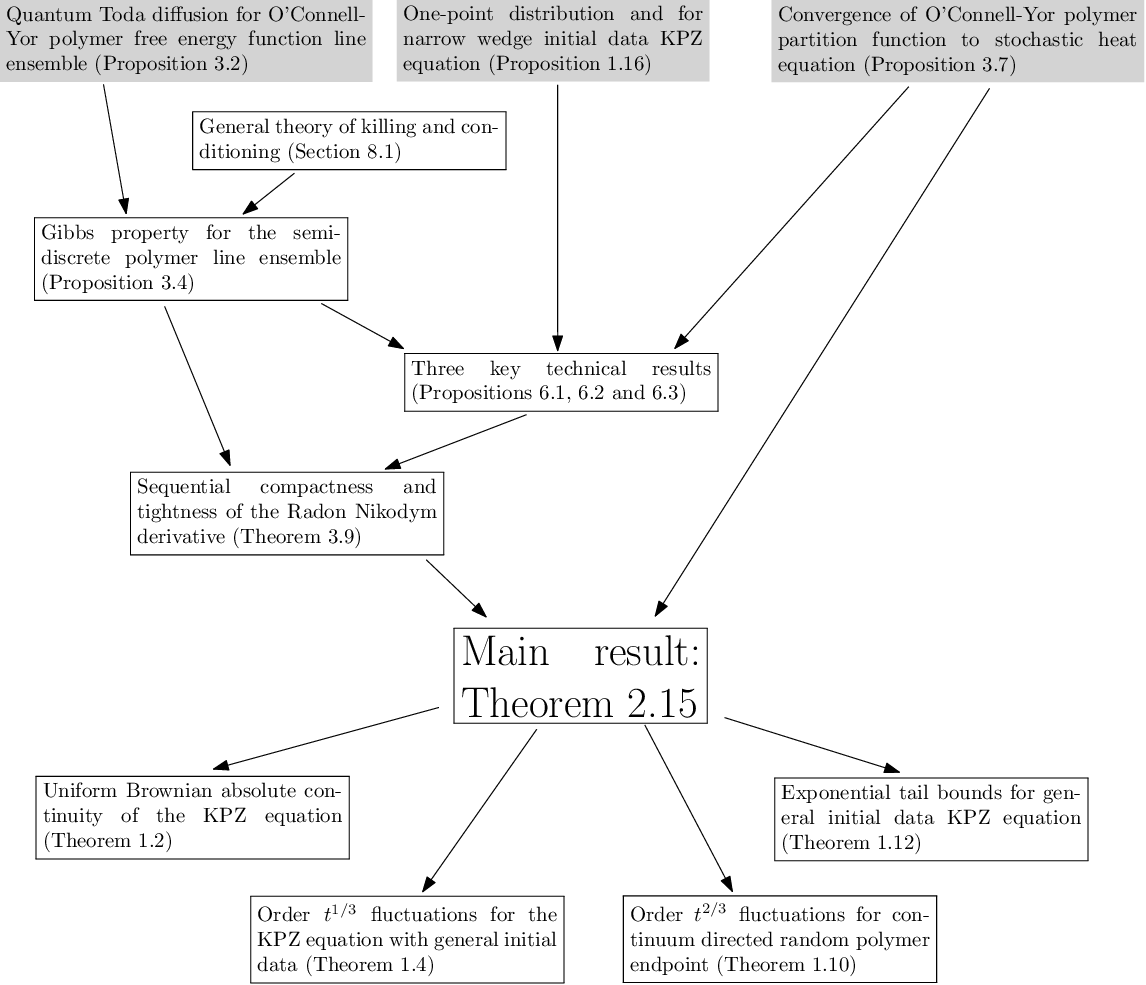, width=16cm}
\caption{Flowchart of this paper. There are three primary inputs towards constructing $\HKPZFPlinet{t}$ in Theorem \ref{mainthm} which are in grey boxes on the top of the chart. Theorem \ref{mainthm} has four applications at the bottom of the chart. These applications also use the inputs of Proposition \ref{ACQprop} though we have left the corresponding arrows out of the figure to avoid it becoming too cluttered.}\label{seqncmptfig}
\end{figure}

The main technical result of this paper (Theorem \ref{mainthm}) is the construction, for all $t\geq 1$, of a family of  $\N$-indexed ensembles $\HKPZlinet{t}:= \big\{\HKPZline{t}{n}:n\in \N\big\}$ of continuous curves $ \HKPZline{t}{n}:\R\to\R$ which has three properties:
\begin{enumerate}
\item The lowest indexed curve $\HKPZline{t}{1}(\cdot)$ is equal in distribution to the time $t$ solution to the narrow wedge initial data KPZ equation $\HKPZ{t}{\cdot}$ (Definition \ref{SHEdef});
\item The ensemble $\HKPZlinet{t}$ has a certain resampling invariance called the $\Ham_1$-Brownian Gibbs property, with $\Ham_1(x) := e^{x}$ (Definition \ref{maindefHBGP}). This is the distributional invariance of $\HKPZlinet{t}$ with respect to resampling increments of curves according to a reweighted Brownian bridge measure which probabilistically rewards configurations in which  lower indexed curves generally stay above the higher indexed curves, and penalizes violations of this order.
%This property (Definition \ref{maindefHBGP}) asserts that for any $k \in \N$ and any interval $(a,b)\subset \R$, the marginal distribution of the curves $\HKPZline{t}{1}(\cdot),\ldots,\HKPZline{t}{k}(\cdot)$ on the interval $(a,b)$ is given according to a reweighting of the law of $k$ independent Brownian bridges which start at the points $\HKPZline{t}{1}(a),\ldots,\HKPZline{t}{k}(a)$ and end at the points $\HKPZline{t}{1}(b),\ldots,\HKPZline{t}{k}(b)$ over the interval $(a,b)$. The reweighting is by the Radon-Nikodym derivative
%    \begin{equation}\label{abovekpz}
%    \frac{1}{Z}\exp\left\{ -\sum_{i=0}^{k} \int_a^b \Ham\big(\HKPZline{t}{i+1}(s)-\HKPZline{t}{i}(s)\big) ds \right\}
%    \end{equation}
%    with the convention that for $i=0$, $\HKPZline{t}{i}\equiv +\infty$. In the above, $Z$ is the normalization necessary so that the expectation of the Radon-Nikodym derivative with respect to the probability measure on $k$ independent Brownian bridges equals 1.
\item Define $\HKPZFPlinet{t}=\big\{\HKPZFPline{t}{n}:n\in \N\big\}$ by the relation
$$
\HKPZline{t}{n}(x) =  -\frac{t}{24} +  t^{1/3} \HKPZFPline{t}{n}(t^{-2/3} x).
$$
This ensemble has the $\Ham_t$-Brownian Gibbs property with $\Ham_t(x) =e^{t^{1/3} x}$. Moreover, the law of $\HKPZline{t}{1}(x)$ on any fixed interval $(a,b)$, has a Radon-Nikodym derivative with respect to Brownian bridge on that interval (with the same starting and ending heights) which is tight as $t$ varies in $[1,\infty)$.
\end{enumerate}

\begin{figure}
\centering\epsfig{file=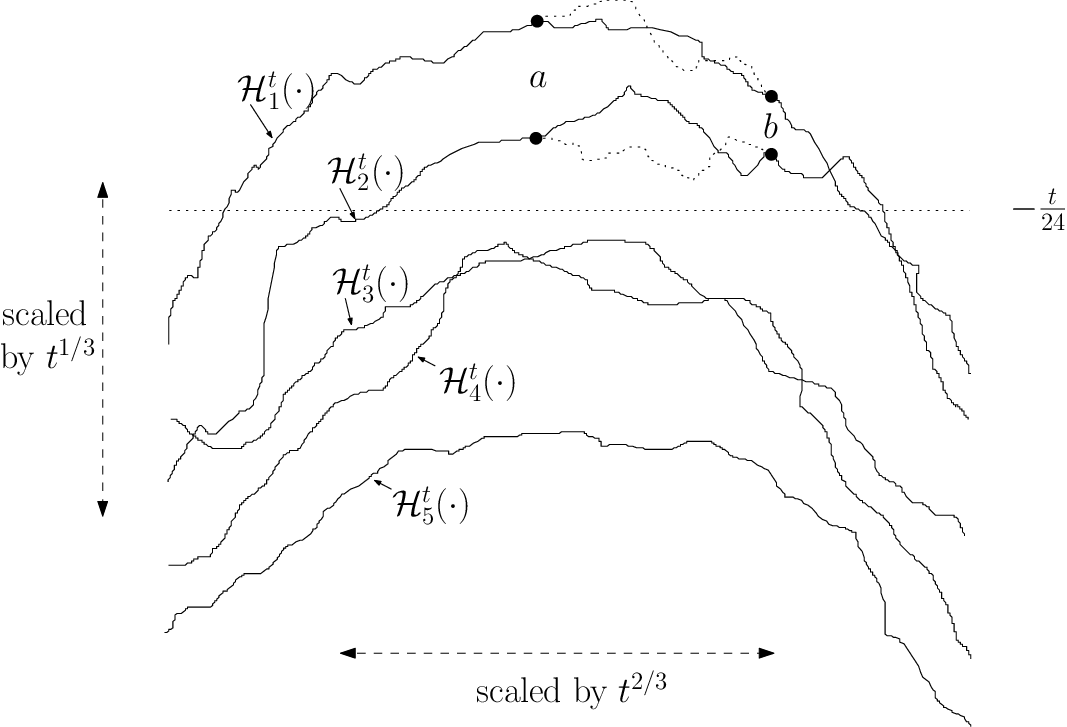, width=14cm}
\caption{Overview of a KPZ$_t$ line ensemble $\HKPZlinet{t}$. Curves $\HKPZline{t}{1}(\cdot)$ through $\HKPZline{t}{5}(\cdot)$ are drawn. The lowest indexed curve $\HKPZline{t}{1}(\cdot)$ is distributed according to the time $t$ solution to the narrow wedge initial data KPZ equation $\HKPZ{t}{\cdot}$. The $\Ham_t$-Brownian Gibbs property is demonstrated by showing a possible resampling (dotted curves) of the curves $\HKPZline{t}{1}(\cdot)$ and $\HKPZline{t}{2}(\cdot)$ between times $a$ and $b$. The lowest indexed curve scaled down vertically by $t^{1/3}$ and horizontally by $t^{2/3}$ has a Radon-Nikodym derivative on any interval $(a,b)$ which is tight for $t\geq1$. The scaled KPZ$_t$ line ensemble $\HKPZFPlinet{t}$ corresponds to this $t^{1/3},t^{2/3}$ scaling (as well as centering at $-\tfrac{t}{24}$).}\label{f.overview}
\end{figure}

We will call any such measure $\HKPZlinet{t}$ satisfying the first two properties a {\it KPZ$_t$ line ensemble} and its scaled version $\HKPZFPlinet{t}$ a {\it scaled KPZ$_t$ line ensemble} (see Figure \ref{f.overview} for an illustration of this line ensemble). Theorem \ref{mainthm} proves the existence of a family of such measures for all $t\geq 1$ which additionally displays the uniformity (of the third property above) as $t\in [1,\infty)$ varies. The uniqueness of such measures is an open problem (our proof is through subsequence extraction). In Section \ref{s.oconwar} we touch on this question as well as a conjectural relationship with O'Connell-Warren's multilayer extension of the solution to the stochastic heat equation with narrow wedge initial data~\cite{OConWar}.

The four applications of Theorem \ref{mainthm} presented earlier do not rely upon uniqueness. The existence of this Gibbsian line ensemble extension of the time $t$ narrow wedge initial data KPZ equation and the uniformity of the Radon-Nikodym derivative for its scaled version provide very potent tools for establishing regularity from one-point information and for ruling out exceptional events. This is illustrated in the proofs of our four application theorems. For instance, for Theorems \ref{t.univonethird} and \ref{t.tailbounds} which deal with general initial data, we rely on the fact that the stochastic heat equation is linear in its initial data and that the narrow wedge initial data corresponds to the fundamental solution to the stochastic heat equation. These observations enable us to represent the one-point distribution for the general initial data KPZ equation in terms of the (fixed time) spatial process for the narrow wedge initial data (Lemma \ref{l.genindata}). Then, given the regularity
established from the identification of this narrow wedge spatial initial data process with the lowest indexed curve of a KPZ$_t$ line ensemble, we are able to transfer our one-point information about the narrow wedge initial data solution to the setting of general initial data. The uniformity as $t\to \infty$ described as the third property of $\HKPZFPlinet{t}$ is key to our ability to study the $t\to\infty$ limit in these applications.

The existence of a KPZ$_t$ line ensemble is not at all obvious. There were two primary motivations for our present investigation. The first came from work of O'Connell \cite{OCon} (Section \ref{OConYor}) and O'Connell-Warren \cite{OConWar} (Section \ref{s.oconwar}) while the second came from our earlier work \cite{CH} on non-intersecting Gibbsian line ensembles related to Dyson Brownian motion and the Airy line ensemble (Section \ref{s.airy}). Even given a KPZ$_t$ line ensemble (satisfying the first two properties outlined above) the uniformity we show as property three of $\HKPZFPlinet{t}$ is itself very far from evident.

There are three inputs involved in our construction which we briefly summarize here:
\begin{enumerate}
\item O'Connell's \cite{OCon} proof  of the existence of an $\{1,\ldots, N\}$-indexed line ensemble whose lowest indexed curve is distributed as the free energy function for the O'Connell-Yor semi-discrete Brownian polymer and whose measure enjoys the $\Ham_1$-Brownian Gibbs property (as a consequence of the description of the ensemble as a diffusion with generator given by a Doob-$h$ transform of the quantum Toda lattice Hamiltonian). This is described in Sections \ref{OConYor} and~\ref{QTLHgibbs}.
\item Moreno Flores-Quastel-Remenik's \cite{MRQ} result (which is a special case of the results of [Nica])  that under suitable scaling, the lowest indexed curve of O'Connell's line ensemble converges to the solution of the narrow wedge initial data KPZ equation. This is described in Section \ref{s.MRQ}.
\item Amir-Corwin-Quastel's \cite{ACQ} proof of the exact formula for the one-point distribution of the narrow wedge initial data KPZ equation, and its large $t$ asymptotics. This is described in Section \ref{s.onepointdist}.
\end{enumerate}

In addition to these three inputs, there are three main tools which we employ in our construction and also in the proof of the applications:
\begin{enumerate}
\item Resampling arguments using the $\Ham$-Brownian Gibbs property (Definition \ref{maindefHBGP}) and the {\it strong} $\Ham$-Brownian Gibbs property (Lemma \ref{s.stronghbn}) enable us to establish regularity from one-point information, rule out exceptional behavior at random points and reduce problems concerning line ensemble behavior to simpler ones involving Brownian bridges.
\item Monotone coupling of line ensembles displaying the $\Ham$-Brownian Gibbs property (Lemmas \ref{monotonicity1} and \ref{monotonicity2}) reduces the complexity of calculations so as to enable us to prove explicit bounds on the probability of complicated events.
\item Explicit Brownian bounds (Section \ref{s.explicitfacts}) give us the means to solve explicitly the computations provided by the above two tools.
\end{enumerate}

The construction provided herein to prove Theorem \ref{mainthm} is achieved through sequential compactness of the sequence (labeled by $N\in \N$) of O'Connell's $\{1,\ldots,N\}$-indexed line ensembles. This ensemble is defined in Section \ref{OConYor} and denoted by $\HSDFPlinet{t}{N} := \big\{\HSDFPline{t}{N}{n}:1\leq n\leq N \big\}$. This compactness is achieved in Theorem \ref{seqncmpt}(1). Its proof relies only on inputs (1) and (2) (in fact, for input (2), it relies only on the one-point convergence to the KPZ equation), and the three tools highlighted above. At the heart of this proof are three key technical propositions (Section \ref{threekey}) which establish inductively the regularity of the curve $\HSDFPline{t}{N}{1}$ then $\HSDFPline{t}{N}{2}$ and so on. The proof of these propositions (Section \ref{s.threekeyproofs}) constitutes a major portion of this paper. In a more general sense, the proof of these three results develops a means to transfer the input of one-point convergence by means of a
Gibbs
property into the
output of compactness.

The sequential compactness and convergence of the lowest indexed curve given by input (2) suffices to show the existence of a line ensemble with the first two properties listed. In order to establish tightness of the Radon-Nikodym derivative for $t\geq 1$, we make use of the tightness of the one-point distribution for the narrow wedge initial data KPZ equation which is afforded by input~(3). This, in conjunction with the three key propositions (Section \ref{s.threekeyproofs}) and the three tools explained above, suffices to prove the tightness (see also Theorem \ref{seqncmpt}(2)).

A relatively simple coupling argument shows that subsequential limits of the ensemble $\HSDFPlinet{t}{N}$ enjoy the Gibbs property and that their Radon-Nikodym derivative is tight (see Theorem \ref{seqncmpt}). By input~(2), any subsequential limit will necessarily have a lowest indexed curve distributed as the solution at time $t$ of the narrow wedge initial data KPZ equation; thus, the construction of our KPZ$_t$ line ensemble is accomplished, and Theorem \ref{mainthm} is proved. This overview is illustrated in Figure \ref{seqncmptfig}; the proof of Theorem \ref{mainthm} is given in Section \ref{s.mainthmproof}.

\subsection{One-point distribution for the narrow wedge initial data KPZ equation}\label{s.onepointdist}

We record here the inputs which exist in the literature (and to which we appeal) pertaining to the one-point distribution for the narrow wedge initial data KPZ equation. An exact formula for the one-point marginal distribution for the narrow wedge initial data KPZ equation was computed rigorously in \cite{ACQ}. We quote part of this theorem here. The definition of Fredholm determinant (and more background) can be found in \cite{ACQ}; $\Ai(\cdot)$ is the classical Airy function.
\begin{theorem}[Theorem 1.1 of \cite{ACQ}]\label{t.acqtheorem}
For all $t>0$, the one-point distribution of the solution at time~$t$ of the narrow wedge initial data KPZ equation is given by
$$
\PP\Big(\frac{\HKPZnw{t}{0}+\frac{t}{24}}{t^{1/3}}\leq s\Big) = \PP\Big(\HKPZFPlinetnw{t}(0)\leq s\Big) = \frac{1}{2\pi i}\int_{\subset} \frac{d\mu}{\mu} e^{-\mu} \det\big(I- K_{\mu}\big)_{L^2(2^{1/3}s,\infty)},
$$
where the infinite contour $\subset$ encloses $\R_{\geq}$ and is positively oriented, and the kernel $K_{\mu}$ is given by
$$
K_{\mu}(x,y) = \int_{-\infty}^{\infty} \frac{\mu}{\mu - e^{-2^{-1/3}t^{1/3} r}}\, \Ai(x+r)\, \Ai(y+r) \dd r.
$$
\end{theorem}

Even though this is an exact formula, it is not trivial to extract meaningful asymptotics from it. In the next proposition, the first two results as well as the exponential upper tail bound in the third result arise from asymptotics of this formula. The exponential lower tail bound in the third result is not presently accessible by these means and comes from work of  Moreno Flores \cite{MorenoFlores}.

\begin{proposition}\label{ACQprop}
The following properties hold true for the narrow wedge initial data KPZ equation $\HKPZnw{t}{x}$ as well as its scaled version $\HKPZFPlinetnw{t}(x)$ defined in (\ref{crossairy2}).
\begin{enumerate}
\item {\it Stationarity and tightness}: the one-point distribution of $\HKPZFPlinetnw{t}(x)+\frac{x^2}{2}$ is independent of $x$ and is tight for $t\geq 1$.
\item {\it Uniform in $t\geq 1$ exponential upper tail bound}: there exist constants $c_1,c_2>0$ such that, for all $t\geq 1$ and $s\geq 1$,
$$
\PP\big(\HKPZFPlinetnw{t}(0)\geq s\big) \leq c_1 e^{-c_2 s}.
$$
\item {\it Exponential lower and upper tail bound for each $t$}: for each $t>0$, there exist constants $c_1=c_1(t)>0$ and $c_2=c_2(t)>0$ such that, for all $s\geq 1$,
$$
\PP\big(\HKPZnw{t}{0} <-s\big) \leq c_1e^{-c_2s^2}
\qquad \textrm{and} \qquad
\PP\big( \HKPZnw{t}{0} >s\big) \leq c_1 e^{-c_2 s}.
$$
\end{enumerate}
\end{proposition}

\begin{proof}
The stationarity in (1) is given as \cite[Proposition 1.4]{ACQ} and the tightness is an immediate consequence of \cite[Corollary 1.3]{ACQ} which further proves that, as $t\to \infty$, the distribution converges to $F_{{\rm GUE}}(2^{1/3} s)$, the GUE Tracy-Widom distribution\footnote{It is also conjectured therein that the spatial process converges to the Airy process (for which the GUE Tracy-Widom distribution is a one-point marginal). For more on this conjecture, see the discussion of Section \ref{s.airy}.}. A proof of (2) is given in Section~\ref{s.tail} via asymptotics of Theorem \ref{t.acqtheorem}. A very similar computation was made in \cite{CQ} for the solution of the KPZ equation corresponding to $\ZSHE{0}{x} = {\bf 1}_{x\geq 0} e^{B(x)}$. The narrow wedge case treated in Section~\ref{s.tail} is considerably easier. Finally, consider (3). This result differs in that it does not deal with the ($1/3,2/3$)-scaling and it does not provide results which are
uniform over $t\in [1,\infty)$. The upper tail bound follows from the same computation of Section \ref{s.tail} by neglecting the uniformity in $t$. The lower tail bound should (in principle) be accessible from the exact formula of Theorem \ref{t.acqtheorem}; however, to date the necessary asymptotics have not been achieved. The stated result was proved by Moreno Flores \cite[Theorem 1]{MorenoFlores} and is a strengthening of earlier work of M\"{u}ller and Nualart \cite{MN}. It has been speculated (but not proved) that the arguments in the exponential are actually  $s^3$ and $s^2$, as is the case for $F_{\textrm{GUE}}$, and that these arguments may arise in a sense which is uniform in $t$ for the scaled solution $\HKPZFPlinetnw{t}(0)$.
\end{proof}

The following result shows how the one-point distribution for the solution of the general initial data KPZ equation is related to the spatial trajectory of $\HKPZnw{t}{\cdot}$. The result will be important in the proof of Theorems~\ref{t.univonethird} and~\ref{t.tailbounds}.

%\begin{proposition}\label{ACQprop} \note{revisit references}
%The process $\HKPZFPlinetnw{t}(x)+\frac{x^2}{2}$ is stationary in $x$ and its one-point distribution is tight for $t\geq 1$.
%\end{proposition}
%\begin{proof}
%This proposition follows immediately from the work of \cite{ACQ}\footnote{Note also the (non-rigorous) work of \cite{SaSp,Dot,CDR} in which this one-point distribution was discovered independently and in parallel to \cite{ACQ}.} which computes an exact formula for the one-point distribution $\PP\big(\HKPZnw{t}{x} +\frac{x^2}{2} \leq s\big)$ and proves that as $t\to \infty$ the distribution converges to $F_{{\rm GUE}}(2^{1/3} s)$, the GUE Tracy-Widom distribution\footnote{It is also conjectured therein that the spatial process should converge to the Airy process (for which the GUE Tracy-Widom distribution is a one-point marginal). For more on this conjecture, see the discussion of Section \ref{s.airy}.}.
%\end{proof}

\begin{lemma}\label{l.genindata}
For general initial data $\Hzeroshort$, the solution to the KPZ equation $\HKPZ{t}{x}$ at a fixed time~$t$ and location~$x$ is equal in distribution\footnote{This representation is valid only for the marginal distribution corresponding to a single pair $(t,x)$.} to
\begin{equation*}
\log \left( \int_{-\infty}^{\infty} e^{\HKPZnw{t}{y} + \Hzero{x-y}}\dd y\right),
\end{equation*}
or (after scaling)
\begin{equation*}
-\frac{t}{24} +\log \left( \int_{-\infty}^{\infty} e^{t^{1/3} \HKPZFPlinetnw{t}\big(t^{-2/3} y\big) + \Hzero{x-y}} \dd y\right).
\end{equation*}
\end{lemma}
\begin{proof}
The stochastic heat equation is linear in the initial data. Let $\ZZSHE^{\textrm{nw},y}(t,x)$ denote the solution to the stochastic heat equation with initial data $\ZZSHE^{\textrm{nw},y}(0,x)=\delta_{x=y}$. For general initial data $\ZZSHE_0$, we may express the solution as
$$\ZSHE{t}{x}= \int_{\R} \ZZSHE^{\textrm{nw},y}(t,x) \ZZSHE_0(y) \dd y.$$
Note that as processes in $y\in \R$ we have that $\ZZSHE^{\textrm{nw},y}(t,x)$ is equal in distribution to $\ZZSHE^{\textrm{nw},y-x}(t,0)$ and likewise equal in distribution to $\ZZSHE^{\textrm{nw},0}(t,x-y)$ and in turn to $\ZSHEnw{t}{x-y}$. Making the change of variables that replaces $x-y$ by $y$, we arrive at the desired result.
\end{proof}

\subsection{Outline}
Section \ref{s.constructingww} contains our main technical result, Theorem \ref{mainthm}. It also contains various definitions necessary to its statement, basic lemmas related to the three tools used in its proof, and a discussion surrounding this result and its connection to some other topics. Section \ref{semidiscete} introduces the O'Connell-Yor semi-discrete Brownian polymer, its associated line ensembles, and records the compactness result, Theorem \ref{seqncmpt}. At this point we turn to proving the results described in these early sections. Section \ref{appsproofs} contains proofs of the applications of Theorem \ref{mainthm} which are highlighted in this introduction. Section \ref{s.mainthmproof} proves Theorem \ref{mainthm} by appealing to Theorem \ref{seqncmpt}. Section \ref{s.seqncmptproof} proves Theorem \ref{seqncmpt} by appealing to three key propositions which are subsequently proved in Section \ref{s.threekeyproofs}. An appendix is provided as Section \ref{s.apend}.

\subsection{Notation}
There is a glossary at the end of this work in which the symbols used in this paper are recorded along with the page on which they are defined. We additionally set some notation here. The natural numbers are defined to be $\N = \{1,2,\ldots\}$ \glossary{$\N$, Natural numbers $\{1,2,\ldots\}$}. Events are denoted in a special font $\mathsf{E}$, their indicator function is written either as $\mathbf{1}_{\mathsf{E}}$ \glossary{$\mathbf{1}_{\mathsf{E}}$, Indicator function for the event $\mathsf{E}$} or $\mathbf{1}\{\mathsf{E}\}$, and their complement is written as $\mathsf{E}^c$. \glossary{$\mathsf{E}^c$, Complement of the event $\mathsf{E}$}  The definition of events may change from proof to proof (though generally will remain constant within a given proof).
The Dirac delta function at $y$ is denoted by $\delta_{x=y}$. When two random variables have the same distribution (or law) we indicate this by using $\stackrel{(d)}{=}$. Constants will be generally denoted by $c$ or $C$ (and possibly adorned with primes,
tildes, or sub/super scripts). These constants may change value between results, though within results we will generally try to keep their use consistent.

\subsection{Acknowledgments}
We extend thanks to Neil O'Connell for valuable discussions about the Gibbs property for the diffusion associated with the quantum Toda lattice Hamiltonian, as well as helpful comments on Proposition \ref{NeilGibbsprop}. We likewise appreciate important discussions with Jeremy Quastel involving his work with Daniel Remenik and with Gregorio Moreno Flores; we benefited from receiving an early draft of the paper \cite{MRQ}. We are also very appreciative to our referees for a very close reading of this paper and many helpful comments. Parts of this article were written during a visit of AH to Microsoft Research New England.  IC was partially supported by the NSF through DMS-1208998, by Microsoft Research through the Schramm Memorial Fellowship, and by the Clay Mathematics Institute through a Clay Research Fellowship. AH was supported principally by EPSRC grant EP/I004378/1.

This 2020 revision corrects a mistake in the proof of Lemma \ref{l.EN} and several typographical errors. The authors are grateful to Xuan Wu for pointing out many of these problems. The revision also notes two minor factors missing previously in Definition \ref{d.cntx} whose presence was revealed by Mihai Nica in his proof \cite{Nica} of Conjecture \ref{OWconjecture}.
\section{Constructing a KPZ$_t$ line ensemble}\label{s.constructingww}

\subsection{Line ensembles and the $\Ham$-Brownian Gibbs property}\label{s.lineensemblesandh}

We now turn to the setup necessary to state our main result, Theorem \ref{mainthm}. We introduce the concept of a line ensemble and the $\Ham$-Brownian Gibbs property. Figure \ref{f.overview} provides an illustration of a portion of an $\N\times \R$-indexed line ensemble with the $\Ham_1$-Brownian Gibbs property. When graphically representing such line ensembles we plot the curves on the same axes.

\begin{definition}\label{maindefline}
Let $\Sigma$ \glossary{$\Sigma$, Set of line ensemble indices (interval of $\Z$)} be an interval of $\Z$, and let $\Lambda$ \glossary{$\Lambda$, Set of line ensemble times (interval of $\R$)} be an interval of $\R$. %The extended reals $\R\cup\{-\infty,+\infty\}$ will be denoted by $\R^*$; we adopt the convention that $+\infty>+\infty$ and $-\infty>-\infty$. We endow $\R^*$ with the usual topology of $\R$ and the discrete topology for $+\infty$ and $-\infty$.
Consider the set $X$ of continuous functions $f:\Sigma\times \Lambda \rightarrow \R$ endowed with the topology of uniform convergence on compact subsets of $\Sigma\times\Lambda$, and let $\mathcal{C}$ denote the sigma-field  generated by Borel sets in $X$.
% generated by the cylinder sets $\{f\in X: f(i,x)\in U\}$, as the parameters range over $i\in \Sigma$, $x\in \Lambda$ and $U \subseteq \R^*$ open.

A {\it $\Sigma\times\Lambda$-indexed line ensemble} $\mathcal{L}$ \glossary{$\mathcal{L}$, A line ensemble} is a random variable defined on a probability space $(\Omega,\mathcal{B},\PP)$, taking values in $X$ such that $\mathcal{L}$ is a $(\mathcal{B},\mathcal{C})$-measurable function. The symbol $\EE$ denotes the expectation operator associated with $\PP$. All statements are to be understood as being almost sure with respect to $\PP$. As in Definition \ref{maindefHBGP} below, we will often decorate $\PP$ and $\EE$ with additional sub- and superscripts so as to specify a particular line ensemble. We will generally reserve $\PP$ and $\EE$ without decorations for ``global'' line ensembles such as $\HSDFPlinet{t}{N}$ or $\HKPZFPlinet{t}$ -- with the noted exception of $\PP$ in the proof of Proposition \ref{p.fiveone} (Section \ref{s.fiveone}).
Intuitively one may think of $\mathcal{L}$ as a collection of random continuous curves (even though we use the word ``line'' we are referring to continuous curves\footnote{The term line ensemble seems to have arisen in the study of discrete models such as the multilayer polynuclear growth (PNG) model for which the curves are piecewise constant (and hence are unions of line segments). See \cite{SpohnLineEnsemble} for example.}) indexed by $\Sigma$, each of which maps $\Lambda$ into $\R$. We will generally abuse notation and write $\mathcal{L}:\Sigma\times \Lambda \rightarrow \R$, even though it is not $\mathcal{L}$ which is such a function, but rather $\mathcal{L}(\omega)$ for each $\omega \in \Omega$. Furthermore, we write $\mathcal{L}_i:=(\mathcal{L}(\omega))(i,\cdot)$ for the curve indexed by $i\in\Sigma$.

Given a $\Sigma\times\Lambda$-indexed line ensemble $\mathcal{L}$ and a sequence of such ensembles $\big\{ \mathcal{L}^N: N \in \N \big\}$,  a natural notion of convergence is the weak-$*$ convergence of the measure $\mathcal{L}^N$ to the measure $\mathcal{L}$; we call this notion {\it weak convergence as a line ensemble}. %and denote it by $\mathcal{L}^N\Rightarrow \mathcal{L}$.
In order words, this means that, for all bounded continuous functionals $f$, $\int d\PP(\omega) f(\mathcal{L}^{N}(\omega)) \to \int d\PP(\omega) f(\mathcal{L}(\omega))$ as $N\to \infty$. We will sometimes also consider sequences of line ensembles for which $\Sigma=\{1,\ldots,N\}$ and $\Lambda = (a_N,b_N)$ with $a_N\to -\infty$ and $b_N\to +\infty$ as $N\to \infty$. We may extend the notation of weak convergence by embedding this sequence into $\N\times \R$-indexed line ensembles by setting $\mathcal{L}_n(x) = -\infty$ for $n>N$ or $x\notin [a_n,b_n]$ with $n\leq N$. Though the curves are no longer continuous, the notation of weak convergence still makes sense if we assume that the functionals extend continuously to $-\infty$.
\end{definition}

We turn now to formulating the $\Ham$-Brownian Gibbs property. As a matter of convention, all Brownian motions and bridges have diffusion parameter one.

\begin{definition}\label{maindefHBGP}
Fix $k_1\leq k_2$ with $k_1,k_2\in \Z$, an interval $(a,b) \subset \R$ and two vectors $\vec{x},\vec{y}\in \R^{k_2-k_1+1}$. A $\{k_1,\ldots,k_2\}\times (a,b)$-indexed line ensemble $\mathcal{L}_{k_1},\ldots, \mathcal{L}_{k_2}$ is called a {\it free Brownian bridge} line ensemble with entrance data $\vec{x}$ and exit data $\vec{y}$ if its law $\Pfree{k_1}{k_2}{(a,b)}{\vec{x}}{\vec{y}}$ \glossary{$\Pfree{k_1}{k_2}{(a,b)}{\vec{x}}{\vec{y}}$, Law of $k_2-k_1+1$ independent Brownian bridges starting at time $a$ at the points $\vec{x}$ and ending at time $b$ at the points $\vec{y}$} is that of $k_2-k_1+1$ independent Brownian bridges starting at time $a$ at the points $\vec{x}$ and ending at time $b$ at the points $\vec{y}$. We write $\PfreeExp{k_1}{k_2}{(a,b)}{\vec{x}}{\vec{y}}$ \glossary{$\PfreeExp{k_1}{k_2}{(a,b)}{\vec{x}}{\vec{y}}$, Expectation operator associated with $\Pfree{k_1}{k_2}{(a,b)}{\vec{x}}{\vec{y}}$} for the associated expectation operator.
When there is no threat of confusion we will sometimes use the shorthand $\PfreeShort$ \glossary{$\PfreeShort$, Shorthand for $\Pfree{k_1}{k_2}{(a,b)}{\vec{x}}{\vec{y}}$} and $\PfreeExpShort$ \glossary{$\PfreeExpShort$, Shorthand for $\PfreeExp{k_1}{k_2}{(a,b)}{\vec{x}}{\vec{y}}$}.
If $k_1,k_2$ is replaced by $k_1+1,k_2+1$, then the measure is unchanged except for a reindexing of the curves.

A {\it Hamiltonian} $\Ham$ \glossary{$\Ham$, Hamiltonian for a $\Ham$-Brownian Gibbs property} is defined to be a continuous function $\Ham:\R\to [0,\infty)$. Throughout, we will make use of the special Hamiltonian
\begin{equation}\label{Hamt}\glossary{$\Ham_t$, Particular Hamiltonian given by $\Ham_t(x) = e^{t^{1/3} x}$}
\Ham_t(x) = e^{t^{1/3} x}.
\end{equation}
Given a Hamiltonian $\Ham$ and two measurable functions $f,g:(a,b)\to \R\cup\{\pm \infty\}$, we define the $\{k_1,\ldots,k_2\}\times (a,b)$-indexed {\it $\Ham$-Brownian bridge} line ensemble with entrance data $\vec{x}$ and exit data $\vec{y}$ and boundary data $(f,g)$ to be the law $\PH{k_1}{k_2}{(a,b)}{\vec{x}}{\vec{y}}{f}{g}{\Ham}$ \glossary{$\PH{k_1}{k_2}{(a,b)}{\vec{x}}{\vec{y}}{f}{g}{\Ham}$, $\{k_1,\ldots,k_2\}\times (a,b)$-indexed {\it $\Ham$-Brownian bridge} line ensemble with entrance data $\vec{x}$, exit data $\vec{y}$ and boundary data $(f,g)$} on $\mathcal{L}_{k_1},\ldots, \mathcal{L}_{k_2}:(a,b)\to \R$ given in terms of the following Radon-Nikodym derivative (with respect to the free Brownian bridge line ensemble $\Pfree{k_1}{k_2}{(a,b)}{\vec{x}}{\vec{y}}$):
\begin{equation*}
\frac{\dd\PH{k_1}{k_2}{(a,b)}{\vec{x}}{\vec{y}}{f}{g}{\Ham}}{\dd\Pfree{k_1}{k_2}{(a,b)}{\vec{x}}{\vec{y}}}\left(\mathcal{L}_{k_1},\ldots, \mathcal{L}_{k_2}\right) =
\frac{\bolt{k_1}{k_2}{(a,b)}{\vec{x}}{\vec{y}}{f}{g}{\Ham}\Big(\mathcal{L}_{k_1},\ldots, \mathcal{L}_{k_2}\Big)}{\partfunc{k_1}{k_2}{(a,b)}{\vec{x}}{\vec{y}}{f}{g}{\Ham}}.
\end{equation*}
Here we call $\mathcal{L}_{k_1-1} = f$, $\mathcal{L}_{k_2+1} = g$ and define the {\it Boltzmann weight}
\begin{equation*}\glossary{$\bolt{k_1}{k_2}{(a,b)}{\vec{x}}{\vec{y}}{f}{g}{\Ham}$, Boltzmann weight in the Radon-Nikodym derivative defining $\PH{k_1}{k_2}{(a,b)}{\vec{x}}{\vec{y}}{f}{g}{\Ham}$}
\bolt{k_1}{k_2}{(a,b)}{\vec{x}}{\vec{y}}{f}{g}{\Ham}\left(\mathcal{L}_{k_1},\ldots, \mathcal{L}_{k_2}\right) := \exp\left\{-\sum_{i=k_1-1}^{k_2}\int_{a}^{b} \Ham\Big(\mathcal{L}_{i+1}(u)-\mathcal{L}_{i}(u)\Big) \dd u \right\},
\end{equation*}
and the {\it normalizing constant}
\begin{equation}\label{partfunceqn}\glossary{$\partfunc{k_1}{k_2}{(a,b)}{\vec{x}}{\vec{y}}{f}{g}{\Ham}$, Normalizing constant in the Radon-Nikodym derivative defining $\PH{k_1}{k_2}{(a,b)}{\vec{x}}{\vec{y}}{f}{g}{\Ham}$}
\partfunc{k_1}{k_2}{(a,b)}{\vec{x}}{\vec{y}}{f}{g}{\Ham}:=\PfreeExp{k_1}{k_2}{(a,b)}{\vec{x}}{\vec{y}}\left[\bolt{k_1}{k_2}{(a,b)}{\vec{x}}{\vec{y}}{f}{g}{\Ham}\left(\mathcal{L}_{k_1},\ldots, \mathcal{L}_{k_2}\right)\right],
\end{equation}
where recall that on the right-hand side in~(\ref{partfunceqn}) the curves $\mathcal{L}_{k_1},\ldots, \mathcal{L}_{k_2}:(a,b)\to \R$ are distributed according to the measure $\Pfree{k_1}{k_2}{(a,b)}{\vec{x}}{\vec{y}}$.
Note that since $\Ham$ takes values in $[0,\infty)$, the normalizing constant takes values in $(0,1]$. Remark \ref{normrem} explains how this normalizing constant can be thought of as an acceptance probability.
%When there is no threat of confusion we will sometimes use the shorthand $\PHShort{\Ham}$, $\boltShort{\Ham}$ and $\partfuncShort{\Ham}$.
The expectation operator associated to $\PH{k_1}{k_2}{(a,b)}{\vec{x}}{\vec{y}}{f}{g}{\Ham}$ will be written as $\PHExp{k_1}{k_2}{(a,b)}{\vec{x}}{\vec{y}}{f}{g}{\Ham}$ \glossary{$\PHExp{k_1}{k_2}{(a,b)}{\vec{x}}{\vec{y}}{f}{g}{\Ham}$, Expectation operator associated to $\PH{k_1}{k_2}{(a,b)}{\vec{x}}{\vec{y}}{f}{g}{\Ham}$}. We will use this notation for the probability and expectation operators throughout the paper, but will not always use $\mathcal{L}$ to denote the line ensemble curves.

We will say that a $\Sigma \times \Lambda$-indexed line ensemble $\mathcal{L}$ has the {\it $\Ham$-Brownian Gibbs property} if for all $K=\{k_1,\ldots, k_2\}\subset \Sigma$ and $(a,b)\subset \Lambda$, the following distributional invariance holds:
\begin{equation*}
\textrm{Law}\Big(\mathcal{L} \big\vert_{K \times (a,b)} \textrm{conditional on } \mathcal{L} \big\vert_{(\Sigma \times \Lambda) \setminus ( K \times (a,b) )} \Big) =\PH{k_1}{k_2}{(a,b)}{\vec{x}}{\vec{y}}{f}{g}{\Ham}.
\end{equation*}
Here we have set $f=\mathcal{L}_{k_1-1}$ and $g=\mathcal{L}_{k_2+1}$ with the convention that if $k_1-1\notin\Sigma$ then $f\equiv +\infty$ and likewise if $k_2+1\notin \Sigma$ then $g\equiv -\infty$; we have also set $\vec{x}=\big(\mathcal{L}_{k_1}(a),\ldots ,\mathcal{L}_{k_2}(a)\big)$ and $\vec{y}=\big(\mathcal{L}_{k_1}(b),\ldots ,\mathcal{L}_{k_2}(b)\big)$. %The $K\times (a,b)$-indexed line ensemble $\mathcal{\tilde{L}}$ is the $\Ham$-Brownian bridge line ensemble with entrance data $\vec{x}$, exit data $\vec{y}$ and boundary data $(f,g)$.

An equivalent way to express this Gibbs property is as follows. For $K\subset \Z$ and $S\subset \R$, let $C^K(S)$ be the space of continuous functions from $K\times S\to \R$. Then
a $\Sigma \times \Lambda$-indexed line ensemble $\mathcal{L}$
enjoys the $\Ham$-Brownian Gibbs property if and only if for any $K=\{k_1,\ldots, k_2\}\subset \Sigma$ and $(a,b)\subset \Lambda$, and any Borel function $F$ from $C^{K}(a,b)\to \R$, $\PP$-almost surely
\begin{equation*}
\E\Big[F(\mathcal{L}_{k_1}\big\vert_{(a,b)},\ldots \mathcal{L}_{k_2}\big\vert_{(a,b)}) \big\vert \Fext\big(K\times (a,b)\big)\Big] = \PHExp{k_1}{k_2}{(a,b)}{\vec{x}}{\vec{y}}{f}{g}{\Ham} \Big[F(\mathcal{L}_{k_1},\ldots, \mathcal{L}_{k_2})\Big],
\end{equation*}
where $\vec{x}$,$\vec{y}$,$f$ and $g$ are defined in the previous paragraph and where
\begin{equation}\label{starstareleven}\glossary{$\Fext\big(K\times (a,b)\big)$, Sigma-field generated by a line ensemble outside $K\times (a,b)$}
\Fext\big(K\times (a,b)\big) := \sigma\big(\mathcal{L}_{i}(s): (i,s)\in \Sigma\times \Lambda \setminus K\times (a,b)\big)
\end{equation}
is the sigma-field generated by the line ensemble outside $K\times (a,b)$. On the left-hand side of the above equality  $\mathcal{L}_{k_1}\big\vert_{(a,b)},\ldots \mathcal{L}_{k_2}\big\vert_{(a,b)}$ are the restriction to $(a,b)$ of curves distributed according to $\PP$, while on the right-hand side $\mathcal{L}_{k_1},\ldots, \mathcal{L}_{k_2}$ are curves on $(a,b)$ distributed according to $\PH{k_1}{k_2}{(a,b)}{\vec{x}}{\vec{y}}{f}{g}{\Ham}$.
%Finally, to fix notation, we write $\N=\{1,2,\ldots\}$.
\end{definition}

\begin{rem}\label{normrem}
The normalizing constant $\partfunc{k_1}{k_2}{(a,b)}{\vec{x}}{\vec{y}}{f}{g}{\Ham}$ in (\ref{partfunceqn}) can be understood to be the acceptance probability for the following resampling procedure. Choose $\mathcal{L}_{k_1},\ldots, \mathcal{L}_{k_2}$ according to the free Brownian bridge line ensemble with entrance data $\vec{x}$, exit data $\vec{y}$ and boundary data $(f,g)$ (i.e. the law $\PH{k_1}{k_2}{(a,b)}{\vec{x}}{\vec{y}}{f}{g}{\Ham}$). {\it Accept} the ensemble if
$$
\bolt{k_1}{k_2}{(a,b)}{\vec{x}}{\vec{y}}{f}{g}{\Ham}\left(\mathcal{L}_{k_1},\ldots, \mathcal{L}_{k_2}\right) \geq U
$$
where $U$ is an independent random variable whose distribution is uniform on the interval $[0,1]$. The probability that the ensemble is accepted is exactly the normalizing constant.
\end{rem}

\subsubsection{Strong $\Ham$-Brownian Gibbs property}\label{s.stronghbn}

A Gibbs property for a line ensemble may be considered to be a spatial generalization of the Markov property. Just as for the Markov property, there is a concept of a strong Gibbs property where intervals are replaced by stopping domains. We will make use of a version of the strong Markov property where stopping domain (Definition \ref{defstopdom}) plays the role of stopping time. We now present an adaptation of the strong Brownian Gibbs property given in \cite[Lemma 2.5]{CH} to the present $\Ham$-Brownian Gibbs setting. The proof of the present result is essentially the same as its counterpart in \cite{CH} and so we do not record it here.

\begin{definition}\label{defstopdom}
Let $\Sigma$ be an interval of $\Z$, and $\Lambda$ be an interval of $\R$. Consider a $\Sigma\times\Lambda$-indexed line ensemble $\mathcal{L}$ which has the $\Ham$-Brownian Gibbs property for some Hamiltonian $\Ham$. For $K=\{k_1,\ldots, k_2\}\subseteq \Sigma$ and $(\ell,r)\subseteq \Lambda$, recall from (\ref{starstareleven}) the sigma field $\Fext\big(K\times (\ell,r)\big)$ generated by the curves outside $K\times (\ell,r)$. The random variable $(\mathfrak{l},\mathfrak{r})$ \glossary{$(\mathfrak{l},\mathfrak{r})$, Stopping domain} is called a {\it $K$-stopping domain} if for all $\ell<r$,
\begin{equation*}
\big\{\mathfrak{l} \leq \ell , \mathfrak{r}\geq r\big\} \in \Fext\big(K\times (\ell,r)\big).
\end{equation*}

Define $C^K(\ell,r)$ to be the set of continuous functions $(f_{k_1},\ldots, f_{k_2})$ with each $f_i:(\ell,r)\to\R$ (Definition \ref{maindefHBGP}) and
$$\glossary{$C^K$, Set of increments $(\ell,r)$ and $K$-indexed continuous functions on those increments}
C^K := \Big\{ \big(\ell,r, f_{k_1},\ldots, f_{k_2}\big): \ell<r \textrm{ and } (f_{k_1},\ldots, f_{k_2})\in C^K(\ell,r)\Big\}.
$$
Let $\mathcal{B}(C^K)$ \glossary{$\mathcal{B}(C^K)$, Borel measurable functions from $C^K\to \R$} denote the set of Borel measurable functions from $C^K\to \R$.
\end{definition}

\begin{lemma}\label{stronggibbslemma}
Consider a $\Sigma\times\Lambda$-indexed line ensemble $\mathcal{L}$ which has the $\Ham$-Brownian Gibbs property. Fix $K=\{k_1,\ldots,k_2\}\subseteq \Sigma$. For all random variables $(\mathfrak{l},\mathfrak{r})$ which are $K$-stopping domains for $\mathcal{L}$, the following strong $\Ham$-Brownian Gibbs property holds: for all $F\in \mathcal{B}(C^K)$, $\PP$-almost surely,
\begin{equation}\label{strongeqn}
\EE\bigg[ F\Big(\mathfrak{l},\mathfrak{r}, \mathcal{L}_{k_1}\big\vert_{(\mathfrak{l},\mathfrak{r})},\ldots  \mathcal{L}_{k_2}\big\vert_{(\mathfrak{l},\mathfrak{r})}\Big) \Big\vert \Fext\big(K\times (\mathfrak{l},\mathfrak{r})\big) \bigg]
=\PHExp{k_1}{k_2}{(a,b)}{\vec{x}}{\vec{y}}{f}{g}{\Ham}\Big[F\big(a,b, \mathcal{L}_{k_1},\ldots, \mathcal{L}_{k_2}\big) \Big],
\end{equation}
where $a=\mathfrak{l}$, $b=\mathfrak{r}$, $\vec{x} = \{\mathcal{L}_i(\mathfrak{l})\}_{i=k_1}^{k_2}$, $\vec{y} = \{\mathcal{L}_i(\mathfrak{r})\}_{i=k_1}^{k_2}$, $f(\cdot)=\mathcal{L}_{k_{1}-1}(\cdot)$ (or $\infty$ if $k_1-1\notin \Sigma$), $g(\cdot)=\mathcal{L}_{k_2+1}(\cdot)$ (or $-\infty$ if $k_2+1\notin \Sigma$). On the left-hand side $\mathcal{L}_{k_1}\big\vert_{(\mathfrak{l},\mathfrak{r})},\ldots  \mathcal{L}_{k_2}\big\vert_{(\mathfrak{l},\mathfrak{r})}$ is the restriction of curves distributed according to $\PP$ and on the right-hand side $\mathcal{L}_{k_1},\ldots, \mathcal{L}_{k_2}$ is distributed according to $\PH{k_1}{k_2}{(a,b)}{\vec{x}}{\vec{y}}{f}{g}{\Ham}$ (Definition \ref{maindefHBGP}).
\end{lemma}

%\begin{rem}\label{r.expbound}
%Under the same setup as Lemma \ref{stronggibbslemma} if $G$ and $G'$ are $\mathcal{F}_{\Sigma\times \Lambda}$ measurable random variables such that $\PP$-almost surely
%$$\EE\bigg[ F\Big(\mathfrak{l},\mathfrak{r}, \mathcal{L}_{k_1}\big\vert_{(\mathfrak{l},\mathfrak{r})},\ldots  \mathcal{L}_{k_2}\big\vert_{(\mathfrak{l},\mathfrak{r})}\Big) \Big\vert \Fext\big(K\times (\mathfrak{l},\mathfrak{r})\big) \bigg] \leq G
%$$
%then
%$$
%\EE\bigg[ F\Big(\mathfrak{l},\mathfrak{r}, \mathcal{L}_{k_1}\big\vert_{(\mathfrak{l},\mathfrak{r})},\ldots  \mathcal{L}_{k_2}\Big) \, G' \bigg] \leq \EE\big[G\, G'\big].
%$$
%
%We will generally use the above fact when we can find $G$ by way of a monotone coupling within the stopping domain. It will also generally be the case that we will apply the above in terms of probabilities (not expectations) though such applications can be readily reduced to that above by the use of indicator functions.
%\end{rem}
%\subsubsection{Combining the (strong) $\Ham$-Brownian Gibbs property and monotonicity}
%
%
%
%Have gibbs ensemble and a stopping domain. Consider a function which is increasing in the functions and their endpoints. Have a monotone coupled set of curves for which we know bounds on the conditional expectation. This implies similar bounds for the full expectation.

\subsubsection{Monotone couplings and the $\Ham$-Brownian Gibbs property}
For a convex Hamiltonian $\Ham$ (such as $\Ham_t$), we prove two monotonicity results for $\Ham$-Brownian bridge line ensembles $\PH{k_1}{k_2}{(a,b)}{\vec{x}}{\vec{y}}{f}{g}{\Ham}$ (Definition \ref{maindefHBGP}). The first result states that if $f$ and $g$ are both increased pointwise, then the resulting line ensemble can be coupled to the original one so as to dominate it pointwise. The second  result asserts that if $\vec{x}$ and $\vec{y}$ are both increased (in the partial order of domination in each coordinate) then the line ensemble resulting from this increase and the original one also enjoy the same type of domination. We will often use these forms of monotonicity to prove inequalities about probabilities of events that increase under such pointwise domination.

\begin{lemma}\label{monotonicity1}
Fix $k_1\leq k_2\in \Z$, $a<b$, two vectors $\vec{x},\vec{y}\in \R^{k_2-k_1+1}$ and two pairs of measurable functions $(f^{(i)},g^{(i)})$ for $i\in \{1,2\}$ such that $f^{(i)}:(a,b)\rightarrow \R\cup\{\infty\}$, $g^{(i)}:(a,b)\rightarrow \R\cup\{-\infty\}$ and for all $s\in (a,b)$, $f^{(1)}(s)\geq f^{(2)}(s)$ and $g^{(1)}(s)\geq g^{(2)}(s)$. Recalling Definitions \ref{maindefline} and \ref{maindefHBGP}, for $i\in \{1,2\}$, let $\mathcal{Q}^{(i)}=\{\mathcal{Q}^{(i)}_j\}_{j=k_1}^{k_2}$ be a $\{k_1,\ldots,k_2\}\times (a,b)$-indexed line ensemble on a probability space $\big(\Omega^{(i)},\mathcal{B}^{(i)},\PHShort{\Ham}^{(i)}\big)$ where $\PHShort{\Ham}^{(i)}$ equals $\PH{k_1}{k_2}{(a,b)}{\vec{x}}{\vec{y}}{f^{(i)}}{g^{(i)}}{\Ham}$
(i.e. $\mathcal{Q}^{(i)}$ has the $\Ham$-Brownian Gibbs property with entrance data $\vec{x}$, exit data $\vec{y}$ and boundary data $(f^{(i)},g^{(i)})$).

If the Hamiltonian $\Ham:\R\to [0,\infty)$ is convex then there exists a coupling of the probability measures $\PHShort{\Ham}^{(1)}$ and $\PHShort{\Ham}^{(2)}$ such that almost surely $\mathcal{Q}^{(1)}_j(s)\leq \mathcal{Q}^{(2)}_j(s)$ for all $j\in \{k_1,\ldots, k_2\}$ and all $s\in (a,b)$.
\end{lemma}

\begin{lemma}\label{monotonicity2}

Fix $k_1\leq k_2 \in \Z$, $a<b$ and two measurable functions $(f,g)$ such that $f:(a,b)\rightarrow \R\cup\{\infty\}$, $g:(a,b)\rightarrow \R\cup\{-\infty\}$. Consider two pairs of vectors $\vec{x}^{(i)},\vec{y}^{(i)}\in \R^{k_2-k_1+1}$ for $i\in \{1,2\}$ such that $x^{(1)}_{j}<x^{(2)}_{j}$ and $y^{(1)}_{j}<y^{(2)}_{j}$ for all $j=k_1,\ldots,k_2$. Recalling Definitions \ref{maindefline} and \ref{maindefHBGP}, for $i\in \{1,2\}$, let $\mathcal{Q}^{(i)}=\{\mathcal{Q}^{(i)}_j\}_{j=k_1}^{k_2}$ be a $\{k_1,\ldots,k_2\}\times (a,b)$-indexed line ensemble on a probability space $\big(\Omega^{(i)},\mathcal{B}^{(i)},\PHShort{\Ham}^{(i)}\big)$ where $\PHShort{\Ham}^{(i)} = \PH{k_1}{k_2}{(a,b)}{\vec{x}^{(i)}}{\vec{y}^{(i)}}{f}{g}{\Ham}$ (i.e. $\mathcal{Q}^{(i)}$ has the $\Ham$-Brownian Gibbs property with entrance data $\vec{x}^{(i)}$, exit data $\vec{y}^{(i)}$ and boundary data $(f,g)$).

If the Hamiltonian $\Ham:\R\to [0,\infty)$ is convex then there exists a coupling of the probability measures $\PHShort{\Ham}^{(1)}$ and $\PHShort{\Ham}^{(2)}$ such that almost surely $\mathcal{Q}^{(1)}_j(s)\leq \mathcal{Q}^{(2)}_j(s)$ for all $j\in \{k_1,\ldots, k_2\}$ and all $s\in (a,b)$.
\end{lemma}

These lemmas are proved in Section \ref{monotoneproofs}. They generalize Lemmas 2.6 and 2.7 of \cite{CH} which correspond to formally using the Hamiltonian $\Ham_{+\infty}(x)$ \glossary{$\Ham(x)_{+\infty}$, Hamiltonian corresponding to non-intersection conditioning} which is $+\infty$ for $x>0$ and $0$ for $x<0$.  This particular Hamiltonian corresponds to conditioning consecutively labeled curves not to touch. The convexity assumption is used to show (\ref{r.convexd}) in the proof.

\begin{definition}\label{moregengibbs}
We may formulate a more general set of monotonicity results which involves measures specified by multiple boundary interactions. Fix a Hamiltonian $\Ham$ and, for $r\geq 1$, fix two vectors of Hamiltonians $\vec{\Ham}^f= (\Ham^f_1,\ldots, \Ham^f_r)$\glossary{$\vec{\Ham}^f$, Vector of boundary interaction Hamiltonians for lowest indexed curve} and $\vec{\Ham}^g= (\Ham^g_1,\ldots, \Ham^g_r)$ \glossary{$\vec{\Ham}^g$, Vector of boundary interaction Hamiltonians for highest indexed curve}. Let $\vec{f} = (f_1,\ldots ,f_r)$ \glossary{$\vec{f}$, Vector of boundary interaction functions with lowest indexed curve} and $\vec{g}=(g_1,\ldots, g_r)$, \glossary{$\vec{g}$, Vector of boundary interaction functions with highest indexed curve} where $f_j:(a,b)\rightarrow \R\cup\{\infty\}$ and $g_j:(a,b)\rightarrow \R\cup\{-\infty\}$ for $1\leq j\leq r$. Generalizing Definition \ref{maindefHBGP}, we define a measure $\PH{k_1}{k_2}{(a,b)}{\vec{x}}{\vec{y}}{\vec{f}}{\vec{g}}{\Ham,\vec{\Ham}^f,\vec{\Ham}^g}$ \glossary{$\PH{
k_1}{k_2}{(a,b)}{\vec{x}}{\vec{y}}{\vec{f}}{\vec{g}}{\Ham,\vec{\Ham}^f,\vec{\Ham}^g}$, $\{k_1,\ldots,k_2\}\times (a,b)$-indexed {\it $\Ham$-Brownian bridge} line ensemble with entrance data $\vec{x}$ and exit data $\vec{y}$ and multiple boundary interactions determined by the vectors of Hamiltonians $\vec{\Ham}^f$, $\vec{\Ham}^g$ and the vectors of boundary curves $\vec{f}$, $\vec{g}$} on curves $\mathcal{L}_{k_1},\ldots, \mathcal{L}_{k_2}:(a,b)\to \R$ by specifying its Radon-Nikodym derivative with respect to $\Pfree{k_1}{k_2}{(a,b)}{\vec{x}}{\vec{y}}$ to
be proportional to the Boltzmann weight
\begin{eqnarray*}
\bolt{k_1}{k_2}{(a,b)}{\vec{x}}{\vec{y}}{\vec{f}}{\vec{g}}{\Ham,\vec{\Ham}^f,\vec{\Ham}^g}\big(\mathcal{L}_{k_1},\ldots, \mathcal{L}_{k_2}\big)&\hskip-.1in:=&\hskip-.1in \exp\Bigg\{\!-\sum_{i=k_1}^{k_2-1} \int_{a}^{b} \Ham\Big(\mathcal{L}^j_{i+1}(s)-\mathcal{L}^j_{i}(s)\Big) \dd s\\
 &&- \sum_{j=1}^{r} \int_{a}^{b} \Ham^f_j\Big(\mathcal{L}_{k_1}(s)-f_j(s)\Big) \dd s - \sum_{j=1}^{r} \int_{a}^{b} \Ham^g_j\Big(g_j(s) - \mathcal{L}_{k_2}(s)\Big) \dd s \Bigg\},
\end{eqnarray*}
where the normalization is chosen so that the integral equals one.
\end{definition}

Both Lemmas \ref{monotonicity1} and \ref{monotonicity2} can be generalized to this setting. We will make use of the generalization for Lemma \ref{monotonicity1}, and so we state it here.

\begin{lemma}\label{moregengibbslemma}
Consider two sets of boundary condition functions $\vec{f}^{(i)},\vec{g}^{(i)}$ for $i\in \{1,2\}$ such that, for all $1\leq j\leq r$ and all $s\in (a,b)$, $f^{(1)}_j(s)\geq f^{(2)}_j(s)$ and $g^{(1)}_j(s)\geq g^{(2)}_j(s)$. For $i\in \{1,2\}$, let $\mathcal{Q}{(i)}=\{\mathcal{Q}^{(i)}_j\}_{j=k_1}^{k_2}$ be a $\{k_1,\ldots, k_2\}\times (a,b)$-indexed line ensemble on a probability space $(\Omega^{(i)},\mathcal{B}^{(i)},\PHShort{\Ham,\vec{\Ham}^f,\vec{\Ham}^g}^{(i)})$, where $\PHShort{\Ham,\vec{\Ham}^f,\vec{\Ham}^g}^{(i)}=\PH{k_1}{k_2}{(a,b)}{\vec{x}}{\vec{y}}{\vec{f}^{(i)}}{\vec{g}^{(i)}}{\Ham,\vec{\Ham}^f,\vec{\Ham}^g}$. If each of the Hamiltonians $\Ham$, $\vec{\Ham}^f$ and $\vec{\Ham}^g$ are convex, then there exists a coupling of the probability measures $\PHShort{\Ham,\vec{\Ham}^f,\vec{\Ham}^g}^{(1)}$ and $\PHShort{\Ham,
\
\vec{\Ham}^f,\vec{\Ham}^g}^{(2)}$ such that almost surely $\mathcal{Q}^{(1)}_j(s)\leq \mathcal{Q}^{(2)}_j(s)$ for all $j\in \{k_1,\ldots, k_2\}$ and all $s\in (a,b)$.
\end{lemma}
\begin{proof}
The proof follows that of Lemma \ref{monotonicity1}, with equations (\ref{e.ri}) and (\ref{r.convexd}) changing to accommodate multiple Hamiltonians. In the vicinity of (\ref{e.phnny}), we will apply this general form of monotonicity with $r=2$ and with one of the Hamiltonians given by $\Ham_{+\infty}$ which we introduced above. This function is not (strictly speaking) convex, but the proof works just as well (as $\Ham_{+\infty}$ is a limit of convex functions).
\end{proof}

\subsubsection{Facts about Brownian motion}\label{s.explicitfacts}

We recall a few basic facts about Gaussian random variables and Brownian bridge and motion. The first are two tail bounds for the Gaussian distribution, the second is the exact distribution for the supremum of a Brownian bridge, the third is about the probability that Brownian motion hits a parabola, and the fourth and fifth are upper bounds on the probability that Brownian motion and bridge drastically change height in a short period of time.

\begin{lemma}\label{l.normallb}
Let $N$ be a standard normal random variable (of variance one). Then, for each $s \geq 0$,
$$ \frac{1}{\sqrt{2\pi}}\, \frac{s}{s^2 + 1}\, e^{-s^2/2} < \PP\big(N \geq s\big) <  \frac{1}{\sqrt{2\pi}} \,\frac{1}{s}\, e^{-s^2/2}.$$
\end{lemma}
\begin{proof}
These follow from straightforward computation.
\end{proof}

\begin{lemma}\label{l.bridgesup}
The supremum of a Brownian bridge $B:[0,L] \to \R$, $B(0) = B(L) = 0$,
satisfies
%\begin{equation}\label{e.bridgesup}
$$
\PP \bigg( \sup_{x \in [0,L]} B(x) > s \bigg) =  e^{ - 2s^2/L} \, .
$$
%\end{equation}
\end{lemma}
\begin{proof}
This amounts to a use of the reflection principle -- see (3.40) in Chapter 4 of \cite{KS}.
\end{proof}

\begin{lemma}\label{l.wegwew}
Let $B(\cdot)$ be a two-sided Brownian motion with $B(0)=0$. Then, for all $c>0$, there exist constants $c_1,c_2>0$ such that, for all $s\geq 1$,
$$
\PP\Big(B(x) \geq s+ cx^2 \textrm{ for some }x\in \R \Big) \leq c_1 e^{-c_2 s^{3/2}}.
$$
\end{lemma}
\begin{proof}
This follows from \cite[Theorem 2.1]{Gro} wherein the constants and corrections to this exponential term are carefully calculated.
\end{proof}

\begin{lemma}\label{l.bm}
Let $s > 0$ and consider Brownian motion $B:[0,s] \to \R$. For $M> 0$ and $r \in (0,s)$,
$$ \PP\Big( \exists t_1,t_2 \in [0,s] : 0 \leq t_2 - t_1 \leq r \textrm{ and }B(t_2) - B(t_1) \leq -M\Big)\leq 16 \pi^{-1/2} \frac{s}{r^{1/2}M} e^{-M^2/(16r)}.$$
\end{lemma}
\begin{proof}
If such times $t_1$ and $t_2$ exist, then $B$ drops in height by at least $M$ between some pair of times during one of the at most $\lfloor s r^{-1} \rfloor +1$ intervals of length $2r$ whose left-hand endpoints are the multiples of $r$ between $0$ and $s$.

Thus, the existences of such times entails the occurrence of the event
$$
\bigcup_{0 \leq i \leq \lfloor sr^{-1} \rfloor} \left( \Big\{ \sup_{s \in[0, 2r]} B(ri + s) - B(ri) \geq M/2 \Big\} \cup
 \Big\{ \inf_{s\in [0 , 2r]} B(ri + s) - B(ri) \leq -M/2 \Big\} \right) \, .
 $$
The reflection principle determines the probability of each of these events. In this way, the probability considered in the statement of the lemma is (by a union bound) at most
$$\big( \lfloor s r^{-1} + 1 \rfloor \big) \,\cdot\, 2\, \cdot\, 2\cdot\, \PP \Big( N \geq \frac{M}{2 \sqrt{2r}} \Big),$$
where here $N$ is a standard Gaussian random variable (of variance one). Lemma~\ref{l.normallb} and $s r^{-1} \geq 1$ show this bound to be at most  $16 \pi^{-1/2} \tfrac{s}{r^{1/2}M} e^{-M^2/(16r)}$.
\end{proof}

We may use the preceding result to deduce an analogue for Brownian bridge.

\begin{lemma}\label{c.bb}
Let $s > 0$, $a\in \R$ and consider Brownian bridge $B:[0,s] \to \R$, $B(0)= 0$ and $B(s) = a$. For $M> 0$ and $r \in (0,s)$, the probability that there exist $t_1,t_2 \in [0,s]$ with $0 \leq t_2 - t_1 \leq r$ and $B(t_2) - B(t_1) \leq -M$ is bounded above by
$$
\frac{1}{\PP(N_s\leq a)} 16 \pi^{-1/2} \frac{s}{r^{1/2}M} e^{-M^2/(16r)},
$$
where $N_s$ is a centered normal random variable of variance $s$.
\end{lemma}
\begin{proof}
Let $B'$ be distributed as Brownian motion on $[0,s]$ with $B'(0)=0$. Define the event
$$\mathsf{E} = \Big\{\exists \, t_1,t_2 \in [0,s]: 0 \leq t_2 - t_1 \leq r, \textrm{ and }B'(t_2) - B'(t_1) \leq -M\Big\}.$$
Note that
$$
 \PP \big( \mathsf{E}\, \big\vert\, B'(s) \leq a \big) \leq \frac{\PP(\mathsf{E})}{\PP\big( B'(s) \leq a \big)} \, .
$$

The Brownian bridge probability that we seek to bound is nothing other than
$\PP\big(\mathsf{E}\,\big| \,B'(s)=a\big)$.
Since Lemma~\ref{l.bm} bounds $\PP(\mathsf{E})$ from above, the proof reduces to arguing that
 $\PP \big( \mathsf{E} \big\vert B'(s) = a \big) \leq  \PP \big( \mathsf{E} \big\vert B'(s) \leq a \big)$.

To this end, it suffices to show that $\PP\big(\mathsf{E}\,|\,B'(s)=a\big)\leq \PP\big(\mathsf{E}\,|\,B'(s)=a'\big)$ when $a'<a$. To see this bound, note that conditioning
Brownian motion $B'$ on ending at $a$ is the same as declaring $B'$ to be a Brownian bridge to $a$. Two Brownian bridges starting from height $0$ and ending at different heights $a$ and $a'$ can be coupled to a single Brownian bridge which starts and ends at height $0$ by adding a linear interpolation of $0$ and either $a$ or $a'$. Lowering the slope of the linear interpolation renders the event $\mathsf{E}$ more likely, whence we obtain the sought inequality.
\end{proof}

\subsection{Main result: KPZ$_t$ line ensemble construction and properties}

%Our main result shows that for any $t>0$, the KPZ$_t$ line ensemble has the $\Ham$-Brownian Gibbs property with $\Ham(x) = e^{x}$. As an immediate application of this we show that each line in the ensemble is locally absolutely continuous with respect to Brownian motion. We then compare these results to our earlier work \cite{CH} on the Airy line ensemble and its non-intersecting Brownian Gibbs property (formally $\Ham(x)=\infty$ for $x\geq 0$ and $\Ham(x)=0$ for $x<0$).

The following theorem constitutes the main result of this paper, from which all of our applications (discussed in Sections \ref{s.unifbwn}, \ref{s.orderthird}, \ref{transec} and \ref{s.tailbounds}) follow.

\begin{theorem}\label{mainthm}
For all $t\geq 1$, there exist $\N\times \R$-indexed line ensembles $\HKPZlinet{t}:= \big\{\HKPZline{t}{n}:n\in \N\big\}$ \glossary{$\HKPZlinet{t}$, KPZ$_t$ line ensemble} \glossary{$\HKPZline{t}{n}$, Index $n$ curve of KPZ$_t$ line ensemble}  enjoying the following three properties:
\begin{enumerate}
\item The lowest indexed curve $\HKPZline{t}{1}:\R\to \R$ is equal in distribution to $\HKPZnw{t}{\cdot}$, the time $t$ Hopf-Cole solution to the narrow wedge initial data KPZ equation (Definition \ref{SHEdef});
\item The ensemble $\HKPZlinet{t}$ has the $\Ham_1$-Brownian Gibbs property (recall from Definition \ref{maindefHBGP} that $\Ham_t(x)=e^{t^{1/3} x}$).
\item Define $\HKPZFPlinet{t}$  by the relation \glossary{$\HKPZFPlinet{t}$, scaled KPZ$_t$ line ensemble} \glossary{$\HKPZFPline{t}{n}$, Index $n$ curve of scaled KPZ$_t$ line ensemble}
\begin{equation}\label{e.gfewgwe}
\HKPZline{t}{n}(x) =  -\frac{t}{24} +  t^{1/3} \HKPZFPline{t}{n}(t^{-2/3} x).
\end{equation}
This ensemble has the $\Ham_t$-Brownian Gibbs property. For any interval $(a,b)\subset \R$ and $\e>0$, there exists $\delta>0$ such that, for all $t\geq 1$, recalling (\ref{partfunceqn}),
$$
\PP\left(\partfunc{1}{1}{(a,b)}{\HKPZFPline{t}{1}(a)}{\HKPZFPline{t}{1}(b)}{+\infty}{\HKPZFPline{t}{2}}{\Ham_t}<\delta \right)\leq \e.
$$
\end{enumerate}
\end{theorem}

This theorem is proved in Section \ref{s.mainthmproof}. See Section \ref{s.kpzlovers} and Figure \ref{seqncmptfig} for an overview of the ideas which contribute to proving this.

We call any line ensemble (such as $\HKPZlinet{t}$ above) which satisfies the two properties in Theorem \ref{mainthm} a {\it KPZ$_t$ line ensemble}, and its scaled version (such as $\HKPZFPlinet{t}$ above) a {\it scaled KPZ$_t$ line ensemble}. The theorem shows the existence of a family of such ensembles for all $t\geq 1$ and additionally shows that the lowest indexed curve of this family (i.e. $\HKPZFPline{t}{1}$) remains uniformly absolutely continuous (with respect to Brownian bridge) over $t\in [1,\infty)$ (see Remark \ref{r.abrbdsd}). This uniformity property is essential in our first three applications.

The theorem does not rule out the existence of more than one KPZ$_t$ line ensemble; (see however the discussion of Section \ref{s.oconwar}). That said, for the applications we provide, it is sufficient to know the existence of such a family of line ensembles provided by the theorem. Note finally that the choice of considering $t\geq 1$ could just as well have been $t\geq t_0$ for any $t_0>0$ fixed.

\begin{rem}\label{r.abrbdsd}
Theorem \ref{mainthm}(3) can be understood in terms of the Radon-Nikodym derivative for the law of the lowest indexed curve on the interval $(a,b)$ with respect to the law of free Brownian bridge with starting height $\HKPZFPline{t}{1}(a)$ at $a$ and ending height $\HKPZFPline{t}{1}(b)$ at $b$. By Definition \ref{maindefHBGP} and the positivity of the Hamiltonian $\Ham_t$, it is clear that when $\partfunc{1}{1}{(a,b)}{\HKPZFPline{t}{1}(a)}{\HKPZFPline{t}{1}(b)}{+\infty}{\HKPZFPline{t}{2}}{\Ham_t}>\delta$, the Radon-Nikodym derivative is bounded above by $\delta^{-1}$:
$$\frac{d\PH{1}{1}{(a,b)}{\HKPZFPline{t}{1}(a)}{\HKPZFPline{t}{1}(b)}{+\infty}{\HKPZFPline{t}{2}}{\Ham_t}}{d\Pfree{1}{1}{(a,b)}{\HKPZFPline{t}{1}(a)}{\HKPZFPline{t}{1}(b)}}<\delta^{-1}.$$
This means that the curve $\HKPZFPline{t}{1}$ on $(a,b)$ remains uniformly absolutely continuous with respect to Brownian bridge, or equivalently has a tight Radon-Nikodym derivative with respect to Brownian bridge as $t\to\infty$. One expects that a similar result will hold for the normalizing constant (and hence Radon-Nikodym derivative) for any finite collection of curves in $\HKPZFPlinet{t}$, though we do not pursue that here.
\end{rem}

\subsection{Extensions and discussion}
Theorem \ref{mainthm} may be elaborated in several directions, some of which here we briefly mention. It should be possible to develop an analogous theorem corresponding to a slightly larger, but still very special, class of KPZ initial data beyond narrow wedge. There are discrete analogues of such line ensembles related to the log-gamma discrete directed polymer, and to degenerations of Macdonald processes. One can conjecture a relationship between our KPZ$_t$ line ensemble and the Airy line ensemble in the limit as $t\to \infty$, and even sketch a possible route of proof. Finally, there should be a more explicit description of a KPZ$_t$ line ensemble in terms of the O'Connell-Warren multilayer extension to the solution of the stochastic heat equation. We now describe briefly each of these directions. This discussion is closely related to the content of Section~\ref{semidiscete} as well.

\subsubsection{Gibbsian line ensembles corresponding with other KPZ initial data}\label{s.otherdata}
Narrow-wedge initial data is not the only type of initial data for which there should exist $\N\times\R$-indexed line ensembles whose lowest indexed curve is the relevant time $t$  KPZ solution, which have the $\Ham_t$-Brownian Gibbs property, and over which control of the Radon-Nikodym derivative is uniform in $t\in [1,\infty)$. One example of initial data for which an analogue of Theorem \ref{mainthm} should exist is half stationary / half narrow wedge (which, recall, corresponds with starting the stochastic heat equation from $\mathcal{Z}(0,x)={\bf 1}_{x\geq 0} e^{B(x)}$ where $B(\cdot)$ is a one-sided Brownian motion chosen independently of all other randomness). This type of initial data arises from scaling the O'Connell-Yor semi-discrete Brownian polymer where the underlying noise Brownian motion $B_1$ has a suitably large drift added to it and all others have no drift (see \cite{MRQ,BCF}). In fact, for all of the types of KPZ initial data considered in \cite{BCF}, there should be analogues of Theorem \ref{mainthm}.

In the last passage percolation setting, this type of perturbed line ensemble is discussed in \cite{ADvM} under the name $r$-Airy processes. These Airy-like line ensembles (and their Gibbs properties and applications) are developed in \cite{CH}.

\subsubsection{Gibbsian line ensembles corresponding with other solvable growth processes and polymers}\label{s.egweg}
The fact that the line ensemble associated to the O'Connell-Yor semi-discrete Brownian polymer (Sections \ref{OConYor} and \cite{OCon}) has the $\Ham_1$-Brownian Gibbs property (Section \ref{QTLHgibbs}) is a consequence of the {\it integrability} or {\it exact solvability} of that process discovered by O'Connell \cite{OCon} using a continuous version of the geometric lifting of the Robinson-Schensted-Knuth (RSK) correspondence. There are only a handful of other probabilistic models for which analogously constructed line ensembles display such a Gibbs property, namely: geometric, exponential, Poissonian, and semi-discrete Brownian last passage percolation (equivalently, various versions of the polynuclear growth model or TASEP \cite{SpohnLineEnsemble}), as well as the log-gamma discrete directed polymer \cite{S,COSZ}.

For instance, for the log-gamma polymer, one constructs a line ensemble (no longer with continuous curves, but rather with piecewise constant curves) by recording the evolution of the ``shape'' of the geometric lifting of the RSK correspondence applied to an input matrix of independent inverse Gamma distributed random variables with a growing number of columns; (adding a column is like advancing time). The generator for this discrete time Markov process on the shape is computed in \cite{COSZ} and takes the form of a Doob-$h$ transform of a collection of independent random walks subject to a two-body potential acting on consecutively labeled curves. This property of the generator results in a discrete time version of the $\Ham_1$-Brownian Gibbs property (where the underlying Brownian bridge measure is replaced by a corresponding random walk). A scaling limit of this discrete line ensemble results in the one studied earlier by O'Connell \cite{OCon}. O'Connell's diffusion (or equivalently,
line ensemble) is introduced in Section \ref{OConYor} and its associated properties play a central role herein.

The last passage percolation models mentioned above also have Gibbs properties, though they involve conditioning on non-intersection (rather than probabilistically penalizing curves for crossing). For instance, the line ensemble $\big\{\OConM{N}{n}(s)\big\}_{n=1}^{N}$ defined in (\ref{e.OConM}) exactly corresponds with the GUE Dyson Brownian motion which displays a non-intersecting Brownian Gibbs property. This type of line ensemble was the subject of our earlier work \cite{CH}.

In the discrete polymer cases discussed above (analogous facts hold true for all of the other cases mentioned above as well), the Gibbs property can be understood also as a consequence of the fact that the line ensemble that records the evolution of the ``shape'' under the geometric lifting of the RSK correspondence is distributed according to a particular {\it Whittaker process}. In fact, in the hierarchy of Macdonald processes \cite{BorCor} and their degenerations, one expects an analogous Gibbs property at the level of $q$-Whittaker processes (at the full Macdonald process level the Gibbs property should not as ``local'' as the one studied here). Though we will not delve any deeper into these matters, we may reference the discussion around \cite[Proposition 2.3.6]{BorCor} and \cite[Definition 3.3.3]{BorCor} as relevant in these considerations. The locality mentioned here can be understood as the fact that in the $q$-Whittaker (second reference) case, the dynamics discussed relies only upon nearest index particle locations, whereas in the full Macdonald case, the dynamics rely on the entire partition (i.e. all particle locations).

\subsubsection{The conjectural relation to Airy line ensemble}\label{s.airy}
It has been conjectured for some time now (see, for instance \cite[Conjecture 1.5]{ACQ}) that the solution of the narrow wedge initial data KPZ equation converges (as $t\to \infty$) under horizontal scaling by $t^{2/3}$ and vertical scaling by $t^{1/3}$ to the Airy process (sometimes called the Airy$_2$ process) minus a parabola. We now expand upon this conjecture to include the full KPZ$_t$ and Airy line ensembles, and provide a sketch of how, by working with this larger structure of line ensembles (rather than their lowest indexed curves), it may be possible to provide a purely probabilistic proof of both conjectures.

To fix notation, let us recall the $\N\times \R$-indexed Airy line ensemble $\mathcal{A} = \{\mathcal{A}_i\}_{i\in \N}$ defined and studied in \cite{CH}. It is shown in \cite[Theorem 3.1]{CH} that the line ensemble defined by $\mathcal{L}_i(x)= 2^{-1/2} \big(\mathcal{A}_i(x)-x^2\big)$ has the non-intersecting Brownian Gibbs property (defined by formally setting $\Ham(x)= \Ham_{+\infty}(x)$ which is $+\infty$ for $x>0$ and $0$ for $x<0$). The Airy line ensemble is stationary under shifts in the $x$-coordinate, and the one-point distribution of the lowest indexed curve (the Airy process) $\mathcal{A}_1(x)$ is the GUE Tracy-Widom distribution. In \cite[Corollary 1.3]{ACQ}, it was proved that, as $t\to\infty$, the distribution of the
random variable $2^{1/3}\big(\HKPZFPline{t}{1}(x)+x^2/2\big)$ (which is also stationary under shifts in $x$ -- see Proposition \ref{ACQprop}(1)) also converges to the GUE Tracy-Widom distribution.

\cite[Conjecture 1.5]{ACQ} claims that the entire spatial process for the narrow wedge initial data KPZ equation should converge (under ($1/3,2/3$)-scaling and parabolic shift) to the Airy process. Evidence for this conjecture was the one-point convergence result as well as a belief in KPZ universality and the knowledge that the Airy process $\mathcal{A}_1(x)$ (recall this lowest indexed curve is sometimes called the Airy$_2$ process) arises in a similar manner for growth processes such as TASEP (see the review~\cite{CorwinReview}). We now conjecture the convergence of the entire scaled KPZ$_t$ line ensemble to the Airy line ensemble.

\begin{conjecture}\label{c.kpzairy}
The line ensemble defined (for each $t>0$) via the map $(n,x)\mapsto 2^{1/3}\big(\HKPZFPline{t}{n}(x)+x^2/2\big)$ converges as a line ensemble (Definition \ref{maindefHBGP}) to the Airy line ensemble $\big\{ \mathcal{A}_n(x): n \in \N, x \in \R \big\}$ as $t\to \infty$.
\end{conjecture}
Let us sketch the steps of a plausible but speculative route to the proof of this result. Notice that the route is purely probabilistic and avoids trying to compute the multi-point distribution for the KPZ equation.

\begin{enumerate}
\item Prove that the line ensemble $2^{1/3}\big(\HKPZFPline{t}{n}(x)+x^2/2\big)$ is stationary in shifting $x$. A priori is this not obvious since the precursor (finite $N$) line ensemble $\HSDFPline{t}{N}{n}(x)+x^2/2$ (see Definition \ref{d.cntx}) used in constructing $\HKPZFPlinet{t}$ does not display this invariance at finite $N$ (after all, the finite $N$ ensemble is only defined for $x$ larger than some fixed $N$ dependent value). The stationarity of $2^{1/3}\big(\HKPZFPline{t}{n}(x)+x^2/2\big)$ would follow from the upcoming Conjecture \ref{OWconjecture}. That conjecture provides a precise (path-integral or rather chaos series) description for all curves of $\HKPZFPlinet{t}$ from which stationarity follows in light of the invariance under affine shifts of space-time white noise and Brownian bridge.
\item Prove sequential compactness of $\HKPZFPlinet{t}$ as $t\to \infty$. Theorem~\ref{seqncmpt}(2) provides control over the normalizing constant for the lowest indexed curve of the scaled KPZ$_t$  line ensemble, in a manner which is uniform for $t\geq 1$. This step would require establishing the same type of result for the normalizing constant of any finite set of curves. This may be achievable by methods similar to those that are used to prove Theorem \ref{seqncmpt} (see Section \ref{wegsavg} in particular).
\item The sequential compactness of $\HKPZFPlinet{t}$ would prove the existence of subsequential limiting line ensembles which are stationary in $x$ (after the necessary parabolic shift) and which display the non-intersecting Brownian Gibbs property (since $\Ham_{+\infty}$ is the limit of $\Ham_t$). Now prove that all subsequential limits are extremal (in the sense that they may not be written as the sum of two non-intersecting Brownian Gibbs measures).
\item Prove \cite[Conjecture 3.2]{CH} which states that the Airy line ensemble is the only extremal $\N \times \R$-indexed line ensemble which satisfies the Brownian Gibbs property and is stationary in $x$. In fact, the claimed uniqueness is up to a global shift in height, but this shift is determined by the one-point convergence result of \cite[Corollary 1.3]{ACQ} which we described above. It would follow from this uniqueness result that $\HKPZFPlinet{t}$ converges to the Airy line ensemble (under the scalings in Conjecture~\ref{c.kpzairy}).
\end{enumerate}

Although the task of rendering this sketch into a proof would be challenging in many ways, the sketch at least provides some more evidence for  \cite[Conjecture 1.5]{ACQ} regarding the spatial convergence of the KPZ equation to the Airy process.
\subsubsection{Conjectural relation to O'Connell-Warren line ensemble}\label{s.oconwar}
We start by recalling a continuum analog of the constructions of Section \ref{OConYor} and then remark on its relation to our constructed KPZ$_t$ line ensembles and the question of uniqueness.

O'Connell and Warren \cite{OConWar} define a multi-layer extension of the solution to the stochastic heat equation. For $n\in \N$, $t\geq 0$ and $x\in \R$, define
\begin{equation}\label{Zpartfunc}\glossary{$\ZSHEn{t}{x}{n}$, O'Connell-Warren multi-layer extension of the solution to the stochastic heat equation}
\ZSHEn{t}{x}{n} := p(x;t)^n \sum_{k=0}^{\infty} \int_{\Delta_k(t)}\int_{\R^k} R_k^{(n)}\big((t_1,x_1),\ldots, (t_k,x_k)\big) \whitenoise(\dd t_1 \dd x_1)\cdots \whitenoise(\dd t_k \dd x_k),
\end{equation}
where $p(x;t)=(2\pi t)^{-1/2} \exp(-x^2/2t)$ is the heat kernel, $\Delta_k(t) = \{0<t_1<\cdots <t_k<t\}$, and $R_k^{(n)}$ is the $k$-point correlation function for a collection of $n$ non-intersecting Brownian bridges each of which starts at $0$ at time $0$ and ends at $x$ at time $t$. The above multiple stochastic integrals (which are sometimes called chaos series) are It\^{o} in time and Stratonovich in space. For notational simplicity define $\ZSHEn{t}{x}{0}\equiv 1$. Observe that if we take the same white noise as in (\ref{SHE}) then $\ZSHEn{t}{x}{1} = \ZSHEnw{t}{x}$, the narrow wedge (or delta initial data) solution to the stochastic heat equation (see \cite{ACQ,AKQ2} for details).

The $\ZSHEn{t}{x}{n}$ can be interpreted as partition functions for a continuum directed random polymer. At a heuristic level, we may write
\begin{equation}\label{wick}
\frac{\ZSHEn{t}{x}{n}}{ p(x;t)^n} = \mathcal{E}\left[:{\rm exp}:\, \left\{\sum_{i=1}^{n} \int_{0}^{t} \whitenoise(s,b_i(s)) \dd s \right\} \right]
\end{equation}
where $\mathcal{E}$ is the expectation of the law on $n$ (independent) Brownian bridges $\{b_i\}_{i=1}^{n}$ starting at $0$ at time $0$ and ending at $x$ at time $t$, conditioned not to intersect. The $:{\rm exp}:$ is the {\it Wick exponential}. Intuitively these path integrals represent energies of non-intersecting paths, and thus the expectation of their exponential represents the partition function for this path (or directed polymer) model. The partition function expression above requires interpretation since it does not make sense to take a path integral through a space-time white noise along Brownian trajectories, or to exponentiate the result.

Let us focus for the moment on just the partition function corresponding to $n=1$ (i.e. the stochastic heat equation). There are many approaches to make sense of (\ref{wick}) (see Section 4 of \cite{CorwinReview} for a review). The chaos series definition we have taken in (\ref{Zpartfunc}) can be considered as the result of expanding the exponential as a Taylor series formally, and then evaluating the $\mathcal{E}$ expectation value of each term. It is also possible to make sense of this Wick exponential by turning to discrete~\cite{AKQ2} or semi-discrete \cite{MRQ} versions of the polymer model, or by smoothing the white noise in space~\cite{BC}. In each of these cases, as the discrete mesh or smoothing goes to zero, and after a suitable renormalization (which one should think of as the meaning of the Wick exponential), these models all converge to the same process defined by the chaos series in (\ref{Zpartfunc}). These results represent a form of universality of the partition function, because the
discretization or smoothing mechanism plays essentially no role in the limiting object.

In Section \ref{s.MRQ}, Proposition \ref{QMthm} quotes the result of \cite{Nica} which shows the above mentioned convergence for $n=1$ of the O'Connell-Yor semi-discrete Brownian polymer partition function to its continuum analog (i.e. stochastic heat equation). Below we conjecture the convergence for all $n\in \N$. This conjecture is based on the $n=1$ result, as well as the convergence of moments known from \cite[Section 5.4.2]{BorCor}. Moreover, it is reasonable given the formal convergence of the polymer interpretation.

\begin{conjecture}\label{OWconjecture}
Recall $\ZSD{N}{n}{t}{x}$ from Definition \ref{d.cntx}. As $N\to \infty$,
$$
\ZSD{N}{n}{t}{x}\to \ZSHEn{t}{x}{n}
$$
in the sense of finite dimensional distributions or as a process as $x\in \R$ and $n\in \N$ vary.
\end{conjecture}
A consequence of this conjecture and the sequential compactness of Theorem \ref{seqncmpt} would be the uniqueness of subsequential limits for $\ZSD{N}{n}{t}{x}$ as $N\to \infty$ and thus the identification of this limit with O'Connell-Warren's $\ZSHEn{t}{x}{n}$. In other words, it would show that for $t$ fixed $\HKPZline{t}{n}$ is equal as a process in $x\in \R$ and $n\in \N$ to
\begin{equation}\label{e.sagasgf}
\log\bigg(\frac{\ZSHEn{t}{x}{n}}{\ZSHEn{t}{x}{n-1}}\bigg)
\end{equation}
with the convention that $\ZSHEn{t}{x}{0}\equiv 1$. This would show that $\ZSHEn{t}{x}{n}$ is almost surely everywhere continuous in $x$ (for each fixed $n$) and everywhere positive as $x\in \R$ and $n\in \N$ vary. A direct proof of the continuity and positivity for $\ZSHEn{t}{x}{n}$ is presently unavailable. It would also be interesting to provide a direct proof (at the continuum level) of the $\Ham_1$-Brownian Gibbs property for (\ref{e.sagasgf}) as well as the uniform control of the Radon-Nikodym derivative of the scaled solution as $t$ varies in $[1,\infty)$.

\begin{rem}
Subsequent to publishing this paper, the above conjecture was proved in \cite{Nica}.
\end{rem}

\section{O'Connell-Yor semi-discrete Brownian polymer}\label{semidiscete}

The O'Connell-Yor semi-discrete Brownian polymer model serves as an approximation for the continuum directed random polymer, and hence for the KPZ equation. In this section, we follow results of O'Connell \cite{OCon}. We define the O'Connell-Yor polymer partition function line ensemble and the free energy line ensemble (defined via the logarithm of ratios of consecutively indexed curves), and show that the free energy line ensemble has the $\Ham_1$-Brownian Gibbs property (Proposition \ref{NeilGibbsprop}). Then we recount the result (Proposition \ref{QMthm}) of Moreno Flores-Quastel-Remenik \cite{MRQ}  --  a special case of the results of [Nica] -- that under suitable centering and scaling, the lowest indexed curve of this ensemble converges to the fixed time solution of the narrow wedge initial data KPZ equation. We close the section by stating Theorem \ref{seqncmpt} which shows sequential compactness and certain aspects of uniformity in $t$ of these centered and scaled line ensembles. This theorem is a key step towards proving Theorem \ref{mainthm}.

\subsection{Line ensembles related to the O'Connell-Yor polymer}\label{OConYor}

\begin{definition}
For $S$ a closed interval in $[0,\infty)$ and $M$ an interval of $\Z$, an {\it up/right path} in $S \times M$ is the range of a non-decreasing, surjective function mapping $S$ into $M$. Fix $N\in \N$ and $s>0$. For each sequence $0<s_1<\cdots<s_{N-1}<s$ we can associate an up/right path $\phi$ in $[0,s] \times \{1,\ldots,N \}$ that is the range of the unique non-decreasing and surjective map $[0,s] \to \{ 1,\ldots,N \}$ whose set of jump times is $\{ s_i\}_{i=1}^{N-1}$. Let $B(\cdot)=\big(B_1(\cdot),\ldots, B_N(\cdot)\big)$ be standard Brownian motion in $\R^N$ and define
\begin{equation*}\glossary{$E(\phi)$, Path integral of Brownian motions along an up/right path $\phi$}
E(\phi) := B_1(s_1) +\big(B_2(s_2)-B_2(s_1)\big) + \cdots + \big(B_{N}(s)-B_{N}(s_{N-1})\big).
\end{equation*}
The {\it O'Connell-Yor polymer partition function line ensemble} is a $\{1,\ldots, N\}\times\R_{> 0}$-indexed line ensemble $\big\{Z_{n}^N(s):n\in \{1,\ldots,N\}, s> 0\big\}$ defined by
\begin{equation}\label{Neilparthier}\glossary{$\OConZ{N}{n}(s)$, O'Connell-Yor polymer partition function line ensemble}
\OConZ{N}{n}(s):= \int_{D_n(s)} e^{\sum_{i=1}^{n} E(\phi_i)} \dd\phi_1 \cdots \dd\phi_n,
\end{equation}
where the integral is with respect to Lebesgue measure on the Euclidean set $D_n(s)$ of all $n$-tuples of non-intersecting (disjoint) up/right paths with initial points $(0,1),\ldots,(0,n)$ and endpoints $(s,N-n+1),\ldots, (s,N)$. For notational simplicity set $\OConZ{N}{0}(s) \equiv 1$. See Figure \ref{f.OCONYOR} for an illustration of this construction.

The {\it O'Connell-Yor polymer free energy line ensemble} is a $\{1,\ldots, N\}\times\R_{> 0}$-indexed line ensemble $\big\{\OConX{N}{n}(s):n\in \{1,\ldots,N\}, s> 0\big\}$ defined by
\begin{equation}\label{freeenergy}\glossary{$\OConX{N}{n}(s)$, O'Connell-Yor polymer free energy line ensemble}
\OConX{N}{n}(s) = \log\left(\frac{\OConZ{N}{n}(s)}{\OConZ{N}{n-1}(s)}\right).
\end{equation}
\end{definition}

\begin{figure}
\centering\epsfig{file=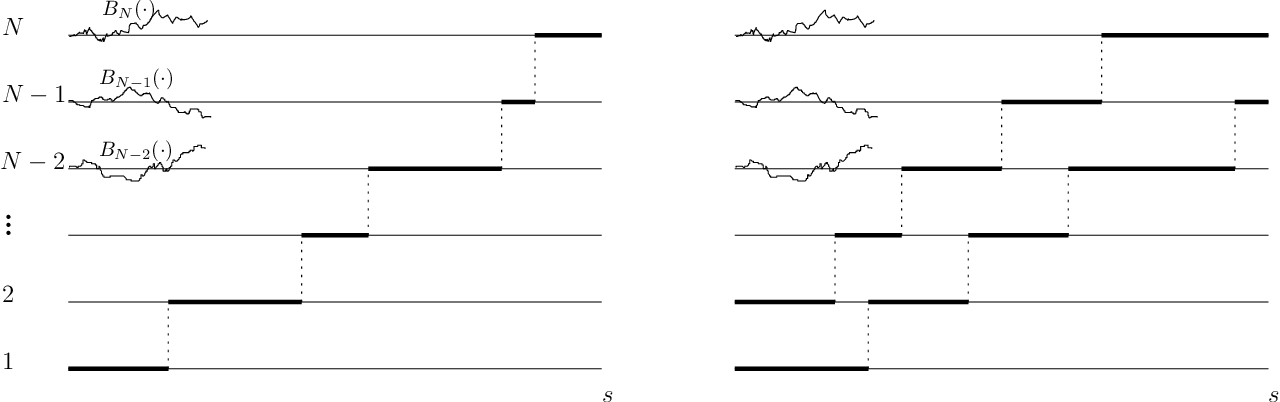, width=14cm}
\caption{The construction of the O'Connell-Yor polymer partition function line ensemble. Left: For $n=1$ the Brownian bridges are summed along the trajectory of a single path (in bold) from $(0,1)$ to $(s,N)$. Right: For $n=2$ the Brownian bridges are summed along the trajectories of two non-intersecting paths (in bold), the first from $(0,1)$ to $(s,N-1)$ and the second from $(0,2)$ to $(s,N)$. The partition function comes from exponentiating these sums and then integrating over the Lebesgue measure on all such sets of paths. The same Brownian motions are used for all $n$ in this construction.}\label{f.OCONYOR}
\end{figure}

The free energy line ensemble can be considered to be a finite temperature analog (or geometric lifting) of the line ensemble $\big\{\OConM{N}{n}(s):n\in \{1,\ldots,N\}, s> 0\big\}$ defined by
\begin{equation}\label{e.OConM}\glossary{$\OConM{N}{n}(s)$, Dyson Brownian motion line ensemble}
\OConM{N}{n}(s) := \max_{(\phi_1,\ldots, \phi_n)\in D_n(s)}\bigg( \sum_{i=1}^{n} E(\phi_i)\bigg) \, -\, \max_{(\phi_1,\ldots, \phi_{n-1})\in D_{n-1}(s)} \bigg(\sum_{i=1}^{n-1} E(\phi_i)\bigg),
\end{equation}
with the convention that for $n=1$, the second term in the right-hand side above is absent. In fact, by inserting an inverse temperature parameter $\beta$ in the exponential in (\ref{Neilparthier}) and calling the result $\OConX{N,\beta}{n}(s)$ one finds that, as $\beta\to \infty$, $\OConX{N,\beta}{n}(s)/\beta \to \OConM{N}{n}(s)$ (i.e. the zero temperature limit).

It was shown in \cite{Bar,BJ,GTW,OConYor2} that the process $\OConMshort{N}(s) = \big(\OConM{N}{1}(s),\ldots, \OConM{N}{N}(s)\big)$ is Markov and coincides with the (GUE) Dyson Brownian motion. That is to say, it has the same law as Brownian motion started from $(0,\ldots, 0)$ conditioned (in the sense of Doob) to never exit the Weyl chamber $W_N=\{x\in \R^N: x_1>\cdots >x_N\}$. This is a diffusion on $\overline{W}_N$ (the closure of $W_N$) with infinitesimal generator
\begin{equation*}
\frac{1}{2} h^{-1} \Delta h = \Delta/2 + \nabla \log h \cdot \nabla, \qquad \textrm{where} \quad h(x) = \prod_{1\leq i<j\leq N} (x_i-x_j),
\end{equation*}
where $\Delta$ is the Dirichlet Laplacian with killing on the boundary of the Weyl chamber $W_N$. The degenerate starting location $(0,\ldots, 0)$ necessitates an entrance law for this process which is given by a scaling of the family of measures corresponding to the eigenvalue distribution of the $N\times N$ Gaussian unitary ensemble. On account of being the Doob-$h$ transform of the Dirichlet Laplacian on $W_N$, the diffusion displays a non-intersecting Brownian Gibbs property (equivalent to a $\Ham_{+\infty}$-Brownian Gibbs property -- see Sections \ref{s.egweg} and \ref{s.airy}). This type of Gibbs property is the central tool  of \cite{CH}.

O'Connell \cite{OCon} discovered that the finite temperature process $\OConXshort{N}(\cdot) = \big(\OConX{N}{1}(\cdot), \cdots, \OConX{N}{N}(\cdot)\big)$ has an analogous diffusion description as a Doob-$h$ transform, where $h(x)$ is no longer the Vandermonde determinant but rather a certain special function called the class one $\mathfrak{gl}_N$-Whittaker function, and where the Dirichlet Laplacian $\Delta$ is replaced by the quantum Toda lattice Hamiltonian $\QTLH$ (\ref{e.QTLHgen}). Note that the term Hamiltonian here has a different meaning than in the context of the definition of the $\Ham$-Brownian Gibbs property (Definition \ref{maindefHBGP}) -- here it is an operator. As explained in \cite{BOCon,Katori,Katori2}, this means that the process $\OConXshort{N}(\cdot)$ still has a description as a Brownian motion conditioned (in the sense of Doob) on the distribution of the terminal values of certain exponential functions of the path trajectory. We record the result of \cite{OCon} to which we appeal. This
result is one of the three key inputs to Theorem \ref{mainthm} (see Section \ref{s.kpzlovers} or Figure \ref{seqncmptfig} for an outline of these inputs).

\begin{proposition}[Theorem 3.1 and Corollary 4.1 of \cite{OCon}]\label{Neilprop}
The stochastic process $\OConXshort{N}(\cdot)$ has the same law as the diffusion process in $\R^N$ with infinitesimal generator
\begin{equation*}
\QTLHgen = \tfrac{1}{2} \psi_0^{-1} \QTLH \psi_0 = \tfrac{1}{2} \Delta + \nabla \log \psi_0 \cdot \nabla,
\end{equation*}
started according to a certain entrance law for $\mu_s(\dd x)$ (Remark \ref{r.belowme}) for $s>0$.
Here $\Delta$ is the Laplacian on $\R^N$, $\nabla$ is the gradient on $\R^N$, $\QTLH$ is the {\it quantum Toda lattice Hamiltonian} on $\R^N$,
\begin{equation}\label{e.QTLHgen}
\QTLH = \Delta - 2\sum_{i=1}^{N-1} e^{x_{i+1}-x_i},
\end{equation}
and $\psi_0$ is the class one $\mathfrak{gl}_N$-Whittaker function.%\footnote{This is a positive harmonic function for $\QTLH$ (i.e. $\psi_0(x)>0$ for all $x$ and $\QTLH\psi_0 =0$).}.
\end{proposition}

\begin{rem}\label{r.belowme}
The exact form of the entrance law $\mu_s(\dd x)$ is not necessary for our purposes, but for completeness we note that it is given by
\begin{equation*}
\mu_s(\dd x) = \psi_0(x) \vartheta_s(x)\dd x, \qquad \vartheta_s(x) := \int_{\sqrt{-1} \R^N} \psi_{-\lambda}(x) \prod_{i=1}^{N} e^{\lambda_i^2 t/2} s_N(\lambda) \dd\lambda.
\end{equation*}
Here $\lambda= (\lambda_1,\ldots,\lambda_N)$, $s_N(\lambda)$ is called the Sklyanin measure
\begin{equation*}
s_N(\lambda) := \frac{1}{(2\pi \sqrt{-1})^N N!} \prod_{j\neq k} \Gamma(\lambda_j-\lambda_k)^{-1},
\end{equation*}
and the $\psi_{\lambda}$ are also eigenfunctions of $\QTLH$ such that $\QTLH\psi_{\lambda} = \left(\sum_{i=1}^{n}\lambda_i^2\right) \psi_{\lambda}$. The $\psi_{\lambda}$ are {\it class one $\mathfrak{gl}_N$ Whittaker functions} and have nice integral expressions which are summarized in \cite{OCon}. The $\psi_{0}$ from Proposition \ref{Neilprop} is $\psi_{\lambda}$ with $\lambda = (0,\ldots, 0)$.
\end{rem}

\subsection{Gibbs property and the quantum Toda lattice Hamiltonian}\label{QTLHgibbs}
The following proposition follows by specializing the theory of killing and conditioning, summarized in Section \ref{genthsec}, and the result of Proposition \ref{Neilprop}. Roughly speaking, the results of Section \ref{genthsec} explain how the Doob-$h$ transform of a free Markov generator with a potential inherits a Gibbs property in which paths are resampled according to the free Markov generator, subject to a Hamiltonian given by the potential.

\begin{proposition}\label{NeilGibbsprop}
For any $\delta>0$ the restriction to times $s\geq \delta$ of the stochastic process $\OConXshort{N}(s)$ defined in Proposition \ref{Neilprop} may be regarded as a line ensemble from $\{1,\ldots, N\}\times [\delta,\infty)\to \R$. This line ensemble has the $\Ham_1$-Brownian Gibbs property (Definition \ref{maindefHBGP}).
\end{proposition}
\begin{proof}
This result can essentially be extracted from work of Katori \cite{Katori,Katori2}. We provide an explicit proof, which is informed by a discussion with Neil O'Connell. We appeal to the general discussion from Section \ref{genthsec}, and specialize the theory explained there. For $\lambda=(\lambda_1>\cdots >\lambda_N)\in \R^N$, consider $L=\Delta/2 + \lambda \cdot \nabla$ (the generator for Brownian motion in $\R^N$ with drift vector $\lambda$), $V= \sum_{i=1}^{N-1}\Ham_1(x_{i+1}-x_{i})$ (a potential) and $L^{(V)} = L-V$. It is shown in \cite[Section 2.1]{BOCon} (see also \cite[(8)]{OCon2}) that
$$
f_{\lambda}(x) = e^{\lambda \cdot x} \,\EE_{0,x}^{L} \left[\exp\left\{-\int_{0}^{\infty} V(X_s)ds\right\}\right]
$$
is the unique solution to the equation $L^{(V)} f_{\lambda}(x) =0$ such that $e^{-\lambda \cdot x}f_{\lambda}(x)$ is bounded and goes to one as $x_i-x_{i+1}\to +\infty$ for each $i\in \{1,\ldots, N-1\}$. (Here, as in Section \ref{genthsec}, $X_s$ is a sample path for the diffusion with generator $L$.) Furthermore, it is shown therein that $f_{\lambda}(x)$ is equal to $e^{-\lambda \cdot x} \psi_{\lambda}(x)$, where in $\psi_{\lambda}(x)$ is the class one $\mathfrak{gl}_N$-Whittaker function corresponding to index $\lambda$. The present Section \ref{genthsec} then implies that $L^{(V,f_{\lambda})} = f_{\lambda}^{-1} L^{(V)} f_{\lambda}$ is the infinitesimal generator for a Markov process, the sample paths of which can be considered as a line ensemble. Since $L$ and $V$ both factor, we may conclude from (\ref{RNFKeqnSpec}) that this line ensemble has the $\Ham_1$-Brownian Gibbs property. The Markov process associated with $L^{(V,f_{\lambda})}$ is the result of conditioning the process $X_s$ to survive for all time $s\geq 0$, given that it is killed at a rate $V(X_s)$. This conditioning is well defined since there is a positive probability of survival (given by $e^{-\lambda\cdot x} f_{\lambda}(x)$). The $\Ham_1$-Brownian Gibbs property follows readily from this interpretation (the details of which are in Section \ref{genthsec}).

Now we take the drift vector $\lambda \to (0,\ldots, 0)$. In this limit $L\to \Delta/2$, $V$ remains unchanged, $f_{\lambda}(x) \to \psi_0(x)$ and hence (identifying with notation of Proposition \ref{Neilprop}) $L^{(V)} \to \tfrac{1}{2}\QTLH$ and $L^{(V,h)}\to \QTLHgen$. Likewise, as $\lambda\to 0$, the measure on sample paths of $L^{(V,h)}$ converges weakly to the corresponding measure for $\QTLHgen$. The $\Ham_1$-Brownian Gibbs property survives this limit as well (see Proposition \ref{HBrownianGibbsLimitProp}). Thus the present proposition follows by applying Proposition \ref{Neilprop} to identify the diffusion related to $\QTLHgen$ with $\OConXshort{N}(s)$.
\end{proof}
%Old proof
%\begin{proof}
%Specialize the general theory of Section \ref{genthsec} to the case where $L=\Delta/2$ (the generator for Brownian motion in $\R^N$), and $V= \sum_{i=1}^{N-1}\Ham_1(x_{i+1}-x_{i})$. Then $L^{(V)} = \tfrac{1}{2}\QTLH$. The unique positive harmonic function for this operator is the class one $\mathfrak{gl}_N$-Whittaker function $h(x) = \psi_0(x)$ (see \cite{BOCon,OCon2}). This implies that $L^{(V,h)}=\QTLHgen$ and, as $L$ and $V$ both factor, we may conclude from (\ref{RNFKeqnSpec}) that the resulting line ensemble has the $\Ham_1$-Brownian Gibbs property. Thus the proposition follows by applying Proposition \ref{Neilprop}.
%\end{proof}

%The process $\OConXshort{N}(s)$ may be understood \cite{Katori} as the result of conditioning $N$ Brownian motions killed in the potential $V$ above on surviving forever.

\subsection{Convergence of lowest indexed curve to narrow wedge initial data KPZ equation}\label{s.MRQ}

The following result was proved in unpublished work of Moreno Flores, Quastel and Remenik~\cite{MRQ}, and is a special case of the results of~\cite{Nica}.  The result is related to work of Alberts-Khanin-Quastel \cite{AKQ2} in the context of fully discrete directed polymers. Roughly speaking, the (suitably centered and scaled) semi-discrete partition function is expanded into a chaos series in the semi-discrete space-time white noise (formed by the Brownian motions $B_1,\ldots, B_N$) whose terms are proved to converge in the $N \to \infty$ limit to analogous expressions coming from the chaos series for the continuum directed random polymer partition function (i.e. the narrow wedge initial data stochastic heat equation). The semi-discrete polymer partition function can be seen as solving a semi-discrete stochastic heat equation (cf. Section 6 of \cite{BCS}). The work of \cite{MRQ} deals with convergence of the lowest indexed curve, though we state centering and scaling under which all curves should have limits (Conjecture \ref{OWconjecture}).

\begin{definition}\label{d.cntx}
Fix $N\in \N$ and for $n\in \{1,\ldots, N\}$, set
\begin{equation*}
C(N,t,x) := \exp\Big\{N+ \frac{\sqrt{tN}+x}{2} + xt^{-1/2}N^{1/2}\Big\} (t^{1/2}N^{-1/2})^{N-1}.
\end{equation*}
For $x> -\sqrt{tN}$ define
\begin{equation}\label{rescaledOConWar}\glossary{$\ZSD{N}{n}{t}{x}$, O'Connell-Yor polymer partition function line ensemble scaled so lowest indexed curve limits to stochastic heat equation}
\ZSD{N}{n}{t}{x} := \frac{\OConZ{N}{n}(\sqrt{tN} + x)}{C(N,t,x)^n t^{-n(n-1)/2}\prod_{i=0}^{n-1} i!}
\end{equation}
%where $x$ is such that the argument of $\OConZ{N}{n}$ is positive.
and let $\HSDlinet{t}{N}:=\big\{\HSDline{t}{N}{n}\big\}_{n\in \N}$ be the $\{1,\ldots, N\}\times(-\sqrt{tN},\infty)$-indexed line ensemble where each curve $\HSDline{t}{N}{n}:(-\sqrt{tN},\infty)\to \R$ is given by
\begin{equation*}\glossary{$\HSDline{t}{N}{n}(x)$, O'Connell-Yor polymer free energy line ensemble scaled so lowest indexed curve limits to KPZ equation}
\HSDline{t}{N}{n}(x) := \log \left(\frac{\ZSD{N}{n}{t}{x}}{\ZSD{N}{n-1}{t}{x}}\right) = \OConX{N}{n}(\sqrt{tN} +x) - \log \big(C(N,t,x)  t^{1-n}(n-1)!\big).
\end{equation*}

Define the scaled $\{1,\ldots,N\}\times(-t^{-1/6} N^{1/2} ,\infty)$-indexed line ensemble $\HSDFPlinet{t}{N}:=\big\{\HSDFPline{t}{N}{n}\big\}_{ n\in \N}$ with $\HSDFPline{t}{N}{n}:(-t^{-1/6}N^{1/2} , \infty)\to \R$ via the relation
\begin{equation*}\glossary{$\HSDFPline{t}{N}{n}(x)$, Scaled O'Connell-Yor polymer free energy line ensemble}
\HSDline{t}{N}{n}(x) = - \frac{t}{24}  + t^{1/3} \HSDFPline{t}{N}{n}(t^{-2/3} x).
\end{equation*}

We will refer to $\HSDlinet{t}{N}$ as the {\it finite $N$ KPZ$_t$ line ensemble} and to $\HSDFPlinet{t}{N}$ as the {\it finite $N$ scaled  KPZ$_t$ line ensemble}.
\end{definition}

The following is an immediate corollary of Proposition \ref{NeilGibbsprop} and the definition of $\HSDlinet{t}{N}$. We will make extensive use of this corollary -- so much so that we will not always refer to it.

\begin{corollary}\label{rescaledHBGP}
For any $t>0$ and $N\geq 1$, the line ensemble $\HSDlinet{t}{N}$ has the $\Ham_1$-Brownian Gibbs property and the line ensemble $\HSDFPlinet{t}{N}$ has the $\Ham_t$-Brownian Gibbs property (Definition \ref{maindefHBGP}).
\end{corollary}

The following result (proved in \cite{Nica}) is one of the three key inputs identified in Figure \ref{seqncmptfig} and the discussion of Section~\ref{s.kpzlovers}. It shows how (for $t$ fixed) the lowest indexed curve of  $\HSDlinet{t}{N}(x)$ converges (weakly as a process in $x$) to $\HKPZnw{t}{x}$. In fact, in proving Theorem \ref{mainthm}, we only need finite dimensional convergence at various values of $x$. Conjecture \ref{OWconjecture}  describes the putative limit of the higher indexed curves of~$\HSDlinet{t}{N}$.

\begin{proposition}\cite{MRQ}\label{QMthm}
Fix $t,T>0$. As $N\to \infty$, as a process in $x\in [-T,T]$, $\ZSD{N}{1}{t}{x}$ converges\footnote{That is, weak convergence with respect to the topology on continuous curves that corresponds to uniform convergence on compact subsets.}  to $\ZSHEnw{t}{x}$, the solution to the narrow wedge initial data stochastic heat equation. Equivalently, as $N\to \infty$, as a process in $x\in [-T,T]$, $\HSDline{t}{N}{1}(x)$ converges to $\HKPZnw{t}{x}$, the solution to the narrow wedge initial data KPZ equation.
\end{proposition}

The next lemma is an immediate consequence of Propositions \ref{QMthm} and \ref{ACQprop}.
\begin{lemma}\label{l.input}
For each $x_0 > 0$ there exists $N_0(x_0) \in \N$ such that the collection of random variables
$\big\{ \HSDFPline{t}{N}{1}(x) + x^2/2  : t \geq 1, x \in [-x_0,x_0], N \geq N_0(x_0) \big\}$ is tight.
\end{lemma}
%\begin{proof}
%The result follows from the weak convergence of the process $\hfixedpt_1^{t,N}(x) + x^2/2:[-x_0,x_0] \to \R$
%to the limiting process  $\hfixedpt_1^t(x) + x^2/2:[-x_0,x_0] \to \R$ assured by [Quastel,Moreno Flores], the stationarity of this limiting process, and the convergence in law of $\hfixedpt_1^t(0)$ to the Tracy-Widom distribution as $t \to \infty$ known by [ACQ].
%\end{proof}

%\begin{corollary}[Of the above proposition and ACQ]
%$\HSDFPline{t}{N}{1}(x)$ has one-point convergence as $N\to \infty$ and the limit has itself a limit in $t\to \infty$.
%\end{corollary}

\begin{rem}
In the published version of the paper, Definition \ref{d.cntx} contained two minor errors that we correct above. In particular, \eqref{rescaledOConWar} previously defined $\ZSD{N}{n}{t}{x}$ as
\begin{equation*}
\ZSD{N}{n}{t}{x} := \frac{\OConZ{N}{n}(\sqrt{tN} + x)}{C(N,t,x)^n},
\end{equation*}
leaving off the term $t^{-n(n-1)/2}\prod_{i=0}^{n-1} i!$ which now appears in the denominator. The mistake filters in the definition of $\HSDline{t}{N}{n}(x)$ which previously was given by
\begin{equation*}
\HSDline{t}{N}{n}(x) := \log \left(\frac{\ZSD{N}{n}{t}{x}}{\ZSD{N}{n-1}{t}{x}}\right) = \OConX{N}{n}(\sqrt{tN} +x) - \log \big(C(N,t,x)\big),
\end{equation*}
leaving off the factor $ t^{1-n}(n-1)!$ in the last log term. These omitted factors are constants for $t$ fixed, and hence do not change any arguments. \cite{Nica} discovered these omitted factors when proving Conjecture \ref{OWconjecture}. When $n=1$ the additional factors are not present.
\end{rem}

\subsection{Sequential compactness and uniform control of the Radon-Nikodym derivative}\label{seqncmptsec}
Here we record a result (proved later in Section \ref{s.seqncmptproof}) that shows how by restricting our finite $N$ scaled KPZ$_t$ line ensemble to $x$ in a fixed interval $[-T,T]$ and to $n$ in a fixed set $\{1,\ldots, k\}$, we have sequential compactness and uniform control in $t$ over the normalizing constant (and thus the Radon-Nikodym derivative)  as $N\to \infty$. This is a key step in our construction of a KPZ$_t$ line ensemble (which we extract via subsequences) and hence in the proof of Theorem \ref{mainthm}.

\medskip

\begin{theorem}\label{seqncmpt}
Fix $t>0$.
\begin{enumerate}
\item For any $k\geq 1$ and $T>0$, the restriction of the line ensemble $\HSDFPlinet{t}{N}$ given by $\big\{\HSDFPline{t}{N}{n}(x):n\in \{1,\ldots,k\}, x\in [-T,T]\big\}$ is sequentially compact as $N$ varies.
\item For all $\e>0$ and $T>0$, there exists $\delta=\delta(\e,T)>0$ such that, for any $t\geq 1$ and $N\geq N_0(t,\e,T)$,
$$
\PP\left(\partfunc{1}{1}{(-T,T)}{\HSDFPline{t}{N}{1}(a)}{\HSDFPline{t}{N}{1}(b)}{+\infty}{\HSDFPline{t}{N}{2}}{\Ham_t}<\delta \right)\leq \e.
$$
(Recall that the normalizing constant is given in Definition \ref{maindefHBGP}.)
\end{enumerate}
\end{theorem}

We prove this theorem in Section \ref{s.seqncmptproof}.

\section{Proof of Theorem \ref{mainthm} applications}\label{appsproofs}

We provide proofs of Theorems \ref{t.univonethird}, \ref{app2} and \ref{t.tailbounds} (recall Theorem \ref{abscontThm} is proved after its statement). These rely heavily upon the constructed KPZ line ensemble in Theorem \ref{mainthm} as well as the input of the one-point narrow wedge initial data KPZ equation information contained in Proposition~\ref{ACQprop}.

\subsection{Preliminaries}
By Theorem \ref{mainthm}(1) $\HKPZlinet{t}$ (and $\HKPZFPlinet{t}$) is related to the narrow wedge initial data KPZ equation (and its scaled version) via the following equalities in distribution:
\begin{equation}\label{e.identifywith}
\HKPZline{t}{1}(\cdot) \stackrel{(d)}{=} \HKPZnw{t}{\cdot},\qquad\textrm{and}\qquad  \HKPZFPline{t}{1}(\cdot) \stackrel{(d)}{=} \HKPZFPlinetnw{t}(\cdot).
\end{equation}
%Here $\HKPZnw{t}{\cdot}$ is the time $t$ narrow wedge initial data KPZ equation (Definition \ref{SHEdef}).

In this section we will generally work with $\HKPZFPlinet{t}$ as its scaling is well adapted for some of our large~$t$ applications.

\begin{lemma}\label{e.tildekappalem}
For any $\nu>0$ and $\e>0$, there exists a constant $C$ such that, for all $t\geq 1$,
\begin{equation}\label{eree}
\PP\Big(\HKPZFPline{t}{1}(y) + y^2/2 < C +\nu y^2 \textrm{  for all } y\in\R\Big) \geq 1-\e.
\end{equation}
\end{lemma}
\begin{proof}
Consider for a moment the effect of restricting (\ref{eree}) to $y\in \Z$. This result would then follow easily from Proposition \ref{ACQprop}(2) (as well as from the stationarity in $y$ of $\HKPZFPline{t}{1}(y) +y^2/2$ provided by Proposition \ref{ACQprop}(1)). After all, for $C$ large enough and $y$ fixed, these propositions imply that
$$
\PP\Big(\HKPZFPline{t}{1}(y) +y^2/2 \geq C +\nu y^2\Big) \leq \e\cdot  e^{-\nu y^2}.
$$
The summation of $e^{-\nu y^2}$ over $y\in \Z$ is bounded by a constant, so redefining $\e$ to absorb this constant would yield (\ref{eree}), up to this replacement of $\R$ with $\Z$.
However, large height spikes at random times may cause a uniformity that is available over $\Z$ to fail over $\R$.

The $\Ham_t$-Brownian Gibbs property enjoyed by $\HKPZFPlinet{t}$ (due to Theorem \ref{mainthm}(2)) affords us the necessary regularity to extend this control over $\Z$ to all of $\R$.
The result in (\ref{eree}) is in the same flavor as \cite[Proposition 4.4]{CH}. However, to implement the argument used in \cite{CH} we   would need to invoke a lower tail bound on  $\HKPZFPline{t}{1}(0)$ which is uniformly controlled as $t\to \infty$. Proposition \ref{ACQprop}(3) records the only available lower tail bound and since that does not behave uniformly as $t\to \infty$, we develop here a different argument which relies only on the tightness of $\HKPZFPline{t}{1}(0)$ for $t\geq 1$.

In order to prove (\ref{eree}) it suffices to show that, for $C$ large enough,

\begin{equation}\label{e.asfwwr}
\PP\Big(\HKPZFPline{t}{1}(y) + y^2/2 \geq  C +\nu y^2 \textrm{  for some } y\in [-1,1]\Big) \leq \frac{\e}{2},
\end{equation}
and that
\begin{equation}\label{e.deisft}
\PP\Big(\HKPZFPline{t}{1}(y) + y^2/2 < C +\nu y^2 \textrm{  for all } y\in\R\setminus [-1,1]\Big) \geq 1-\frac{\e}{2}.
\end{equation}
Observe that (\ref{e.asfwwr}) is proved by appealing to Theorem \ref{mainthm}(3) to see that, on $[-1,1]$, the curve $\HKPZFPline{t}{1}(\cdot)$ is absolutely continuous with respect to Brownian bridge, with a Radon-Nikodym derivative which is tight for $t\geq 1$. This and the tightness and stationarity in $y$ of the one-point distribution of $\HKPZFPline{t}{1}(y)+y^2/2$ (Proposition \ref{ACQprop}(1)) implies that, for $C$ large enough, (\ref{e.asfwwr}) holds. Thus, it remains only to prove (\ref{e.deisft}).

Let $\Z^*=\Z\setminus\{-1,0\}$. For $\ell\in \R$, define the event
$$\mathsf{A}_\ell = \Big\{\HKPZFPline{t}{1}(0)\geq -\ell\Big\}$$
and for $C\in \R$ and $n\in \Z^*$, define
$$\mathsf{E}_{C}(n)= \Big\{ \HKPZFPline{t}{1}(y) +y^2/2 \geq C +\nu y^2 \textrm{  for some } y\in [n,n+1]\Big\}.$$
Notice that
\begin{equation}\label{sfaee}
\textrm{LHS }(\ref{e.deisft}) = 1- \PP\bigg(\bigcup_{n\in \Z^*}\mathsf{E}_{C}(n)\bigg),
\end{equation}
so, to prove (\ref{e.deisft}), it suffices to show that, for $C$ large enough,
\begin{equation}\label{fsfsaee}
\PP\bigg(\bigcup_{n\in \Z^*}\mathsf{E}_{C}(n)\bigg)\leq \frac{\e}{2}.
\end{equation}

Let $\mathcal{F}_0$ be the sigma-field generated by $\HKPZFPline{t}{1}(0)$ as well by all curves $\HKPZFPline{t}{n}$ for $n\geq 2$. Then, using conditional expectations,
\begin{equation}\label{sfaeeff}
\PP\bigg(\bigcup_{n\in \Z^*}\mathsf{E}_{C}(n)\bigg) = \EE\Big[ \EE\big[\mathbf{1}_{\cup_n \mathsf{E}_{C}(n)} \,\big\vert\, \mathcal{F}_0\big]\Big].
\end{equation}

The union bound implies that, $\PP$-almost surely,
\begin{equation}\label{fsfss}
\EE\big[\mathbf{1}_{\cup_n \mathsf{E}_{C}(n)} \,\big\vert\, \mathcal{F}_0\big] \leq
\EE\Big[ \sum_{n\in \Z^*}\mathbf{1}_{\mathsf{A}_\ell}\,\cdot\,  \mathbf{1}_{\mathsf{E}_{C}(n)}\,\Big\vert\, \mathcal{F}_0\Big]+ \mathbf{1}_{\mathsf{A}^c_\ell}.
\end{equation}
The key purpose of using these conditional expectations is that we want to perform the union bound over $n$ only after having fixed or conditioned on our knowledge of $\HKPZFPline{t}{1}(0)$. The reason for this is that we only know the tightness of $\HKPZFPline{t}{1}(0)$ over $t\geq 1$. In order to bound the probability of the union of the $\mathsf{E}_{C}(n)$ we will condition on reasonable behavior of $\HKPZFPline{t}{1}(0)$, i.e. the event $\mathsf{A}_{\ell}$. If we took the union bound for the overall expectation, then each summand in $n$ would come with an error coming from the probability of the complement of $\mathsf{A}_{\ell}$. Summing over $n$ these errors would diverge. This problem could be remedied if we had a good tail bound on $\HKPZFPline{t}{1}(0)$ which was uniform in $t\geq 1$, but we do not have such an estimate. By taking the union bound inside the conditional expectation, the error in assuming that the event $\mathsf{A}_\ell$ holds is only incurred once.

In order to control the right-hand side of (\ref{fsfss}), we will take what may seem to be a slight diversion, but which will lead to the important bound~(\ref{e.asbgweeg}).

Consider $n\in \Z^*$. For $n> 0$, define $\sigma_n$ to be the supremum of those $y\in [n,n+1]$ such that $\HKPZFPline{t}{1}(y) +y^2/2 \geq C +\nu y^2$; and for $n<-1$ define $\sigma_n$ to be the infimum of those $y\in [n,n+1]$ such that $\HKPZFPline{t}{1}(y) +y^2/2 \geq C +\nu y^2$. In both cases, if the set of such $y$ is empty, then set $\sigma_n=0$. The event $\mathsf{E}_{C}(n)$ is equivalent to $\{\sigma_n\neq 0\}$. We will assume for now that $n>0$, although similar results hold by symmetry for $n<-1$ and will be needed to finish the proof. Given $n>0$, the interval $[0,\sigma_n]$ is a $\{1\}$-stopping domain for the line ensemble $\HKPZFPlinet{t}$ (recall the notation of a $K$-stopping domain from Definition \ref{defstopdom}).

Define the event
$$
\mathsf{B}_{C,\ell}(n) = \bigg\{ \HKPZFPline{t}{1}(n) \geq \frac{\sigma_n-n}{\sigma_n}(-\ell) + \frac{n}{\sigma_n} \big(C+\nu \sigma_n^2 - \sigma_n^2/2\big)\bigg\}.
$$
Observe that the $\Ham_t$-Brownian Gibbs property enjoyed by $ \HKPZFPlinet{t}$ (Theorem \ref{mainthm}(2)) along with the strong $\Ham_t$-Brownian Gibbs property (Lemma \ref{stronggibbslemma}) implies that $\PP$-almost surely
$$
\mathbf{1}_{\mathsf{A}_\ell}\cdot \mathbf{1}_{\mathsf{E}_{C}(n)}\cdot  \EE\Big[\mathbf{1}_{\mathsf{B}_{C,\ell}(n)} \,\big\vert \, \Fext\big(\{1\},(0,\sigma_n)\big)\Big] =
\mathbf{1}_{\mathsf{A}_\ell}\cdot \mathbf{1}_{\mathsf{E}_{C}(n)}\cdot \PH{1}{1}{(0,\sigma_n)}{\HKPZFPline{t}{1}(0)}{\HKPZFPline{t}{1}(\sigma_n)}{+\infty}{\HKPZFPline{t}{2}}{\Ham_t}\big(\mathsf{B}_{C,\ell}(n)\big).
$$
The above equality is true even if the term $\mathbf{1}_{\mathsf{A}_\ell}$ is removed from both sides, but does rely upon the inclusion of the $\mathbf{1}_{\mathsf{E}_{C}(n)}$ (as the occurrence of that event implies that $\sigma_n\neq 0$).

The measure on the right-hand side above is given in Definition \ref{maindefHBGP}. On the event $\mathsf{A}_\ell \cap \mathsf{E}_{C}(n)$, we have that $\HKPZFPline{t}{1}(0)\geq -\ell$, $\sigma_n\in [n,n+1]$ and
$\HKPZFPline{t}{1}(\sigma_n) \geq C +\nu \sigma_n^2-\sigma_n^2/2$. On this event, we may use Lemmas \ref{monotonicity1} and \ref{monotonicity2} to construct a coupling of the measure $\PH{1}{1}{(0,\sigma_n)}{\HKPZFPline{t}{1}(0)}{\HKPZFPline{t}{1}(\sigma_n)}{+\infty}{\HKPZFPline{t}{2}}{\Ham_t}$ on the curve $B:(0,\sigma_n)\to \R$ and the measure $\PH{1}{1}{(0,\sigma_n)}{-\ell}{ C +\nu \sigma_n^2-\sigma_n^2/2}{+\infty}{-\infty}{\Ham_t}$ on the curve $B':(0,\sigma_n)\to \R$  such that $B(x)\geq B'(x)$ for all $x\in (0,\sigma_n)$. Since the event $\mathsf{B}_{C,\ell}(n)$ (with $\HKPZFPline{t}{1}(n)$ replaced by $B$) becomes less probable under pointwise decreases in $B$, the coupling's existence implies that
$$
\PH{1}{1}{(0,\sigma_n)}{\HKPZFPline{t}{1}(0)}{\HKPZFPline{t}{1}(\sigma_n)}{+\infty}{\HKPZFPline{t}{2}}{\Ham_t}\big(\mathsf{B}_{C,\ell}(n)\big) \geq
\PH{1}{1}{(0,\sigma_n)}{-\ell}{ C +\nu \sigma_n^2-\sigma_n^2/2}{+\infty}{-\infty}{\Ham_t}\big(\mathsf{B}_{C,\ell}(n)\big) = \frac{1}{2}.
$$
The final equality follows because this measure is Brownian motion from $-\ell$ at time $0$ to $C +\nu \sigma_n^2-\sigma_n^2/2$ at time $\sigma_n$, and the event
$\mathsf{B}_{C,\ell}(n)$ is exactly the probability that this Brownian motion is above the linear interpolation of its endpoints at time $n$ (which is $1/2$).
We find then that
\begin{equation}\label{eq.findthen}
\mathbf{1}_{\mathsf{A}_\ell}\cdot \mathbf{1}_{\mathsf{E}_{C}(n)}\cdot \EE\Big[\mathbf{1}_{\mathsf{B}_{C,\ell}(n)} \,\big\vert \, \Fext\big(\{1\},(0,\sigma_n)\big)\Big]
\geq \mathbf{1}_{\mathsf{A}_\ell}\cdot \mathbf{1}_{\mathsf{E}_{C}(n)}\cdot  \frac{1}{2}.
\end{equation}

%For any $\nu'\in (0,\nu)$, as long as $\sigma_n\in [n,n+1]$ there exists a function $C'(\tilde{C},\ell)$ such that \note{maybe should say something more about this function or give it explicitly...}
On the event $\mathsf{E}_{C}(n)$, we may bound
$$
\frac{\sigma_n-n}{\sigma_n}(-\ell) +  \frac{n}{\sigma_n}\big(C+\nu \sigma_n^2 - \sigma_n^2/2\big)
\geq -\ell + \frac{n}{n+1} C +\nu n^2 - \frac{n(n+1)}{2}.
$$
Therefore, defining the event
$$
\mathsf{B}'_{C,\ell}(n) = \bigg\{ \HKPZFPline{t}{1}(n) \geq -\ell + \frac{n}{n+1} C +\nu n^2 - \frac{n(n+1)}{2} \bigg\},
$$
it follows that $\PP$-almost surely
\begin{eqnarray*}
 \hskip-.25in \mathbf{1}_{\mathsf{E}_{C}(n)}\,\cdot\, \EE\Big[\mathbf{1}_{\mathsf{B}_{C,\ell}(n)} \,\big\vert \, \Fext\big(\{1\},(0,\sigma_n)\big)\Big]
&\leq& \mathbf{1}_{\mathsf{E}_{C}(n)}\,\cdot\, \EE\Big[\mathbf{1}_{\mathsf{B}'_{C,\ell}(n)} \,\big\vert \, \Fext\big(\{1\},(0,\sigma_n)\big)\Big] \\
&\leq& \EE\Big[\mathbf{1}_{\mathsf{B}'_{C,\ell}(n)} \,\big\vert \, \Fext\big(\{1\},(0,\sigma_n)\big)\Big].
\end{eqnarray*}
Thus, we conclude using (\ref{eq.findthen}) that for $n>0$, $\PP$-almost surely
\begin{equation}\label{e.asbgweeg}
\mathbf{1}_{\mathsf{A}_\ell}\cdot \mathbf{1}_{\mathsf{E}_{C}(n)} \leq 2 \cdot \EE\Big[\mathbf{1}_{\mathsf{B}'_{C,\ell}(n)} \,\big\vert \, \Fext\big(\{1\},(0,\sigma_n)\big)\Big].
\end{equation}
A similar result holds for $n<-1$ by symmetry.

Combining this conclusion with (\ref{sfaeeff}) and (\ref{fsfss}), as well as by taking the overall $\PP$-expectation, this shows that
\begin{eqnarray*}
\PP\Bigg(\bigcup_{n\in \Z^*}\mathsf{E}_{C}(n)\Bigg) &\leq& \EE\Bigg[ \sum_{n\in \Z^*} 2 \cdot\, \EE\Big[\mathbf{1}_{\mathsf{B}'_{C,\ell}(n)} \,\big\vert \, \Fext\big(\{1\},(0,\sigma_n)\big)\Big] \bigg\vert \mathcal{F}_0 \Bigg] + \, \PP\big(\mathsf{A}_{\ell}^c\big)\\
& \leq&
2\cdot  \sum_{n\in \Z^*}  \PP\Big(\mathsf{B}'_{C,\ell}(n)\Big) \, + \PP\big(\mathsf{A}_{\ell}^c\big).
\end{eqnarray*}

We have essentially now reduced the problem from one on $\R$ to one on $\Z$ as desired. To bound the right-hand side by $\e/2$, it suffices to take $\ell$ large enough so that $\PP\big(\mathsf{A}_{\ell}^c\big)\leq \e/4$, and then, with $\ell$ fixed in this manner, to take $C$ large enough that $2 \cdot \sum_{n\in \Z^*}  \PP\Big(\mathsf{B}'_{C,\ell}(n)\Big) \leq \e/4$ as well. Proposition~\ref{ACQprop}(1) gives the first bound while it is the upper tail estimate in Proposition \ref{ACQprop}(2) which provides the second. This proves (\ref{fsfsaee}) and completes the proof of Lemma \ref{e.tildekappalem}.
\end{proof}

Let us also record a variant of Lemma \ref{e.tildekappalem} in which $t$ is fixed.
\begin{lemma}\label{e.tildekappalemwwww}
For $t>0$ fixed and any $\nu>0$ there exists constants $C,c_1,c_2>0$ such that, for all $s\geq 1$,
\begin{equation*}
\PP\Big(\HKPZFPline{t}{1}(y) + y^2/2 < C + s +\nu y^2 \textrm{  for all } y\in\R\Big) \geq 1-c_1 e^{-c_2 s}.
\end{equation*}
\end{lemma}
\begin{proof}
Since $t$ is fixed, we have both exponential upper and lower bounds from Proposition \ref{ACQprop}(3).
We prove the present result by using a minor modification of the proof of Lemma \ref{e.tildekappalem}, taking into account the exponential lower bound, which was previously unavailable because $t$ is not fixed.
\end{proof}

Finally, we record the following result which is an immediate corollary of Lemma \ref{l.genindata}.
\begin{lemma}\label{l.varproblem}
For the KPZ equation $\HKPZ{t}{x}$ started from general initial data $\Hzero{\cdot}$, we have the distributional equality for a fixed time $t$ and location $x$ (here $x=0$) given by
\begin{equation*}
\HKPZ{t}{0} \stackrel{(d)}{=}  \log \left( \int_{-\infty}^{\infty} e^{t^{1/3}\big(\HKPZFPline{t}{1}(y) + t^{-1/3}\Hzero{-t^{2/3}y}\big)}\dd y\right) -\frac{t}{24} + \frac{2}{3} \log t .
\end{equation*}
If $\Hzero{\cdot} = t^{1/3} f^{(t)}(t^{-2/3} \cdot)$ for some function $f^{(t)}$, then rearranging terms, we have that
\begin{equation*}
\frac{\HKPZt{t}{0} +\frac{t}{24}}{t^{1/3}} \stackrel{(d)}{=} t^{-1/3} \log\Big(\int_\R e^{t^{1/3}\big(\HKPZFPline{t}{1}(y) +f^{(t)}(-y)\big)} \dd y\Big) + t^{-1/3} \frac{2}{3} \log t.
\end{equation*}
\end{lemma}

\subsection{Proof of Theorem \ref{t.univonethird}}\label{s.univonethird}

There are three parts to Theorem \ref{t.univonethird}, all of which we prove below. All three parts rely on Theorem \ref{mainthm}, Proposition \ref{ACQprop} and Lemmas \ref{e.tildekappalem} and \ref{l.varproblem}. The basic logic behind the proofs of Theorem \ref{t.univonethird}(1) and (2) is as follows: Lemma \ref{l.genindata} relates the general initial data KPZ equation solution to the narrow wedge solution. Theorem \ref{mainthm} relates this solution to $\HKPZFPline{t}{1}$. Theorem \ref{mainthm}(3) shows that this curve remains uniformly Brownian as $t\to \infty$. Thus the one-point tightness of the narrow wedge KPZ equation afforded by Proposition~\ref{ACQprop}(1) is readily transferred to the setting of general initial data. For Theorem \ref{t.univonethird}(3), we require a slightly more involved use of the Gibbs property which is explained at the beginning of that proof.

We now state a lemma which will be used in Theorem \ref{t.univonethird}'s proof.

\begin{lemma}\label{l.considercol}
Fix $C,\delta, \kappa, M>0$ and consider a collection of functions $f^{(t)}:\R\to \R\cup\{-\infty\}$ each of which satisfies hypothesis $\Hyp(C,\delta,\kappa,M)$ for all $t\geq 1$.
\begin{enumerate}
\item For all $\e>0$, there exists $C'>0$ such that
$$
\PP\Big( \textrm{Leb}\big\{y\in [-M,M]: \HKPZFPline{t}{1}(y) +f^{(t)}(-y) \geq -C'\big\}\geq \delta\Big) \geq 1-\e.
$$
\item For any $\kappa'\in (0,\kappa)$ and any $\e>0$, there exists a constant $C'>0$ such that
$$
\PP\Big(\HKPZFPline{t}{1}(y) +f^{(t)}(-y) \leq C' -\kappa' y^2/2 \textrm{ for all } y\in\R\Big) \geq 1-\e.
$$
\end{enumerate}
\end{lemma}
\begin{proof}
Since $f^{(t)}$ satisfies $\Hyp(C,\delta,\kappa,M)$, we have that $\textrm{Leb}\big\{y\in [-M,M]: f^{(t)}(y) \geq -C\big\} \geq \delta$. Theorem \ref{mainthm}(3) implies that on the interval $[-M,M]$ the curve $\HKPZFPline{t}{1}(\cdot)$ is absolutely continuous with respect to Brownian bridge, with a Radon-Nikodym derivative which is tight for $t\geq 1$. This along with the tightness and stationarity of the one-point distribution of $\HKPZFPline{t}{1}(y)+y^2/2$ (Proposition \ref{ACQprop}(1)) implies that there exists a constant $\tilde{C}>0$ for which
$$\PP\Big(\HKPZFPline{t}{1}(y)\geq -\tilde{C}\textrm{  for all }y\in [-M,M]\Big)\geq 1-\e.$$
Thus, setting $C'=C+\tilde{C}$ we readily arrive at Lemma \ref{l.considercol}(1).

By the quadratic growth condition on $f^{(t)}$ implied by $\Hyp(C,\delta,\kappa,M)$, the result of Lemma  \ref{l.considercol}(2) follows clearly from Lemma \ref{e.tildekappalem}.
\end{proof}

\begin{proof}[Proof of Theorem \ref{t.univonethird}(1)]
Fix $\e$ as in the statement of the theorem. Apply Lemma \ref{l.considercol}(1) for $\e/2$ and let $C_1'$ be the constant returned by the lemma. Set
$$C_{1,1}(t) = -C_1' +t^{-1/3} \big(\delta +\tfrac{2}{3}\log t\big).$$
Let $C_{1,1}$ be the (clearly finite) maximal value of $C_{1,1}(t)$ over $t\geq 1$. Then it follows from Lemmas \ref{l.considercol}(1) and \ref{l.varproblem} that
\begin{equation}\label{e.lowereqnacomp}
\PP\bigg( \frac{\HKPZt{t}{0} +\frac{t}{24}}{t^{1/3}} \geq -C_{1,1}\bigg) \geq 1-\frac{\e}{2}.
\end{equation}

Fix $\kappa'\in (0,\kappa)$ and apply Lemma \ref{l.considercol}(2) for $\e/2$, letting $C_2'$ be the constant returned by the lemma. Set
$$
C_{1,2}(t) = -C_2' + t^{-1/3} \Big( \log \sqrt{\tfrac{2\pi}{t^{1/3} \kappa'}} + \tfrac{2}{3}\log t \Big).$$
Let $C_{1,2}$ be the (clearly finite) maximal value of $C_{1,2}(t)$ over $t\geq 1$. Then it follows from Lemmas ~\ref{l.considercol}(2) and \ref{l.varproblem} that
\begin{equation}\label{e.uppereqnacomp}
\PP\bigg( \frac{\HKPZt{t}{0} +\frac{t}{24}}{t^{1/3}} \leq C_{1,2}\bigg) \geq 1-\frac{\e}{2}.
\end{equation}
Here we used the quadratic bound $\HKPZFPlinetnw{t}(y) +f^{(t)}(y) \leq C' -\kappa' y^2/2$ accorded by Lemma~\ref{l.considercol}(2) and performed the resulting Gaussian integral explicitly.

Letting $C_1 = \max\big(C_{1,1},C_{1,2}\big)$, we arrive at the desired result of Theorem \ref{t.univonethird}(1) from (\ref{e.lowereqnacomp}), (\ref{e.uppereqnacomp}) and the union bound.
\end{proof}

\begin{proof}[Proof of Theorem \ref{t.univonethird}(2)]
By virtue of Lemma \ref{l.varproblem} we must show that, for all $\rho,\e>0$, there exists $t_0>1$ such that, for $t\geq t_0$,
\begin{equation}\label{ppgiebbs}
\PP\Bigg(\bigg| \,  t^{-1/3} \log\bigg(\int_{\R} e^{t^{1/3}\big( \HKPZFPline{t}{1}(y) +f^{(t)}(-y)\big)}\dd y\bigg) - t^{-1/3} \log\bigg(\int_{\R} e^{t^{1/3}\big(\HKPZFPline{t}{1}(y) +\tilde{f}^{(t)}(-y)\big)}\dd y\bigg)\bigg| \leq \rho \Bigg) \geq 1-\e.
\end{equation}

For $T>0$, define the integrals
\begin{eqnarray*}
I_{(-T,T)} &:=& \int_{(-T,T)} e^{t^{1/3}\big( \HKPZFPline{t}{1}(y) +f^{(t)}(-y)\big)}\dd y,\\
I_{(-T,T)^c} &:=& \int_{\R\setminus(-T,T)} e^{t^{1/3}\big( \HKPZFPline{t}{1}(y) +f^{(t)}(-y)\big)}\dd y,
\end{eqnarray*}
and likewise $\tilde{I}_{(-T,T)}$ and $\tilde{I}_{(-T,T)^c}$ with $f^{(t)}$ replaced by $\tilde{f}^{(t)}$.

For $C'>0$ define the event
\begin{equation}\label{e.wgnnernr}
\mathsf{E}_{C'}(T) = \Big\{e^{-t^{1/3} C'} \leq I_{(-T,T)} \leq e^{t^{1/3} C'}\Big\} \cap \Big\{ I_{(-T,T)^c} \leq e^{-t^{1/3}T}\Big\},
\end{equation}
and likewise define the event $\mathsf{\tilde{E}}_{C'}(T)$.
Then there exists $C',T_0>0$ such that, for all $T>T_0$ and $t\geq 1$,
\begin{equation}\label{e.econem}
\PP\big(\mathsf{E}_{C'}(T)\big) \geq 1-\frac{\e}{4},
\end{equation}
and likewise for $\mathsf{\tilde{E}}_{C'}(T)$.
The high probability of the occurrence of the first part of the event $\mathsf{E}_{C'}(T)$ is ensured by the results of Theorem \ref{t.univonethird}(1) already shown (see Lemma \ref{l.varproblem}). The high probability of the occurrence of the second part follows by applying Lemma \ref{l.considercol}(2). This lemma shows that for some $C''$, with high probability $\HKPZFPlinetnw{t}(y) +f^{(t)}(y)\leq C'' -\kappa' y^2/2$. Using this and taking $T$ large enough, we can use a Gaussian tail bound (such as in Lemma \ref{l.normallb}) to bound $I_{(-T,T)^c}$ as desired.

Define the event (with $t$ an implicit parameter in the definition)
$$
\mathsf{F}(T) = \bigg\{ \big\vert f^{(t)}(y) - \tilde{f}^{(t)}(y)\big\vert \leq \frac{\rho}{8T} \textrm{  for all } y\in [-T,T]\bigg\}.
$$
Then it follows from the hypothesis of Theorem \ref{t.univonethird}(2) that for all $T>0$ there exists $t_0(T)$ such that, for $t\geq t_0(T)$,
$$
\PP\big(\mathsf{F}(T)\big) \geq 1-\frac{\e}{2}.
$$

Fixing $T>T_0$, let $C',T_0,t_0(T)$ be as above. We have that for all $t\geq t_0(T)$,
$$
\PP\big(\mathsf{E}_{C'}(T)\cap \tilde{\mathsf{E}}_{C'}(T) \cap \mathsf{F}(T)\big) \geq 1-\e.
$$
It is then easily seen that, on the event $\mathsf{E}_{C'}(T)\cap \tilde{\mathsf{E}}_{C'}(T) \cap \mathsf{F}(T)$, the following bounds hold:
$$
\bigg\vert t^{-1/3} \big(\log I_{(-T,T)} - \log \tilde{I}_{(-T,T)}\big) \bigg\vert \leq \frac{\rho}{2},
$$
and
$$
\bigg\vert t^{-1/3} \Big(\log \big(I_{(-T,T)}+I_{(-T,T)^c}\big) - \log I_{(-T,T)}\Big) \bigg\vert \leq \frac{\rho}{4};
$$
and likewise with $\tilde{I}_{(-T,T)}$ and $\tilde{I}_{(-T,T)^c}$. Combining these bounds yields
$$
\bigg\vert t^{-1/3} \Big(\log \big(I_{(-T,T)}+I_{(-T,T)^c}\big) - \log\big(\tilde{I}_{(-T,T)}+\tilde{I}_{(-T,T)^c}\big)\Big) \bigg\vert \leq \rho.
$$
Since the intersection of events which force this bound occurs with probability at least $1-\e$, we obtain~(\ref{ppgiebbs}) and thereby complete the proof of Theorem \ref{t.univonethird}(2).
\end{proof}

\begin{proof}[Proof of Theorem \ref{t.univonethird}(3)]
For a union of intervals $B\subseteq \R$ define
\begin{equation}\label{e.iabearly}
I_B = \int_{B}  e^{t^{1/3} \big( \HKPZFPline{t}{1}(x) + f^{(t)}(-x) \big)} \dd x \, .
\end{equation}
Owing to Lemma \ref{l.varproblem}, in order to prove Theorem \ref{t.univonethird}(3), we must show that for all $\e>0$ there exists $C_2=C_2(\e,C,\delta,\kappa,M)$ such that for all $y\in \R$, $\eta>0$ and $t\geq 1$,
\begin{equation}\label{e.desiredwe}
\PP \Big( I_{(-\infty,\infty)} \in  \big( e^{t^{1/3}y},e^{t^{1/3}(y + \eta)} \big) \Big)\leq C_2 \eta + \e.
\end{equation}

We will argue that this is the case by using a certain resampling procedure.  Our argument starts by sampling the entire line ensemble $\HKPZFPlinet{t}$ according to the measure constructed by Theorem \ref{mainthm}. We then discard a certain piece of information determining this ensemble and resample it according to its conditional law. This resampling results in a new ensemble which we call $\HKPZFPlinet{t,{\rm re}}$ (with ``re'' standing for ``resampled''). By construction, the new ensemble has the same law as $\HKPZFPlinet{t}$. To obtain the desired upper bound~(\ref{e.desiredwe}) on the localization probability we will argue that, should the entire line ensemble $\HKPZFPlinet{t}$ sampled initially satisfy a certain highly probable ``good'' event $\mathsf{G}$, then the partial resampling that this ensemble undergoes is likely to result in the output line ensemble not satisfying
the condition that $I^{{\rm re}}_{(-\infty,\infty)}$ lie in the interval $\big( e^{t^{1/3}y},e^{t^{1/3}(y + \eta)} \big)$. Here we write $I^{{\rm re}}_{B}$ for the analog of $I_{B}$ defined in (\ref{e.iabearly}) with $\HKPZFPline{t}{1}$ replaced by $\HKPZFPline{t,{\rm re}}{1}$. Let us note that we could just as well not talk about resamplings here and instead use the language of conditional laws. We stick with resamplings here for pedagogical purposes, as we believe it makes some of the arguments more transparent.

Let $\PP$ (and $\EE$) represent the probability (and expectation) operator for an augmented probability space on which both $\HKPZFPlinet{t}$ and $\HKPZFPlinet{t,{\rm re}}$ are coupled in the manner we will describe. It follows then that
\begin{eqnarray*}
 & & \hskip-.25in \PP \Big( I_{(-\infty,\infty)} \in  \big( e^{t^{1/3}y},e^{t^{1/3}(y + \eta)} \big) \Big)\nonumber \\
  & = & \PP  \Big( I^{{\rm re}}_{(-\infty,\infty)} \in  \big( e^{t^{1/3}y},e^{t^{1/3}(y +\eta)} \big) \Big) \nonumber \\
   & \leq & \PP \bigg(  \Big\{ I^{{\rm re}}_{(-\infty,\infty)} \in  \big( e^{t^{1/3}y},e^{t^{1/3}(y + \eta)} \big) \Big\} \cap \mathsf{G} \bigg) +
     \PP \big( \mathsf{G}^c \big). \nonumber
\end{eqnarray*}
We will show that $\PP \Big( I_{(-\infty,\infty)} \in  \big( e^{t^{1/3}y},e^{t^{1/3}(y + \eta)} \big) \Big)$ is small by arguing that both of the probabilities on the third line are small.
It is in showing that the first of these two probabilities is small that we will show that the partial resampling described above has a tendency, which is uniform over those $\HKPZFPlinet{t}$ for which $\mathsf{G}$ holds, to delocalize the value of $I^{{\rm re}}_{(-\infty,\infty)}$.
\medskip
Let us now substantiate the argument outlined above and prove the bound in (\ref{e.desiredwe}). Fix $\e$ as in the statement of the theorem. For $\ell>0$ define $\delta_{\ell}$ to be the parameter $\delta$ (not to be confused with the $\delta$ which appears in $\Hyp(C,\delta,\kappa,M)$) returned by Theorem \ref{mainthm}(3) such that for all $t\geq 1$
\begin{equation}\label{e.wbnn}
\PP\big(\partfuncShort{\Ham_t}<\delta_{\ell} \big)\leq \frac{\e}{8}.
\end{equation}
We have used a shorthand (Definition \ref{maindefHBGP})
$$
\partfuncShort{\Ham_t} = \partfunc{1}{1}{(-\ell-1,\ell+1)}{\HKPZFPline{t}{1}(-\ell-1)}{\HKPZFPline{t}{1}(\ell+1)}{+\infty}{\HKPZFPline{t}{2}}{\Ham_t}.
$$
(Note that we have applied Theorem \ref{mainthm}(3) with $\e/8$ instead of $\e$.)% Then define $d_{\ell} = \delta_{\ell} \, \frac{\e}{8}$.

Define
\begin{equation}\label{e.bbre}
\BP_{1,2}^{t} =  \exp \left\{ - \int_{-\ell -1}^{\ell + 1}  e^{t^{1/3} \big( \HKPZFPline{t}{2}(s) -  \HKPZFPline{t}{1}(s)  \big) } \, \dd s \right\}.
\end{equation}
For $\ell>0$, the {\it good} event $\mathsf{G}_{\ell}$ is then defined as
$$
\mathsf{G}_{\ell} = \bigg\{ \max\Big\{\big\vert\HKPZFPline{t}{1}(-\ell) - \HKPZFPline{t}{1}(-\ell-1)\big\vert\, , \, \big\vert \HKPZFPline{t}{1}(\ell) - \HKPZFPline{t}{1}(\ell+1)\big\vert\Big\} \leq 2\ell\bigg\} \cap \bigg\{ \BP_{1,2}^{t}  \geq \delta_{\ell}\, \frac{\e}{8}\bigg\}.
$$
We have not explicitly included $t$ in the notation for this event, though it is implicitly a function of $t$ (as in Lemma \ref{l.gbound}).

\begin{lemma}\label{l.gbound}
For all $\e > 0$, there exists $\ell_0=\ell_0(\e)>1$ such that, for all $\ell \geq \ell_0$ and $t\geq 1$,
$$
\PP\big(\mathsf{G}_{\ell}^c\big) \leq \frac{\e}{2} \, .
$$
\end{lemma}
\begin{proof}
By the tightness and stationarity in $s$ of the one-point distribution of $\HKPZFPline{t}{1}(s)+s^2/2$ (Proposition \ref{ACQprop}(1)), there exists $\ell_0=\ell_0(\e) > 1$ such that, for all $s \in \R$ and all $\ell>\ell_0$,
$$
\PP \Big(  \big\vert \HKPZFPline{t}{1}(s)+s^2/2 \big\vert  \leq \ell/4 \Big) \geq 1 - \frac{\e}{16} \, .
$$
From this (and the fact that $(\ell+1)^2/2 - \ell^2/2 = \ell + 1/2 >\ell + \ell/4$) we find that
\begin{equation}\label{e.maxa}
 \PP\bigg(\max\Big\{\big\vert\HKPZFPline{t}{1}(-\ell) - \HKPZFPline{t}{1}(-\ell-1)\big\vert\, , \, \big\vert \HKPZFPline{t}{1}(\ell) - \HKPZFPline{t}{1}(\ell+1)\big\vert\Big\} \leq 2\ell \bigg) \geq 1 - \frac{\e}{4} \, .
\end{equation}
This controls the probability of the first event used in defining $\mathsf{G}_{\ell}$.

We now turn to controlling the probability of the event $\big\{ \BP_{1,2}^{t}  \geq \delta_{\ell}\, \frac{\e}{8} \big\}$. Recall from Definition \ref{maindefHBGP} that
$$
\partfuncShort{\Ham_t} = \PfreeExpShort\big[\boltShort{\Ham_t}(B)\big] $$
where we are using the shorthand
\begin{eqnarray*}
\PfreeExpShort &=&  \PfreeExp{1}{1}{(-\ell-1,\ell+1)}{\HKPZFPline{t}{1}(-\ell-1)}{\HKPZFPline{t}{1}(\ell+1)},\\
\boltShort{\Ham_t} &=& \bolt{1}{1}{(-\ell-1,\ell+1)}{\HKPZFPline{t}{1}(-\ell-1)}{\HKPZFPline{t}{1}(\ell+1)}{+\infty}{\HKPZFPline{t}{2}}{\Ham_t},
\end{eqnarray*}
and $B$ is distributed according to $\PfreeExpShort$. Let us also introduce the shorthand
$$\PHShort{\Ham_t} = \PH{1}{1}{(-\ell-1,\ell+1)}{\HKPZFPline{t}{1}(-\ell-1)}{\HKPZFPline{t}{1}(\ell+1)}{+\infty}{\HKPZFPline{t}{2}}{\Ham_t}.$$

Whenever $\partfuncShort{\Ham_t}\geq \delta_{\ell}$ we have that, for all $w\geq 0$,
\begin{equation}\label{e.wetwha}
\PHShort{\Ham_t} \Big( \boltShort{\Ham_t}(B)  \leq w \Big)  = \frac{1}{\partfuncShort{\Ham_t}}\, \PfreeExpShort\big[ \mathbf{1}_{\boltShort{\Ham_t}(B)\leq w}\,\cdot\, \boltShort{\Ham_t}(B)\big] \leq \delta_{\ell}^{-1}  w\, \PfreeShort\big[\boltShort{\Ham_t}(B)\leq w\big] \leq \delta_{\ell}^{-1} w.
\end{equation}
Hence,
$$
\PP \Big( \boltShort{\Ham_t}(\HKPZFPline{t}{1}) \leq w \Big) \leq  \PP \Big( \big\{ \boltShort{\Ham_t}(\HKPZFPline{t}{1}) \leq w \big\} \cap \big\{ \partfuncShort{\Ham_t}\geq \delta_{\ell} \big\} \Big)
  + \PP \Big( \partfuncShort{\Ham_t}< \delta_{\ell} \Big) \leq  \delta_{\ell}^{-1} w  + \frac{\e}{8} \, .
$$
The second inequality uses (\ref{e.wetwha}) and (\ref{e.wbnn}). Setting $w = \delta_{\ell} \, \frac{\e}{8}$ and identifying $\boltShort{\Ham_t}(\HKPZFPline{t}{1}) = \BP_{1,2}^{t}$ from (\ref{e.bbre}) we find that
$$
\PP\Big(\BP_{1,2}^{t} \geq  \delta_{\ell}\, \frac{\e}{8}\Big) \geq 1- \frac{\e}{4}.
$$

Combining this bound with that of (\ref{e.maxa}) completes the proof of the lemma.
\end{proof}

Returning to the proof of Theorem \ref{t.univonethird}(3), recall that in proving Theorem~\ref{t.univonethird}(2) we show in (\ref{e.econem}) that there exists $C',T_0>0$ such that, for all $T>T_0$ and $t\geq 1$,
\begin{equation}\label{e.econemtwo}
\PP\big(\mathsf{E}_{C'}(T)\big) \geq 1-\frac{\e}{4}.
\end{equation}
where $\mathsf{E}_{C'}(T)$ is defined in (\ref{e.wgnnernr}). For the rest of this proof we will choose
\begin{equation}\label{e.leqna}
\ell = \max \big\{\ell_0,\, 2C',\,T_0\big\}.
\end{equation}

As we have finished introducing the parameter $\ell$ and {\it good} event $\mathsf{G}_{\ell}$ we may now describe the partial resampling procedure we employ. For an interval $(a,b)$, define $H_{(a,b)}:(a,b)\to \R$ as the difference between $\HKPZFPline{t}{1}$ on $(a,b)$ and the linear interpolation of the endpoint values $\big(a,\HKPZFPline{t}{1}(a)\big)$ and $\big(b,\HKPZFPline{t}{1}(b)\big)$. The curve $\HKPZFPline{t}{1}$ restricted to the interval $[-\ell-1,\ell+1]$ can be {\it decomposed} into the following data:
\begin{enumerate}
\item its endpoint values $\HKPZFPline{t}{1}(-\ell-1)$ and $\HKPZFPline{t}{1}(\ell+1)$;
\item the curves $H_{-\ell-1,-\ell}$, $H_{-\ell,\ell}$ and $H_{\ell,\ell+1}$;
\item the value of $\HKPZFPline{t}{1}(\ell)-\HKPZFPline{t}{1}(-\ell)$;
\item the value of $\HKPZFPline{t}{1}(-\ell)$.
\end{enumerate}
Our resampling procedure will proceed by discarding the data $\HKPZFPline{t}{1}(-\ell)$ (piece (4) in the above list), and resampling it independently according to its conditional distribution (given pieces (1), (2), (3) of data above along with $\HKPZFPline{t}{2}$). From this resampled value we then reconstruct a curve and call the result~$\HKPZFPline{t,{\rm re}}{1}$. By construction, the law of $\HKPZFPline{t,{\rm re}}{1}$ is the same as that of $\HKPZFPline{t}{1}$. The original and resampled line ensembles are coupled on the probability space with measure $\PP$. A virtue of this resampling procedure is that the $\Ham_t$-Brownian Gibbs property implies that the distribution (whose density is denoted by $g(z)$) of $\HKPZFPline{t}{1}(-\ell)$, conditional on the knowledge of pieces (1), (2), (3) and the curve $\HKPZFPline{t}{2}$,  is quite explicit and simple. Figure~\ref{figshift} illustrates the above decomposition, and the discarding, resampling and reconstruction procedure.

We now introduce the formal apparatus which codifies this procedure. Define $\mathcal{F}$  to be
the sigma-field generated by pieces (1), (2), (3) in the above list as well as by $\HKPZFPline{t}{2}$ restricted to $[-\ell-1,\ell+1]$.

Given pieces (1), (2), (3) in the list of data, and $z\in \R$, we uniquely define the {\it reconstructed} curve $\HKPZFPline{t,z}{1}:[-\ell-1,\ell+1]\to \R$ via

\begin{equation*}
\HKPZFPline{t,z}{1}(s) =
\begin{cases}
 H_{-\ell-1,-\ell}(s) - (s + \ell) \HKPZFPline{t}{1}(-\ell-1)  + (s + \ell + 1) z \, & s\in [-\ell-1,-\ell],\\
 H_{-\ell,\ell}(s)    - \tfrac{s - \ell}{2\ell} z   + \tfrac{s + \ell}{2\ell} \big( z + \HKPZFPline{t}{1}(\ell) -  \HKPZFPline{t}{1}(-\ell) \big)  \,  & s\in [-\ell,\ell],\\
 H_{\ell,\ell + 1}(s) + (s - \ell) \HKPZFPline{t}{1}(\ell+1)  + (\ell + 1 - s)  \big( z + \HKPZFPline{t}{1}(\ell) -  \HKPZFPline{t}{1}(-\ell) \big)     \, & s\in [\ell,\ell + 1].
\end{cases}
\end{equation*}
This reconstruction is such that $\HKPZFPline{t,z}{1}(-\ell)=z$ and if $\HKPZFPline{t,z}{1}$  is subsequently decomposed, pieces (1), (2), (3) will be identical to those for the decomposition of the original curve $\HKPZFPline{t}{1}$. For completeness define $\HKPZFPline{t,z}{i}(s) = \HKPZFPline{t}{1}(s)$ for $i=1$ and $s\notin [-\ell-1,\ell+1]$ or $i\geq 2$ and $s\in \R$. Call the entire line ensemble $\HKPZFPlinet{t,z}$.

\begin{figure}
\centering\epsfig{file=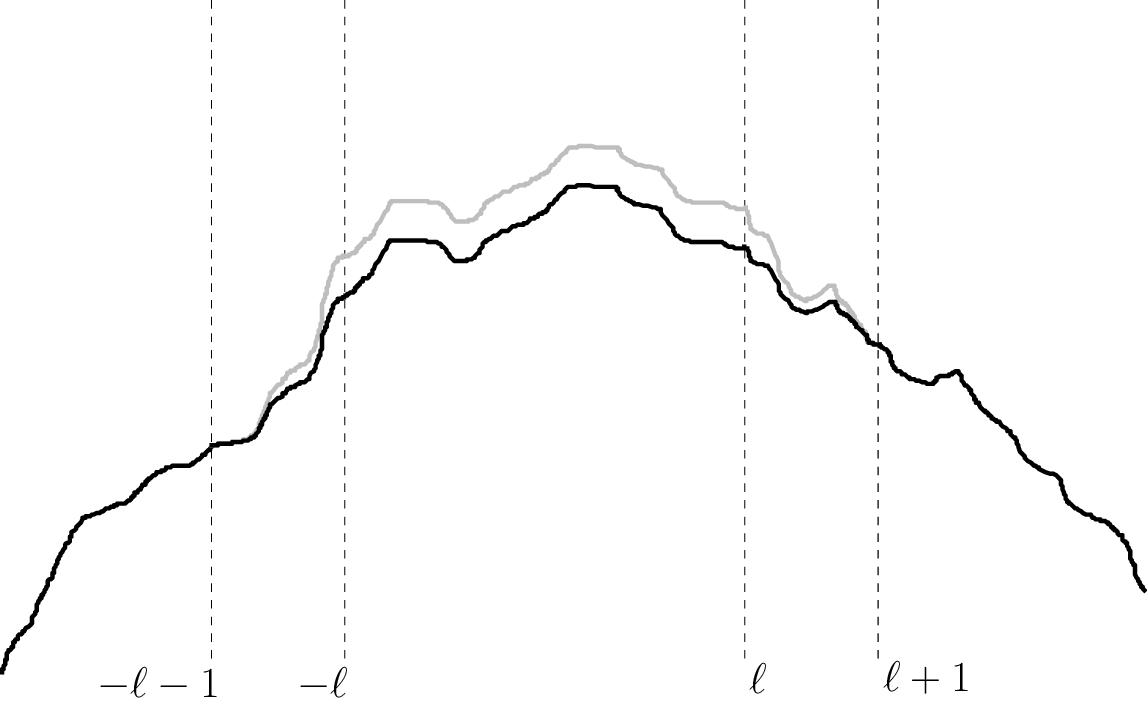, width=14cm}
\caption{The black curve depicts an instance of $\HKPZFPline{t}{1}$. After deconstructing the restriction of this to $[-\ell-1,\ell+1]$ and discarding the data $\HKPZFPline{t}{1}(-\ell)$, an observer can no longer know whether it was the black curve that was initially sampled, since it is merely one of a one-parameter family of curves consistent with the information that remains. Another possibility is the curve which is black on $[-\ell-1,\ell+1]^c$ and grey on $[-\ell-1,\ell+1]$. Note that, above the interval $[-\ell,\ell]$, the two curves differ by a constant while on $[-\ell-1,-\ell]$ and $[\ell,\ell+1]$ they differ by a linear function.}\label{figshift}
\end{figure}

The $\Ham_t$-Brownian Gibbs property implies that the density $g: \R \to (0,\infty)$  (with respect to Lebesgue measure) of the law of $\HKPZFPline{t}{1}(-\ell)$ given the random variables which generate $\mathcal{F}$ may be explicitly computed by the L\'{e}vy-Ciesielski construction of Brownian bridge (cf. \cite[Lemma 2.8]{CH}) to be
\begin{equation}\label{e.gdens}
  g(z) = Z^{-1}  \exp\bigg\{-\frac{1}{2}\Big(\HKPZFPline{t}{1}(-\ell - 1) - z\Big)^2 - \frac{1}{2} \Big( \HKPZFPline{t}{1}(\ell + 1) - \big(z + \HKPZFPline{t}{1}(\ell) -  \HKPZFPline{t}{1}(-\ell) \big) \Big)^2 \bigg\} \,\cdot\,\BP^{t,z}_{1,2}\, ,
\end{equation}
where
$$
\BP_{1,2}^{t,z} =  \exp \bigg\{ - \int_{-\ell -1}^{\ell + 1}  e^{t^{1/3} \big( \HKPZFPline{t}{2}(s) -  \HKPZFPline{t,z}{1}(s)  \big) } \, \dd s \bigg\},
$$
and where the constant $Z>0$ normalizes $g$ to be a probability measure, so that
\begin{equation}\label{e.zint}
 Z = \int_{-\infty}^{\infty}  \exp\bigg\{-\frac{1}{2}\Big(\HKPZFPline{t}{1}(-\ell - 1) - z\Big)^2 - \frac{1}{2} \Big( \HKPZFPline{t}{1}(\ell + 1) - \big(z + \HKPZFPline{t}{1}(\ell) -  \HKPZFPline{t}{1}(-\ell) \big) \Big)^2 \bigg\} \,\cdot\, \BP^{t,z}_{1,2} \, \dd z.
\end{equation}

We now formally define the resampling procedure. Let $\mathcal{Z}$ be a random variable such that $\PP$-almost surely, for all $z\in \R$,
$$
\EE\big[\mathbf{1}_{\mathcal{Z}\geq z}\big\vert \mathcal{F}\big] = \int_{z}^{\infty} g(s) \dd s,
$$
and such that under $\PP$, $\mathcal{Z}$ is conditionally independent of $\HKPZFPline{t}{1}(-\ell)$ given $\mathcal{F}$.
Define $\HKPZFPlinet{t,{\rm re}} =  \HKPZFPlinet{t,\mathcal{Z}}$. Notice that the law of $\HKPZFPlinet{t,{\rm re}}$ is the same as that of  $\HKPZFPlinet{t}$. However, $\HKPZFPlinet{t,{\rm re}}$ has been reconstructed using resampled the data $\mathcal{Z}$ instead of $\HKPZFPline{t}{1}(-\ell)$.

\begin{lemma}\label{l.znorm}
Suppose that $\mathsf{G}_{\ell}$ occurs. Then the normalizing constant $Z$ in~(\ref{e.zint}) satisfies
$$
 Z \geq e^{- 4\ell^2 - 5}  \ell^{-1} \delta_\ell\, \frac{\e}{8} \, .
$$
\end{lemma}
\begin{proof}
This lower bound is obtained by noting that $Z$ is at least the value of the integral in~(\ref{e.zint}) when the range of integration is taken to be the interval $\big[ \HKPZFPline{t}{1}(-\ell), \HKPZFPline{t}{1}(-\ell) + \ell^{-1} \big]$. This may seem to be a surprising choice for a lower bound, because $Z$ is an $\mathcal{F}$-measurable random variable, whereas $\HKPZFPline{t}{1}(-\ell)$ (and hence this lower bound) is not. Indeed, conditionally on $\mathcal{F}$, our lower bound is a random quantity. However, as we will presently see, this random variable is almost surely bounded below by the non-random value which shows up on the right-hand side in the statement of the lemma.

Indeed, since $\mathsf{G}_{\ell}$ occurs, we have that whenever $z\in \big[\HKPZFPline{t}{1}(-\ell),\HKPZFPline{t}{1}(-\ell) + \ell^{-1}\big]$,
$$ \max \Big\{ \big\vert \HKPZFPline{t}{1}(-\ell-1) - z \big\vert , \big\vert \HKPZFPline{t}{1}(\ell + 1) - \big(z + \HKPZFPline{t}{1}(\ell) -  \HKPZFPline{t}{1}(-\ell) \big) \big\vert \Big\} \leq 2\ell + \ell^{-1}.
$$
Note that $\BP_{1,2}^{t,z}$ increases as $z$ increases. Note also that, since $\HKPZFPline{t}{1},\HKPZFPline{t}{2}$ are such that the event $\mathsf{G}_{\ell}$ holds, we must have that $\BP_{1,2}^{t,\HKPZFPline{t}{1}(-\ell)}=\BP_{1,2}^{t} \geq \delta_\ell\, \tfrac{\e}{8}$. We see then that the integrand in~(\ref{e.zint}) is at least
$e^{- \big( 2 \ell + \ell^{-1} \big)^2} \delta_\ell\, \tfrac{\e}{8}$, a quantity which is bounded below by $e^{-4 \ell^2 - 5} \delta_\ell\, \tfrac{\e}{8}$ (due to $\ell \geq  1$). As this lower bound is constant, it proves the lemma.
\end{proof}

For a union of intervals $B\subseteq \R$ and $z\in \R$ define
\begin{equation}\label{e.iab}
I^{z}_B = \int_{B}  e^{t^{1/3} \big( \HKPZFPline{t,z}{1}(x) + f^{(t)}(-x) \big)} \dd x \, .
\end{equation}
That is, $I^{z}_{B}$ is the value of $I_{B}$ when the process $\HKPZFPline{t,z}{1}$ replaces the role of $\HKPZFPline{t}{1}$. Then, for $z \in \R$ and $q > 0$,
\begin{equation}\label{e.iboundone}
I_{(-\ell,\ell)}^{z + q} = e^{t^{1/3} q}
I_{(-\ell,\ell)}^{z}, \qquad\qquad I_{(-\infty,-\ell-1)}^{z+q} + I_{(\ell+1,\infty)}^{z+q} =  I_{(-\infty,-\ell-1)}^{z} + I_{(\ell+1,\infty)}^{z},
\end{equation}
and
\begin{equation}\label{e.iboundtwo}
I_{(-\ell-1,-\ell)}^{z} + I_{(\ell,\ell + 1)}^{z} \leq
I_{(-\ell-1,-\ell)}^{z+q} + I_{(\ell,\ell + 1)}^{z+q}  \leq
    e^{t^{1/3} q}   \big(
I_{(-\ell-1,-\ell)}^{z} + I_{(\ell,\ell + 1)}^{z}  \big) \, .
\end{equation}
The first inequality follows because increasing $z$ to $z+q$ only serves to increase the value of $\HKPZFPline{t,z}{1}(x)$ and the second inequality comes from the fact that on the intervals of interest, $\HKPZFPline{t,z}{1}(x)$ is increases at most by $q$.

Let
$$
 \mathsf{J} = \Big\{ z \in \R: I^{z}_{(-\infty,-\ell)} + I^{z}_{(\ell,\infty)} \leq e^{-t^{1/3}T_0/2} I^{z}_{(-\ell,\ell)} \Big\} ,
$$
where $T_0 > 0$ is the constant fixed before (\ref{e.econemtwo}).
Note that, by (\ref{e.iboundone}) and (\ref{e.iboundtwo}), if $p \in \mathsf{J}$ and $q > p$, then $q \in \mathsf{J}$; note also that $\mathsf{J}$ is closed on the left.
We may thus write $\mathsf{J}$ in the form $\mathsf{J} = [V,\infty)$, where note that $V$ is a $\mathcal{F}$-measurable random variable.

Let $\mathcal{Z}_1,\mathcal{Z}_2$ be $\R$-valued $\mathcal{F}$-measurable random variables such that
$$
I_{(-\infty,\infty)}^{\mathcal{Z}_1} = e^{t^{1/3}y} \textrm{  and } I_{(-\infty,\infty)}^{1,\mathcal{Z}_2} = e^{t^{1/3}(y+\eta)}.
$$

%Write $\mathsf{R}$ for the event $\left\{ I^2_{-\infty,\infty} \in  \big( e^{t^{1/3}y},e^{t^{1/3}(y + h)} \big) \right\}$.
%
%We then define the $\mathcal{F}$-measurable set
%\begin{equation*}
%S=S\Big(  \hfixedptone_1^t(-\ell-1) ,\hfixedptone_1^t(\ell+1), H_{-\ell-1,-\ell}, H_{-\ell,\ell}, H_{\ell,\ell + 1}, \hfixedptone_1^t(\ell) - \hfixedptone_1^t(-\ell),  \hfixedptone_2^t:[-\ell-1,\ell+1] \to \R \Big)
%\end{equation*}
%to be the set of $z \in \R$ such that the event $\mathsf{R}$ is satisfied when $\hfixedpttwo^t_1$ is replaced by $\mathcal{H}^z$ in the definition of this event.

Define the event
$$
 \mathsf{R} = \Big\{ \HKPZFPline{t,{\rm re}}{1}(-\ell) \in \big( \mathcal{Z}_1,\mathcal{Z}_2 \big) \Big\}.
$$
Since
$$
 \PP(\mathsf{R}) =  \PP \Big( I_{(-\infty,\infty)} \in  \big( e^{t^{1/3}y},e^{t^{1/3}(y + \eta)} \big) \Big),
$$
we see that when the next lemma is proved, so too will be Theorem~\ref{t.univonethird}(3).
\begin{lemma}\label{l.rprob}
We have that
$$
 \PP(\mathsf{R}) \leq 	  2\eta e^{4 \ell^2 + 5} \ell \delta_\ell^{-1}  \frac{8}{\e}  +  \e .
$$
\end{lemma}
\begin{proof}
By using conditional expectations, we can write
\begin{equation}\label{e.condexppresnew}
\PP(\mathsf{R}) = \EE\bigg[ \, \EE\Big[ \mathbf{1}_{\mathsf{R}} \big\vert \mathcal{F}\Big]\bigg].
\end{equation}
By the procedure through which we resampled $\HKPZFPline{t,{\rm re}}{1}(-\ell)$, it follows that $\PP$-almost surely,
\begin{equation*}
\EE\Big[ \mathbf{1}_{\mathsf{R}} \big\vert \mathcal{F} \Big] = \int_{\mathcal{Z}_1}^{\mathcal{Z}_2} g(z) \dd z .
\end{equation*}
We may then bound
$$
\int_{\mathcal{Z}_1}^{\mathcal{Z}_2} g(z) \,\dd z \leq A_1 + A_2,
$$
where
$$
 A_1 = \int_{-\infty}^V g(z)\,  \dd z \qquad \textrm{and}\qquad
 A_2 =  \max\bigg\{0,  \int_{\max\{ \mathcal{Z}_1,V \}}^{\mathcal{Z}_2} g(z)\, \dd z \bigg\}.
$$
(The cutoff of negative values in the form for $A_2$ treats the case where $Z_2 < \max\{ Z_1,V \}$.)

Note then that
$$
 A_1 = \EE \Big[   \mathbf{1}_{\HKPZFPline{t,{\rm re}}{1}(-\ell) \in (-\infty,V)}    \, \Big\vert \, \mathcal{F}  \Big] .
$$
Thus,
$$
 \EE \big[ A_1 \big] = \PP \Big( \HKPZFPline{t,{\rm re}}{1}(-\ell) \in (-\infty,V)  \Big) \leq \frac{\e}{4} ,
$$
the inequality being due to $\HKPZFPline{t,{\rm re}}{1}$ having the same law as $\HKPZFPline{t}{1}$, as well as Lemma~\ref{l.av} below. Note also that $A_2  \mathbf{1}_{\mathsf{G}}  \leq 2\eta e^{4 \ell^2 + 5} \ell \delta_\ell^{-1}\, \frac{8}{\e} \mathbf{1}_{\mathsf{G}}$, by Lemmas~\ref{l.znorm} and~\ref{l.zoneztwo} below.

Assembling these bounds and using Lemma \ref{l.gbound},
\begin{eqnarray*}
 & & \hskip-.25in \PP(\mathsf{R}) \leq
 \EE \Big [ \EE\big[ \mathbf{1}_{\mathsf{R}} \big\vert \mathcal{F} \big]  \mathbf{1}_{\mathsf{G}} \Big]  \, + \,  \PP\big(\mathsf{G}^c \big) \\
 &
  \leq & \EE \Big( \big(  A_1 + A_2 \big) \mathbf{1}_{\mathsf{G}} \Big) + \e/2 \leq \e/4 +  2\eta e^{4 \ell^2 + 5} \ell \delta_\ell^{-1}\, \frac{8}{\e}  \,\cdot \PP \big(\mathsf{G} \big) + \e/2 \leq 2\eta e^{4 \ell^2 + 5} \ell \delta_\ell^{-1}\, \frac{8}{\e}   + \e .
\end{eqnarray*}
In this way, we obtain Lemma \ref{l.rprob}.
\end{proof}

The proof of Theorem \ref{t.univonethird}(3) has been reduced to establishing Lemmas~\ref{l.zoneztwo} and \ref{l.av}.

\begin{lemma}\label{l.zoneztwo}
We have that $\mathcal{Z}_2 \leq \max\{ \mathcal{Z}_1,V \}  + 2\eta$.
\end{lemma}
\begin{proof}
Set $W = \max \{ \mathcal{Z}_1, V\}$. For $q > 0$, note that
$$
  I_{(-\infty,\infty)}^{W + q} - I_{(-\infty,\infty)}^{W} \geq \big( e^{t^{1/3}q} - 1 \big) I_{(-\ell,\ell)}^{W}
   \geq \frac{1}{2}  \big( e^{t^{1/3}q} - 1 \big) I_{(-\infty,\infty)}^{W} \, ,
$$
where we used $W \in \mathsf{J}$ to show the latter inequality (i.e. $I_{(-\ell,\ell)}^{W} \geq \frac{e^{t^{1/3} T_0/2}}{1+e^{t^{1/3} T_0/2}} \, I_{(-\infty,\infty)}^W\geq \frac{1}{2}\, I_{(-\infty,\infty)}^W$).
Thus, $I_{(-\infty,\infty)}^{W + q} \geq e^{t^{1/3}\eta} I_{(-\infty,\infty)}^{W}$ provided that
$\frac{1}{2}\big( e^{t^{1/3}q} - 1 \big) \geq e^{t^{1/3}\eta} - 1$ or equivalently $e^{t^{1/3}q} \geq 2e^{t^{1/3}\eta} -1$.
The latter inequality is satisfied whenever $q \geq 2\eta \geq 0$.
If $V < \mathcal{Z}_1$, then $W = \mathcal{Z}_1$, and  $I_{(-\infty,\infty)}^{\mathcal{Z}_1 + 2\eta} \geq e^{t^{1/3}\eta} I_{(-\infty,\infty)}^{\mathcal{Z}_1} = e^{t^{1/3}(y + \eta)}$, so that  $\mathcal{Z}_2 \leq \mathcal{Z}_1 + 2\eta$. If $\mathcal{Z}_1 \leq V$,
then $W = V$ and $I_{(-\infty,\infty)}^{1,V} \geq e^{t^{1/3}y}$. Note then that $I_{(-\infty,\infty)}^{V + 2\eta} \geq e^{t^{1/3}\eta} I_{(-\infty,\infty)}^{V} \geq e^{t^{1/3}(y + \eta)}$, so that $\mathcal{Z}_2 \leq V + 2\eta$.
\end{proof}

Recall the positive constants $C'$ and $T_0$ appearing before (\ref{e.econemtwo}).
\begin{lemma}\label{l.av}
For $t \geq 1$,
$$
\PP \big(  \HKPZFPline{t}{1}(-\ell) < V \big)  \leq \frac{\e}{4}\, .
$$
\end{lemma}
\begin{proof} The statement is equivalent to
\begin{equation}\label{e.minusell}
    \PP \Big(
  I_{(-\infty,-\ell)} + I_{(\ell,\infty)} > e^{-t^{1/3}T_0/2} I_{(-\ell,\ell)} \Big)  \leq \frac{\e}{4} \, .
\end{equation}
The occurrence of the event $E_{C'}(T)$ (\ref{e.wgnnernr}) entails $I_{(-\infty,-T)} + I_{(T,\infty)} \leq e^{-t^{1/3}T}$ and $I_{(-T,T)} \geq e^{-t^{1/3}C'}$.
Thus, provided that $T \geq 2C'$, it also ensures that $I_{(-\infty,-T)} + I_{(T,\infty)} \leq e^{-t^{1/3}T/2} I_{(-T,T)}$ is satisfied. It follows from (\ref{e.econemtwo}) and the definition of $\ell$ in (\ref{e.leqna}) that (\ref{e.minusell}) holds for all $t \geq 1$.
\end{proof}

This completes the proof of Theorem \ref{t.univonethird}(3).
\end{proof}

\subsection{Proof of Theorem \ref{app2}}\label{s.proofapp2}

By (\ref{e.identifywith}) we may replace $\HKPZFPlinetnw{t}(\cdot)$ by $\HKPZFPline{t}{1}(\cdot)$. Let $X$ be distributed according to $\PPendpoint$ and define $\tilde{X}= \frac{X}{t^{2/3}}$. We will make use of the following representation for $\PPendpoint$ which follows from its definition as well as from (\ref{crossairy2}) and (\ref{e.identifywith}). For any Borel set $I\subset \R$ we have that
$$
\PPendpoint\Big(\tilde{X} \in I\Big) \stackrel{(d)}{=} \frac{\int_{I} e^{t^{1/3}\HKPZFPline{t}{1}(y)}dy} { \int_{\R} e^{t^{1/3}\HKPZFPline{t}{1}(y)}dy}\,.
$$

\begin{proof}[Proof of Theorem \ref{app2}(1)]
Fix $\e$ as in the statement of the theorem. By Lemma \ref{e.tildekappalem} there exists a constant $\tilde{C}$ such that
\begin{equation}\label{egeww}
\PP\Big(\HKPZFPline{t}{1}(y) < \tilde{C} -y^2/4 \textrm{  for all } y\in\R\Big) \geq 1-\frac{\e}{2}.
\end{equation}
Likewise, we see that for $\tilde{C}'$ large enough
\begin{equation}\label{weqqbhh}
\PP\Big(\HKPZFPline{t}{1}(y) > -\tilde{C}' \textrm{  for all } y\in[-1,1]\Big) \geq 1-\frac{\e}{2}.
\end{equation}
This follows by appealing to Theorem \ref{mainthm}(3) to see that on $[-1,1]$ the curve $\HKPZFPline{t}{1}(\cdot)$ is absolutely continuous with respect to a Brownian bridge, with a Radon-Nikodym derivative which is tight for $t\geq 1$. This and the tightness and stationarity of the one-point distribution of $\HKPZFPline{t}{1}(y)+y^2/2$ (Proposition \ref{ACQprop}(1)) implies (\ref{weqqbhh}).

When both of the events in (\ref{egeww}) and (\ref{weqqbhh}) hold, a circumstance whose probability is at least $1-\e$, we can bound
$$
\PPendpoint\Big(\tilde{X} \notin [-C,C] \Big) \leq  \frac{ \int_{\R\setminus [-C,C]} e^{t^{1/3}\big(\tilde{C} - y^2/4\big)}dy} { \int_{[-1,1]} e^{-t^{1/3}\tilde{C}'}dy} =
\frac{ \int_{\R\setminus [-C,C]} e^{t^{1/3}\big(\tilde{C} - y^2/4\big)}dy} {2e^{-t^{1/3}\tilde{C}'}}\,.
$$
Taking $C$ large enough and using the Gaussian tail bounds such as in Lemma \ref{l.normallb}, the numerator can be bounded by $\e\cdot 2e^{-t^{1/3}\tilde{C}'}$. Thus, for such a $C$, we obtain
$$
\PPendpoint\Big(\tilde{X} \notin [-C,C] \Big) \leq \e\,.
$$
As  this occurs with $\PP$-probability at least $1-\e$, we have arrived at the result desired to prove Theorem \ref{app2}(1).
\end{proof}

\begin{proof}[Proof of Theorem \ref{app2}(2)]
Fix $\e$ and $x$ as in the statement of the theorem.
For $h_0,h,\eta>0$, define the event
$$
\mathsf{E}_{h_0,h,\eta} = \bigg\{\exists \, y \notin [x-2h_0,x+2h_0]:\, \min_{z\in[y-h_0,y+h_0]} \HKPZFPline{t}{1}(z) \geq \eta + \max_{z\in [x-h,x+h]} \HKPZFPline{t}{1}(z)\bigg\}.
$$
This event is illustrated in Figure \ref{f.polymerendpoint}. The reason for assuming $y \notin [x-2h_0,x+2h_0]$ is to ensure the separation of the intervals around $x$ and $y$.
\begin{figure}
\centering\epsfig{file=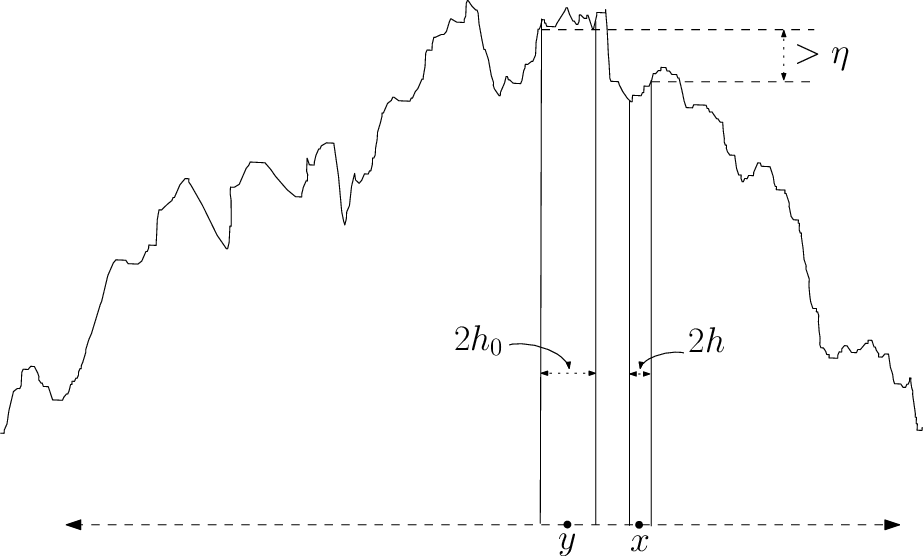, width=12cm}
\caption{Illustration of the event $\mathsf{E}_{h_0,h,\eta}$ in the proof of Theorem \ref{app2}(2).}\label{f.polymerendpoint}
\end{figure}

When $\mathsf{E}_{h_0,h,\eta}$ holds,
\begin{equation}\label{hmhf}
\PPendpoint\Big(\tilde{X} \in [x-h,x+h] \Big) \leq \frac{ 2h\cdot e^{t^{1/3} \big(\max_{z\in [x-h,x+h]} \HKPZFPline{t}{1}(z)\big)}}{2 h_0\cdot e^{t^{1/3}\big( \eta + \max_{z\in [x-h,x+h]} \HKPZFPline{t}{1}(z)\big)}} = e^{-t^{1/3} \eta} h_0^{-1} \cdot h.
\end{equation}

\begin{claim}\label{clfsfas}
There exists $h_0>0$ and $\eta>0$ such that, for all $h<h_0$ and $t\geq 1$,
$$
\PP\Big(\mathsf{E}_{h_0,h,\eta}\Big) \geq 1-\e.
$$
\end{claim}

To show Claim \ref{clfsfas}, observe that, by choosing $h_0$ and $\eta$ small enough, we may be assured that, for all $t\geq 1$,
$$
\PP\bigg(\exists \, y\notin [x-2h_0,x+2h_0]:\, \min_{z\in[y-h_0,y+h_0]} \HKPZFPline{t}{1}(z) \geq 2\eta +\HKPZFPline{t}{1}(x-h)\bigg) \geq 1-\frac{\e}{2}.
$$
This readily follows from the curve $\HKPZFPline{t}{1}(\cdot)$ being absolutely continuous with respect to Brownian bridge, with a Radon-Nikodym derivative which is tight for $t\geq 1$ (Theorem \ref{mainthm}(3)).
For the same reason,
by choosing $h,\eta$ small enough, we find that, for all $t\geq 1$,
$$
\PP\Big(\max_{z\in [x-h,x+h]}  \HKPZFPline{t}{1}(z) \leq \eta + \HKPZFPline{t}{1}(x-h)\Big) \geq 1-\frac{\e}{2}.
$$
Combining these two observations proves Claim \ref{clfsfas}.

For such $h_0$ and $\eta$ as Claim \ref{clfsfas} yields, take $h<h_0$ small enough so that $e^{-\eta} h_0^{-1} \cdot h<\e$. From~(\ref{hmhf}), we then find that, when $\mathsf{E}_{h_0,h,\eta}$ holds,
$$
\PPendpoint\Big(\tilde{X} \in [x-h,x+h] \Big)< \e.
$$
Claim \ref{clfsfas} implies that this holds with $\PP$-probability at least $1-\e$, as desired to prove Theorem~\ref{app2}(2).
\end{proof}

\subsection{Proof of Theorem \ref{t.tailbounds}} \label{s.tailboundsproof}

Consider a fixed time $t>0$. We call a bound on $\PP\big( \HKPZ{t}{0} <-s\big)$ a {\it lower tail bound} and on $\PP\big( \HKPZ{t}{0} >s\big)$ an {\it upper tail bound}.

%Lemma \ref{l.genindata} and the identification of (\ref{e.identifywith}) \note{should include some review of the relationships between different curves} shows how the one-point distribution of $\HKPZ{t}{0}$ is given by \note{in some other places it seems we have forgotten the minus sign in the initial data term... should fix though not a big deal}
%\begin{equation}\label{e.dasqn}
%\HKPZ{t}{0} \stackrel{(d)}{=} -\frac{t}{24} - \frac{2}{3} \log t +  \log \left( \int_{-\infty}^{\infty} e^{\HKPZFPline{t}{1}(y) + \Hzero{-t^{2/3}y}}\dd y\right).
%\end{equation}
%In addition to this, we will use the fact from Theorem \ref{mainthm} that $\HKPZFPline{t}{1}$ arises as the lowest indexed curve of the KPZ$_t$ fixed point line ensemble, which has the $\Ham_1$-Brownian Gibbs property.

\begin{proof}[Proof of Theorem \ref{t.tailbounds}: lower tail]
We have assumed that the KPZ initial data $\Hzero{x}$ satisfies hypothesis $\Hyp(C+t^{1/3} s_0,\delta,\kappa,M)$. This implies that
\begin{equation}\label{e.egbnh}
\textrm{Leb}\big(y\in [-M,M]:\Hzero{-t^{2/3} y} <-C-t^{1/3} s_0\big) \geq \delta.
\end{equation}
For $\tilde{s}\in \R$, define the event
$$
\mathsf{E}_{\tilde{s}} = \Big\{\HKPZFPline{t}{1}(y) \geq  -M^2/2 -\tilde{s}\, \forall y\in [-M,M]\Big\}.
$$
We claim that there are some constants $\tilde{c}_1,\tilde{c}_2>0$ such that, for all $\tilde{s}>1$,
\begin{equation}\label{e.littleclaim}
\PP\big(\mathsf{E}_{\tilde{s}}\big) \geq 1-\tilde{c}_1e^{-\tilde{c}_2 \tilde{s}^2}.
\end{equation}
To prove this claim, observe that by the narrow wedge initial data lower tail bound (Proposition \ref{ACQprop}(3)), the stationarity of the one-point distribution of $\HKPZFPline{t}{1}(y)+y^2/2$ (Proposition \ref{ACQprop}(1)), and the union bound,
$$
\PP\Big(\big\{\HKPZFPline{t}{1}(M) \geq -M^2/2 - \tilde{s}/2\big\} \cap  \big\{\HKPZFPline{t}{1}(-M) \geq -M^2/2 - \tilde{s}/2\big\}\Big) \geq 1- \tilde{c}_3 e^{-\tilde{c}_4 \tilde{s}^2},
$$
for some constants $\tilde{c}_3,\tilde{c}_4>0$.
We can then utilize either the Brownian absolute continuity of $\HKPZFPline{t}{1}$ on the interval $[-M,M]$ or the $\Ham_t$-Brownian Gibbs property and a simple application of Lemma~\ref{monotonicity1} to show that the probability that the value of $\HKPZFPline{t}{1}(\cdot)$ decreases by more than $\tilde{s}/2$ from the value at $-M$ and $M$ is less than $\tilde{c}_5 e^{-\tilde{c}_6 \tilde{s}^2}$ for some constants $\tilde{c}_5,\tilde{c}_6>0$. These observations readily imply~(\ref{e.littleclaim}).

On the event $\mathsf{E}_{\tilde{s}}$, and owing to (\ref{e.egbnh}), we find that
\begin{eqnarray*}
 & &\hskip-.25in -\frac{t}{24} + \frac{2}{3} \log t+ \log \bigg( \int_{-\infty}^{\infty} e^{t^{1/3}\big(\HKPZFPline{t}{1}(y) + t^{-1/3}\Hzero{-t^{2/3} y}\big)}\dd y\bigg) \\
 & \geq &
 -\frac{t}{24} + \frac{2}{3} \log t+ \log\bigg(\delta \cdot e^{t^{1/3} \big(-M^2/2 - \tilde{s} - s_0 -t^{-1/3}C\big)}\bigg)
 \geq
-C' (\tilde{s}+s_0)
\end{eqnarray*}
for some suitably large constant $C'=C'(t,C,\delta,M)$. Allied with Lemma \ref{l.varproblem} and (\ref{e.littleclaim}), this implies that
$$
\PP \Big(\HKPZ{t}{0} \leq -C'(\tilde{s}+s_0)\Big) \leq \tilde{c}_1e^{-\tilde{c}_2 \tilde{s}^2}.
$$
Letting $s= \tilde{s}+s_0$, it follows then that, for all $s\geq s_0+1$,
$$
\PP \Big(\HKPZ{t}{0} \leq -C's\Big) \leq \tilde{c}_1e^{-\tilde{c}_2 (s-s_0)^2},
$$
as desired to prove the lower tail bound in Theorem \ref{t.tailbounds}.
\end{proof}

\begin{proof}[Proof of Theorem \ref{t.tailbounds}: upper tail]
We have assumed the KPZ initial data $\Hzero{x}$ satisfies hypothesis $\Hyp(C+t^{1/3} s_0,\delta,\kappa,M)$ with $\kappa>1-t^{-1}$. For $\nu\in \big(0,1-t(1-\kappa)\big)$ (note that since $\kappa>1-t^{-1}$, $1-t(1-\kappa)>0$), let $C',c_1,c_2$ be constants specified by Lemma \ref{e.tildekappalemwwww} (we put a prime on $C$ to distinguish it from the $C$ in the hypothesis) so that, for all $s-s_0\geq 1$,
$$
\PP\big(\mathsf{E}_{s-s_0}\big) \geq 1- c_1 e^{c_2 (s-s_0)}
$$
where the event
$$
\mathsf{E}_{s-s_0} = \Big\{\HKPZFPline{t}{1}(y) + y^2/2 < C' + s-s_0 + \nu y^2/2\textrm{ for all }y\in \R\Big\}.
$$
We then have that, when the event $\mathsf{E}_{s-s_0}$ occurs,
\begin{eqnarray*}
\HKPZ{t}{0} &\stackrel{(d)}{=}&  \log \left( \int_{-\infty}^{\infty} e^{t^{1/3}\big(\HKPZFPline{t}{1}(y) + t^{-1/3}\Hzero{-t^{2/3}y}\big)}\dd y\right) -\frac{t}{24} + \frac{2}{3} \log t\\
&\leq&  \log \left( \int_{-\infty}^{\infty} e^{t^{1/3}\big(C'+s-s_0 + (-1+\eta + t(1-\kappa))y^2/2 + t^{-1/3}(C+t^{1/3} s_0)\big)}\dd y\right) -\frac{t}{24} + \frac{2}{3} \log t\\
&=&  s + \log \left( \int_{-\infty}^{\infty} e^{t^{1/3}\big(C' + (-1+\eta + t(1-\kappa))y^2/2 + t^{-1/3}C)\big)}\dd y\right) -\frac{t}{24} + \frac{2}{3} \log t\\
&\leq & s+ c_4,
\end{eqnarray*}
for a suitably large constant $c_4>0$.
The first equality in distribution is by Lemma \ref{l.varproblem}. The inequality between the first and second lines is due to the fact that we are assuming that $\mathsf{E}_{s-s_0}$ occurs, and due to the hypothesis that $\Hzero{y} \leq C+t^{1/3} s_0 + (1-\kappa) y^2/2$ (and hence $t^{-1/3} \Hzero{-t^{2/3}y} \leq t^{-1/3} C + s_0 + t(1-\kappa) y^2/2$). The equality between the second and third lines comes from canceling $s_0$ factors and moving the $s$ outside of the exponential. The inequality between the third and fourth lines comes from the fact that by assumption on $\eta$, $-1+\eta + t(1-\kappa)<0$ and hence the Gaussian integral is convergent and bounded by a constant.

This means that for the constant $c_4$ as defined above,
$$
\PP \big(\HKPZ{t}{0} \leq s+ c_4\big) \geq \PP\big(\mathsf{E}_{s-s_0}\big) \geq 1- c_1 e^{c_2 (s-s_0)}.
$$
as desired to prove the theorem.
\end{proof}

\section{Proof of Theorem \ref{mainthm}}\label{s.mainthmproof}
We turn now to the proof of Theorem \ref{mainthm}, the main result of this paper. The main input for this proof is Theorem \ref{seqncmpt}, which provides the sequential compactness of the line ensembles $\HSDFPlinet{t}{N}$ (Definition \ref{d.cntx}) along with a uniform control over $t\in[1,\infty)$ of its normalizing constant. The ensembles $\HSDFPlinet{t}{N}$ also enjoy the $\Ham_t$-Brownian Gibbs property (Corollary \ref{rescaledHBGP}). Using the Skorohod representation theorem and some simple coupling arguments, we transfer these finite $N$ properties into those of a subsequential limit $\HKPZFPlinet{t}$,  hence proving Theorem \ref{mainthm}.

The first small hurdle to overcome in making use of Theorem \ref{seqncmpt}(1) is the fact that Theorem~\ref{seqncmpt}(1) only deals with sequential compactness on a $\{1,\ldots,k\}\times[-T,T]$ restriction of  $\HSDFPlinet{t}{N}$. The next lemma uses a diagonalization argument to extend this to weak convergence as a line ensemble (Definition \ref{maindefline}) along a suitable subsequence.

\begin{lemma}\label{allsubseq}
There exists a strictly increasing sequence $\{N_1,N_2,\ldots\}\subset \N$ such that, simultaneously for all $k\geq 1$ and $T>0$, the $\{1,\ldots,k\}\times [-T,T]$-indexed line ensemble $\big\{\HSDFPline{t}{N_i}{n}(x):n\in \{1,\ldots,k\}, x\in [-T,T]\big\}$ has a weak limit (as a line ensemble) as $i\to \infty$.
\end{lemma}
\begin{proof}
Fix a strictly increasing sequence $\{k_1,k_2,\ldots\}\subset \N$ as well as a strictly increasing sequence $\{T_1,T_2,\ldots\}\subset (0,\infty)$ such that $T_j\to \infty$ as $j\to \infty$. Let $N^1=\big\{N^1_i:i\in \N\big\}\subset \N$ be a strictly increasing sequence such that $\big\{\HSDFPline{t}{N^1_i}{n}(x):n\in \{1,\ldots,k_1\}, x\in [-T_1,T_1]\big\}$ is convergent as a line ensemble as $i\to \infty$. The existence of such a subsequence is assured by the sequential compactness afforded by Theorem \ref{seqncmpt}(1). Inductively extract further strictly increasing subsequences $N^j=\big\{N^j_i:i\in \N\big\}\subset \N^{j-1}$ such that
$\big\{\HSDFPline{t}{N^j_i}{n}(x):n\in \{1,\ldots,k_j\}, x\in [-T_j,T_j]\big\}$ is convergent as a line ensemble as $i\to \infty$. Define $N_i= N^{i}_i$ and observe that, for any $j$, $\big\{\HSDFPline{t}{N_i}{n}(x):n\in \{1,\ldots,k_j\}, x\in [-T_j,T_j]\big\}$ is convergent as a line ensemble as $i\to \infty$. Finally, note that if $\big\{\HSDFPline{t}{N_i}{n}(x):n\in \{1,\ldots,k\}, x\in [-T,T]\big\}$ is convergent as a line ensemble as $i\to \infty$, then so too is the further restriction where $k$ is replaced by $k'\leq k$ and $T$ by $T'\leq T$.
\end{proof}

Proposition \ref{HBrownianGibbsLimitProp} below shows the continuity of the $\Ham$-Brownian Gibbs property under weak convergence of line ensembles. We consider a line ensemble with $k$ curves, but only require that the Gibbs property holds for the lowest $k-1$ indexed curves. This is because in applications of this proposition, the $k$ curves considered may be the $k$ lowest indexed curves of an $N$ curve line ensemble ($N>k$). Hence, curve $k$ does not have the Gibbs property without keeping track of curve $k+1$.

\begin{proposition}\label{HBrownianGibbsLimitProp}
Fix $k\geq 1 $ and $T>0$. Consider a Hamiltonian $\Ham$ (Definition \ref{maindefHBGP}) and a sequence of line ensembles $\mathcal{L}^{N}:\{1,\ldots, k\}\times [-T,T]\to \R$ which have the $\Ham$-Brownian Gibbs property when restricted to indices $1,\ldots, k-1$ and which converge weakly as line ensembles (Definition \ref{maindefline}) to a limit line ensemble $\mathcal{L}^{\infty}:\{1,\ldots, k\}\times [-T,T]\to \R$.
\begin{enumerate}
\item Then $\mathcal{L}^{\infty}$ also has the $\Ham$-Brownian Gibbs property when restricted to indices $1,\ldots, k-1$.
\item Moreover, for any interval $(a,b)\subset [-T,T]$ and any $1\leq k'\leq k-1$, the normalizing constant (which is $\Fext\big(\{1,\ldots, k\}\times (a,b)\big)$-measurable) $\partfunc{1}{k'}{(a,b)}{\vec{x}^N}{\vec{y}^N}{f^N}{g^N}{\Ham}$ converges in distribution to $\partfunc{1}{k'}{(a,b)}{\vec{x}^\infty}{\vec{y}^\infty}{f^\infty}{g^\infty}{\Ham}$, where, for $N\in \N\cup \{\infty\}$, we define
$\vec{x}^N = \big(\mathcal{L}^N_1(a),\ldots, \mathcal{L}^N_{k'}(a)\big)$, $\vec{y}^N = \big(\mathcal{L}^N_1(b),\ldots, \mathcal{L}^N_{k'}(b)\big)$, $f^N\equiv +\infty$ and $g^N = \mathcal{L}^N_{k'+1}$ restricted to $(a,b)$.
\end{enumerate}
\end{proposition}

Given the proposition (proved below) let us conclude now the proof of Theorem \ref{mainthm}. Lemma \ref{allsubseq} implies that, along an increasing subsequence of values of $N$, the finite $N$ line ensemble $\HSDFPlinet{t}{N}$ converges weakly as a line ensemble. Call the weak limit $\HKPZFPlinet{t}$ and let $\HKPZlinet{t}$ be related to $\HKPZFPlinet{t}$ according to (\ref{e.gfewgwe}). In order to prove Theorem \ref{mainthm}, it remains to show that $\HKPZFPlinet{t}$ has the three properties claimed by that theorem. Property (1) follows immediately from Proposition \ref{QMthm}. Property (2) follows from the $\Ham_t$-Brownian Gibbs property for $\HSDFPlinet{t}{N}$ (Proposition \ref{NeilGibbsprop}), in conjunction with Lemma \ref{allsubseq} and Proposition \ref{HBrownianGibbsLimitProp}(1). Property (3) also follows from the $\Ham_t$-Brownian Gibbs property for $\HSDFPlinet{t}{N}$ (Proposition \ref{NeilGibbsprop}), in conjunction with Lemma \ref{allsubseq}, Theorem \ref{seqncmpt}(2) and Proposition \ref{HBrownianGibbsLimitProp}(2). Thus, up to proving Proposition \ref{HBrownianGibbsLimitProp}, this
completes the proof of Theorem \ref{mainthm}.

\begin{proof}[Proof of Proposition \ref{HBrownianGibbsLimitProp}]
Proposition \ref{HBrownianGibbsLimitProp}(1) claims that the $\Ham$-Brownian Gibbs property survives under weak convergence as a line ensemble. We will demonstrate this via a coupling argument. Proposition \ref{HBrownianGibbsLimitProp}(2) claims convergence of the normalizing constant.  We will explain at the end of the proof how this follows quite simply.

Let us focus on proving Proposition \ref{HBrownianGibbsLimitProp}(1). By assumption, the limit of $\mathcal{L}^N$ is supported on the space of $k$ continuous curves on $[-T,T]$. As we are dealing with a separable space we may apply the Skorohod representation theorem (see \cite{Bill} for instance). This implies that there exists a probability space $(\Omega,\mathcal{B},\PP)$ on which all of the $\mathcal{L}^N$ for $N\in \N\cup \{\infty\}$ are defined with the correct $\PP$-marginals and for which $\mathcal{L}^N(\omega)\rightarrow \mathcal{L}^{\infty}(\omega)$ in the uniform topology (for almost every $\omega\in \Omega$).

We will show that, for any fixed line index $i\in\{1,\ldots, k-1\}$ and any two times $a,b\in [-T,T]$ with $a<b$, the law of $\mathcal{L}^{\infty}$ is unchanged if $\mathcal{L}^{\infty}_i$ is resampled between times $a$ and $b$ according to the law $\PH{i-1}{i+1}{(a,b)}{\mathcal{L}_i^\infty(a)}{\mathcal{L}_i^\infty(b)}{\mathcal{L}^\infty_{i-1}}{\mathcal{L}^\infty_{i+1}}{\Ham}$. The argument we now present (for one line resampling) clearly works for resampling several lines consecutive lines. This invariance under resampling arbitrary numbers of consecutive lines is equivalent to the $\Ham$-Brownian Gibbs property and hence implies Proposition \ref{HBrownianGibbsLimitProp}(1). For simplicity, in what follows let us assume $i\neq 1$ (the argument is modified with minor notational changes for the remaining case $i=1$).

We show this invariance of $\mathcal{L}^{\infty}$ by coupling the resampling procedure for all values of $N$. In order to do this, we perform the resampling in two steps. Let $\big\{U_\ell\}_{\ell \in \N}$ be a sequence of independent random variables, each having the uniform distribution on $[0,1]$, and let $\big\{B_{\ell}\big\}_{\ell\in \N}$ be a sequence of independent Brownian bridges $B_\ell:(a,b)\rightarrow \R$ such that $B_\ell(a)=B_\ell(b)=0$; our probability space may be augmented to accommodate these independently of existing data.

In the first step of the resampling, we construct the $\ell$-th candidate resampling of line $i$, which is given by
$$
\mathcal{L}_i^{N,\ell}(t)= B_\ell(t) + \frac{b-t}{b-a}
\mathcal{L}_i^N(a) + \frac{t-a}{b-a} \mathcal{L}_i^N(b).
$$
The last two terms on the right add the necessary affine shift to the Brownian bridge to ensure that $\mathcal{L}_i^{N,\ell}(a)=\mathcal{L}_i^{N}(a)$ and $\mathcal{L}_i^{N,\ell}(b)=\mathcal{L}_i^{N}(b)$.

In the second step, we check whether
\begin{equation}\label{pagetwentyone}
U_\ell \leq \boltNew(N,\ell):=\bolt{i-1}{i+1}{(a,b)}{\mathcal{L}_i^N(a)}{\mathcal{L}_i^N(b)}{\mathcal{L}^N_{i-1}}{\mathcal{L}^N_{i+1}}{\Ham}\big(\mathcal{L}_i^{N,\ell}\big),
\end{equation}
and {\it accept} the candidate Brownian bridge sample $B_{\ell}$ if this event occurs (see Remark \ref{normrem}). For $N\in \N\cup \{\infty\}$ define $\ell(N)$ to be the minimal value of $\ell$ for which we accept the sample $B_{\ell}$. Write $\mathcal{L}^{N,\rm{re}}$ for the line ensemble with the $i^{\rm{th}}$ line replaced by $\mathcal{L}_{i}^{N,\ell(N)}$. (Here, $re$ stands for ``resampled''.) The (single curve case of the) Brownian Gibbs property for $\mathcal{L}^N$ is equivalent to the fact that $$\mathcal{L}^N \stackrel{(d)}{=} \mathcal{L}^{N,\rm{re}}.$$

If we show that this same equality in distribution holds for $N=\infty$, it will imply the desired (single curve case of the) $\Ham$-Brownian Gibbs property $\mathcal{L}^{\infty}$. For this, it suffices to show that almost surely
\begin{equation}\label{elllim}
\lim_{N\rightarrow \infty} \ell(N) = \ell(\infty),
\end{equation}
with $\ell(\infty)$ almost surely finite.
This is because we already know that $\mathcal{L}^N$ is converging to $\mathcal{L}^\infty$ and the above fact would further show that $\mathcal{L}^{N,\rm{re}}$ is converging to $\mathcal{L}^{\infty,\rm{re}}$ (all in the uniform topology). The two sets of line ensembles being equal in law for each finite $N$ (by the Gibbs property for such $N$), their almost sure $\Omega$-pointwise limits must also be equal in law. Thus, we see that the $\Ham$-Brownian Gibbs property for $\mathcal{L}^\infty$ follows from proving (\ref{elllim}).  Two lemmas will prove~(\ref{elllim}).

\begin{lemma}\label{prevlemma}
The sequence $\big\{ \ell(N)\big\}_{N \in \N}$ is bounded almost surely. Moreover, $\ell(\infty)$ is finite almost surely.
\end{lemma}
\begin{proof}
For fixed $\mathcal{L}_i^\infty(a),\mathcal{L}_i^\infty(b),\mathcal{L}^\infty_{i-1},\mathcal{L}^\infty_{i+1}$, and a choice of the random input $\mathcal{L}^{\infty,\ell}_i$, the law of $\boltNew(\infty,\ell)$ (defined as in (\ref{pagetwentyone}) but with $\mathcal{L}^{N,\ell}_{i}$ replaced by $\mathcal{L}^{\infty,\ell}_i$) is supported on $(0,1)$. Hence, for some $\e > 0$, this random variable $\boltNew(\infty,\ell)$ is at least $\e$ with probability at least $\e$, which implies that $\ell(\infty)$ is finite almost surely. By uniform convergence of the approximating line ensembles and continuity of the function $\Ham$, we find that the corresponding random variable $\boltNew(N,\ell)$ associated to the $N$-th line ensemble exceeds $\e/2$ with probability at least $\e/2$, for all $N$ high enough. This implies that there exists a geometric random variable $G$ of parameter at most $1 - (\e/2)^2$, such that $\ell(N) \leq G$ for all high enough $N$, hence completing the proof of the lemma.
\end{proof}

\begin{lemma}\label{lemuniquel}
There exists a unique limit point for $\big\{\ell(N)\big\}_{N\in \N}$ and this limit point is $\ell(\infty)$.
\end{lemma}
\begin{proof}
Condition on the values of the (almost surely finite length) sequence $\big\{\boltNew(\infty,j)\big\}_{j=1}^{\ell(\infty)}$. Note that $\boltNew(\infty, \ell(\infty))\in (0,1)$ as readily follows from the fact that $\ell(\infty)$ is almost surely finite and each $\boltNew(\infty,\ell)\in (0,1)$. Note that the conditional distribution of $U_{\ell(\infty)}$ is the uniform distribution on $\big[0,\boltNew(\infty,\ell(\infty))\big]$. Therefore, the inequality $U_{\ell(\infty)} < \boltNew\big(\infty,\ell(\infty)\big)$ is satisfied almost surely, so that, by uniform convergence,  $U_{\ell(\infty)} < \boltNew\big(N,\ell(\infty)\big)$ for $N$ high enough. Hence, $\limsup_{N\to \infty} \ell(N) \leq \ell(\infty)$ almost surely.

Conversely, for any given $1 \leq j < \ell(\infty)$, the conditional distribution of $U_j$ is the uniform distribution on $\big[\boltNew\big(\infty,\ell(\infty)\big),1\big]$. Therefore, $U_j > \boltNew(\infty,j)$ almost surely, so that, by uniform convergence, $U_j > \boltNew(N,j)$ for $N$ high enough. Hence, $\liminf_{N\to \infty} \ell(N) \geq \ell(\infty)$ almost surely.
\end{proof}
These two lemmas complete the proof of the first claim of Proposition \ref{HBrownianGibbsLimitProp}. The second claim follows readily from the definition of the normalizing constant and the uniform convergence afforded by the Skorohod representation theorem.
\end{proof}

\section{Proof of Theorem \ref{seqncmpt}}\label{s.seqncmptproof}

The proof of Theorem \ref{seqncmpt} relies heavily on three key technical propositions which we now state (their proofs are delayed until Section \ref{s.threekeyproofs}). Theorem \ref{seqncmpt}(1) deals with the sequential compactness of the $k$ lowest indexed curves of $\HSDFPlinet{t}{N}$ as $N\to \infty$. Using the three key propositions, we prove Proposition \ref{p.one} which shows that, for $t>0$ fixed, there is uniform control as $N\to \infty$ on the normalizing constant for the $k$ lowest indexed curves of $\HSDFPlinet{t}{N}$ on a fixed interval $(\ell,r)$. This suffices to prove sequential compactness. Theorem \ref{seqncmpt}(2) deals with the uniform control (as both $t\geq 1$ and $N\geq N_0$ vary) of the normalizing constant for the lowest indexed curve of $\HSDFPlinet{t}{N}$ on a fixed interval $(\ell,r)$. Using the three key technical propositions, we prove Proposition \ref{p.oneunif} which exactly demonstrates this control.

\subsection{Three key technical propositions}\label{threekey}

Together the following three propositions will show that, for each given $n \in \N$, the curve $\HSDFPline{t}{N}{n}(\cdot)$ is bounded above and below on compact intervals, with bounds which are independent of spatial location (after a parabolic shift) and which hold uniformly in $t \geq \tzero$ and $N\geq N_0$ (with $N_0$ depending on $t$ and the choice of compact interval). %In the ensuing statements, the function $N_0 = N_0(x_0,t)$ can be specified by Lemma~\ref{l.input}.

\begin{proposition}\label{p.epstr}
Fix $n \in \N$. For each $\e > 0$, there exists $R_n = R_n(\e) > 0$ such that, for any $x_0 > 0$ and $t \geq \tzero$, there exists $N_0(x_0,t,\e)$ so that, for all $N \geq N_0(x_0,t,\e)$ and $\bar{x} \in [-x_0,x_0]$,
$$
\PP\Bigg(\inf_{s\in \big[\bar{x}-\frac{1}{2},\bar{x}+\frac{1}{2}\big]}\big(\HSDFPline{t}{N}{n}(s) + s^2/2\big) < - R_n \Bigg) < \e \, .
$$

\end{proposition}

\begin{proposition}\label{p.lbpara}
Fix $n \in \N$. For all $\e\geq 0$ and $\delta \in(0,\tfrac{1}{8})$, there exists $T_0 > 0$ such that,
for any $x_0 > T_0$ and $t \geq \tzero$, there exists $N_0(x_0,t,\e,\delta)$ so that, for all $N \geq N_0(x_0,t,\e, \delta)$,
$T \in [ T_0 , x_0 ]$ and $y_0 \in [-x_0,x_0 - T]$,
$$
\PP\Bigg( \inf_{s\in[y_0, y_0 + T]} \big(\HSDFPline{t}{N}{n} (s) + s^2/2 \big) < - \delta T^2 \Bigg) < \e \, .
$$
\end{proposition}

\begin{proposition}\label{p.oneptub}
Fix $n \in \N$. For each $\e > 0$, there exists $\Rstar_n = \Rstar_n(\e) > 0$ such that, for any $x_0 > 0$ and $t \geq \tzero$, there exists $N_0(x_0,t,\e)$ so that, for all $N \geq N_0(x_0,t,\e)$ and $\bar{x} \in [-x_0,x_0 - 1]$,
$$
\PP\Bigg( \sup_{s \in (\bar{x},\bar{x} + 1)} \big(\HSDFPline{t}{N}{n}(s) + s^2/2 \big) > \Rstar_n \Bigg) < \e \, .
$$
\end{proposition}

\subsection{Controlling the normalizing constant}

Using the three key propositions, we prove two results about the probability that normalizing constants for  $\HSDFPlinet{t}{N}$ becomes small. Proposition~\ref{p.one} works with a fixed $t$ and an arbitrary indexed curve, while Proposition~\ref{p.oneunif} works with all $t\geq 1$ in a uniform manner and only considers the lowest indexed curve. Before stating and proving these propositions, let us remark on how to interpret their statements. In Proposition~\ref{p.one} we consider $\partfunc{n}{n}{(\ell,r)}{\HSDFPline{t}{N}{n}(\ell)}{\HSDFPline{t}{N}{n}(r)}{\HSDFPline{t}{N}{n-1}}{\HSDFPline{t}{N}{n+1}}{\Ham_t}$ as a random variable under the measure $\PP$. Given the information provided by $\HSDFPline{t}{N}{n}(\ell), \HSDFPline{t}{N}{n}(r)$ and the curves $\HSDFPline{t}{N}{n-1},\HSDFPline{t}{N}{n+1}$ on the interval $(\ell,r)$, this normalizing constant is deterministic. However, since this input data is random with respect to $\PP$, the normalizing constant is likewise
random. It is this random variable whose distribution we are studying. A similar consideration applies to Proposition \ref{p.oneunif} below.

\begin{proposition}\label{p.one}
Fix $t \geq \tzero$, $k_1\leq k_2 \in \N$  and an interval $(\ell,r)\subset \R$. Then, for all $\e > 0$, there exists $\delta > 0$ and $N_0(t,n,\ell,r,\e,\delta)>n$ such that, for all $N \geq N_0(t,k_1,k_2,\ell,r,\e,\delta)$,
$$
\PP\Big(\partfunc{k_1}{k_2}{(\ell,r)}{\vec{x}}{\vec{y}}{\HSDFPline{t}{N}{k_1-1}}{\HSDFPline{t}{N}{k_2+1}}{\Ham_t}< \delta \Big) < \e \,
$$
where $\vec{x}=\big(\HSDFPline{t}{N}{i}(\ell)\big)_{i=k_1}^{k_2}$, and $\vec{y}=\big(\HSDFPline{t}{N}{i}(r)\big)_{i=k_1}^{k_2}$, and with the convention that $\HSDFPline{t}{N}{0}\equiv +\infty$.
\end{proposition}
\begin{proof}
Propositions \ref{p.epstr}, \ref{p.lbpara} and \ref{p.oneptub} imply that (given $t, \ell,r,k_1,k_2,\e$) there exists $M>0$ such that the event
$$
\mathsf{E} = \Big\{ \inf_{s\in[\ell, r]} \HSDFPline{t}{N}{k_1-1}(s) > - M\Big\} \bigcap \Big\{ \sup_{s\in [\ell, r]} \HSDFPline{t}{N}{k_2+1}(x) < M\Big\}\bigcap \Big\{\big\vert\HSDFPline{t}{N}{i}(s)\big\vert \leq M \textrm{ for }s\in \{ \ell,r \}, i\in(k_1,\ldots, k_2)\Big\}
$$
has probability $\PP(\mathsf{E})\geq 1-C\e$ for some constant $C=C(\ell,r)$. For instance $C(\ell,r) = r-\ell+2+4(k_2-k_1+1)$ works since the first term in $\mathsf{E}$ contributed $\e$ due to a single application of Proposition \ref{p.lbpara}, the second term contributed $\e \,\lceil r-\ell \rceil\geq \e(r-\ell +1)$ due to that many applications of Proposition  \ref{p.oneptub} and the third term in $\mathsf{E}$ contributed $4\e(k_2-k_1+1)$ due to applications of Proposition \ref{p.epstr} and \ref{p.oneptub} for the lower and upper bounds (respectively) and for $s=\ell$ and $s=r$, as well as all $i\in (k_1,\ldots, k_2)$.

For $\delta$ to be specified soon, define the event
%Choose $\delta > 0$ to be
%$$
%\delta =  \big( 1 - 2e^{-2} \big)  \exp \bigg\{ - (r- \ell) \Big(  e^{-t^{1/3} \big(  M_n + M_{n+1} + (r - \ell)^{1/2} \big) }  +  e^{-t^{1/3} \big( M_{n-1} + M_n + (r - \ell)^{1/2} \big) }   \Big)  \bigg\}  \, .
%$$and define the event
$$
\mathsf{D} = \left\{\partfunc{k_1}{k_2}{(\ell,r)}{\vec{x}}{\vec{y}}{\HSDFPline{t}{N}{k_1-1}}{\HSDFPline{t}{N}{k_2+1}}{\Ham_t}<\delta\right\}.
$$
We may bound
$$
\PP\big(\mathsf{D}\big) \leq \PP\big(\mathsf{D}\cap \mathsf{E}\big) + \PP\big(\mathsf{E}^c\big).
$$

Since we have already observed that $\PP\big(\mathsf{E}^c\big)\leq C\e$, the proposition will follow (by replacing $\e$ with $\e/C$) if we can show that
\begin{equation}\label{e.claimde}
\PP\big(\mathsf{D}\cap \mathsf{E}\big) =0.
\end{equation}
%In order to show this claim observe that
%$$
%\PP\big(\mathsf{D}\cap \mathsf{E}\big) = \EE\Big[\mathbf{1}_{\mathsf{E}} \,\EE\big[\mathbf{1}_{\mathsf{D}}\,\vert\, \Fext\big(\{n\}\times (\ell,r)\big)\big]\Big].
%$$
%Since the event $\mathsf{E}$ is measurable with respect to $\Fext\big(\{n\}\times (\ell,r)\big)$ the claim of (\ref{e.claimde}) follows from showing that
%\begin{equation}\label{e.claimdee}
%\EE\big[\mathbf{1}_{\mathsf{D}}\,\vert\, \Fext\big(\{n\}\times (\ell,r)\big)\big] \leq 0\cdot \mathbf{1}_{\mathsf{E}} + \mathbf{1}_{\mathsf{E}^c}.
%\end{equation}
%The bound on the complement of $\mathsf{E}$ is trivial, hence
Let us assume that $\mathsf{E}$ occurs. In that case, observe that, due to the monotonicity of Lemma \ref{monotonicity1},
$$
\partfunc{k_1}{k_2}{(\ell,r)}{\vec{x}}{\vec{y}}{\HSDFPline{t}{N}{k_1-1}}{\HSDFPline{t}{N}{k_2+1}}{\Ham_t} \geq \partfunc{k_1}{k_2}{(\ell,r)}{\vec{x}}{\vec{y}}{-M}{M}{\Ham_t}.
$$
Clearly, there exists some $\delta>0$ such that $\partfunc{k_1}{k_2}{(\ell,r)}{\vec{x}}{\vec{y}}{-M}{M}{\Ham_t}>\delta$ given that $\mathsf{E}$ occurs. Thus, for such a $\delta$, (\ref{e.claimde} holds and the proof is completed.
\end{proof}

We now prove bounds, valid uniformly in $t$, on the normalization constant for the lowest indexed curve. It is possible that a variant of the method of proof can be applied to curves of higher index, but we do not pursue such an inquiry.

\begin{proposition}\label{p.oneunif}
Consider an interval $(\ell,r)\subset \R$. Then, for all $\e > 0$, there exists $\delta = \delta(\ell,r,\e) > 0$ such that, for all $t \geq \tzero$ and $N \geq N_0\big(t,\ell,r,\e,\delta \big)$,
$$
\PP\Big(\partfunc{1}{1}{(\ell,r)}{\HSDFPline{t}{N}{1}(\ell)}{\HSDFPline{t}{N}{1}(r)}{+\infty}{\HSDFPline{t}{N}{2}}{\Ham_t}< \delta \Big) < \e\,.
$$
\end{proposition}

\begin{proof}
We start this proof by considering a general setup. In this setup we state and prove Lemmas \ref{lemmaaps} and \ref{l.secondlem}. With these lemmas, we readily reach the desired conclusion of the proposition.

In this general setup, we consider two random curves $\mathcal{L}$ and $\mathcal{\tilde{L}}$ from $[\ell-1,r+1]\to \R$ which start at $x$ and end at $y$, as well as a measurable function $g:[\ell-1,r+1]\to \R$. Here $x,y\in \R$ and the function $g$ are arbitrary. The laws of the curves $\mathcal{L}$ and $\tilde{\mathcal{L}}$ will be denoted respectively by $\PP_{\mathcal{L}}$ and $\PP_{\mathcal{\tilde{L}}}$ are both are defined via Radon-Nikodym derivatives with respect to the law of a single Brownian bridge on $[\ell-1,r+1]$ which starts at $x$ and ends at $y$. As in Definition \ref{maindefHBGP} we denote the law of this Brownian bridge law by $\Pfree{1}{1}{(\ell-1,r+1)}{x}{y}$. The law of $\mathcal{L}$ is equal to
$\PP_{\mathcal{L}}=\PH{1}{1}{(\ell,r)}{x}{y}{+\infty}{g}{\Ham_t}$ (Definition \ref{maindefHBGP}). The law of $\mathcal{\tilde{L}}$ is defined by the Radon-Nikodym derivative
$$
 \frac{\dd \PP_{\mathcal{\tilde L}}}{\dd \Pfree{1}{1}{(\ell-1,r+1)}{x}{y}} (B) = \tilde{Z}^{-1} \exp \bigg\{ - \int_{[\ell - 1,\ell] \cup [r,r+1]} e^{t^{1/3} \big( g(s) - B(s)  \big) }   \dd s\bigg\}  \, ,
$$
where $\tilde{Z}> 0$ is equal to $\PfreeExp{1}{1}{(\ell-1,r+1)}{x}{y}$ evaluated on the exponential term on the right-hand side. This definition for the law of $\mathcal{\tilde{L}}$ means that Gibbs-type conditioning is only on the intervals $[\ell-1,\ell]$ and $[r,r+1]$.

\begin{lemma}\label{lemmaaps}
For any $M_1,M_2>0$ if
$$\sup_{s\in[\ell-1, r+1]} g(s) \leq M_2\,\qquad \mathrm{ and }\qquad x,y \geq -M_1,$$
then, setting $z=M_1+M_2+(r-\ell)^{1/2}$, we have that
\begin{equation}\label{e.ellrap}
 \PP_{\mathcal{\tilde L}}\Big(\partfunc{1}{1}{(\ell,r)}{\mathcal{\tilde L}(\ell)}{\mathcal{\tilde L}(r)}{+\infty}{g}{\Ham_t} \geq e^{ - (r - \ell)} (1 - 2e^{-2})  \Big) \geq   \frac{1}{2\pi} \,\frac{2z^2}{(2z^2 + 1)^2} \,e^{-2 z^2} \,.
\end{equation}
\end{lemma}
\begin{proof}
Define $\mathsf{E}=\Big\{\min \{ \mathcal{\tilde L}(\ell),\mathcal{\tilde L}(r) \} \geq M_2 + (r - \ell)^{1/2}\Big\}$. We claim that
\begin{equation}\label{desiredab}
 \PP_{\mathcal{\tilde L}}(\mathsf{E}) \geq \frac{1}{2\pi}\, \frac{2z^2}{(2z^2 + 1)^2}\, e^{-2 z^2}.
\end{equation}
To prove this, observe that by use of the monotonicity results, Lemmas \ref{monotonicity1} and \ref{monotonicity2}, we can couple to $\mathcal{\tilde L}$ a Brownian bridge $\tilde{B}$ with $\tilde{B}(\ell-1)=\tilde{B}(r+1)= -M_1$ in such a way that, for all $s\in [\ell-1,r+1]$, $\mathcal{\tilde L}(s)\geq \tilde{B}(s)$.
Thus,
$$
 \PP_{\mathcal{\tilde L}}(\mathsf{E})  \geq \PP\Bigg( \min \{ \tilde{B}(\ell),\tilde{B}(r) \} \geq M_2 + (r - \ell)^{1/2}\Bigg)\,,
$$
where the right-hand side $\PP$ is Brownian bridge. The bound in (\ref{desiredab}) follows now from the positive associativity of Brownian bridge, and the Gaussian tail bound of Lemma~\ref{l.normallb}. In particular, the probability that Brownian bridge starting and ending at height 0 on the interval $[\ell-1,r+1]$ is above $z$ at $\ell$ and at $r$ is at least the product of the one-dimensional probabilities which are themselves at least as high as $\frac{1}{2\pi} \frac{\sqrt{2}z}{2z^2 + 1} e^{-z^2}$ (which follows because the variance at $\ell$ and $r$ is $\frac{r-\ell+1}{r-\ell+2}\geq 1/2$).

Having shown the claim in (\ref{desiredab}), we now assert that, when the event $\mathsf{E}$ holds, so does
$$\mathsf{A}= \left\{\partfunc{1}{1}{(\ell,r)}{\mathcal{\tilde L}(\ell)}{\mathcal{\tilde L}(r)}{+\infty}{g}{\Ham_t} \geq e^{ - (r - \ell)} (1 - 2e^{-2})\right\}.$$
To see this, observe that

\begin{eqnarray}\label{above123}
\partfunc{1}{1}{(\ell,r)}{\mathcal{\tilde L}(\ell)}{\mathcal{\tilde L}(r)}{+\infty}{g}{\Ham_t} &=& \PfreeExp{1}{1}{(\ell,r)}{\mathcal{\tilde L}(\ell)}{\mathcal{\tilde L}(r)}\Bigg[ \exp \bigg\{ - \int_\ell^r e^{t^{1/3} \big( g(s)- B(s) \big) }  \dd s  \bigg\} \Bigg]\\
\nonumber &\geq & e^{-(r-\ell)} \big( 1 - 2 e^{-2} \big).
\end{eqnarray}
The equality follows directly from Definition \ref{maindefHBGP}. To derive the inequality, observe that, owing to Lemma~\ref{l.bridgesup}, Brownian bridge $B$, distributed according to the measure on the right-hand side, differs by less than $(r-\ell)^{1/2}$ from the linear interpolation of its endpoints with probability at least $1-2e^{-2}$. On this event for $B$ (and recalling the inequalities satisfied on the event $\mathsf{E}$), it follows that we can bound $g(s)- B(s) \leq 0$ and hence obtain $e^{t^{1/3} \big( g(s)- B(s) \big) } \leq 1$. This proves the inequality in (\ref{above123}), so that $\mathsf{E} \subseteq \mathsf{A}$ holds, as we sought to show. This inclusion and~(\ref{desiredab}) prove the lemma.
\end{proof}

%
%Indeed,
%$$
%  \AP \big( \ell,r,   \mathcal{L}_1(\ell)  ,  \mathcal{L}_1(r)  , h\vert_{[\ell,r]} \big) \\
%  =   \int \exp \bigg\{ - \int_{\ell}^{r}   e^{t^{1/3} \big( h(x) - W(x)  \big) }     \dd x  \bigg\} \, \dd W \, ,
%$$
%where  $W:[\ell,r] \to \R$, $W(\ell) = \mathcal{L}_1(\ell)$, $W(r) = \mathcal{L}_1(r)$, denotes Brownian bridge.
%On the event that $\min \{ \mathcal{L}_1(\ell),\mathcal{L}_1(r) \} \geq R_{2}^* + (r - \ell)^{1/2}$, consider an instance of $W$ for which
% the infimum on $[\ell,r]$ of the difference of $W$ and the linear interpolation of $\mathcal{L}_1(\ell)$ and $\mathcal{L}_1(r)$ is at least $-(r - \ell)^{1/2}$, and note that
%$$
%\exp \bigg\{ - \int_{\ell}^{r} e^{t^{1/3} \big( h(x) -  W(x)  \big) }   dx  \bigg\} \geq e^{-(r-\ell)} \, .
%$$
%By Lemma~\ref{l.bridgesup}, the ${\rm d}W$-probability of such an instance equals $1 - e^{-2}$. Since $\big(\mathcal{L}_1(\ell),\mathcal{L}_1(r)\big)$ stochastically dominates $(W_1(\ell),W_1(r))$, [Alan: this is an irregular monotonicity lemma, since interaction is considered only on $[-\ell-1,-\ell] \cup [r,r+1]$, but it follows directly from the monotonicty of the interaction  on $[-\ell-1,-\ell] \cup [r,r+1]$ as a function of the lowest indexed curve]
%$$
% P \Big( \min \{ \mathcal{L}_1(\ell),\mathcal{L}_1(r) \} \geq R_2^* + (r - \ell)^{1/2} \Big) \geq \tfrac{1}{2\pi} \tfrac{2z^2}{(2z^2 + 1)^2} e^{-2 z^2} \, ,
%$$
%and we have verified~(\ref{e.ellrap}). Here, we used Lemma~\ref{l.normallb}.
%\end{proof}

\begin{lemma}\label{l.secondlem}
Let $M_1,M_2>0$. If
$$\sup_{s\in[\ell-1, r+1]} g(s) \leq M_2\,\qquad \mathrm{ and }\qquad x,y \geq -M_1,$$
then, setting $z=M_1+M_2+(r-\ell)^{1/2}$, we have that, for all $\tilde \e\in [0,1]$,
\begin{equation}\label{e.ellrnew}
\PP_{\mathcal{L}} \Big(\partfunc{1}{1}{(\ell,r)}{\mathcal{L}(\ell)}{\mathcal{L}(r)}{+\infty}{g}{\Ham_t}\leq \delta(\tilde \e) \Big)\leq \tilde \e ,
\end{equation}
where
$$\delta(\tilde \e) = \tilde \e e^{ - (r - \ell)} (1 - 2e^{-2})  \frac{1}{2\pi} \,\frac{2z^2}{(2z^2 + 1)^2} \,e^{-2 z^2}.$$
%where
%$$\delta(\tilde \e) = \tilde \e (1 - 2e^{-2})^{-1}  2\pi \,\frac{(2z^2 + 1)^2}{2z^2} \,e^{2 z^2} e^{  r - \ell  }.$$
\end{lemma}

\begin{proof}
Define $\PP_{\mathcal{L}'}$ and $\PP_{\mathcal{\tilde L}'}$ as the measures onto curves $\mathcal{L}',\mathcal{\tilde L}':[\ell-1,\ell] \cup [r,r+1]\to \R$ (with $\mathcal{L}'(\ell-1)=\mathcal{\tilde L}'(\ell-1)=x$ and $\mathcal{L}'(r+1)=\mathcal{\tilde L}'(r+1)=y$) induced by restriction of the measures $\PP_{\mathcal{L}}$ and $\PP_{\mathcal{\tilde L}}$ to these intervals. The Radon-Nikodym derivative between these two restricted measures is given on curves $B:[\ell-1,\ell] \cup [r,r+1]\to \R$ by
\begin{equation}\label{e.rnthree}
\frac{\dd \PP_{\mathcal{L}'}}{\dd \PP_{\mathcal{\tilde L}'}}(B)=
\big(Z'\big)^{-1}
 %\exp \bigg\{ - \int_\ell^r   \Big(  \sum_{i=1}^{k-1} e^{t^{1/3} \big( B_{i+1}(x) - B_i(x) \big) } \, + \,  %e^{t^{1/3} \big( h(x) - B_k(x) \big) }   \Big) \dd x  \bigg\} \, .
\partfunc{1}{1}{(\ell,r)}{B(\ell)}{B(r)}{+\infty}{g}{\Ham_t}\, .
\end{equation}
Here $Z'$ is the normalizing constant. Specifically,
$$Z'=  \EE_{\mathcal{\tilde L}'}\Big[ \partfunc{1}{1}{(\ell,r)}{B(\ell)}{B(r)}{+\infty}{g}{\Ham_t}\Big]= \EE_{\mathcal{\tilde L}}\Big[ \partfunc{1}{1}{(\ell,r)}{B(\ell)}{B(r)}{+\infty}{g}{\Ham_t}\Big],$$
where $B$ is distributed according to the measure specified by the expectation. The first equality is by definition and the second is from the fact that restricting the measure $\PP_{\mathcal{\tilde L}}$ does not change its marginal at $\ell$ and $r$. Owing to this explicit form of $Z'$ and Lemma \ref{lemmaaps}, we have the lower bound
\begin{equation}\label{e.Zprime}
Z' \geq e^{ - (r - \ell)} (1 - 2e^{-2})  \frac{1}{2\pi} \,\frac{2z^2}{(2z^2 + 1)^2} \,e^{-2 z^2} = \frac{\delta(\tilde \e)}{\tilde \e}\,.
\end{equation}

Observe now that
\begin{eqnarray*}
\PP_{\mathcal{L}} \Big(\partfunc{1}{1}{(\ell,r)}{\mathcal{L}(\ell)}{\mathcal{L}(r)}{+\infty}{g}{\Ham_t}\leq \delta(\tilde \e) \Big) &=&  \PP_{\mathcal{L}'}\Big(\partfunc{1}{1}{(\ell,r)}{\mathcal{L}'(\ell)}{\mathcal{L}'(r)}{+\infty}{g}{\Ham_t}\leq \delta(\tilde \e) \Big)\\
&\leq & \big(Z'\big)^{-1}\, \delta(\tilde \e)\,  \PP_{\mathcal{\tilde L}'}\Big(\partfunc{1}{1}{(\ell,r)}{\mathcal{\tilde L}'(\ell)}{\mathcal{\tilde L}'(r)}{+\infty}{g}{\Ham_t}\leq \delta(\tilde \e) \Big)\\
&\leq &  \big(Z'\big)^{-1}\, \delta(\tilde \e)  \leq  \tilde \e.
\end{eqnarray*}
The equality is by definition. The first inequality is an instance of size biasing and relies upon (\ref{e.rnthree}).  The next inequality is trivial and the final one follows from (\ref{e.Zprime}).
\end{proof}

We now complete the proof of Proposition \ref{p.oneunif} by applying Lemma \ref{l.secondlem} to the situation at hand. Define the event
$$
\mathsf{B} = \Big\{\sup_{s\in[\ell - 1, r+1]} \HSDFPline{t}{N}{2}(s) \leq M_2\Big\}\cap \Big\{\min \big\{ \HSDFPline{t}{N}{1}(\ell-1) , \HSDFPline{t}{N}{1}(r+1) \big\} \geq - M_1\Big\},
$$
where $M_1$ and $M_2$ have been chosen here to be large enough so that, for $N\geq N_0$, $\PP(\mathsf{B}^c) \leq 3\e$. The existence of such $M_1$ and $M_2$ is assured by Propositions~\ref{p.oneptub} and~\ref{p.epstr}. Define $z$ and $\delta(\tilde \e)$ with respect to $M_1$ and $M_2$ as in the statement of Lemma \ref{l.secondlem}.

Consider the probability
\begin{eqnarray}\label{e.abovebc}
&&\hskip-.25in\PP\bigg(\Big\{\partfunc{1}{1}{(\ell,r)}{\HSDFPline{t}{N}{1}(\ell)}{\HSDFPline{t}{N}{1}(r)}{+\infty}{\HSDFPline{t}{N}{2}}{\Ham_t}< \delta(\tilde \e)\Big\}\cap \mathsf{B}\bigg)\\
\nonumber &=&\EE\bigg[ \mathbf{1}_{\mathsf{B}} \, \EE\Big[\mathbf{1}\big\{\partfunc{1}{1}{(\ell,r)}{\HSDFPline{t}{N}{1}(\ell)}{\HSDFPline{t}{N}{1}(r)}{+\infty}{\HSDFPline{t}{N}{2}}{\Ham_t}< \delta(\tilde \e)\big\} \, \Big\vert \Fext\big(\{1\}\times (\ell-1,r+1)\big)\Big] \bigg].
\end{eqnarray}
Observe that as $\Fext\big(\{1\}\times (\ell-1,r+1)\big)$-measurable random variables,
$$\EE\Big[\mathbf{1}\big\{\partfunc{1}{1}{(\ell,r)}{\HSDFPline{t}{N}{1}(\ell)}{\HSDFPline{t}{N}{1}(r)}{+\infty}{\HSDFPline{t}{N}{2}}{\Ham_t}< \delta(\tilde \e)\big\} \, \Big\vert \, \Fext\big(\{1\}\times (\ell-1,r+1)\big)\Big] =
\PP_{\mathcal{L}} \Big(\partfunc{1}{1}{(\ell,r)}{\mathcal{L}(\ell)}{\mathcal{L}(r)}{+\infty}{g}{\Ham_t}\leq \delta(\tilde \e) \Big)$$
where $\PP_{\mathcal{L}}$ is specified (as in the beginning of the proof of this proposition) with respect to $x=\HSDFPline{t}{N}{1}(\ell - 1)$, $y=\HSDFPline{t}{N}{1}(r+1)$ and $g(\cdot) =\HSDFPline{t}{N}{2}(\cdot)$ on $[\ell - 1,r+1]$.
When the $\Fext\big(\{1\}\times (\ell-1,r+1)\big)$-measurable event $\mathsf{B}$ holds, $\sup_{s\in[\ell-1, r+1]} g(s) \leq M_2$ and $x,y \geq -M_1$. Hence, on the event $\mathsf{B}$ we may apply Lemma \ref{l.secondlem}. This implies that, as a $\Fext\big(\{1\}\times (\ell-1,r+1)\big)$-measurable random variable,
$$\PP_{\mathcal{L}} \Big(\partfunc{1}{1}{(\ell,r)}{\mathcal{L}(\ell)}{\mathcal{L}(r)}{+\infty}{g}{\Ham_t}\leq \delta(\tilde \e) \Big) \leq \tilde \e \,\cdot\, \mathbf{1}_{\mathsf{B}} + \mathbf{1}_{\mathsf{B}^c}.$$
The bound on $\mathsf{B}^c$ is the obvious upper bound of one for any probability. Plugging this into (\ref{e.abovebc}), we find that
$$\PP\bigg(\Big\{\partfunc{1}{1}{(\ell,r)}{\HSDFPline{t}{N}{1}(\ell)}{\HSDFPline{t}{N}{1}(r)}{+\infty}{\HSDFPline{t}{N}{2}}{\Ham_t}< \delta(\tilde \e)\Big\}\cap \mathsf{B}\bigg) \leq \tilde \e.$$

From this and $\PP(\mathsf{B}^c) \leq 3\e$, we find that
\begin{eqnarray*}
& & \PP\bigg(\partfunc{1}{1}{(\ell,r)}{\HSDFPline{t}{N}{1}(\ell)}{\HSDFPline{t}{N}{1}(r)}{+\infty}{\HSDFPline{t}{N}{2}}{\Ham_t}< \delta(\tilde \e)\bigg) \\
&=& \PP\bigg(\Big\{\partfunc{1}{1}{(\ell,r)}{\HSDFPline{t}{N}{1}(\ell)}{\HSDFPline{t}{N}{1}(r)}{+\infty}{\HSDFPline{t}{N}{2}}{\Ham_t}< \delta(\tilde \e)\Big\}\cap \mathsf{B}\bigg) + \PP \big(\mathsf{B}^c \big) \leq  \tilde \e  +  3\e  \, .
\end{eqnarray*}
We take $\tilde \e = \e$ and replace $\e$ by $\e/4$ to complete the proof of Proposition~\ref{p.oneunif}.
\end{proof}

\subsection{Concluding the proof of Theorem \ref{seqncmpt}}\label{wegsavg}

Theorem \ref{seqncmpt}(2) follows immediately from Proposition \ref{p.oneunif}. Theorem \ref{seqncmpt}(1) follows from Proposition \ref{p.one}; however, this requires a more explanation, which we now provide. What follows is a modification of an argument used in \cite[Proof of Proposition 3.7]{CH}.

The tightness criterion for $k$ continuous functions is the same as for a single function. For $a < b$, and $f_i:(a,b) \to \R$, $1 \leq i \leq k$, define the $k$-line modulus of continuity
\begin{equation}\label{eqkline}
w_{a,b}\big(\{f_1,\ldots, f_k\},r\big) = \sup_{1\leq i\leq k} \sup_{\substack{s,t\in (a,b)\\ |s-t|<r}} \big\vert f_i(s)-f_i(t)\big\vert.
\end{equation}
Consider a sequence of probability measures $P_N$ on $k$ functions $f=\{f_1,\ldots f_k\}$ on the interval $(a,b)$, and define the event
\begin{equation*}
\mathsf{W}_{a,b}(\rho,r) = \Big\{w_{a,b}\big(\{f_1,\ldots, f_k\},r\big) \leq \rho\Big\}.
\end{equation*}
As an immediate generalization of \cite[Theorem 8.2]{Bill}, the sequence $P_N$
is tight (or sequentially compact) if, for each $1\leq i\leq k$, the one-point distribution of $f_i(x)$ at a fixed $x\in (a,b)$ is tight and if, for each positive $\rho$ and $\eta$, there exists a $r>0$ and an integer $N_0$ such that
$$
P_N\big(\mathsf{W}_{a,b}(\rho,r)\big)\geq 1-\eta, \qquad \textrm{for $N \geq N_0$} \, .
$$
We will apply this tightness criterion when $P_N$ is a line ensemble on at least $k$ curves. When we do so,  $\mathsf{W}_{a,b}(\rho,r)$ remains defined with respect to the curves on $(a,b)$ of index between $1$ and $k$. Let us also choose $a=-T$ and $b=T$ with $T$ from the statement of Theorem \ref{seqncmpt}.

Propositions~\ref{p.oneptub} and~\ref{p.epstr} show tightness of the one-point distribution. In particular, for each given $t \geq \tzero$ and $n \in \{1,\ldots k\}$, the one-point distribution of $\HSDFPline{t}{N}{n}(x)$ is tight in $N\in \N$, uniformly as $x$ varies over $[-T,T]$. Thus, in light of the criterion provided above, to show tightness of the ensemble $\big\{\HSDFPline{t}{N}{n}(x):n\in \{1,\ldots,k\},x\in[-T,T]\big\}$, it suffices to verify that, for all $\rho, \eta > 0$, we may find some small $r>0$ and some large $N_0$ for which $\PP\big(\mathsf{W}_{-T,T}(\rho,r)\big)\geq 1-\eta$ for all $N \geq N_0$. Here, $\mathsf{W}_{-T,T}(\rho,r)$ is defined as above with $f_n = \HSDFPline{t}{N}{n}$ on the interval $[-T,T]$.

Towards this end, let us adopt the shorthand that $Z$ denotes the $\Fext\big(\{1,\ldots, k\}\times [-T,T]\big)$-measurable normalizing constant (Definition \ref{maindefHBGP})
\begin{equation*}
Z = \partfunc{1}{k}{(-T,T)}{\vec{x}}{\vec{y}}{+\infty}{\HSDFPline{t}{N}{k+1}}{\Ham_t},
\end{equation*}
where
\begin{equation}\label{e.twentyseven}
\vec{x}= \big(\HSDFPline{t}{N}{1}(-T),\ldots, \HSDFPline{t}{N}{k}(-T)\big),\qquad \vec{y}= \big(\HSDFPline{t}{N}{1}(T), \ldots, \HSDFPline{t}{N}{k}(T)\big).
\end{equation}

%Observe that
%\begin{equation*}
%\PP\big(W_{-T,T}(\rho,r)\big) \geq \PP \Big(W_{-T,T}(\rho,r)\, \cup\, \{ Z \geq \delta\}\, \cup\, S_{M} \Big)
%\end{equation*}
For $M>0$, define the event
\begin{equation*}
\mathsf{S}_{M}= \bigcap_{n=1}^k \Big\{ -M \leq \HSDFPline{t}{N}{n}(-T),\HSDFPline{t}{N}{n}(T) \leq M \Big\} \, .
\end{equation*}
We claim that, for any $\rho,\eta>0$, there exists $r>0$, $\delta>0$, $M>0$ and $N_0$ such that, for $N \geq N_0$,
\begin{equation}\label{lhsabove}
\PP \Big(\mathsf{W}_{-T,T}(\rho,r)\, \cap\, \{ Z \geq \delta\}\, \cap\, \mathsf{S}_{M} \Big) > 1-\eta \, .
\end{equation}
It is obvious that from this claimed bound we may conclude that $\PP\big(\mathsf{W}_{-T,T}(\rho,r)\big)\geq 1-\eta$ as is needed for the tightness criterion. Hence, if we can show the above claim then we will have completed the proof of Theorem \ref{seqncmpt}.

In order to prove the claim, note that the events $\{ Z \geq \delta\}\, \cap\, \mathsf{S}_{M}$ in (\ref{lhsabove}) are $\Fext\big(\{1,\ldots, k\}\times [-T,T]\big)$-measurable. We can thus write the left-hand side of (\ref{lhsabove}) as
\begin{equation}\label{condexp}
\EE\bigg[\mathbf{1}_{Z \geq \delta}\, \mathbf{1}_{\mathsf{S}_{M}}\,\EE\Big[\mathbf{1}_{\mathsf{W}_{-T,T}(\rho,r)} \,\big\vert\, \Fext\big(\{1,\ldots, k\}\times [-T,T]\big)\Big] \bigg].
\end{equation}

From the $\Ham_t$-Brownian Gibbs property (Definition \ref{maindefHBGP}) enjoyed by the measure $\PP$ on $\HSDFPlinet{t}{N}$ (Corollary \ref{rescaledHBGP}), $\PP$-almost surely,
\begin{equation}\label{strongapp}
\EE\Big[\mathbf{1}_{\mathsf{W}_{-T,T}(\rho,r)} \,\big\vert\, \Fext\big(\{1,\ldots, k\}\times [-T,T]\big)\Big] = \PH{1}{k}{(-T,T)}{\vec{x}}{\vec{y}}{+\infty}{\HSDFPline{t}{N}{k+1}}{\Ham_t} \big[\mathsf{W}_{-T,T}(\rho,r) \big],
\end{equation}
with $\vec{x}$ and $\vec{y}$ as in (\ref{e.twentyseven}). On the right-hand side, $\mathsf{W}_{-T,T}(\rho,r)$ is defined with respect to the $k$ paths specified by the measure $\PH{1}{k}{(-T,T)}{\vec{x}}{\vec{y}}{+\infty}{\HSDFPline{t}{N}{k+1}}{\Ham_t}$.

\begin{lemma}\label{l.abovelemma}
Let $\rho, \eta,\delta,M,T > 0$ and $k\geq 1$. There exists $r>0$ such that, if $\vec{x},\vec{y} \in \R^k$ and $g:[-T,T] \to \R$ satisfy the conditions $|x_i|,|y_i|\leq M$ for $1 \leq i \leq k$ and $\partfunc{1}{k}{(-T,T)}{\vec{x}}{\vec{y}}{+\infty}{g}{\Ham_t}~\geq~\delta$, then
\begin{equation*}
 \PH{1}{k}{(-T,T)}{\vec{x}}{\vec{y}}{+\infty}{g}{\Ham_t} \big[\mathsf{W}_{-T,T}(\rho,r)\big] \geq 1-\eta/2.
\end{equation*}
\end{lemma}

Let us assume the lemma for the moment and complete the proof of the claimed equation (\ref{lhsabove}). By choosing $r$ small enough (depending on $\rho,\eta,\delta,M$), using (\ref{strongapp}) and Lemma \ref{l.abovelemma} we may bound
\begin{equation}\label{claimapp}
(\ref{condexp}) \geq (1-\eta/2)\, \EE\big[\mathbf{1}_{Z \geq \delta}\, \mathbf{1}_{\mathsf{S}_{M}}\big] = (1-\eta/2)\,\PP\big(\{Z \geq \delta\}\, \cap\, \mathsf{S}_{M}\big).
\end{equation}

Now choose $\delta>0$ small enough and $M,N_0>0$ large enough so that
\begin{equation}\label{e.sfm}
\PP\big(\{Z \geq \delta\}\, \cap\, \mathsf{S}_{M}\big) \geq 1-\eta/2 \, .
\end{equation}
Let us see why this can be achieved. Proposition~\ref{p.one} shows that there exists $\delta>0$ and $N_0$ such that, for $N \geq N_0$,
$\PP\big(Z< \delta \big) \leq \eta/4$. Propositions \ref{p.oneptub} and \ref{p.epstr} imply that we may choose $M,N_0$ large enough and $\delta$ small enough so that
$\PP\big(\mathsf{S}^c_{M}\big)\leq \eta/4$ for for $N\geq N_0$. This implies (\ref{e.sfm}).

In light of (\ref{claimapp}) and (\ref{e.sfm}), we find that $(\ref{condexp}) \geq (1-\eta/2)^2 \geq 1-\eta$, as desired to prove the claim (\ref{lhsabove}) and complete the proof of this theorem.

All that now remains is to prove Lemma \ref{l.abovelemma}.

\begin{proof}[Proof of Lemma \ref{l.abovelemma}]
Let us start by considering $\{\tilde{B}_i\}_{i=1}^k$, $k$ independent Brownian bridges on $[0,1]$ with $\tilde{B}_i(0)=0$ and $\tilde{B}_i(1)=0$. For each $\tilde{r}$, associate the random modulus of continuity $w_{0,1}(\tilde{B}_i,\tilde{r})\in[0,\infty)$. From these $\tilde{B}_i$, we construct Brownian bridges $B_i$ on $(a,b)$ with $B_i(a)=x_i$ and $B_i(b)=y_i$ by setting
\begin{equation*}
B_{i}(t)= ( b-a)^{1/2} \tilde{B}_i\left(\frac{t-a}{b-a}\right) +  \left(\frac{b-t}{b-a}\right) x_i +  \left(\frac{t-a}{b-a}\right) y_i \, .
\end{equation*}
The collection $\{B_i\}_{i=1}^{k}$ is thus distributed according to the measure $\Pfree{1}{k}{(a,b)}{\vec{x}}{\vec{y}}$. For what follows, set $a = -T$ and $b = T$.

The $k$-line modulus of continuity $w_{-T,T}\big(\{B_1,\ldots, B_k\}, 2T\tilde{r}\big)$ may then be bounded by
\begin{equation}\label{modcontr}
w_{-T,T}\big(\{B_1,\ldots, B_k\}, 2T\tilde{r}\big) \leq \sup_{1\leq i\leq k} \Big((2T)^{1/2} w_{0,1}(\tilde{B}_i,\tilde{r}) + |x_i-y_i| \tilde{r} \Big).
\end{equation}
Setting $\tilde{r} = \tfrac{r}{2T}$ and using the fact that $|x_i-y_i|\leq 2M$ (as follows from the assumption that $|x_i|,|y_i|\leq M$), we find that
\begin{equation}\label{allabeqn}
w_{-T,T}\big(\{B_1,\ldots, B_k\},r\big) \leq (2T)^{1/2} \sup_{1\leq i\leq k}w_{0,1}(\tilde{B}_i, \tfrac{r}{2T}) \, + \frac{Mr}{T} \, .
\end{equation}

Consider any event $\mathsf{E}$ such that $\Pfree{1}{k}{(a,b)}{\vec{x}}{\vec{y}}(\mathsf{E})>\delta$. Observe the general fact that conditioning the measure $\Pfree{1}{k}{(a,b)}{\vec{x}}{\vec{y}}$ on $\mathsf{E}$ is equivalent to conditioning the measure of the $k$ Brownian bridges $\{\tilde{B}_i\}_{i=1}^k$ on some other related event $\mathsf{\tilde E}$ also of probability at least $\delta$. Since the random variables $w_{0,1}(\tilde{B}_{i},\tfrac{r}{2T})$ are supported on $[0,\infty)$ and converge to zero as $r$ tends to zero (since the Brownian bridges are continuous almost surely), it follows that, by choosing $r$ small enough (with all of the other variables fixed), we can be assured that, conditioned on the event $\mathsf{\tilde E}$,
\begin{equation*}
 (2T)^{1/2} \sup_{1\leq i\leq k}w_{0,1}(\tilde{B}_i, \tfrac{r}{2T}) \, + \frac{Mr}{T} \leq \rho
\end{equation*}
with probability at least $1-\eta/2$.

We now apply this deduction to the present case. As in Remark \ref{normrem}, note that the law $\PH{1}{k}{(a,b)}{\vec{x}}{\vec{y}}{+\infty}{g}{\Ham_t}$ is obtained from  $\Pfree{1}{k}{(-T,T)}{\vec{x}}{\vec{y}}$ by conditioning on the event
$$
\mathsf{E} = \Bigg\{ \exp \bigg\{ - \int_{-T}^T   \Big(  \sum_{i=1}^{k-1} e^{t^{1/3} \big( B_{i+1}(x) - B_i(x) \big) } \, + \,  e^{t^{1/3} \big( g(x) - B_k(x) \big) }   \Big) \dd x  \bigg\}  \geq U \Bigg\}\, ,
$$
where $U$ is an independent random variable distributed uniformly over the interval $[0,1]$. Note that
$\partfunc{1}{k}{(-T,T)}{\vec{x}}{\vec{y}}{+\infty}{g}{\Ham_t}$ is precisely $\Pfree{1}{k}{(-T,T)}{\vec{x}}{\vec{y}}(\mathsf{E})$, which, by the hypothesis of the lemma, is at least~$\delta$. Hence we can make use of the general discussion above to conclude the proof of the lemma.
\end{proof}

\section{Proof of three key propositions}\label{s.threekeyproofs}
Here we prove the three key propositions of Section \ref{threekey} by an induction on the index $n\in \N$. The induction does not proceed independently for each proposition but instead in order to deduce all three propositions for index $n$ we utilize the three propositions for index $n-1$ and $n-2$. The order of deductions for index $n$ is illustrated in Figure \ref{keypropfig}. To summarize, we start by deducing Proposition~\ref{p.epstr} for index $n$ from the knowledge of all three propositions for index $n-1$. Then we deduce Proposition~\ref{p.lbpara} for index $n$ from the knowledge (just gained) of Proposition~\ref{p.epstr} for index $n$ as well as that of Proposition~\ref{p.lbpara} for index $n-1$. Finally, we deduce Proposition~\ref{p.oneptub} from the knowledge of Proposition~\ref{p.epstr} for index $n$, $n-1$ and $n-2$ and Proposition~\ref{p.oneptub} for index $n-1$ (the case $n=1$ of Proposition~\ref{p.oneptub} requires a slightly different argument as explained at the beginning of its proof). In order to start the induction we trivially observe that for $n\leq 0$ Propositions~\ref{p.epstr} and \ref{p.lbpara} are satisfied
under the convention that such indexed curves are identically $+\infty$. Proposition~\ref{p.epstr} for index $n=1$ follows immediately from Lemma~\ref{l.input}. This suffices to start the induction.

\begin{figure}
\centering\epsfig{file=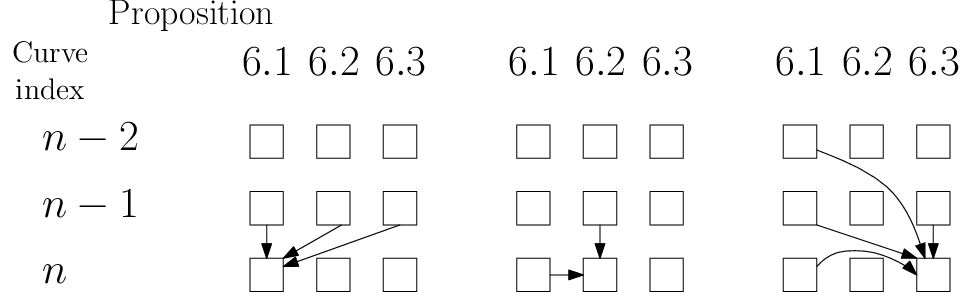, width=14cm}
\caption{Schematic for the inductive proof of Propositions \ref{p.epstr}, \ref{p.lbpara} and \ref{p.oneptub} at index $n$.}\label{keypropfig}
\end{figure}

\subsection{Proof of Proposition \ref{p.epstr}}
We begin presenting the inductive argument. The base case of $n=1$ of the present proposition follows from Lemma~\ref{l.input}. Let us therefore assume $n\geq 2$ and that for index $n-1$ all three of Propositions \ref{p.epstr}, \ref{p.lbpara} and \ref{p.oneptub} have been proved.

Consider $\e>0$ given in the hypothesis of Proposition \ref{p.epstr}. We start this proof by defining $R_n=R_n(\e)$ as well as a few other parameters.
\begin{definition}\label{d.parts}
Let $R_{n-1}$ and $\Rstar_{n-1}$ be constants provided respectively by Propositions \ref{p.epstr} and \ref{p.oneptub} (at index $n-1$) for $\e>0$. Let $K>1$ be such that $2(1-e^{-1/2})^{-1} e^{-K^2/2}=\e$.
Let $T_0>0$ be the parameter provided by Proposition \ref{p.lbpara} (at index $n-1$) for the given $\e$ and for any $\delta \in (1/128,1/8)$. Let $T>1/2$ be large enough that
\begin{equation*}
T>T_0,\qquad T e^{-T^{1/2}}\leq \frac{1}{4} \log 2,\qquad \Rstar_{n-1} \leq \frac{1}{16} T^2- KT^{1/2},\qquad R_{n-1} \leq \frac{1}{16} T^2.
\end{equation*}
Define
\begin{equation*}
M = \frac{1}{8} T^2 -\Rstar_{n-1} +(K+1)T^{1/2}, \qquad\mathrm{and} \qquad R_n = M+ 2T^2 +KT^{1/2}.
\end{equation*}
\end{definition}

We will argue that for  $\e$ given, if we choose $R_n$ as above then, for any $x_0 > 0$ and $t \geq \tzero$, there exists $N_0=N_0(x_0,t,\e)$ such that for
$N \geq N_0(x_0,t,\e)$ and $\bar{x} \in [-x_0,x_0]$,
\begin{equation}\label{e.teneps}
\PP\bigg(\inf_{s\in \big[\bar{x}-\frac{1}{2},\bar{x}+\frac{1}{2}\big]}\big(\HSDFPline{t}{N}{n}(s) + s^2/2\big) < - R_n \bigg) < 10\e \, .
\end{equation}
Of course, since $\e$ is arbitrary, we may as well have taken $\e/10$ in place of $\e$; thus, by verifying  (\ref{e.teneps}), we will likewise verify Proposition~\ref{p.epstr} at index $n$.

Now consider arbitrary $x_0>0$, $t \geq \tzero$ and $\bar{x} \in [-x_0,x_0]$. For $T$ and $M$ as above, define the events
\begin{eqnarray*}
\mathsf{E}^{N,-}_n &=& \Big\{ \sup_{x \in [\bar{x}-2T,\bar{x}-T]} \big( \HSDFPline{t}{N}{n}(x) + x^2/2 \big) > - M \Big\}\\
\mathsf{E}^{N,+}_n &=& \Big\{ \sup_{x \in [\bar{x}+T,\bar{x}+2T]} \big( \HSDFPline{t}{N}{n}(x) + x^2/2 \big) > - M \Big\}
\end{eqnarray*}
and their intersection
\begin{equation*}
\mathsf{E}^N_n = \mathsf{E}^{N,-}_n \cap \mathsf{E}^{N,+}_n.
\end{equation*}
When the event $\mathsf{E}^N_n$ holds, $\HSDFPline{t}{N}{n}(x)$ exceeds $-M$ at some $x$ in $[\bar{x}-2T,\bar{x}-T]$ as well as at another $x$ in $[\bar{x}+T,\bar{x}+2T]$.

We will prove two lemmas involving this event. Lemma \ref{l.ENC} shows this is likely since otherwise the $\Ham_t$-Brownian Gibbs property implies that the $n-1$ indexed curve sags lower than it is allowed to by Proposition \ref{p.lbpara} (for index $n-1$). Lemma \ref{l.EN} shows that if $\mathsf{E}^{N}_n$ occurs, then the index $n$ curve can not sag too much in $[\bar{x}-T,\bar{x}+T]$.

\begin{lemma}\label{l.ENC}
For arbitrary $x_0>0$ and $t \geq \tzero$, there exists $N_0(x_0,t,\e)$ such that, for $N\geq N_0(x_0,t,\e)$ and $\bar{x} \in [-x_0,x_0]$,
$$\PP\Big(\big(\mathsf{E}^{N}_n\big)^c \Big) \leq 8\e.$$
\end{lemma}

\begin{lemma}\label{l.EN}
For arbitrary $x_0>0$ and $t \geq \tzero$, there exists $N_0(x_0,t,\e)$ such that, for $N\geq N_0(x_0,t,\e)$ and $\bar{x} \in [-x_0,x_0]$,
$$
\PP\bigg(\Big\{\inf_{x\in [\bar{x}-T,\bar{x}+T]}\big(\HSDFPline{t}{N}{n}(x) + x^2/2\big) < - R_n\Big\} \cap \mathsf{E}^N_n \bigg) < 2\e.
$$
\end{lemma}

Before proving these lemmas, observe that combining them readily proves (\ref{e.teneps}) (in fact with the $1/2$ replaced by $T$). Of course, the two functions $N_0(x_0,t,\e)$ returned by the two lemmas may not be the same, but we may take their maximum.

Thus, to complete the proof of  Proposition \ref{p.epstr} at index $n$, it remains prove these two lemmas. It is worth noting that in these lemmas, $x_0>0$ and $t\geq 1$ are chosen arbitrarily, and the parameters $T_0,T,R_{n-1},\Rstar_{n-1},K,M$ and $R_n$ specified in Definition \ref{d.parts} are independent of $x_0$ and $t$. The reason why this can be is that, in proving these lemmas, we appeal only to (the inductively already shown instances of) Propositions  \ref{p.epstr}, \ref{p.lbpara} and \ref{p.oneptub}, all of which display a similar independence of $x_0$ and $t$. This uniformity in $x_0>0$ and $t\geq 1$ is important since it allows us to make statements which hold true uniformly as $t$ grows (the $x_0$ uniformity is also useful, but in a more technical way in the proofs). We will not labor this point further in the ensuing proofs.

\begin{proof}[Proof of Lemma \ref{l.ENC}]
On the event $\big(\mathsf{E}^{N}_n\big)^c$, either $\big(\mathsf{E}^{N,-}_n\big)^c$ or $\big(\mathsf{E}^{N,+}_n\big)^c$ holds (or both). Consider $x_0>0$, $t\geq 1$ and $\bar{x}\in [-x_0,x_0]$. We will show that there exists $N_0(x_0,t,\e)$ such that, for $N\geq N_0(x_0,t,\e)$,
\begin{equation}\label{e.foureps}
\PP\Big(\big(\mathsf{E}^{N,-}_n\big)^c \Big) \leq 4\e.
\end{equation}
From (\ref{e.foureps}), the lemma will immediately follow by the union bound since the analogous result for $\big(\mathsf{E}^{N,+}_n\big)^c$ is shown in the same manner.

The idea of the proof of (\ref{e.foureps})  is that we have chosen $T$ and $M$ large enough that, conditioned on the event $\big(\mathsf{E}^{N,-}_n\big)^{c}$, the curve $\HSDFPline{t}{N}{n-1}(x)$ on the interval $[\bar{x}-2T,\bar{x}-T]$ sags so as to interpolate linearly (rather than quadratically) between its values at $x= \bar{x}-2T$ and $x=\bar{x}-T$. However, this sagging of the $(n-1)$-indexed curve is known to occur with very small probability (due to our inductive knowledge about curves of index $n-1$) and thus so too must $\big(\mathsf{E}^{N,-}_n\big)^c$ occur only with very small probability. In the course of this proof, we will require $N$ to be sufficiently large so as to be able to apply various (already inductively established) results. The value of $N_0(x_0,t,\e)$ in Lemma~\ref{l.ENC} may be taken to be the maximum over all such requirements on $N$.

Define the event
$$\mathsf{H}^N_{n-1} = \Big\{\HSDFPline{t}{N}{n-1}(x)+x^2/2 \leq \Rstar_{n-1} \textrm{ for }x=\bar{x}-2T \textrm{ and } x=\bar{x}-T\Big\}.$$
By Proposition \ref{p.oneptub}, the union bound, and the definition of $\Rstar_{n-1}$ (given in Definition \ref{d.parts}), there exists $N_0(x_0,t,\e)$ such that $\PP\big((\mathsf{H}^N_{n-1})^c\big) \leq 2 \e$ for $N\geq N_0(x_0,t,\e)$.

\begin{claim}\label{c.labeled}
There exists $N_0(x_0,t,\e)$ such that
$$\PP\big((\mathsf{E}^{N,-}_n)^c \cap \mathsf{H}^N_{n-1}\big) \leq 2 \e$$
for $N\geq N_0(x_0,t,\e)$.
\end{claim}

It follows from this claim that
$$
\PP\big((\mathsf{E}^{N,-}_{n})^c\big) \leq \PP\big((\mathsf{E}^{N,-}_{n})^c \cap \mathsf{H}^N_{n-1}\big) +\PP\big((\mathsf{H}^N_{n-1})^c\big) \leq 4\e,
$$
and thus (\ref{e.foureps}) holds.

Thus, to complete the proof of Lemma \ref{l.ENC}, we must show Claim \ref{c.labeled}. Define the event
$$\mathsf{A} = \Big\{ \HSDFPline{t}{N}{n-1}(\bar{x}-3T/2) + (\bar{x}-3T/2)^2/2 \leq -T^2/16\Big\}.$$
We bound above the probability of this event by applying Proposition~\ref{p.lbpara} with curve index $n-1$. To do so, we specify that the parameter $x_0$ be set as a function of curve index so that $x_0(n) = x_0(n-1)/2$; (in this way, if any particular choice of $x_0$ at index $n$ is desired, it may be obtained by setting $x_0(1)$ equal to a $2^n$ multiple of the sought value). The parameter $\bar{x}$ satisfies $\vert \bar{x} \vert \leq x_0(n) = x_0(n-1)/2$. As we seek to apply Proposition~\ref{p.lbpara} at index~$n-1$, we must ensure that $\bar{x} - 3T/2$ lies in the interval $[y_0,y_0 + T]$, where the constraint on $y_0$ is that it lies in $[-x_0,x_0 - T]$ for $x_0 = x_0(n-1)$. Since $\big\vert \bar{x} - 3T/2 \big\vert \leq x_0(n) + 3T/2 = x_0(n-1)/2 + 3T/2$, we wish to make sure that $x_0(n-1)/2 + 3T/2 \leq x_0(n-1) - T$, for which a lower bound of $5T$ on $x_0(n-1)$ suffices. Setting $x_0(n-1)$ to be any such value also exceeding $T_0$, we may then apply Proposition~\ref{p.lbpara} at index $n-1$ for any given choice of $\delta \in (0,1/16)$ to learn that there exists $N_0(x_0,t,\e)$ such that $\PP(\mathsf{A})\leq \e$ for $N\geq N_0(x_0,t,\e)$. If we can also show that \begin{equation}\label{e.ifwecanshow}
\PP\big((\mathsf{E}^{N,-}_n)^c \cap \mathsf{H}^N_{n-1} \cap \mathsf{A}^c\big) \leq \e,
\end{equation}
then Claim \ref{c.labeled} will follow.

We will use conditional expectations, the $\Ham_t$-Brownian Gibbs property and some monotonicity to reduce the problem of showing (\ref{e.ifwecanshow}) to a simple problem involving a single Brownian bridge. Observe that the event $(\mathsf{E}^{N,-}_n)^c \cap \mathsf{H}^N_{n-1}$ is $\Fext\big(\{n-1\},[\bar{x}-2T,\bar{x}-T]\big)$-measurable, whereas $\mathsf{A}$ is not. Thus, using conditional expectations, we have
$$
\PP\big((\mathsf{E}^{N,-}_n)^c \cap \mathsf{H}^N_{n-1} \cap \mathsf{A}^c\big) = \EE\bigg[\mathbf{1}_{(\mathsf{E}^{N,-}_n)^c \cap \mathsf{H}^N_{n-1}}\, \EE\Big[\mathbf{1}_{\mathsf{A}^c} \, \big\vert\, \Fext\big(\{n-1\},[\bar{x}-2T,\bar{x}-T]\big)\Big]\bigg].
$$
To establish (\ref{e.ifwecanshow}), it suffices to show that, $\PP$-almost surely,
\begin{equation}\label{e.letsshowthis}
\EE\Big[\mathbf{1}_{\mathsf{A}^c} \, \big\vert\, \Fext\big(\{n-1\},[\bar{x}-2T,\bar{x}-T]\big)\Big] \leq \e\, \mathbf{1}_{(\mathsf{E}^{N,-}_n)^c \cap \mathsf{H}^N_{n-1}} + \mathbf{1}_{\big((\mathsf{E}^{N,-}_n)^c \cap \mathsf{H}^N_{n-1}\big)^c}\, .
\end{equation}
Of course, the bound by $\mathbf{1}_{\big((\mathsf{E}^{N,-}_n)^c \cap \mathsf{H}^N_{n-1}\big)^c}$ is trivial. In what follows, let us therefore assume that the event
$(\mathsf{E}^{N,-}_n)^c \cap \mathsf{H}^N_{n-1}$ occurs. The $\Ham_t$-Brownian Gibbs property for $\HSDFPlinet{t}{N}$ (Corollary \ref{rescaledHBGP}) implies that $\PP$-almost surely
$$
\EE\Big[\mathbf{1}_{\mathsf{A}^c} \, \big\vert\, \Fext\big(\{n-1\},[\bar{x}-2T,\bar{x}-T]\big)\Big] = \PH{n-1}{n-1}{(\bar{x}-2T,\bar{x}-T)}{\HSDFPline{t}{N}{n-1}(\bar{x}-2T)}{\HSDFPline{t}{N}{n-1}(\bar{x}-T)}{\HSDFPline{t}{N}{n-2}}{\HSDFPline{t}{N}{n}}{\Ham_t} (\mathsf{A}^c).
$$
To explain the notation on the right-hand side: we let $B:[\bar{x}-2T,\bar{x}-T] \to \R$ be the curve distributed according to the measure $\PH{n-1}{n-1}{(\bar{x}-2T,\bar{x}-T)}{\HSDFPline{t}{N}{n-1}(\bar{x}-2T)}{\HSDFPline{t}{N}{n}(\bar{x}-T)}{\HSDFPline{t}{N}{n-2}}{\HSDFPline{t}{N}{n}}{\Ham_t} $ and then let $\mathsf{A}$ be defined in terms of $B$ (i.e. we replace $\HSDFPline{t}{N}{n-1}$ by $B$ in the definition).

Since we have assumed that the event $(\mathsf{E}^{N,-}_n)^c \cap \mathsf{H}^N_{n-1}$ occurs, we may use Lemmas \ref{monotonicity1} and \ref{monotonicity2} to construct a coupling of the measure
$\PH{n-1}{n-1}{(\bar{x}-2T,\bar{x}-T)}{\HSDFPline{t}{N}{n-1}(\bar{x}-2T)}{\HSDFPline{t}{N}{n-1}(\bar{x}-T)}{\HSDFPline{t}{N}{n-2}}{\HSDFPline{t}{N}{n}}{\Ham_t}$ on the curve $B:[\bar{x}-2T,\bar{x}-T] \to \R$ and the measure $\PH{n-1}{n-1}{(\bar{x}-2T,\bar{x}-T)}{-(\bar{x}-2T)^2 + \Rstar_{n-1}}{-(\bar{x}-T)^2 + \Rstar_{n-1}}{+\infty}{\HSDFPline{t}{N}{n}}{\Ham_t}$ on the curve $\tilde{B}:[\bar{x}-2T,\bar{x}-T] \to \R$ such that almost surely $B(x)\leq \tilde{B}(x)$ for all $x\in [\bar{x}-2T,\bar{x}-T]$. Since the event $\mathsf{A}^c$ becomes more probable under pointwise increases in $B$, the existence of the coupling implies that
$$
\EE\Big[\mathbf{1}_{\mathsf{A}^c} \, \big\vert\, \Fext\big(\{n-1\},[\bar{x}-2T,\bar{x}-T]\big)\Big] \leq \PH{n-1}{n-1}{(\bar{x}-2T,\bar{x}-T)}{-(\bar{x}-2T)^2 + \Rstar_{n-1}}{-(\bar{x}-T)^2 + \Rstar_{n-1}}{+\infty}{\HSDFPline{t}{N}{n}}{\Ham_t}(\mathsf{A}^c),
$$
where $\mathsf{A}$ is now defined with respect to $\tilde{B}$. The measure on $\tilde{B}$ is now relatively simple and a straightforward analysis will reveal that
\begin{equation}\label{e.finalthingtoprove}
\PH{n-1}{n-1}{(\bar{x}-2T,\bar{x}-T)}{-(\bar{x}-2T)^2 + \Rstar_{n-1}}{-(\bar{x}-T)^2 + \Rstar_{n-1}}{+\infty}{\HSDFPline{t}{N}{n}}{\Ham_t}(\mathsf{A}^c) \leq \e;
\end{equation}
thus (\ref{e.letsshowthis}), Claim \ref{c.labeled} and Lemma \ref{l.ENC} will follow.

In order to prove (\ref{e.finalthingtoprove}), let us recall that the law of $\tilde{B}$ is given in Definition \ref{maindefHBGP} by specifying its Radon-Nikodym derivative
$Z^{-1} W(\tilde{B})$ with respect to Brownian bridge with the same starting and ending points. Here we have used a shorthand
$$Z=\partfunc{n-1}{n-1}{(\bar{x}-2T,\bar{x}-T)}{-(\bar{x}-2T)^2 + \Rstar_{n-1}}{-(\bar{x}-T)^2 + \Rstar_{n-1}}{+\infty}{\HSDFPline{t}{N}{n}}{\Ham_t}$$ \
and
$$W(\tilde{B})= \bolt{n-1}{n-1}{(\bar{x}-2T,\bar{x}-T)}{-(\bar{x}-2T)^2 + \Rstar_{n-1}}{-(\bar{x}-T)^2 + \Rstar_{n-1}}{+\infty}{\HSDFPline{t}{N}{n}}{\Ham_t}(\tilde{B}).$$
The Boltzmann weight is given explicitly as
$$W(\tilde{B})= \exp\left\{ -\int_{\bar{x}-2T}^{\bar{x}-T} e^{t^{1/3}\big(\HSDFPline{t}{N}{n}(x) - \tilde{B}(x)\big)} \dd x\right\}$$
and $Z$ is the expectation of $W(\tilde{B})$ with respect to the Brownian bridge measure on $\tilde{B}$.

Let $L:[\bar{x}-2T,\bar{x}-T]\to \R$ denote the linear interpolation between $L(\bar{x}-2T) = -(\bar{x}-2T)^2 / 2 +\Rstar_{n-1}$ and $L(\bar{x}-T) = -(\bar{x}-T)^2 / 2 +\Rstar_{n-1}$. It is readily verified that
$$
\inf_{x\in[\bar{x}-2T,\bar{x}-T]} \big(L(x)+x^2/2 \big) \geq \Rstar_{n-1} -\tfrac{1}{8} T^2.
$$
Since we have assumed that the event $(\mathsf{E}^{N,-}_n)^c \cap \mathsf{H}^N_{n-1}$ occurs, it follows that
$$
\sup_{x \in [\bar{x}-2T,\bar{x}-T]} \big( \HSDFPline{t}{N}{n}(x) + x^2/2 \big) \leq  - M.
$$
From the formula for $M$ in Definition \ref{d.parts}, we find that
\begin{equation}\label{e.supjsf}
\inf_{x \in [\bar{x}-2T,\bar{x}-T]} \big(L(x)-\HSDFPline{t}{N}{n}(x) \big) \geq (K+1)T^{1/2}.
\end{equation}

Observe now that for curves $\tilde{B}$ such that
\begin{equation}\label{e.bdev}
\inf_{x\in[\bar{x}-2T,\bar{x}-T]} \big(\tilde{B}(x)-L(x)\big) \geq -T^{1/2},
\end{equation}
the weight $W(\tilde{B})$ satisfies
$$W(\tilde{B})\geq \exp \left\{ -T e^{-t^{1/3} K T^{1/2}} \right\} \geq 1/2.$$
The first inequality follows from (\ref{e.supjsf}) and the second from the explicit conditions on $T$ and $K$ in Definition \ref{d.parts}, as well as $t\geq 1$.

In computing the normalizing constant $Z$, we average $W(\tilde{B})$ over $\tilde{B}$ distributed according to the Brownian bridge measure. From Lemma \ref{l.bridgesup}, the probability that (\ref{e.bdev}) occurs for a Brownian bridge $\tilde{B}$ is known to be at least $1-e^{-2}$. Since for such $\tilde{B}$, $W(\tilde{B})\geq 1/2$, it follows that
\begin{equation}\label{e.zgeqen}
Z\geq \tfrac{1}{2} (1-e^{-2}).
\end{equation}

We further find that
\begin{eqnarray}\label{e.wefurtherfind}
\nonumber&&\hskip-.25in\PH{n-1}{n-1}{(\bar{x}-2T,\bar{x}-T)}{-(\bar{x}-2T)^2 + \Rstar_{n-1}}{-(\bar{x}-T)^2 + \Rstar_{n-1}}{+\infty}{\HSDFPline{t}{N}{n}}{\Ham_t}
\bigg(\sup_{x \in [\bar{x}-2T,\bar{x}-T]} \big( \tilde{B}(x) - L(x)\big) \geq  KT^{1/2} \bigg)\\
& \leq& Z^{-1} e^{-2K^2} \leq \e.
\end{eqnarray}
The inequality between the first and second lines uses a similar reasoning as above (i.e. Lemma \ref{l.bridgesup} and the trivial bound that $W$ is always bounded above by one) and the inequality in the second line follows from (\ref{e.zgeqen}) and the choice of $K$ specified in Definition \ref{d.parts}. All that remains is to observe that (\ref{e.wefurtherfind}) implies the desired (\ref{e.finalthingtoprove}).

To this end, observe that
$$
L(\bar{x}-3T/2) + (\bar{x}-3T/2)^2/2 = \Rstar_{n-1} - \tfrac{T^2}{8}
$$
and hence
$$
\tilde{B}(\bar{x}-3T/2) + (\bar{x}-3T/2)^2/2 = \tilde{B}(\bar{x}-3T/2) - L(\bar{x}-3T/2) + \Rstar_{n-1} - \tfrac{T^2}{8}
$$
Thus, if $\sup_{x \in [\bar{x}-2T,\bar{x}-T]} \big( \tilde{B}(x) - L(x)\big) <  KT^{1/2}$ holds, then
$$
\tilde{B}(\bar{x}-3T/2) + (\bar{x}-3T/2)^2/2 < \Rstar_{n-1} - \tfrac{T^2}{8} + KT^{1/2} \leq -\tfrac{T^2}{16},
$$
where the second inequality uses the assumption on $T$ in Definition~\ref{d.parts} that $\Rstar_{n-1}+KT^{1/2} \leq \tfrac{T^2}{16}$. The event $\mathsf{A}^c$ (with $\HSDFPline{t}{N}{n-1}$ replaced by $\tilde{B}$ in the definition) exactly coincides with the event that $\tilde{B}(\bar{x}-3T/2) + (\bar{x}-3T/2)^2/2 \geq -\tfrac{T^2}{16}$ and hence (\ref{e.wefurtherfind}) implies (\ref{e.finalthingtoprove}). This completes the proofs of Claim~\ref{c.labeled} and Lemma \ref{l.ENC}.
\end{proof}

\begin{proof}[Proof of Lemma \ref{l.EN}]
When the event $\mathsf{E}^N_{n}$ holds, the curve $\HSDFPline{t}{N}{n}(x) + x^2/2$   rises above the level $-M$ on both of the intervals $[\bar{x}-2T,\bar{x}-T]$ and $[\bar{x}+T,\bar{x}+2T]$. By resampling the trajectory of $\HSDFPline{t}{N}{n}(x)$ between the outermost such times and using the strong Gibbs property as well as certain monotonicities, we will infer that, on the interval $[\bar{x}-T,\bar{x}+T]$, the curve $\HSDFPline{t}{N}{n}(x) + x^2/2$ is very likely to {\it always} (as $x$ varies) be larger than $-R_n$, thus proving the lemma. Recall that the parameters $T_0,T,R_{n-1},\Rstar_{n-1},K,M$ and $R_n$ are specified in Definition \ref{d.parts}.

Define the event
$$\mathsf{F}^N_{n-1} = \Big\{ \inf_{x\in [\bar{x}-2T,\bar{x}+2T]} (\HSDFPline{t}{N}{n-1}(x)+ x^2/2) \geq -M+2T^{1/2} \Big\}.$$
We now provide an upper bound on the probability of the complement of~$\mathsf{F}^N_{n-1}$ by invoking Proposition~\ref{p.lbpara} at index~$n-1$. Recall that we made the choice $x_0(n) = x_0(n-1)/2$ when we found an upper bound on the probability of the event
$\mathsf{A}$ in the proof of Claim \ref{c.labeled}. We again make this choice. We may split the interval $[\bar{x}-2T,\bar{x}+2T]$ in the definition of $\mathsf{F}^N_{n-1}$ into four consecutive intervals of length~$T$ and seek to apply  Proposition~\ref{p.lbpara} at index~$n-1$ to each of them.  Each of these intervals must be contained in  $\big[-x_0(n-1),x_0(n-1) - T\big]$ if we are to apply the proposition. Such containment is ensured by the conditions $\bar{x} - 2T \geq - x_0(n-1)$  and $\bar{x} + 2T \leq x_0(n-1) - T$,
and, in light of $\vert \bar{x} \vert  \leq x_0(n) = x_0(n-1)/2$, these requirements are met if
$-x_0(n-1)/2 - 2T \geq - x_0(n-1)$, (i.e., $x_0(n-1) \geq 4T$)
and $x_0(n-1)/2 + 2T \leq x_0(n-1) - T$, (i.e., $x_0(n-1) \geq 6T$).
That is, setting $x_0(n-1)$ to be any value that exceeds $6T$ (as well as $T_0$), we may indeed apply Proposition~\ref{p.lbpara} at index $n-1$. Since our specifications on $M$ and $T$ cause $-M + 2T^{1/2}$ to be less than $-T^2/16$, we find by doing so with $\delta = 1/16$ that  there exists $N_0(x_0,t,\e)$ such that $\PP \big((\mathsf{F}^N_{n-1})^c\big) \leq \e$ for $N\geq N_0(x_0,t,\e)$.

We also define the event
\begin{equation}\label{e.Aeqn}
\mathsf{G} =  \Big\{ \inf_{x\in [\bar{x}-T,\bar{x}+T]} (\HSDFPline{t}{N}{n}(x)+ x^2/2) \leq -R_n \Big\}.
\end{equation}
\begin{claim}\label{c.secondmclaim}
There exists $N_0(x_0,t,\e)$ such that
$$
\PP\big(\mathsf{E}^N_n\cap \mathsf{F}^{N}_{n-1}\cap \mathsf{G}\big) \leq \e
$$
for $N\geq N_0(x_0,t,\e)$.
\end{claim}

Observe that given this claim,
$$
\PP\big(\mathsf{E}^N_n\cap \mathsf{G}\big) \leq \PP\big( \mathsf{E}^N_n \cap \mathsf{F}^{N}_{n-1}\cap \mathsf{G}\big) + \PP\big((\mathsf{F}^N_{n-1})^c\big) \leq \e +\e.
$$
The above inequality is exactly (up to replacing $\e$ by $\e/2$) Lemma \ref{l.EN}. Thus, the proof of the lemma reduces to verifying the above claim. We will follow a very similar route to that used to establish Claim \ref{c.labeled} in the proof of Lemma \ref{l.ENC}. The main difference here is that we will work with stopping domains and the associated external sigma-fields and conditional expectations.

Define $\sigma^N_{-,n}$ to be the infimum over those $x\in [\bar{x}-2T,\bar{x}-T]$ such that $\HSDFPline{t}{N}{n}(x) + x^2/2 \geq -M$ (if no such $x$ exists, set $\sigma^N_{-,n} = \bar{x}-T$). Likewise define $\sigma^N_{+,n}$ to be the infimum over those $x\in [\bar{x}+T,\bar{x}+2T]$ such that $\HSDFPline{t}{N}{n}(x) + x^2/2 \geq -M$ (if no such $x$ exists, set $\sigma^N_{+,n} = \bar{x}+T$). The event $\mathsf{E}^N_n$ is equivalent (up to negligible events) to the event that both $\sigma^N_{-,n}<\bar{x}-T$ and $\sigma^N_{+,n}>\bar{x}+T$. Therefore, the event $\mathsf{E}^N_n \cap \mathsf{F}^{N}_{n-1}$ is
$\Fext\big(\{n\},(\sigma^N_{-,n},\sigma^N_{+,n})\big)$-measurable. This implies that
$$
\PP\big(\mathsf{E}^N_n\cap \mathsf{F}^{N}_{n-1}\cap \mathsf{G}\big) = \EE\bigg[\mathbf{1}_{\mathsf{E}^N_n\cap \mathsf{F}^{N}_{n-1}} \, \EE\Big[\mathbf{1}_{\mathsf{G}}\, \big\vert\, \Fext\big(\{n\},(\sigma^N_{-,n},\sigma^N_{+,n})\big)\Big]\bigg].
$$
To establish Claim \ref{c.secondmclaim}, it therefore suffices to show that $\PP$-almost surely
\begin{equation}\label{e.Gvert}
\EE\Big[\mathbf{1}_{\mathsf{G}}\, \big\vert\, \Fext\big(\{n\},(\sigma^N_{-,n},\sigma^N_{+,n})\big)\Big] \leq \e \, \mathbf{1}_{\mathsf{E}^N_n\cap \mathsf{F}^{N}_{n-1}}  + \mathbf{1}_{\big(\mathsf{E}^N_n\cap \mathsf{F}^{N}_{n-1}\big)^c}\, .
\end{equation}

The interval $(\sigma^N_{-,n},\sigma^N_{+,n})$ forms a $\{n\}$-stopping domain (Definition~\ref{defstopdom}). By the strong Gibbs property (Lemma~\ref{stronggibbslemma}) and the $\Ham_t$-Brownian Gibbs property enjoyed by $\HSDFPlinet{t}{N}$ (Corollary \ref{rescaledHBGP}), it follows that $\PP$-almost surely
$$
\EE\Big[\mathbf{1}_{\mathsf{G}}\, \big\vert\, \Fext\big(\{n\},(\sigma^N_{-,n},\sigma^N_{+,n})\big)\Big] = \PH{n}{n}{(\sigma^N_{-,n},\sigma^N_{+,n})}{\HSDFPline{t}{N}{n}(\sigma^N_{-,n})}{\HSDFPline{t}{N}{n}(\sigma^N_{+,n})}{\HSDFPline{t}{N}{n-1}}{\HSDFPline{t}{N}{n+1}}{\Ham_t}(\mathsf{G}).
$$
To explain the notation on this right-hand side, let $B:(\sigma^N_{-,n},\sigma^N_{+,n}) \to \R$ be the curve distributed according to the given measure and let $\mathsf{G}$ be defined now in terms of $B$ (i.e. replace $\HSDFPline{t}{N}{n}$ by $B$ in the definition).

By Lemma \ref{monotonicity1}, there exists a coupling of the measure $\PH{n}{n}{(\sigma^N_{-,n},\sigma^N_{+,n})}{\HSDFPline{t}{N}{n}(\sigma^N_{-,n})}{\HSDFPline{t}{N}{n}(\sigma^N_{+,n})}{\HSDFPline{t}{N}{n-1}}{\HSDFPline{t}{N}{n+1}}{\Ham_t}$ on the curve $B:(\sigma^N_{-,n},\sigma^N_{+,n})\to \R$ with the measure $\PH{n}{n}{(\sigma^N_{-,n},\sigma^N_{+,n})}{\HSDFPline{t}{N}{n}(\sigma^N_{-,n})}{\HSDFPline{t}{N}{n}(\sigma^N_{+,n})}{\HSDFPline{t}{N}{n-1}}{-\infty}{\Ham_t}$ on the curve $\tilde{B}:(\sigma^N_{-,n},\sigma^N_{+,n})\to \R$ such that almost surely $B(x)\geq \tilde{B}(x)$ for $x\in (\sigma^N_{-,n},\sigma^N_{+,n})$. Since the event $\mathsf{G}$ becomes more probable under pointwise decrease in $B$, this implies that
$$
\EE\Big[\mathbf{1}_{\mathsf{G}}\, \big\vert\, \Fext\big(\{n\},(\sigma^N_{-,n},\sigma^N_{+,n})\big)\Big]\leq \PH{n}{n}{(\sigma^N_{-,n},\sigma^N_{+,n})}{\HSDFPline{t}{N}{n}(\sigma^N_{-,n})}{\HSDFPline{t}{N}{n}(\sigma^N_{+,n})}{\HSDFPline{t}{N}{n-1}}{-\infty}{\Ham_t}(\mathsf{G})
$$
where $\mathsf{G}$ is now defined with respect to $\tilde{B}$.

We are seeking to prove (\ref{e.Gvert}). The bound by $\mathbf{1}_{\big(\mathsf{E}^N_n\cap \mathsf{F}^{N}_{n-1}\big)^c}$ is trivial; thus, from here on in we will assume that the event $\mathsf{E}^N_n\cap \mathsf{F}^{N}_{n-1}$ occurs. On this event we know that $\HSDFPline{t}{N}{n}(\sigma^N_{\pm,n}) = -(\sigma^N_{\pm,n})^2/2 - M$. Therefore, in order to prove (\ref{e.Gvert}) and hence complete the proofs of Claim \ref{c.secondmclaim} and Lemma \ref{l.EN}, we must prove that, when $\mathsf{E}^N_n\cap \mathsf{F}^{N}_{n-1}$ occurs,
\begin{equation}\label{e.phsigm}
\PH{n}{n}{(\sigma^N_{-,n},\sigma^N_{+,n})}{-(\sigma^N_{-,n})^2/2 - M}{-(\sigma^N_{+,n})^2/2 - M}{\HSDFPline{t}{N}{n-1}}{-\infty}{\Ham_t}(\mathsf{G})\leq \e.
\end{equation}

In order to prove (\ref{e.phsigm}), recall that the law of $\tilde{B}$ is specified (Definition \ref{maindefHBGP}) via a Radon-Nikodym derivative $Z^{-1}W(\tilde{B})$ with respect to the law of Brownian bridge with the same starting and ending points. Here we have used a shorthand
$$
Z = \partfunc{n}{n}{(\sigma^N_{-,n},\sigma^N_{+,n})}{-(\sigma^N_{-,n})^2/2 - M}{-(\sigma^N_{+,n})^2/2 - M}{\HSDFPline{t}{N}{n-1}}{-\infty}{\Ham_t}
$$
and
$$
W(\tilde{B}) = \bolt{n}{n}{(\sigma^N_{-,n},\sigma^N_{+,n})}{-(\sigma^N_{-,n})^2/2 - M}{-(\sigma^N_{+,n})^2/2 - M}{\HSDFPline{t}{N}{n-1}}{-\infty}{\Ham_t}(\tilde{B}).
$$
The Boltzmann weight is given explicitly as
$$W(\tilde{B})= \exp\bigg\{ -\int_{\sigma^N_{-,n}}^{\sigma^N_{+,n}} e^{t^{1/3}\big(\tilde{B}(x) - \HSDFPline{t}{N}{n-1}(x)\big)} \dd x\bigg\}$$
and $Z$ is the expectation of $W(\tilde{B})$ with respect to the Brownian bridge measure on $\tilde{B}$.

Let $L:\big[(\sigma^N_{-,n}),(\sigma^N_{+,n})\big]\to \R$ denote the linear interpolation between $L(\sigma^N_{-,n}) =  -(\sigma^N_{-,n})^2/2 -M$ and $L(\sigma^N_{+,n}) =  -(\sigma^N_{+,n})^2/2 -M$.
Consider curves $\tilde{B}$ for which
\begin{equation}\label{e.wlsup}
\sup_{x\in [\sigma_{-,n}^N, \sigma_{+,n}^N]} \big( \tilde{B}(x) - L(x) \big) \leq T^{1/2} \, .
\end{equation}
Note then that, for all $x\in  [\sigma_{-,n}^N, \sigma_{+,n}^N]$,
$$
\HSDFPline{t}{N}{n-1}(x) \geq -x^2/2 -M + 2T^{1/2} \geq L(x) + 2T^{1/2} \geq \tilde{B}(x) + T^{1/2},
$$
where the first inequality is due to the occurrence of $\mathsf{F}^{N}_{n-1}$, the second is due to the concavity of $-x^2/2$, and the third is due to (\ref{e.wlsup}).

For such curves $\tilde{B}$ that satisfy (\ref{e.wlsup}), it follows that the weight $W(\tilde{B})$ is bounded below by
$$
W(\tilde{B}) \geq \exp \left\{-4T e^{-t^{1/3} T^{1/2}} \right\} \geq \frac{1}{2}
$$
where the first inequality uses  $\sigma^N_{+,n}\,-\, \sigma^N_{-,n} \leq 4T$ and the second is due to the assumptions on $T$ and the fact that $t\geq 1$.

In computing the normalizing constant $Z$, we average $W(\tilde{B})$ over $\tilde{B}$ distributed according to the Brownian bridge measure. From Lemma~\ref{l.bridgesup} and $\sigma_{+,n}^N \,-\, \sigma_{-,n}^N \leq 4T$, the probability that (\ref{e.wlsup}) occurs for a Brownian bridge $\tilde{B}$ is at least $1 - e^{-1/2}$. Since for such $\tilde{B}$, $W(\tilde{B})\geq 1/2$, it follows that
\begin{equation}\label{e.newz}
Z\geq \tfrac{1}{2} (1-e^{-1/2}).
\end{equation}

We further find that
\begin{eqnarray*}
&&\hskip-.25in\PH{n}{n}{(\sigma^N_{-,n},\sigma^N_{+,n})}{-(\sigma^N_{-,n})^2/2 - M}{-(\sigma^N_{+,n})^2/2 - M}{\HSDFPline{t}{N}{n-1}}{-\infty}{\Ham_t}\bigg(\inf_{x\in [\sigma_{-,n}^N, \sigma_{+,n}^N]} \big(B(x) - L(x)\big) \leq -KT^{1/2}\bigg)\\
&\leq& Z^{-1} e^{-2K^2} \leq \e.
\end{eqnarray*}
The inequality between the first and second lines uses a similar reasoning as above (i.e. Lemma \ref{l.bridgesup} and the trivial bound that $W$ is always bounded above by one) and the inequality in the second line follows from (\ref{e.newz}) and the choice of $K$ specified in Definition \ref{d.parts}. All that remains is to observe that the above bound implies the desired inequality (\ref{e.phsigm}).
To this end, we claim that
$$
\inf_{x\in [\sigma_{-,n}^N, \sigma_{+,n}^N]} \big(B(x) + x^2/2 + M + 2T^2 \big) \geq \inf_{x\in [\sigma_{-,n}^N, \sigma_{+,n}^N]} \big(B(x) - L(x)\big) \, .
$$
Indeed, the left-hand term is at least the sum of the right-hand term and a further term that is given by $\inf_{x\in [\sigma_{-,n}^N, \sigma_{+,n}^N]} \big(L(x) + x^2/2 + M + 2T^2 \big)$; and the latter infimum is attained at $x = - 2^{-1} \big( \sigma_{-,n}^N +  \sigma_{+,n}^N \big)$ and thus equals $-8^{-1} \big( \sigma_{+,n}^N -  \sigma_{-,n}^N \big)^2 + 2T^2$, a quantity that is non-negative in view of
$0 \leq \sigma_{+,n}^N -  \sigma_{-,n}^N \leq 4T$. Recalling from Definition~\ref{d.parts} that $R_n \geq M + 2T^2$, we confirm (\ref{e.phsigm}) by means of the two preceding displays.

%This, however, readily follows from the bound that $L(x)\leq -x^2/2 -M$ and the fact that $R_n=M+KT^{1/2}$ ().
This completes the proofs of Claim \ref{c.secondmclaim} and Lemma \ref{l.EN}.
\end{proof}

As explained earlier, having proved Lemmas \ref{l.ENC} and \ref{l.EN}, we conclude the proof of Proposition~\ref{p.epstr}.

\subsection{Proof of Proposition \ref{p.lbpara}}\label{s.rwweg}
We prove this proposition by induction on the index $n$. In order to deduce the proposition for index $n$, we rely  on Proposition \ref{p.epstr} for index $n$ as well as Proposition \ref{p.lbpara} for index $n-1$. As a base case one finds trivially that the result holds true for index $n=0$. Thus we assume below that $n\geq 1$. See Figure~\ref{keypropfig} for a schematic illustration of the induction.

Consider $\e>0$ and $\delta\in (0,\tfrac{1}{8})$ fixed from the statement of the proposition. As the desired result is trivially satisfied for $\e\geq 1$ we may assume that $\e \in (0,1)$. We start the proof by specifying the value of $T_0$ for which we will derive Proposition \ref{p.lbpara} (for index $n$).

Proposition \ref{p.epstr} for index $n$ (which we have already proved) implies the existence of a constant $R_n$ such that, for all $t\geq 1$, $x_0>0$, and $x\in [-x_0,x_0]$,
\begin{equation}\label{e.RNeta}
\PP\big(\HSDFPline{t}{N}{n}(x) + x^2/2 < -R_n \big)\leq \frac{\e \delta}{3}
\end{equation}
whenever $N\geq N_0(x_0,t,\e\delta/3)$. Here, the value of $N_0(x_0,t,\e\delta/3)$ is also provided by Proposition \ref{p.epstr} (for index $n$). For what follows we fix this constant $R_n$. Let us also define, for $y_0,T>0$, the event
$$
\mathsf{C}_{y_0,T} = \bigg\{ \inf_{x\in[y_0, y_0 +T]} \big(\HSDFPline{t}{N}{n-1}(x) + x^2/2\big) \geq -\tfrac{1}{2}\delta T^2\bigg\}.
$$

%\begin{definition}\label{d.towards}
For the whole duration of the proof, we fix a constant $T_0>0$ large enough that the following conditions hold:
\begin{enumerate}
\item $R_n \leq \tfrac{5}{8} \delta (T_0)^2$;
\item for all $T>T_0$,
$$
\max\Big\{\big(\delta (T+1)\big)^{1/2},\, K(T)\,\big(\delta (T+1)\big)^{1/2},\, \delta^2 (T+1)^2  \Big\} \leq \tfrac{1}{8}\delta  T^2,\qquad e^{-(T+1)\delta e^{-\frac{1}{8}\delta T^2}} \geq \frac{1}{2},
$$
where we have defined
\begin{equation}\label{e.kt}
K(T) = \Big(\tfrac{1}{2} \log \big(2 (1-e^{-2})^{-1} 3T \e^{-1}\big)\Big)^{1/2};
\end{equation}
\item for all $t\geq 1$, $x_0\geq T_0$, $T\in [T_0,x_0]$ and $y_0\in [-x_0,x_0-T]$,
\begin{equation}\label{e.Cnbdd}
\PP\big(\mathsf{C}_{y_0,T}\big) \geq 1-\frac{\e}{3}
\end{equation}
for $N\geq N_0(x_0,t,\e,\delta)$ large enough. The existence of such a $T_0$ (as well as $N_0(x_0,t,\e,\delta)$) for which this final condition holds is assured by Proposition \ref{p.lbpara} for index $n-1$.
\end{enumerate}
%\end{definition}

For $x_0\geq T_0$ and $t\geq 1$, define $N_0(x_0,t,\e,\delta)$ to be the maximum of $N_0(x_0,t,\e\delta/3)$ (specified around (\ref{e.RNeta}) by Proposition \ref{p.epstr} for index $n-1$) and $N_0(x_0,t,\e,\delta)$ (specified around (\ref{e.Cnbdd}) by Proposition \ref{p.lbpara} for index $n-1$). For the rest of the proof of Proposition~\ref{p.lbpara}, we consider a parameter choice satisfying $t\geq 1$, $x_0\geq T_0$, $N\geq N_0(x_0,t,\e,\delta)$, $T\in [T_0,x_0]$ and $y_0\in [-x_0,x_0-T]$.

Define the event
$$
\mathsf{E}_{y_0,T} = \bigg\{\inf_{x\in[y_0,y_0+T]} \big(\HSDFPline{t}{N}{n}(x) + x^2/2\big) \leq -\delta T^2\bigg\}.
$$
Proving Proposition \ref{p.lbpara} amounts to showing that
\begin{equation}\label{e.eyt}
\PP(\mathsf{E}_{y_0,T})\leq \e.
\end{equation}
The remainder of the proof is devoted to this aim.

We will say that $x\in \Z\cap [-x_0,x_0]$ is $(\e\delta/3)$-good if $\HSDFPline{t}{N}{n}(x) + x^2/2 \geq -R_n$ where $R_n$ is defined by means of~(\ref{e.RNeta}). We say that $x\in \Z\cap [-x_0,x_0]$ is $(\e\delta/3)$-bad if it is not $(\e\delta/3)$-good. Define the event $\mathsf{B}_{y_0,T}$ that the number of $(\e\delta/3)$-bad $x$ in $\Z\cap [y_0,y_0+T]$ is at most $(T+1)\delta$.

It follows from (\ref{e.RNeta}) that the probability that any given $x\in\Z\cap [y_0,y_0+T]$ is $(\e\delta/3)$-good is at least $1-\e\delta/3$. The mean number of $(\e\delta/3)$-bad $x \in \Z\cap[y_0,y_0+T]$ is therefore at most $(T+1)\e\delta/3$. Thus, by  the Markov inequality,
\begin{equation}\label{e.byt}
\PP\big(\mathsf{B}_{y_0,T}\big)\geq 1-\frac{\e}{3}.
\end{equation}

Observe that
$$
\PP\big(\mathsf{E}_{y_0,T}\big) \leq \PP\big(\mathsf{E}_{y_0,T} \cap \mathsf{B}_{y_0,T} \cap \mathsf{C}_{y_0,T} \big) + \PP\big( (\mathsf{B}_{y_0,T})^c \cup (\mathsf{C}_{y_0,T})^c\big).
$$
By the bounds (\ref{e.byt}) and (\ref{e.Cnbdd}), we find that $\PP\big((\mathsf{B}_{y_0,T})^c \cup (\mathsf{C}_{y_0,T})^c\big)\leq \tfrac{2}{3} \e$. Thus, to prove (\ref{e.eyt}), it remains to show that
\begin{equation}\label{e.eyttwo}
\PP\big(\mathsf{E}_{y_0,T} \cap \mathsf{B}_{y_0,T} \cap \mathsf{C}_{y_0,T} \big)\leq \frac{\e}{3}.
\end{equation}

The event $\mathsf{C}_{y_0,T}$ is concerned with the curve of index $n-1$; hence, it is $\Fext\big(\{n\}\times [y_0,y_0+T]\big)$-measurable. Using conditional expectation we have that
$$
\PP\big(\mathsf{E}_{y_0,T} \cap \mathsf{B}_{y_0,T} \cap \mathsf{C}_{y_0,T} \big)= \EE\bigg[\mathbf{1}_{\mathsf{C}_{y_0,T}} \, \EE\Big[\mathbf{1}_{\mathsf{E}_{y_0,T} \cap \mathsf{B}_{y_0,T}}\, \big\vert\,
\Fext\big(\{n\}\times [y_0,y_0+T]\big)\Big]\bigg].
$$

Showing (\ref{e.eyttwo}) then reduces to showing that $\PP$-almost surely,
\begin{equation}\label{e.claimeb}
\EE\Big[\mathbf{1}_{\mathsf{E}_{y_0,T} \cap \mathsf{B}_{y_0,T}}\, \big\vert\,
\Fext\big(\{n\}\times [y_0,y_0+T]\big)\Big] \leq \tfrac{\e}{3} \mathbf{1}_{\mathsf{C}_{y_0,T}} + \mathbf{1}_{(\mathsf{C}_{y_0,T})^c}.
\end{equation}
The bound by $\mathbf{1}_{(\mathsf{C}_{y_0,T})^c}$ is trivial. In what follows, let us therefore assume that the event $\mathsf{C}_{y_0,T}$ occurs. On this event the $\Ham$-Brownian Gibbs property for $\HSDFPlinet{t}{N}$ (Corollary \ref{rescaledHBGP}) implies that $\PP$-almost surely
$$
\EE\Big[\mathbf{1}_{\mathsf{E}_{y_0,T} \cap \mathsf{B}_{y_0,T}}\, \big\vert\,\Fext\big(\{n\}\times [y_0,y_0+T]\big)\Big]
=\PH{n}{n}{(y_0,y_0+T)}{\HSDFPline{t}{N}{n}(y_0)}{\HSDFPline{t}{N}{n}(y_0+T)}{\HSDFPline{t}{N}{n-1}}{\HSDFPline{t}{N}{n+1}}{\Ham_t}\big(\mathsf{E}_{y_0,T} \cap \mathsf{B}_{y_0,T}\big).
$$
On the right-hand side the events $\mathsf{E}_{y_0,T} \cap \mathsf{B}_{y_0,T}$ are now defined in terms of $B'$ (i.e. by replacing $\HSDFPline{t}{N}{n}$ by $B'$ in the definition of the events) where $B':[y_0,y_0+T]\to \R$  is distributed according to $\PH{n}{n}{(y_0,y_0+T)}{\HSDFPline{t}{N}{n}(y_0)}{\HSDFPline{t}{N}{n}(y_0+T)}{\HSDFPline{t}{N}{n-1}}{\HSDFPline{t}{N}{n+1}}{\Ham_t}$.

%Lemma \ref{monotonicity1} implies the existence of a coupling of the measure $\PH{n}{n}{(y_0,y_0+T)}{\HSDFPline{t}{N}{n}(y_0)}{\HSDFPline{t}{N}{n}(y_0+T)}{\HSDFPline{t}{N}{n-1}}{\HSDFPline{t}{N}{n+1}}{\Ham_t}$ on the curve $B':[y_0,y_0+T]\to \R$ with the measure
%$\PH{n}{n}{(y_0,y_0+T)}{\HSDFPline{t}{N}{n}(y_0)}{\HSDFPline{t}{N}{n}(y_0+T)}{\HSDFPline{t}{N}{n-1}}{-\infty}{\Ham_t}$ on the curve $B'':[y_0,y_0+T]\to \R$ such that almost surely $B'(x)\geq B''(x)$ for $x\in [y_0,y_0+T]$. Since the event $\mathsf{E}_{y_0,T} \cap \mathsf{B}_{y_0,T}$ becomes  more probable as $B'$ decreases \note{this is untrue}, we find that
%$$
%\EE\Big[\mathbf{1}_{\mathsf{E}_{y_0,T} \cap \mathsf{B}_{y_0,T}}\, \big\vert\,\Fext\big(\{n\}\times [y_0,y_0+T]\big)\Big]
%\leq \PH{n}{n}{(y_0,y_0+T)}{\HSDFPline{t}{N}{n}(y_0)}{\HSDFPline{t}{N}{n}(y_0+T)}{\HSDFPline{t}{N}{n-1}}{-\infty}{\Ham_t}\big(\mathsf{E}_{y_0,T} \cap \mathsf{B}_{y_0,T}\big).
%$$
We must show that when the event $\mathsf{C}_{y_0,T}$ holds, the left-hand side of (\ref{e.claimeb}) is bounded  by $\tfrac{\e}{3}$.

In order to establish this, we will decompose the event $\mathsf{E}_{y_0,T} \cap \mathsf{B}_{y_0,T}$ further. For any set $A\subseteq \Z\cap [y_0,y_0+T]$, let $\mathsf{G}_A$ denote the event that the set of $(\e\delta/3)$-good $x \in \Z\cap[y_0,y_0+T]$ is exactly the set $A$. Write $\ell_A$ as the maximal gap in $A$ (i.e. the maximal length among the intervals that comprise $[y_0,y_0+T]\cap A^c$) and denote by $S_{T,\delta}$ the set of all $A\subseteq \Z\cap [y_0,y_0+T]$ such that $\ell_A\leq (T+1)\delta$.

Observe that the event $\mathsf{B}_{y_0,T}$ is a subset of the union of $\mathsf{G}_A$ over all $A\in S_{T,\delta}$ (i.e. $\mathsf{B}_{y_0,T} \subseteq \bigcup_{A\in S_{T,\delta}} \mathsf{G}_A$). This is because having at most $(T+1)\delta$ integers $x\in \Z\cap[y_0,y_0+T]$ which are $(\e\delta/3)$-bad implies that the maximal number of such consecutive integers is at most $(T+1)\delta$.
This implies that
\begin{eqnarray}\label{e.PHnnyo}
&&\hskip-.25in \PH{n}{n}{(y_0,y_0+T)}{\HSDFPline{t}{N}{n}(y_0)}{\HSDFPline{t}{N}{n}(y_0+T)}{\HSDFPline{t}{N}{n-1}}{\HSDFPline{t}{N}{n+1}}{\Ham_t}\big(\mathsf{E}_{y_0,T} \cap \mathsf{B}_{y_0,T}\big)\\
\nonumber&\leq & \sum_{A\in S_{T,\delta}} \PH{n}{n}{(y_0,y_0+T)}{\HSDFPline{t}{N}{n}(y_0)}{\HSDFPline{t}{N}{n}(y_0+T)}{\HSDFPline{t}{N}{n-1}}{\HSDFPline{t}{N}{n+1}}{\Ham_t}\big(\mathsf{E}_{y_0,T} \cap \mathsf{G}_{A}\big)\\
\nonumber& = & \sum_{A\in S_{T,\delta}} p_A\, \cdot \PH{n}{n}{(y_0,y_0+T)}{\HSDFPline{t}{N}{n}(y_0)}{\HSDFPline{t}{N}{n}(y_0+T)}{\HSDFPline{t}{N}{n-1}}{\HSDFPline{t}{N}{n+1}}{\Ham_t}\big(\mathsf{E}_{y_0,T} \,\big\vert\, \mathsf{G}_{A}\big),
\end{eqnarray}
where $$p_A = \PH{n}{n}{(y_0,y_0+T)}{\HSDFPline{t}{N}{n}(y_0)}{\HSDFPline{t}{N}{n}(y_0+T)}{\HSDFPline{t}{N}{n-1}}{\HSDFPline{t}{N}{n+1}}{\Ham_t}\big(\mathsf{G}_A\big)$$ and where
$$\PH{n}{n}{(y_0,y_0+T)}{\HSDFPline{t}{N}{n}(y_0)}{\HSDFPline{t}{N}{n}(y_0+T)}{\HSDFPline{t}{N}{n-1}}{\HSDFPline{t}{N}{n+1}}{\Ham_t}\big(\, \cdot \,\big\vert\, \mathsf{G}_{A}\big)$$ represents the measure conditioned on the occurrence of the event $\mathsf{G}_A$.

This conditioned measure is a special case of a more general class of measures from Definition~\ref{moregengibbs}. In particular,
\begin{equation}\label{e.phnny}
\PH{n}{n}{(y_0,y_0+T)}{\HSDFPline{t}{N}{n}(y_0)}{\HSDFPline{t}{N}{n}(y_0+T)}{\HSDFPline{t}{N}{n-1}}{\HSDFPline{t}{N}{n+1}}{\Ham_t}\big( \cdot \,\big\vert\, \mathsf{G}_{A}\big)=
\PH{n}{n}{(y_0,y_0+T)}{\HSDFPline{t}{N}{n}(y_0)}{\HSDFPline{t}{N}{n}(y_0+T)}{\vec{f}}{\vec{g}}{\Ham,\vec{\Ham}^f,\vec{\Ham}^g}\big(\,\cdot\,\big)
\end{equation}
where $\vec{f} = (f_1,f_2)$ is given by
$$f_1(x)= \HSDFPline{t}{N}{n-1}(x),\qquad \mathrm{and}\qquad f_2(x) = (-x^2/2 -R_n)\cdot \mathbf{1}_{x\in \Z\cap A^c} + \infty \cdot\mathbf{1}_{x\notin \Z\cap A^c},$$
and $\vec{g} = (g_1,g_2)$ by
$$g_1(x) \equiv \HSDFPline{t}{N}{n+1},\qquad \mathrm{and}\qquad g_2(x) = (-x^2/2-R_n)\cdot \mathbf{1}_{x \in \Z \cap A} -\infty \cdot \mathbf{1}_{x \notin \Z\cap A};$$
we also set $\Ham= \Ham_t$ and $\vec{\Ham}^f=\vec{\Ham}^g = (\Ham_t, \Ham_{+\infty})$. As in Definition \ref{moregengibbs}, the notation $\Ham_{+\infty}$ corresponds to conditioning on the non-intersection event that $g_2(\cdot)<\tilde{B}(\cdot) < f_2(\cdot)$ on $[y_0,y_0+T]$. Here we are writing $\tilde{B}$ as the curve distributed according to the law $\PH{n}{n}{(y_0,y_0+T)}{\HSDFPline{t}{N}{n}(y_0)}{\HSDFPline{t}{N}{n}(y_0+T)}{\vec{f}}{\vec{g}}{\Ham,\vec{\Ham}^f,\vec{\Ham}^g}$.
This measure is illustrated on the left-hand side of Figure \ref{f.goodbad}.

In light of (\ref{e.PHnnyo}), if we can show that (letting $\mathsf{E}_{y_0,T}$ now be defined with respect to $\tilde{B}$)
\begin{equation}\label{e.llate}
\PH{n}{n}{(y_0,y_0+T)}{\HSDFPline{t}{N}{n}(y_0)}{\HSDFPline{t}{N}{n}(y_0+T)}{\vec{f}}{\vec{g}}{\Ham,\vec{\Ham}^f,\vec{\Ham}^g}\big( \mathsf{E}_{y_0,T} \big) \leq \frac{\e}{3}
\end{equation}
then, since $\sum_{A\in S_{T,\delta}}p_A \leq 1$, it follows from (\ref{e.PHnnyo}) that
$$
\PH{n}{n}{(y_0,y_0+T)}{\HSDFPline{t}{N}{n}(y_0)}{\HSDFPline{t}{N}{n}(y_0+T)}{\HSDFPline{t}{N}{n-1}}{\HSDFPline{t}{N}{n+1}}{\Ham_t}\big(\mathsf{E}_{y_0,T} \cap \mathsf{B}_{y_0,T}\big) \leq \frac{\e}{3},
$$
as desired.

\begin{figure}
\centering\epsfig{file=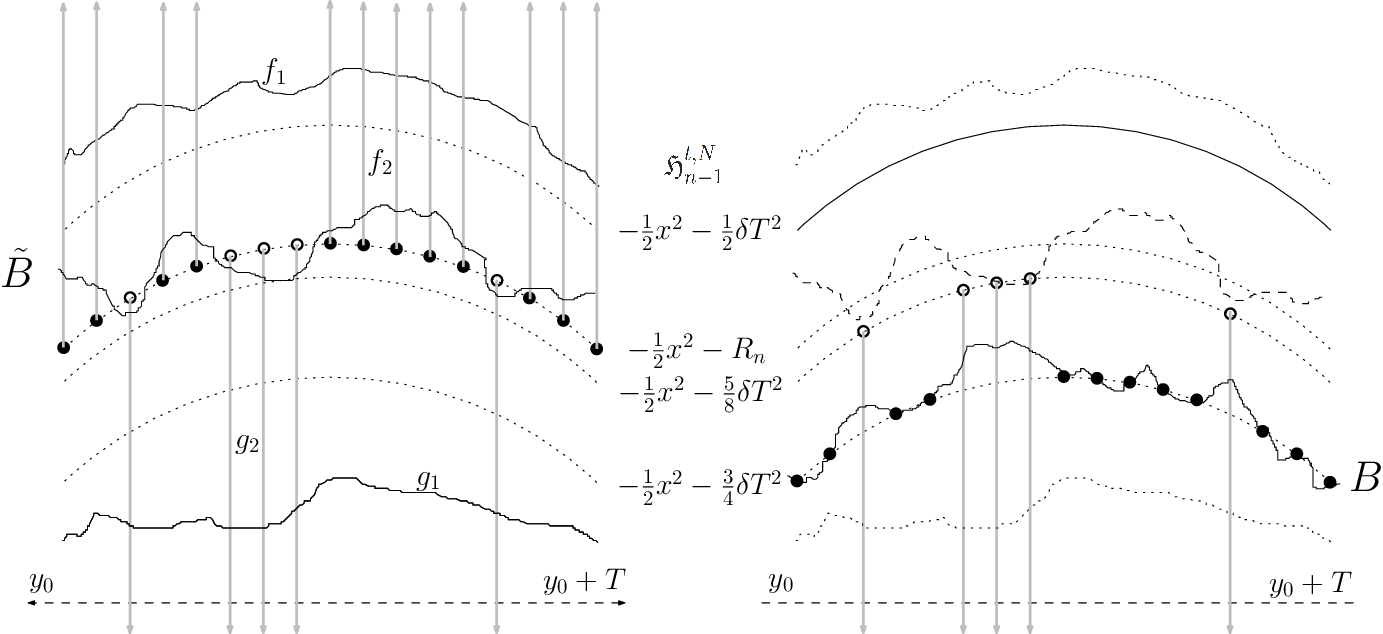, width=15cm}
\caption{Overview of the coupling between the measure (on the left of the figure) $\PH{n}{n}{(y_0,y_0+T)}{\HSDFPline{t}{N}{n}(y_0)}{\HSDFPline{t}{N}{n}(y_0+T)}{\vec{f}}{\vec{g}}{\Ham,\vec{\Ham}^f,\vec{\Ham}^g}$ and the measure (on the right of the figure) $P$, assuming that the event $\mathsf{C}_{y_0,T}$ occurs. Left: The bold solid curve $\tilde{B}$ is distributed according to the measure $\PH{n}{n}{(y_0,y_0+T)}{\HSDFPline{t}{N}{n}(y_0)}{\HSDFPline{t}{N}{n}(y_0+T)}{\vec{f}}{\vec{g}}{\Ham,\vec{\Ham}^f,\vec{\Ham}^g}$. This measure is that of a Brownian bridge which is weakly conditioned (via the probabilistic penalization of $\Ham_t$) to stay below $f_1=\HSDFPline{t}{N}{n-1}$ and above $g_1=\HSDFPline{t}{N}{n-1}$, and strictly conditioned to stay above the bullets (represented by $f_2$) and below the circles (represented by $g_2$) along the curve $-\tfrac{1}{2}x^2- R_n$ (the bullets are exactly the points of $A$). Right: The bold solid curve $B$ is distributed according to the measure $P$. This measure is
that of a Brownian bridge weakly conditioned (via the probabilistic penalization of $\Ham_t$) to stay below $-\tfrac{1}{2}x^2 - \tfrac{1}{2}\delta T^2$ and strictly conditioned to touch each bullet along the curve $-\tfrac{1}{2}x^2- \tfrac{3}{4}\delta T^2$ and be below each circle along the curve $-\tfrac{1}{2}x^2- \tfrac{5}{8}\delta T^2$.  When the event $\mathsf{C}_{y_0,T}$ occurs, $\HSDFPline{t}{N}{n-1}(x)\geq -\tfrac{1}{2} x^2 - \tfrac{1}{2} \delta T^2$. Note also that $-\tfrac{1}{2} x^2 - R_n \geq -\tfrac{1}{2} x^2 - \tfrac{5}{8} \delta T^2$ for all $x\in [y_0,y_0+T]$. This implies (see Lemma \ref{moregengibbslemma} for more details) that it is possible to couple the two measures so that $\tilde{B}(x)\geq B(x)$ for all $x\in [y_0,y_0+T]$. On the right-hand side, the curve $B$ is bounded above by the dashed curve which is $\tilde{B}$ from the left-hand side.}\label{f.goodbad}
\end{figure}

To show (\ref{e.llate}), we will utilize the more general monotonicity from Lemma \ref{moregengibbslemma} for the measure on the left-hand side of (\ref{e.llate}). This monotonicity is illustrated in Figure \ref{f.goodbad}. On the event $\mathsf{C}_{y_0,T}$ (and owing to the assumption that $T$ is sufficiently large that $R_n\leq \tfrac{5}{8}\delta T^2$), we may couple the measure $\PH{n}{n}{(y_0,y_0+T)}{\HSDFPline{t}{N}{n}(y_0)}{\HSDFPline{t}{N}{n}(y_0+T)}{\vec{f}}{\vec{g}}{\Ham,\vec{\Ham}^f,\vec{\Ham}^g}$ on the curve $\tilde{B}:[y_0,y_0+T]\to \R$ with the measure $P$ (which we will introduce below) on the curve $B:[y_0,y_0+T]\to \R$ so that $\tilde{B}(x)\geq B(x)$ for all $x\in [y_0,y_0+T]$. Since the event $\mathsf{E}_{y_0,T}$ becomes more probable as $\tilde{B}$ decreases, this monotonicity implies that
\begin{equation}\label{e.asdf}
\PH{n}{n}{(y_0,y_0+T)}{\HSDFPline{t}{N}{n}(y_0)}{\HSDFPline{t}{N}{n}(y_0+T)}{\vec{f}}{\vec{g}}{\Ham,\vec{\Ham}^f,\vec{\Ham}^g}\big( \mathsf{E}_{y_0,T} \big) \leq P\big( \mathsf{E}_{y_0,T} \big).
\end{equation}
Above, $\mathsf{E}_{y_0,T}$ is defined with $\HSDFPline{t}{N}{n}$ replaced by $B$. Therefore, to show (\ref{e.llate}) and hence complete the proof of Proposition \ref{p.lbpara}, we must define the measure $P$ and then show that
\begin{equation}\label{e.lefttoshowE}
P\big( \mathsf{E}_{y_0,T} \big)\leq \frac{\e}{3}.
\end{equation}

The measure $P$ (illustrated on the right-hand side of Figure \ref{f.goodbad}) on the curve $B:[y_0,y_0+T] \to \R$ is given as follows:
\begin{itemize}
\item For $x\in A$, fix $B(x)+ x^2/2 = -\tfrac{3}{4} \delta T^2$;
\item For $a<a'$ that are consecutive elements in $A$ (i.e. no $b\in A$ is such that $a<b<a'$) the law of $B$ on the interval $(a,a')$ is specified by stipulating that it has Radon-Nikodym derivative $Z^{-1} W(B)$ with respect to the law of Brownian bridge from $B(a)= -a^2/2 - \tfrac{3}{4} \delta T^2$ to   $B(a')= -(a')^2/2 - \tfrac{3}{4} \delta T^2$. Here, the Boltzmann weight is given by
\begin{equation}\label{e.wrn}
W(B)= \exp \bigg\{ - \int_a^{a'} e^{  t^{1/3} \big( B(x) + \tfrac{1}{2}\delta T^2 + x^2/2   \big)} \dd x \bigg\} \mathbf{1}_{B(b) + \tfrac{b^2}{2} <  - \tfrac{5}{8} \delta T^2 \, \forall b \in  \Z  \cap (a,a')} \, ,
\end{equation}
and $Z$ is the expectation of $W(B)$ with respect to the Brownian bridge measure;
\item For the minimal $a\in A$, the law of $B$ on the interval $[y_0,a)$ is given by the Radon-Nikodym derivative $Z^{-1}W(B)$ with respect to the law of a Brownian bridges from $B(y_0) = \HSDFPline{t}{N}{n}(y_0)$ to $B(a)= -a^2/2 - \tfrac{3}{4} \delta T^2$. Here, the Boltzmann weight
\begin{equation*}
W(B)=  \exp \bigg\{ - \int_{y_0}^{a} e^{  t^{1/3} \big( B(x) + \tfrac{1}{2}\delta T^2 + x^2/2   \big)} \dd x \bigg\} \mathbf{1}_{B(b) + \tfrac{b^2}{2} <  -  \tfrac{5}{8} \delta T^2 \, \forall b \in  \Z  \cap [y_0,a)} \, ,
\end{equation*}
and $Z$ is the expectation of $W(B)$ with respect to the Brownian motion measure;
\item For the maximal $a\in A$, the law of $B$ is similarly defined as for the minimal $b\in A$.
\end{itemize}

The monotone coupling explained in Lemma \ref{moregengibbslemma} implies that on the event $\mathsf{C}_{y_0,T}$, the inequality (\ref{e.asdf}) holds and thus reduces the proof to showing (\ref{e.lefttoshowE}).

To show (\ref{e.lefttoshowE}), we employ an argument whose style is similar to those used in the proofs of Claims \ref{c.labeled} and \ref{c.secondmclaim}. Consider $a<a'$, consecutive elements in $A$. Let $B_{a,a'}:[a,a']\to \R$ represent the restriction of the curve $B$ to the interval $[a,a']$. For the moment, consider $B_{a,a'}$ which satisfies
\begin{equation}\label{e.axa}
\sup_{x\in[a, a']} \big(B_{a,a'}(x)-L(x)\big) \leq \big(\delta (T+1)\big)^{1/2},
\end{equation}
where $L:[a,a']\to \R$ denotes the linear interpolation of $L(a)=-a^2/2-\tfrac{3}{4}\delta T^2$ and $L(a') = -(a')^2/2 - \tfrac{3}{4} \delta T^2$. By the concavity of $-x^2/2$ and the bound $\big(\delta (T+1)\big)^{1/2} \leq \frac{1}{8} \delta T^2$ (which is valid for all $T\geq T_0$ by our assumptions on $T_0$), we see that, for all $x\in [a,a']$ and $B_{a,a'}$ satisfying (\ref{e.axa}),
\begin{equation}\label{e.abovebx}
B_{a,a'}(x) \leq L(x)+ \big(\delta(T+1)\big)^{1/2} \leq -x^2/2 - \frac{3}{4} \delta T^2 +\big(\delta (T+1)\big)^{1/2} \leq -x^2/2 - \frac{5}{8} \delta T^2.
\end{equation}
For such $B_{a,a'}$, we may bound $B_{a,a'}(x)+\tfrac{1}{2}\delta T^2 + x^2/2 \leq -\tfrac{1}{8}\delta T^2$; hence, $W(B)$ from (\ref{e.wrn}) is at least
$$
W(B_{a,a'})\geq \exp \bigg\{ - \int_a^{a'} e^{ - t^{1/3} \frac{1}{8}\delta T^2} \dd x \bigg\} \geq \exp\big\{ -(T+1)\delta e^{-\frac{1}{8}\delta T^2}\big\} \geq \frac{1}{2}.
$$
The second of these inequalities is due to $a'-a\leq (T+1)\delta$, which is implied by $A\in S_{T,\delta}$, along with the fact that $B_{a,a'}$ satisfies (\ref{e.axa}). The third inequality is satisfied by virtue of $T\geq T_0$ and the assumptions made on $T_0$.

The above bound on $W(B_{a,a'})$ immediately translates into the lower bound $Z\geq \tfrac{1}{2} (1-e^{-2})^{-1}$ since under the Brownian bridge law on $B_{a,a'}$, the event in (\ref{e.axa}) occurs with probability at least $1-e^{-2}$ (owing to Lemma \ref{l.bridgesup} and $a'-a\leq (T+1)\delta$). By the same reasoning, we find that
$$
P\bigg(\inf_{x\in[a, a']} \big(B_{a,a'}(x)-L(x)\big) \leq - K(T)\big((T+1)\delta\big)^{1/2}\bigg) \leq Z^{-1} \exp\Big\{-2\big(K(T)\big)^2\Big\} \leq \frac{\e}{3T},
$$
where the final inequality is due to the definition of $K(T)$ given in (\ref{e.kt}). The curve $-x^2/2 - \tfrac{3}{4}\delta T^2$ and $L(x)$ differ by at most $(a'-a)^2$ as $x$ varies in $[a,a']$. Thus, since $(a'-a)^2 \leq \delta ^2 (T+1)^2$, we find that
$$
P\Big(\inf_{x\in[a, a']} \big(B_{a,a'}(x)+x^2/2\big) \leq -\frac{3}{4}\delta T^2 - \delta^2 (T+1)^2- K(T)\big((T+1)\delta\big)^{1/2}\Big) \leq \frac{\e}{3T}.
$$
Given that $\delta<1/8$, by assumption on $T_0$, for $T\geq T_0$ we have $\delta^2 (T+1)^2 \leq \tfrac{1}{8}\delta T^2$ and $K(T)\big((T+1) \delta\big)^{1/2} \leq \tfrac{1}{8}\delta T^2$. Thus,
$$
P\bigg(\inf_{x\in[a, a']} \big(B_{a,a'}(x)+x^2/2\big) \leq -\delta T^2 \bigg) \leq \frac{\e}{3T}.
$$

There are at most $T$ pairs $(a,a')$ of consecutive elements in $A$. Let $a_*$ and $a^*$ denote the minimal and maximal elements of $A$. Then the above bound readily implies that
$$
P\bigg(\inf_{x\in[a_*, a^*]} \big(B(x)+x^2/2\big) \leq -\delta T^2 \bigg) \leq \frac{\e}{3}.
$$

Given that $A\in S_{T,\delta}$, it follows that $a_* \leq y_0+(T+1)\delta$ and $a^*\geq y_0+T -(T+1)\delta$, implying that
$$
P\bigg(\inf_{x\in[y_0+(T+1)\delta, y_0+T-(T+1)\delta]} \big(B(x)+x^2/2\big) \leq -\delta T^2 \bigg) \leq \frac{\e}{3};
$$
which is to say,
$P(\mathsf{E}_{y_0+(T+1)\delta,T-2(T+1)\delta}) \leq \e/3$. Thus, were we to start with slightly changed values of $T\in [T_0,x_0]$ and $y_0\in [-x_0,x_0-T]$, the same argument as above would yield the conclusion $P(\mathsf{E}_{y_0,T}) \leq \e/3$. This completes that proof of (\ref{e.lefttoshowE}) and thus also of Proposition \ref{p.lbpara}.

\subsection{Proof of Proposition \ref{p.oneptub}}
The proof proceeds, of course, by induction on the curve index~$n$.
The general case is $n \geq 3$, and the case $n=2$ is a specialization. This is not quite true for $n=1$, however.
In this case, it is simple to see that the $\Ham_t$-Brownian Gibbs property enjoyed by the ensemble $\HSDFPline{t}{N}{}$ (Corollary~\ref{rescaledHBGP}) along with the one-point tightness of  $\HSDFPline{t}{N}{1}(x) + x^2/2$ (Lemma~\ref{l.input}) imply Proposition~\ref{p.oneptub} when $n=1$: the argument is a minor adaptation of that of \cite[Lemma 5.1]{CH} whose details we omit.
%\note{Ideally, add: , and we direct the reader to Figure~\ref{} for an overview. + figure}

In order to deduce the proposition for general index $n$, we will apply Proposition \ref{p.epstr} for indices $n-2$, $n-1$ and $n$, and Proposition~\ref{p.oneptub} for index $n-1$. We may assume $n\geq 2$ and note that Propositions \ref{p.epstr}  and \ref{p.lbpara} hold trivially for non-positive indices by declaring that $\HSDFPline{t}{N}{n}\equiv +\infty$ for $n\leq 0$. %See Figure \ref{keypropfig} for a schematic illustration of this induction.

%\begin{figure}
%\centering\epsfig{file=KPZnobigmax.eps, width=14cm}
%\caption{Illustrating the proof of Lemma \ref{l.nobigmax}. If $\mathcal{L}_1$ is at least $M_1$ at some moment during $[x,x+1]$, locate the leftmost such time $\phi$. The law of $\mathcal{L}_1:[\phi,x+2] \to \R$ given the external data is specified by the $H$-Brownian Gibbs rule, and by Lemma \ref{} stochastically dominates Brownian bridge $B:[\phi,x+2] \to \R$ with $B(\phi) = M_1$ and $B(x+2) = \mathcal{L}_1(x+2)$.}
%\label{f.nobigmax}
%\end{figure}
%\begin{proof} Let $x \in [-x_0,x_0 - 1]$. For $M_1 \in \R$, let $\phi = \phi_{M_1} \in [x,x+1]$ be given by
%$\phi = \inf \big\{ y \in [x,x+1]: \mathcal{L}_1(y) \geq M_1  \big\}$; should the set in question be empty, set $\phi = x+1$.
%Write  $U_x(M_1) =  \sup_{y \in [x,x+1]} \mathcal{L}_1(y) > M_1$, and note that $U_x(M_1) = \big\{ \phi < x+1 \big\}$.
%Note that $[\phi,x+2] \times \{ 1 \}$ is a strong stopping domain. Since $\mathcal{L}$ satisfies the $H$-Brownian Gibbs property, Lemma~\ref{} implies that
%$$
% \PP \Big( U_x(M_1) \cap \big\{ \mathcal{L}_1(x+2) \geq - M \big\} \Big)
%$$
%\end{proof}

In this proof we will argue that, should $\HSDFPline{t}{N}{n}(x) + x^2/2$ be very high (which means, at least the height $\Rstar_n$ we will define) at some time $x \in \big[\bar{x},\bar{x}+\tfrac{1}{2}\big]$, then, should several events (which are known to be typical) also occur, the process  $\HSDFPline{t}{N}{n-1}(x) + x^2/2$ will typically become high (at least $\Rstar_{n-1}$) at some time $x \in [\bar{x},\bar{x}+2]$. Proposition \ref{p.oneptub} at index $n-1$ shows that this eventuality is unlikely; thus so too must the event that $\HSDFPline{t}{N}{n}(x) + x^2/2\geq \Rstar_n$ for some $x\in \big[\bar{x},\bar{x}+\tfrac{1}{2}\big]$ be unlikely. In this way, Proposition~\ref{p.oneptub} at index $n$ will be established.

For arbitrary $t\geq 1$, $x_0>0$, $\bar{x}\in [-x_0,x_0-1]$ and $N\in \N$, define then the following events. These are determined by parameters $K_n, R_n, \Rstar_{n-1}$ and $\Rstar_n$ which we will choose later in the proof and which we take, for the moment, to be arbitrary.

First, define the event whose probability we seek to show is small:
$$
\mathsf{E}^N_n(\Rstar_n) = \bigg\{ \sup_{x\in [\bar{x},\bar{x}+\frac{1}{2}]} \big(\HSDFPline{t}{N}{n}(x) + x^2/2\big) \geq \Rstar_n\bigg\}.
$$
We may reexpress $\mathsf{E}^N_n(\Rstar_n)$ as follows. Define
$$\chi(R) = \inf\Big\{x\in \big[\bar{x},\bar{x}+\tfrac{1}{2}\big]: \big(\HSDFPline{t}{N}{n}(x) + x^2/2\big) \geq R \Big\}.$$
If this infimum is not attained, then define $\chi(R)$ to be $\bar{x}+\tfrac{1}{2}$. In all of what follows we will work with $\chi(\Rstar_n)$ and thus for short-hand we wrote $$\chi=\chi(\Rstar_n).$$
Thus, up to negligible events, $\mathsf{E}^N_n(\Rstar_n)$   is nothing other than $\big\{ \chi < \bar{x}+\tfrac{1}{2} \big\}$.

Second, define events which the inductive hypotheses show to be typical:
\begin{eqnarray*}
\mathsf{Q}^N_{n-2}(K_n) &=& \bigg\{\inf_{x \in [\bar{x}, \bar{x}+2]} \big(\HSDFPline{t}{N}{n-2}(x) + x^2/2\big) \geq -K_n\bigg\},\\
\mathsf{A}^N_{n-1,n}(R_n) & = & \Big\{\HSDFPline{t}{N}{n-1}\big(\chi\big)+ \chi^2/2 \geq -R_n\Big\}\cap \Big\{\HSDFPline{t}{N}{j}\big(\bar{x}+2\big)+\big(\bar{x}+2\big)^2/2 \geq -R_n\, \textrm{ for } j=n-1 \textrm{ and }j=n\Big\}%\cap  \Big\{\HSDFPline{t}{N}{n}\big(\bar{x}+2\big)+\big(\bar{x}+2\big)^2/2 \geq -R_n\Big\}
\end{eqnarray*}

Third, define the event which the inductive hypotheses show to be atypical:
$$
\mathsf{B}^N_{n-1}(\Rstar_{n-1}) = \bigg\{\sup_{x\in [\chi,\bar{x}+2]} \big(\HSDFPline{t}{N}{n-1}(x)+x^2/2\big) \geq \Rstar_{n-1}\bigg\}.
$$

Our plan is to show, roughly speaking, that the occurrence of $\mathsf{E}^N_n(\Rstar_n)$ and the above typical events will entail the atypical one. This wil imply that $\mathsf{E}^{N}_n(\Rstar_n)$ is equally atypical, and yield the proof.

%\mathsf{A}^N_{n-1,n}(R_n) &=& \bigg\{\inf_{\bar{x}\leq x\leq \bar{x}+2} \big(\HSDFPline{t}{N}{j}(x)+x^2/2\big) \geq -R_n \textrm{ for }j=n-1\textrm{ and }j= n\bigg\},\\

Observe now that the interval $[\chi,\bar{x}+2]$ forms an $\{n\}$-stopping domain for $\HSDFPlinet{t}{N}$ (Definition \ref{maindefHBGP}). Consequently, it also forms an $\{n-1,n\}$-stopping domain. The events $\mathsf{E}^N_n(\Rstar_n)$, $\mathsf{Q}^N_{n-2}(K_n)$ and $\mathsf{A}^N_{n-1,n}(R_n)$ are all $\Fext\big(\{n-1,n\},[\chi,\bar{x}+2]\big)$-measurable. For $\mathsf{E}^N_n(\Rstar_n)$, this is due to this event equalling $\big\{\chi<\bar{x}+\tfrac{1}{2}\big\}$; for the other events, it is clear.  However, the event $\mathsf{B}^N_{n-1}(\Rstar_{n-1})$ is not measurable with respect to this external sigma-field.

By using conditional expectations we have
\begin{eqnarray*}
&&\hskip-.25in\PP\Big[\mathsf{E}^N_n(\Rstar_n) \cap \mathsf{Q}^N_{n-2}(K_n) \cap \mathsf{A}^N_{n-1,n}(R_n) \cap \mathsf{B}^N_{n-1}(\Rstar_{n-1})\Big]\\
&=& \EE\bigg[\mathbf{1}_{\mathsf{E}^N_n(\Rstar_n) \cap \mathsf{Q}^N_{n-2}(K_n) \cap \mathsf{A}^N_{n-1,n}(R_n)}\, \EE\Big[\mathbf{1}_{\mathsf{B}^N_{n-1}(\Rstar_{n-1})} \,\big\vert \, \Fext\big(\{n-1,n\},[\chi,\bar{x}+2]\big)\Big]\bigg].
\end{eqnarray*}

We claim the following $\PP$-almost sure lower bound
\begin{equation}\label{e.ppaslb}
\EE\Big[\mathbf{1}_{\mathsf{B}^N_{n-1}(\Rstar_{n-1})} \,\big\vert \, \Fext\big(\{n-1,n\},[\chi,\bar{x}+2]\big)\Big]
\geq p(R_n,K_n,\Rstar_n) \,\cdot \mathbf{1}_{\mathsf{E}^N_n(\Rstar_n) \cap \mathsf{Q}^N_{n-2}(K_n) \cap \mathsf{A}^N_{n-1,n}(R_n)}
\end{equation}
where
$$
p(R_n,K_n,\Rstar_n) = \PH{n-1}{n}{(\chi,\bar{x}+2)}{(-R_n - \chi^2/2,\Rstar_n - \chi^2/2)}{(-R_n - (\bar{x} + 2)^2/2,-R_n - (\bar{x} + 2)^2/2)}{-K_n-x^2/2}{-\infty}{\Ham_t}\big(\mathsf{B}^N_{n-1}(\Rstar_{n-1})\big).
$$
The lower bound when the event $\mathsf{E}^N_n(\Rstar_n) \cap \mathsf{Q}^N_{n-2}(K_n) \cap \mathsf{A}^N_{n-1,n}(R_n)$ does not hold is trivial. Thus we may focus only on the case where $\mathsf{E}^N_n(\Rstar_n) \cap \mathsf{Q}^N_{n-2}(K_n) \cap \mathsf{A}^N_{n-1,n}(R_n)$ holds.

By the strong Gibbs property (Lemma~\ref{stronggibbslemma}) and the $\Ham_t$-Brownian Gibbs property enjoyed by $\HSDFPlinet{t}{N}$ (Corollary \ref{rescaledHBGP}), it follows that $\PP$-almost surely
\begin{eqnarray*}
&&\hskip-.25in \EE\Big[\mathbf{1}_{\mathsf{B}^N_{n-1}(\Rstar_{n-1})} \,\big\vert \, \Fext\big(\{n-1,n\},[\chi,\bar{x}+2]\big)\Big] \\
&=&
\PH{n-1}{n}{(\chi,\bar{x}+2)}{\big(\HSDFPline{t}{N}{n-1}(\chi),\,\HSDFPline{t}{N}{n}(\chi)\big)}{\big(\HSDFPline{t}{N}{n-1}(\bar{x}+2),\,\HSDFPline{t}{N}{n}(\bar{x}+2)\big)}{\HSDFPline{t}{N}{n-2}}{\HSDFPline{t}{N}{n+1}}{\Ham_t}\big(\mathsf{B}^N_{n-1}(\Rstar_{n-1})\big).
\end{eqnarray*}
The meaning of this right-hand side is given by letting $B_{n-1},B_{n}:(\chi,\bar{x}+2) \to \R$ be the curves distributed according to the  measure $\PH{n-1}{n}{(\chi,\bar{x}+2)}{\big(\HSDFPline{t}{N}{n-1}(\chi),\,\HSDFPline{t}{N}{n}(\chi)\big)}{\big(\HSDFPline{t}{N}{n-1}(\bar{x}+2),\,\HSDFPline{t}{N}{n}(\bar{x}+2)\big)}{\HSDFPline{t}{N}{n-2}}{\HSDFPline{t}{N}{n+1}}{\Ham_t}$ and then letting $\mathsf{B}^N_{n-1}(\Rstar_n)$ be defined in terms of $B_{n-1},B_{n}$ (i.e. by replacing $\HSDFPline{t}{N}{n-1}$ and $\HSDFPline{t}{N}{n}$ by $B_{n-1}$ and $B_n$ in the definition).

\begin{figure}
\centering\epsfig{file=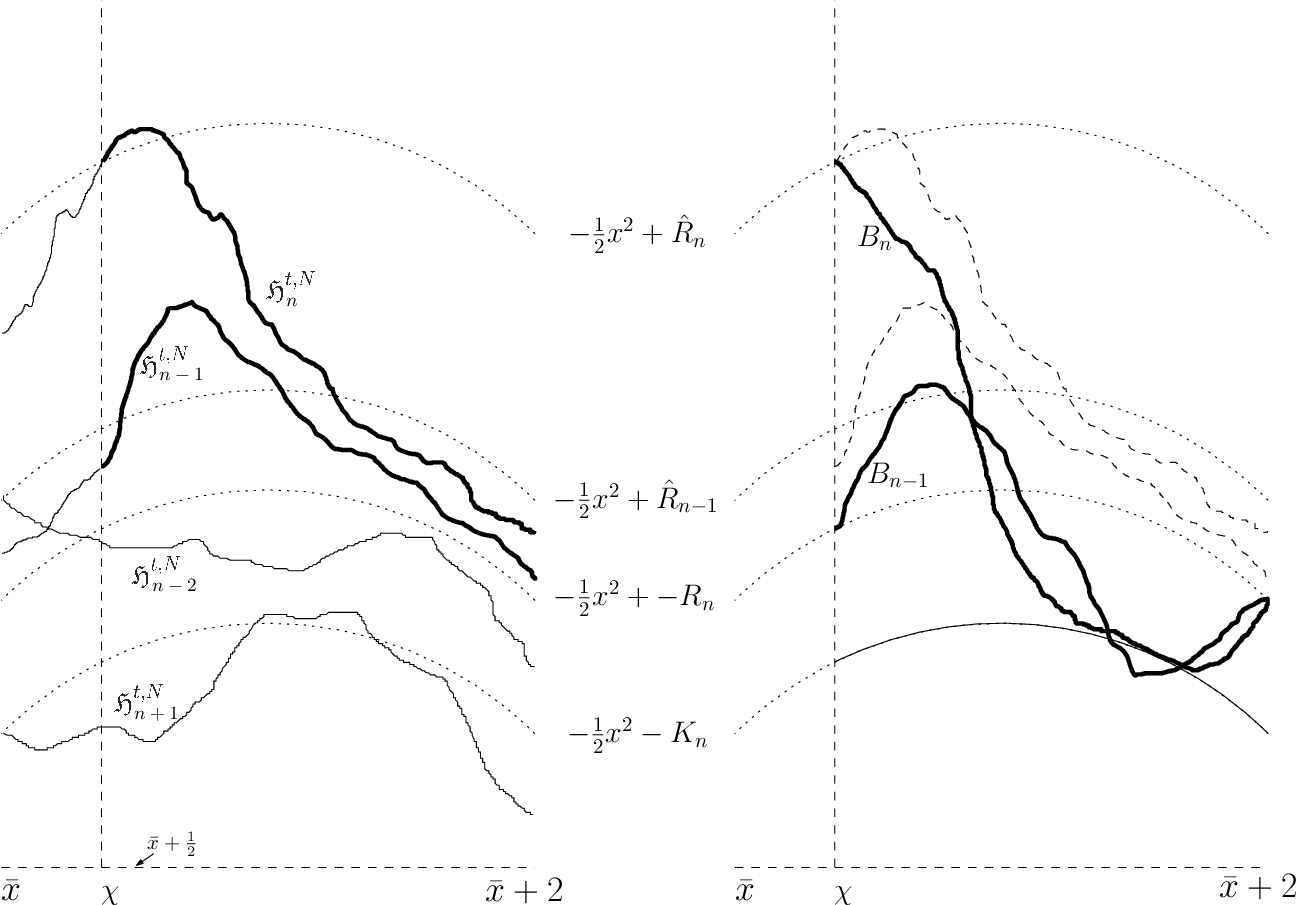, width=13cm}
\caption{An overview of the coupling constructed between the measure $\PH{n-1}{n}{(\chi,\bar{x}+2)}{\big(\HSDFPline{t}{N}{n-1}(\chi),\,\HSDFPline{t}{N}{n}(\chi)\big)}{\big(\HSDFPline{t}{N}{n-1}(\bar{x}+2),\,\HSDFPline{t}{N}{n}(\bar{x}+2)\big)}{\HSDFPline{t}{N}{n-2}}{\HSDFPline{t}{N}{n+1}}{\Ham_t}$
and the measure $\PH{n-1}{n}{(\chi,\bar{x}+2)}{(-R_n,\Rstar_n)}{(-R_n,-R_n)}{-x^2/2-K_n}{-\infty}{\Ham_t}$ assuming the occurrence of the event  $\mathsf{E}^N_n(\Rstar_n) \cap \mathsf{Q}^N_{n-2}(K_n) \cap \mathsf{A}^N_{n-1,n}(R_n)$, and a tentative explanation of why we may expect Proposition \ref{p.fiveone} to hold. Left: The two thick black curves are distributed according to the first measure. Right: When the event $\mathsf{E}^N_n(\Rstar_n) \cap \mathsf{Q}^N_{n-2}(K_n) \cap \mathsf{A}^N_{n-1,n}(R_n)$ holds, $\HSDFPline{t}{N}{n}(\chi)= -\chi^2/2+\Rstar_n$, $\HSDFPline{t}{N}{n-1}(\chi)\geq -\chi^2/2 + \Rstar_n$, $\HSDFPline{t}{N}{n}(\bar{x}+2)\geq -(\bar{x}+2)^2/2+\Rstar_n$, $\HSDFPline{t}{N}{n-1}(\bar{x}+2)\geq -(\bar{x}+2)^2/2 + \Rstar_n$, and $\HSDFPline{t}{N}{n-2}(x)\geq -x^2/2 - K_n$ for all $x\in [\bar{x},\bar{x}+2]$. This (along with the fact that $\HSDFPline{t}{N}{n+1}(x)\geq -\infty$) implies the monotonicity which is illustrated on the right with
the thick black curves lying beneath the dashed ones (which are the thick black curves from the left). Proposition \ref{p.fiveone} claims that by choosing first $R_n$, then $K_n$ and finally $\Rstar_n$ large enough, there is probability at least $\mu$ that $B_{n-1}$ ($\mathcal{L}_1$ in the proposition) rises above $-x^2/2 +\Rstar_{n-1}$ (here $\Rstar_{n-1}$ is specified by an application of Proposition \ref{p.fiveone} for index $n-1$). Why might we believe that such a claim holds? The curve $B_{n-1}$ on $[\chi,\bar{x}+2]$ begins high above the curve $-x^2/2 - K_n$, but we are permitted to insist that $B_n$ at time $\chi$ is extremely high above $B_{n-1}$. Now, during $[\chi,\bar{x}+2]$, $B_{n-1}$ is pushed in two directions by its interactions: downwards towards $-x^2/2 - K_n$, and upwards towards $B_n$. Initially at least, at times just above $\chi$, the upward push may be the greater, because $B_n$ is then farther away from $B_{n-1}$ than  $-x^2/2 - K_n$ is. If  $\Rstar_n$ is taken large compared to $K_n$,
then, we may believe that this upward force is
enough to ensure that $B_{n-1}$ rises above $-x^2/2 +\Rstar_{n-1}$ at some moment during $[\chi,\bar{x} +2]$ with the desired probability. However, the situation is rendered more complicated because $B_n$ just after time $\chi$ is likely to fall precipitously in order to attenuate its highly costly interaction with $B_{n-1}$. The proof of Proposition \ref{p.fiveone} is a detailed investigation of this circumstance.}\label{f.ABEQpic}
\end{figure}

Given that the event $\mathsf{E}^N_n(\Rstar_n) \cap \mathsf{Q}^N_{n-2}(K_n) \cap \mathsf{A}^N_{n-1,n}(R_n)$ holds, it follows that (by the definition of $\chi=\chi(\Rstar_n)$)
\begin{align*}
\HSDFPline{t}{N}{n-1}(\chi)&\geq -R_n -\chi^2/2\\
\HSDFPline{t}{N}{n}(\chi)&\geq \Rstar_n - \chi^2/2,\\
\HSDFPline{t}{N}{n-1}(\bar{x}+2)&\geq -R_n  - (\bar{x} + 2)^2/2,\\
\HSDFPline{t}{N}{n}(\bar{x}+2)&\geq -R_n  - (\bar{x} + 2)^2/2,\\
\end{align*}
and that $\HSDFPline{t}{N}{n-2}(x) \geq - K_n  -x^2/2$ and $\HSDFPline{t}{N}{n}(x) \geq -\infty$ for $x\in [\chi,\bar{x}+2]$. Therefore, by Lemmas \ref{monotonicity1} and \ref{monotonicity2}, there exists a coupling of the measure
$$\PH{n-1}{n}{(\chi,\bar{x}+2)}{\big(\HSDFPline{t}{N}{n-1}(\chi),\,\HSDFPline{t}{N}{n}(\chi)\big)}{\big(\HSDFPline{t}{N}{n-1}(\bar{x}+2),\,\HSDFPline{t}{N}{n}(\bar{x}+2)\big)}{\HSDFPline{t}{N}{n-2}}{\HSDFPline{t}{N}{n+1}}{\Ham_t}$$
on the curves $B_{n-1},B_n:[\chi,\bar{x}+2]\to \R$ with the measure
$$\PH{n-1}{n}{(\chi,\bar{x}+2)}{(-R_n -\chi^2/2,\Rstar_n - \chi^2/2)}{(-R_n - (\bar{x} +2)^2/2,-R_n  - (\bar{x} +2)^2/2)}{-K_n-x^2/2}{-\infty}{\Ham_t}$$ on the curves $B'_{n-1},B'_n:[\chi,\bar{x}+2]\to \R$ such that almost surely $B_i(x)\geq B'_i(x)$ for all $i\in \{n-1,n\}$ and $x\in [\chi,\bar{x}+2]$. Since the event $\mathsf{B}^N_{n-1}(\Rstar_n)$ becomes less probable as the curves $B_{n-1},B_n$ decrease,
$$
\PH{n-1}{n}{(\chi,\bar{x}+2)}{\big(\HSDFPline{t}{N}{n-1}(\chi),\,\HSDFPline{t}{N}{n}(\chi)\big)}{\big(\HSDFPline{t}{N}{n-1}(\bar{x}+2),\,\HSDFPline{t}{N}{n}(\bar{x}+2)\big)}{\HSDFPline{t}{N}{n-2}}{\HSDFPline{t}{N}{n+1}}{\Ham_t}\big(\mathsf{B}^N_{n-1}(\Rstar_{n-1})\big) \geq p(R_n,K_n,\Rstar_n),
$$
thus proving the claim (\ref{e.ppaslb}). The above coupling is illustrated in Figure \ref{f.ABEQpic}.

Taking expectations in (\ref{e.ppaslb}), we obtain
\begin{eqnarray}\label{e.oneoverp}
&&\hskip-.25in\PP\big(\mathsf{E}^N_n(\Rstar_n) \cap \mathsf{Q}^N_{n-2}(K_n) \cap \mathsf{A}^N_{n-1,n}(R_n)\big) \\
\nonumber&\leq& \frac{1}{\E \, p(R_n,K_n,\Rstar_n)} \PP\big(\mathsf{B}^N_{n-1}(\Rstar_{n-1})\big).
\end{eqnarray}

Choose $\Rstar_{n-1}$ so that $\PP\big(\mathsf{B}^N_{n-1}(\Rstar_{n-1})\big)\leq 2\e$. This can be achieved for $N\geq N_0(x_0,t,\e)$ owing to  Proposition \ref{p.oneptub} applied for index $n-1$ (and the union bound).

We wish to choose the other parameters $K_n$,$R_n$ and $\Rstar_n$ so that $\E \, p(R_n,K_n,\Rstar_n) \geq 1/2$. Showing that this may be done is the core of the argument, since this bound in essence asserts that should $\chi < \bar{x}+1/2$ and the known-to-be-typical events occur, then so does $\mathsf{B}^N_{n-1}(\Rstar_{n-1})$.
It is the next result which permits us to make such a choice of parameters; the result  is a key one, and its proof is given in Section \ref{s.fiveone}. See Figure \ref{f.ABEQpic} for some preliminary discussion.

\begin{proposition}\label{p.fiveone} Fix any $\mu\in (0,1)$.
There exists $\delta>0$, $R^0>0$, and functions $K^0(R)>0$ and $\Rstar^{0}(R,K)>0$ such that, for all $R> R^0$, $K> K^{0}(R)$, $\Rstar\geq \Rstar^{0}(R,K)$ and all $t\geq 1$, $\bar{x}\in \R$, and $\chi \in \big[\bar{x},\bar{x}+\tfrac{1}{2}\big]$,
\begin{equation*}%\label{e.twocurve}
\PH{1}{2}{(\chi,\bar{x}+2)}{\big(-R-\frac{\chi^2}{2},\,\Rstar-\frac{\chi^2}{2}\big)}{\big(-R-\frac{(\bar{x}+2)^2}{2},\,-R-\frac{(\bar{x}+2)^2}{2}\big)}{-\frac{x^2}{2}-K}{-\infty}{\Ham_t} \bigg(\sup_{x\in [\chi,\bar{x}+2]} \big(\mathcal{L}_1(x) +  x^2/2 \big) \geq \tfrac{1}{2}\delta \Rstar\bigg) \geq \mu.
\end{equation*}
The measure above is on the curves $\mathcal{L}_1$ and $\mathcal{L}_2$ (Definition \ref{maindefHBGP}).
\end{proposition}

Let the parameters $\delta>0$, $R^0>0$, and the functions $K^{0}(R)$ and $\Rstar^{0}(R,K)$ be specified via Proposition \ref{p.fiveone}. Choose $R_n>R^0$, $K_n>K^{0}(R_n)$ and $\Rstar_n>\Rstar^{0}(R_n,K_n)$, with the additional stipulations that $N\geq N_0(x_0,t,\e)$ and
\begin{equation}\label{e.boundsre}
\delta \Rstar_n /2 >\Rstar_{n-1},\qquad \PP\big(\mathsf{Q}^N_{n-2}(K_n)\big) \geq 1-2\e, \qquad \PP\big(\mathsf{A}^N_{n-1,n}(R_n)\big) \geq 1-3\e.
\end{equation}
The existence of such parameters that $\PP\big(\mathsf{Q}^N_{n-2}(K_n)\big) \geq 1-2\e$ holds follows from Proposition \ref{p.epstr} for index $n-2$, and that $\PP\big(\mathsf{A}^N_{n-1,n}(R_n)\big) \geq 1-3\e$ holds, from Proposition \ref{p.epstr} for indices $n-1$ and $n$. The value of  $N_0(x_0,t,\e)$ is determined by these applications of Proposition \ref{p.epstr}.

With this choice of parameters (and in particular recalling the stipulation that $\delta \Rstar_n /2 >\Rstar_{n-1}$),
it follows from Proposition \ref{p.fiveone} (identifying $\mathcal{L}_1$ with $B_{n-1}$ and $\mathcal{L}_2$ with $B_n$) with the choice $\mu=1/2$ that $\E \, p(R_n,K_n,\Rstar_n)\geq 1/2$. Thus, from (\ref{e.oneoverp}), we conclude that
$$
\PP\big(\mathsf{E}^N_n(\Rstar_n) \cap \mathsf{Q}^N_{n-2}(K_n) \cap \mathsf{A}^N_{n-1,n}(R_n)\big)  \leq 4 \,\e.
$$

Finally, observe that
$$
\PP\big(\mathsf{E}^N_n(\Rstar_n)\big) \leq  \PP\big(\mathsf{E}^N_n(\Rstar_n) \cap \mathsf{Q}^N_{n-2}(K_n) \cap \mathsf{A}^N_{n-1,n}(R_n)\big) + \PP\big( \big(\mathsf{Q}^N_{n-2}(K_n)\big)^c \cup \big(\mathsf{A}^N_{n-1,n}(R_n)\big)^{c}\big)  \leq 4\e + 2\e +3\e,
$$
where we have also used the bounds in (\ref{e.boundsre}). This is valid for $N\geq N_0(x_0,t,\e)$, the maximum over those lower bounds on $N$ assumed above.

Recalling the definition of $\mathsf{E}^N_n(\Rstar_n)$, we see that we have proved a form of Proposition \ref{p.oneptub} for index $n$ where the interval $[\bar{x},\bar{x}+1]$ is replaced by $\big[\bar{x},\bar{x}+\tfrac{1}{2}\big]$. By shifting $\bar{x}$ to $\bar{x}+\tfrac{1}{2}$ and reapplying the argument above, we obtain Proposition~\ref{p.oneptub} with $\e$ replaced by $18 \e$; however, this discrepancy is inconsequential because $\e > 0$ is arbitrary and may be replaced by $\e/18$.

\subsection{Proof of Proposition \ref{p.fiveone}} \label{s.fiveone}

Before embarking on this rather lengthy proof, we will expand a little on the explanation offered in the caption of Figure \ref{f.ABEQpic}. First of all, however, we will reduce Proposition \ref{p.fiveone} to a slightly simpler statement, which works in a flat rather than parabolically curved set of coordinates. The statement of Proposition \ref{p.fiveone} involves a factor $\tfrac{1}{2}\delta \Rstar$. Below, we will prove the same statement but with the  $\tfrac{1}{2}$ removed. This amounts to the same result (by changing $\delta$ to $\tfrac{1}{2} \delta$) and is done simply to avoid too many extra factors of $\tfrac{1}{2}$.

\subsubsection{Removing the parabola}

We consider a counterpart to the law in the statement of Proposition \ref{p.fiveone} in which parabolic terms are absent.
During the entire proof of Proposition \ref{p.fiveone} we use the shorthand
\begin{equation}\label{shorthandppp}
\PP := \PH{1}{2}{(\chi,2)}{(-R,\Rstar)}{(-R,-R)}{-K}{-\infty}{\Ham_t};
\end{equation}
the law is supported on a pair of curves $\mathcal{L}_1,\mathcal{L}_2:[\chi,2]\to \R$. The use of these flat coordinates (justified by Lemma \ref{l.para} below) has the merit of making evident that the value of $\bar{x} \in \R$ plays no role, and we have thus chosen to set $\bar{x} = 0$ in this definition (so that $\chi$ is now some fixed value in $[0,1/2]$).  It is with $\PP$ that we will work throughout the proof of Proposition \ref{p.fiveone}.

The expectation operator associated to $\PP$ will be denoted by $\EE$.  Recall from Definition \ref{maindefHBGP} that $\PP$ is specified by stipulating that it has Radon-Nikodym derivative given by $Z^{-1}W(\mathcal{L}_1,\mathcal{L}_2)$ with respect to the law $\PfreeShort$ of two independent Brownian bridges with the same starting and ending points. Here, we have used some shorthand, namely
\begin{eqnarray}\label{e.zwnew}
Z &=& \partfunc{1}{2}{(\chi,2)}{(-R,\Rstar)}{(-R,-R)}{-K}{-\infty}{\Ham_t},\\
\nonumber W(\mathcal{L}_1,\mathcal{L}_2) &=& \bolt{1}{2}{(\chi,2)}{(-R,\Rstar)}{(-R,-R)}{-K}{-\infty}{\Ham_t}(\mathcal{L}_1,\mathcal{L}_2), \\
\nonumber \PfreeShort &=& \Pfree{1}{2}{(\chi,2)}{(-R,\Rstar)}{(-R,-R)}.
\end{eqnarray}
The Boltzmann weight is given explicitly by
\begin{equation}\label{e.znew}
W(\mathcal{L}_1,\mathcal{L}_2) = \exp \left\{ -  \int_{\chi}^{2} \left(  e^{t^{1/3} \big(  \mathcal{L}_1(x) + K  \big)}  +
 e^{t^{1/3} \big(  \mathcal{L}_2(x) - \mathcal{L}_1(x) \big)} \right)   \dd x \right\} \, ,
\end{equation}
and $Z$ is the expectation of $W(\mathcal{L}_1,\mathcal{L}_2)$ with respect to $\PfreeShort$.

Define the event
\begin{equation}\label{e.mathadef}
\mathsf{A} = \bigg\{ \sup_{x\in [\chi,2]}\mathcal{L}_1(x) \leq \delta \Rstar\bigg\}.
\end{equation}

The next lemma reduces the task of proving Proposition~\ref{p.fiveone} to showing that
$\PP\big( \mathsf{A}^c \big) \geq \mu$.

\begin{lemma}\label{l.para}
Assume that
\begin{equation}\label{e.twocurveprimenotpara}
\PP\big( \mathsf{A}^c \big) \geq \mu \, ;
\end{equation}
then likewise
\begin{equation*}
\PH{1}{2}{(\chi,\bar{x}+2)}{\big(-R-\frac{\chi^2}{2},\Rstar-\frac{\chi^2}{2}\big)}{\big(-R-\frac{(\bar{x}+2)^2}{2},-R-\frac{(\bar{x}+2)^2}{2}\big)}{-\frac{x^2}{2}-K}{-\infty}{\Ham_t} \Big(\sup_{x\in [\chi,\bar{x}+2]} \big(\mathcal{L}_1(x) + x^2/2\big) \geq \delta \Rstar -1\Big) \geq \mu.
\end{equation*}
%\note{This last sentence should be removed! It is fine to prove this lemma for $\delta$ instead of $\delta/2$. When we call it we should have already noted that we will be working with $\delta$ not $\delta/2$}(This conclusion is that of Proposition \ref{p.fiveone} if $\delta \Rstar$ is replaced by $\delta \Rstar-1$.)
\end{lemma}
\begin{rem}
In fact, there is a small abuse of notation in the lemma. In specifying $\PP$, $\bar{x} = 0$ and $\chi \in [0,1/2]$, while, in specifying the second measure, $\bar{x} \in \R$ is some general value and $\chi \in [\bar{x},\bar{x} + 1/2]$. As we come to relate the two measures in the ensuing proof, it is understood that we translate $\bar{x}$ to $0$ and $\chi$ to $\chi - \bar{x} \in [0,1/2]$ as we change from considering $\PP$ to considering the other measure.
\end{rem}
\begin{proof}
Let $L:[\chi,\bar{x}+2]\to \R$ be the linear interpolation between $-\chi^2/2$ at time $\chi$ and $-(\bar{x}+2)^2/2$ at time $\bar{x}+2$. Then
\begin{eqnarray*}
&&\hskip-.25in \PH{1}{2}{(\chi,\bar{x}+2)}{\big(-R-\frac{\chi^2}{2},\Rstar-\frac{\chi^2}{2}\big)}{\big(-R-\frac{(\bar{x}+2)^2}{2},-R-\frac{(\bar{x}+2)^2}{2}\big)}{-\frac{x^2}{2}-K}{-\infty}{\Ham_t} \bigg(\sup_{x\in [\chi,\bar{x}+2]} \big(\mathcal{L}_1(x) + x^2/2\big) \geq \delta \Rstar -1\bigg)\\
&=&\PH{1}{2}{(\chi,\bar{x}+2)}{\big(-R,\Rstar \big)}{\big(-R,-R\big)}{-L(x)-\frac{x^2}{2}-K}{-\infty}{\Ham_t} \bigg(\sup_{x\in [\chi,\bar{x}+2]} \big( \mathcal{L}_1(x) + x^2/2 + L(x) \big) \geq \delta \Rstar -1 \bigg)\\
&\geq&\PH{1}{2}{(\chi,\bar{x}+2)}{\big(-R,\Rstar \big)}{\big(-R,-R\big)}{-K}{-\infty}{\Ham_t} \bigg(\sup_{x\in [\chi,\bar{x}+2]} \mathcal{L}_1(x) \geq \delta \Rstar \bigg)\\
&=& \PP\big(\mathsf{A}^{c}\big) \geq \mu.
\end{eqnarray*}
To obtain the equality between the first and second line, we apply the affine shift which sends $(\chi,-\chi^2/2)$ to $(\chi,0)$
and $\big(\bar{x},-(\bar{x}+2)^2/2\big)$ to $(\bar{x},0)$, and use the invariance of the law in question under such a shift.
The inequality between the second and third lines follows from the fact that, for $x\in [\chi,\bar{x}+2]$, we have $0\leq -L(x)-\frac{x^2}{2} \leq 1$ (recall that the interval has length at most two), as well as the monotonicity of Lemma \ref{monotonicity1} which shows that lowering $-L(x)-\frac{x^2}{2}-K$ to $-K$ has the effect of decreasing the height of $\mathcal{L}_1$. The equality between the third and fourth lines is by the definition of $\PP$ given above in (\ref{shorthandppp}) and the independence of the third line with respect to the value of $\bar{x}$; and the inequality in the fourth line is the assumption (\ref{e.twocurveprimenotpara}).
\end{proof}

\subsubsection{An outline of the proof of Proposition~\ref{p.fiveone}}\label{secoverview}

The reader may wish to glance ahead to Figure \ref{f.Rfigure} for a depiction of the curves $\mathcal{L}_1$ and $\mathcal{L}_2$, each defined on $[\chi,2]$, under the law $\PP$.
These two curves correspond to the curves $B_{n-1}$ and $B_n$ in Figure \ref{f.ABEQpic}, and the overview that we now offer develops the short explanation in the caption of that figure.

The behaviour of the pair of curves during $[\chi,2]$ will be guided by an effort to minimise the sum of potential and kinetic energies. The potential energy associated to the pair is given by the integral expression in~(\ref{e.znew}): it is a sum of two terms, corresponding to interaction  between the pair of curves $(\mathcal{L}_2,\mathcal{L}_1)$ and between the pair $(\mathcal{L}_1,-K)$. The kinetic energy describes the probabilistic cost of drastic changes for the Brownian bridge.

We will be setting the parameters so that $R$ is at least $K$, and $\Rstar$ far exceeds $R$ (in the form that $\Rstar \geq \delta^{-1} R$ for the small parameter $\delta > 0$).  Thus, while $\mathcal{L}_1$ may begin at time $\chi$ high above the level $- K$, the curve $\mathcal{L}_2$ must begin at this time extremely high above $\mathcal{L}_1$.  Following the discussion in Figure \ref{f.ABEQpic}, note that, due to its potential energetic interaction with $\mathcal{L}_1$, the curve $\mathcal{L}_2$ may be expected to drop very rapidly in an effort to fall below this other curve: suppose that $s > 0$ denotes a characteristic time scale such that, during $[\chi,\chi + s]$, $\mathcal{L}_2$ drops about halfway towards $\mathcal{L}_1$; ($s$ will depend on $t$, being very small if $t \geq \tzero$ is high). How does $\mathcal{L}_1$ behave during the short interval $[\chi,\chi + s]$? Interactions with adjacent curves push $\mathcal{L}_1$ in opposite directions: upwards towards $\mathcal{L}_2$ and downwards towards $-K$. The great elevation of $\mathcal{L}_2$ during this interval ensures that the upward pressure on $\mathcal{L}_1$ is much stronger than the downward pressure exerted by the curve $- K$ underneath. Now, whatever the value
of $s$ may be, it was worthwhile for $\mathcal{L}_2$ to expend the kinetic cost of travelling halfway towards $\mathcal{L}_1$ in duration $s$ in order to save the potential energetic cost of the interaction of these two curves; and since an upward movement by $\mathcal{L}_1$ during $[\chi,\chi + s]$ will similarly serve to alleviate this potential energetic cost, it seems reasonable to believe that it is also worthwhile that $\mathcal{L}_1$ expend the comparable kinetic cost for this upward movment as $\mathcal{L}_2$ did for its downward one. If this happens, $\mathcal{L}_1$ will
rise to level $\delta \Rstar$ shortly after time $\chi$; the event $\mathsf{A}^c$ will be typical under $\PP$, and Proposition~\ref{p.fiveone} (with $\tfrac{1}{2}\delta$ replaced by $\delta$) will follow in light of Lemma~\ref{l.para}.

The notion that  $\mathcal{L}_2$ drops about halfway towards $\mathcal{L}_1$ during $[\chi,\chi + s]$ will have a formal counterpart given by the event  $\mathsf{J}_s := \big\{ \frac{1}{2}\Rstar \leq \mathcal{L}_2 ( \chi+s) \leq (1-\e)\Rstar \big\}$ in the forthcoming rigorous argument; ($\e > 0$ will be a fixed small constant). The heuristic summary above suggests that, whatever the value of $s$, it is very unlikely under $\PP$ that both $\mathsf{A}$ and $\mathsf{J}_s$ will occur. In Lemma~\ref{l.Alem}, we will present a result to this effect.
In essence, this lemma states that: for all $t\geq 1$, $\chi\in [0,\tfrac{1}{2}]$ and $s\in (0,10^{-3})$,
\begin{equation}\label{e.oversimple}
\textrm{`` }\PP\big(  \mathsf{A} \cap \mathsf{J}_s \big)  \leq \exp\Big\{- 10^{-4}s^{-1} \Rstar^2\Big\}\textrm{ ''}.
\end{equation}
In truth, this is an oversimplification which we will not attempt to show, but it is helpful for our expository purpose to state a simple form now.

We are trying to establish that $\PP(\mathsf{A}^c) \geq \mu$. We will do so by finding a contradiction to the negation: in practice, then, think of $\mathsf{A}$ as being highly likely (with probability at least $1-\mu$).
From knowing its boundary values at $\chi$ and $2$, we certainly see that the curve $\mathcal{L}_2:[\chi,2] \to \R$ must pass through the interval of values $\big[\Rstar/2,3\Rstar/4\big]$ from the upper to the lower end. Moreover, for as long as $\mathcal{L}_2$ remains above level $\Rstar/2$, its interaction with $\mathcal{L}_1$ incurs a great cost; as we will see in (\ref{c.ler}) (which is in essence a corollary of the global energy estimate Lemma~\ref{l.z}), the downward passage of $\mathcal{L}_2$ through  $\big[\Rstar/2,3\Rstar/4\big]$ must occur before time $\chi + s_0$, where $s_0 = \Rstar \exp \big\{ - t^{1/3} \Rstar/4 \big\}$. Note that $s_0 \to 0$ as $\Rstar \to \infty$ at an exponential rate. On the other hand, (\ref{e.oversimple}) shows that, at any given time during $[\chi,\chi + s_0]$, $\mathcal{L}_2 \in \big[\Rstar/2,3\Rstar/4\big]$ has probability at most $\exp \big\{ - 10^{-4} s_0^{-1} \Rstar^2 \big\}$. How can $\mathcal{L}_2$ pass through $\big[\Rstar/2,3\Rstar/4 \big]$ in an interval of
duration $s_0$ while having probability to belong to  $\big[\Rstar/2,3\Rstar/4 \big]$  at any given time in this interval which is doubly exponentially small in $\Rstar$? The downward passage through  $\big[\Rstar/2,3\Rstar/4 \big]$  must occur at a random time of duration of order at most $\exp \big\{ - \Theta(s_0^{-1}) \Rstar^2 \big\}$. (Here $\Theta(x)$ means a function which is asymptotically bounded between $c x$ and $c^{-1} x$ for some $c>0$.) This however incurs a huge kinetic cost: for a Brownian motion to move a distance of order $\Rstar$ in such a time has probability which tends to zero with $\Rstar$ at a triply exponential rate. Clearly, however, the overall potential energy in the system is at most doubly exponential in $\Rstar$, since this is true for the expression~(\ref{e.znew}) when $\mathcal{L}_1$ and $\mathcal{L}_2$ are typical Brownian bridges under $\PfreeShort$.  Thus, this huge kinetic cost is far too great for the system to bear, so that such rapid traversal of $\big[\Rstar/2,3\Rstar/
4\big]$ by $\mathcal{L}_2$ is highly unlikely under $\PP$. We have seen that, on the event $\mathsf{A}$, this rapid traversal is both likely and unlikely. This is the desired contradiction, which shows that $\PP(\mathsf{A}^c)\geq \mu$.

\subsubsection{Structure of the proof.}
The rest of this proof proceeds by first specifying the parameters $\delta>0$, $R^0>0$, and the functions $K^{0}(R)>0$ and $\Rstar^{0}(R,K)>0$. We will also specify a small fixed quantity $\e$ and a few other parameters used during the proof. Then we will present some lemmas which enable us to flesh out the preceding sketch and arrive at the desired contradiction.
The proofs of these lemmas are delayed until Section \ref{s.prooflems}.

\subsubsection{Definitions and assumptions on parameters}\label{fiveonedefs}
Set $\nu:=1-\mu$ where $\mu$ is as in the statement of Proposition \ref{p.fiveone}. Fix any positive parameter $\e<\tfrac{1}{20}$ and then fix $\delta>0$ so that $\delta\leq \e/80$. Let $R^0>0$ be arbitrary and choose any function $K^0(R)>0$ so that, for all $R>R^0$, we have that $K^0(R)>R$ and
$$\Big(K^0(R)\Big)^2 \geq \max\left\{ 256 \log \left( \tfrac{2\sqrt{2}}{\sqrt{2}-1} \right) \, ,\,16 - 4 \log\big(\tfrac{\nu}{12}\big)\right\}.$$

For $s\geq 0$ and $t\geq 1$, define
\begin{eqnarray*}
\Gamma_s= \Gamma^t_s(\Rstar) &:=&  \exp\Big\{s e^{t^{1/3} \frac{1}{8}\Rstar}\Big\},\\
s_0=s^t_0(\Rstar) &:=&  \Rstar \exp\Big\{-t^{1/3} \frac{\Rstar}{4}\Big\}, \\
r_0=r^t_0(\Rstar) &:=& \sqrt{\tfrac{4}{\nu}s_0}\exp\Big\{ -\frac{\e \delta}{40 s_0} \Rstar^2\Big\}.
\end{eqnarray*}

Finally, let the function $\Rstar^{0}(R,K)>0$ be large enough that if $R>R^0$, $K>K^0(R)$
and $\Rstar>\Rstar^{0}(R,K)$, then the following conditions are satisfied: first,
$$\Rstar \geq \max\Big\{\delta^{-1} K,\,\delta^{-1} R,\,  \delta^{-1} \sqrt{\log 4}\,,\, \tfrac{2}{3(1+\delta)}\log(\tfrac{2}{1+\delta})\, , \, \sqrt{1600\log\big(16 \cdot \tfrac{3}{\nu}\big)} \Big\};$$
second, for all $s >  0$, $t\geq 1$ and $s_0\leq \frac{\e}{2}$,
\begin{eqnarray*}
&&\frac{4}{\nu}\, \cdot\,  \exp \Big\{ -\frac{\e \delta}{20 s_0} \Rstar^2  \Big\} < \frac{1}{4}\, ,\\
&&(\Gamma_s)^{-1/10}< \exp\Big\{ - 3\e^{-1} s e^{t^{1/3} 4 \delta \Rstar}\Big\}\, ,\\
&& \frac{3}{\nu}\,\cdot\, \exp \Big\{  8 \Rstar\, e^{t^{1/3} \big( \frac{1}{2} + \frac{3}{2}\delta \big) \Rstar} \Big\} \, \cdot\,\exp\Big\{  - s_0 e^{t^{1/3} \frac{4 }{5}\Rstar} \Big\}
   <   \exp \Big\{  -  \Rstar e^{t^{1/3} \frac{1}{2} \Rstar} \Big\} \, ,\\
&& \frac{3}{\nu}\,\cdot\,   \exp \Big\{ -\frac{\e \delta}{20 s_0} \Rstar^2  \Big\} +  \exp \Big\{  -  \Rstar e^{t^{1/3} \frac{1}{2}\Rstar} \Big\}<\frac{1}{2}\\
& & \frac{3}{\nu}\cdot  (2 \pi)^{1/2}\,\cdot\, \frac{\big((1+2\delta)\Rstar\big)^2 s_0^{-1} + 1}{(1+2\delta)\Rstar s_0^{-1/2}}\, \cdot\, \exp\bigg\{\frac{\big((1+2\delta)\Rstar\big)^2}{2s_0}\bigg\}\, , \\
&& \qquad\qquad \cdot\, \,  64 \pi^{-1/2} s_0\, r_0^{-1} \Rstar^{-1} \exp\left\{-\frac{\Rstar^2}{256 r_0}\right\}\, \cdot\, \exp \Big\{ 8 \Rstar  e^{t^{1/3} \big( \frac{1}{2} + \frac{3}{2}\delta \big) \Rstar} \Big\} <
 \frac{1}{4}\,.
\end{eqnarray*}
Notice that, as $R\to +\infty$, so too does $K^0(R)$ and $\Rstar^{0}(R,K)$.

It is not a priori obvious that this last inequality can be satisfied. However, note that
$$
\exp\bigg\{-\frac{\Rstar^2}{256 r_0}\bigg\} = \exp\Bigg\{ - \frac{\sqrt{\nu}(\Rstar)^{3/2}}{512} \, \cdot\, \exp\bigg\{t^{1/3}\Big(\tfrac{1}{8} \Rstar + \tfrac{\e \delta \Rstar}{40} \exp\big\{t^{1/3} \tfrac{1}{4}\Rstar\big\}\Big)\bigg\}\Bigg\}.
$$
This term goes to zero triple exponentially fast in $\Rstar$ whereas all other terms on the left-hand side of the last inequality grow at most double exponentially fast. Hence the inequality will be satisfied for large enough $\Rstar$.

%With our parameters satisfying the specifications made in Section \ref{fiveonedefs}, we now proceed to prove that,
%for all $R> R^0$, $K> K^{0}(R)$, $\Rstar> \Rstar^{0}(R,K)$, and for all $t\geq 1$, $\bar{x}\in \R$ and $\chi \in [\bar{x},\bar{x}+\tfrac{1}{2}]$,
%\begin{equation}\label{e.twocurveprime}
%\PH{1}{2}{(\chi,\bar{x}+2)}{\big(-R-\frac{\chi^2}{2},\Rstar-\frac{\chi^2}{2}\big)}{\big(-R-\frac{(\bar{x}+2)^2}{2},-R-\frac{(\bar{x}+2)^2}{2}\big)}{-\frac{x^2}{2}-K}{-\infty}{\Ham_t} \Big(\sup_{x\in [\chi,\bar{x}+2]} \big(\mathcal{L}_1(x) + x^2/2\big) \geq \delta \Rstar\Big) \geq \mu.
%\end{equation}
%The measure recorded here is on the curves $\mathcal{L}_1$ and $\mathcal{L}_2$ (Definition \ref{maindefHBGP}).  The bound (\ref{e.twocurveprime}) is exactly Proposition \ref{p.fiveone} after $\delta$ is replaced by $\delta/2$, and hence it suffices to prove this inequality.
\subsubsection{Events and related lemmas}

Recall that Lemma \ref{l.para} reduces the proof of Proposition \ref{p.fiveone} to verifying (\ref{e.twocurveprimenotpara}), and that we take $\bar{x} = 0$ in specifying $\PP$. To be explicit, our task is to show that, for all $R> R^0$, $K> K^{0}(R)$ and $\Rstar> \Rstar^{0}(R,K)$, and for all $t\geq 1$ and $\chi \in [0,\tfrac{1}{2}]$,
\begin{equation}\label{e.pbigacs}
\PP\big(\mathsf{A}^{c}\big) \geq \mu.
\end{equation}

%For the rest of this section we will be working with the measure $\PP = \PH{1}{2}{(\chi,\bar{x}+2)}{(-R,\Rstar)}{(-R,-R)}{f}{g}{\Ham_t}$. Let us remind the reader explicitly that under this measure the joint law of  $\mathcal{L}_j:[\chi,\bar{x} + 2] \to \R$ for $j \in \{ 1,2 \}$ with $\mathcal{L}_2(\chi) = \Rstar$, and $\mathcal{L}_2(\bar{x} + 2)   = \mathcal{L}_1(\chi) = \mathcal{L}_1(\bar{x} + 2)  = - R$ is specified by the Radon-Nikodym derivative with respect to independent Brownian bridges (with the same starting and ending heights) given explicitly by
%\begin{equation}\label{e.znew}
% Z^{-1} \exp \left\{ -  \int_{\chi}^{\bar{x} + 2} \left(  e^{t^{1/3} \big(  \mathcal{L}_1(x) + K  \big)}  +
% e^{t^{1/3} \big(  \mathcal{L}_2(x) - \mathcal{L}_1(x) \big)} \right)   \dd x \right\} \, ,
%\end{equation}
%where $Z > 0$ is a normalization constant.

We now define a collection of events which will be used throughout the rest of this proof. For $s\in [0,\tfrac{\e}{2}]$, it follows that $3\e^{-1}s\leq \tfrac{3}{2}$ and hence $\chi+3\e^{-1}s\leq 2$ (since $\chi\in [0,\tfrac{1}{2}]$). First we generalize the event $\mathsf{A}$, defining
$$
\mathsf{A}_s  = \bigg\{\sup_{u\in[\chi, \chi + 3\e^{-1}s]} \mathcal{L}_1(u) \leq \delta \Rstar\bigg\}
$$
for all such $s$. Again for such $s$, we set
\begin{eqnarray}\label{e.events}
\nonumber \mathsf{C}_s &=& \bigg\{\mathcal{L}_2(\chi + s) \geq (1 - \e) \Rstar\bigg\}\, ,\\
\nonumber \mathsf{J}_s &=& \bigg\{ \frac{1}{2}\Rstar \leq \mathcal{L}_2 ( \chi+s) \leq (1-\e)\Rstar\bigg\}\, ,\\
\nonumber \mathsf{N}_s &=& \bigg\{\inf_{u\in[\chi + s/2,\chi +s]} \mathcal{L}_2 (u) \geq \frac{1}{4}\Rstar\bigg\}\, ,\\
\nonumber \mathsf{H}_s &=& \mathsf{J_s} \cap \mathsf{B} \cap \mathsf{A}_s \cap \mathsf{N}_s\,.
\end{eqnarray}
We also define the {\em no upcrossing} event
$$
\mathsf{NU} = \bigg\{\forall\,  s_1,s_2 \in [\chi,2]: \, \textrm{ if } s_1 < s_2 \textrm{ and } \mathcal{L}_2(s_1) \geq \tfrac{\Rstar}{4} ,
\textrm{ then } \mathcal{L}_2(s_2) \leq \mathcal{L}_2(s_1) + \tfrac{\Rstar}{20} \bigg\}\, ,
$$
and also set
$$
\nonumber \mathsf{B}   = \bigg\{\inf_{u\in[\chi, 2]} \mathcal{L}_i(u) \geq -2K \textrm{ for } i=1 \textrm{ and }i=2\bigg\}\, .
$$

The following lemmas will enable us to prove (\ref{e.pbigacs}). Their proofs (except when very short) are delayed until Section \ref{s.prooflems}.

\begin{lemma}\label{l.abovelemmaone}
For all $t\geq 1$ and $\chi\in [0,\tfrac{1}{2}]$,
$$
\PP \big( \mathsf{A} \cap \mathsf{NU}^c \big) \leq \PP\big(\mathsf{NU}^c \big)\leq   16 \exp \Big\{ - \frac{1}{1600} \Rstar^2 \Big\} \, .
$$
\end{lemma}

\begin{lemma}\label{l.b}
For all $t\geq 1$ and $\chi\in [0,\tfrac{1}{2}]$,
$$
\PP \big( \mathsf{B}^c \big) \leq 4 e^{4-K^2/4}  \, .
$$
\end{lemma}

\begin{lemma}\label{l.Alem}
For all $t\geq 1$, $\chi\in [0,\tfrac{1}{2}]$ and $s\in (0,\tfrac{\e}{2})$,
$$
\PP(\mathsf{H}_s) \leq \exp\Big\{-\frac{\e \delta}{20s}\Rstar^2\Big\}.
$$
(See Figure \ref{f.Rfigure} for an illustration of $\mathsf{H}_s$ and this lemma.)
\end{lemma}

\begin{figure}
\centering\epsfig{file=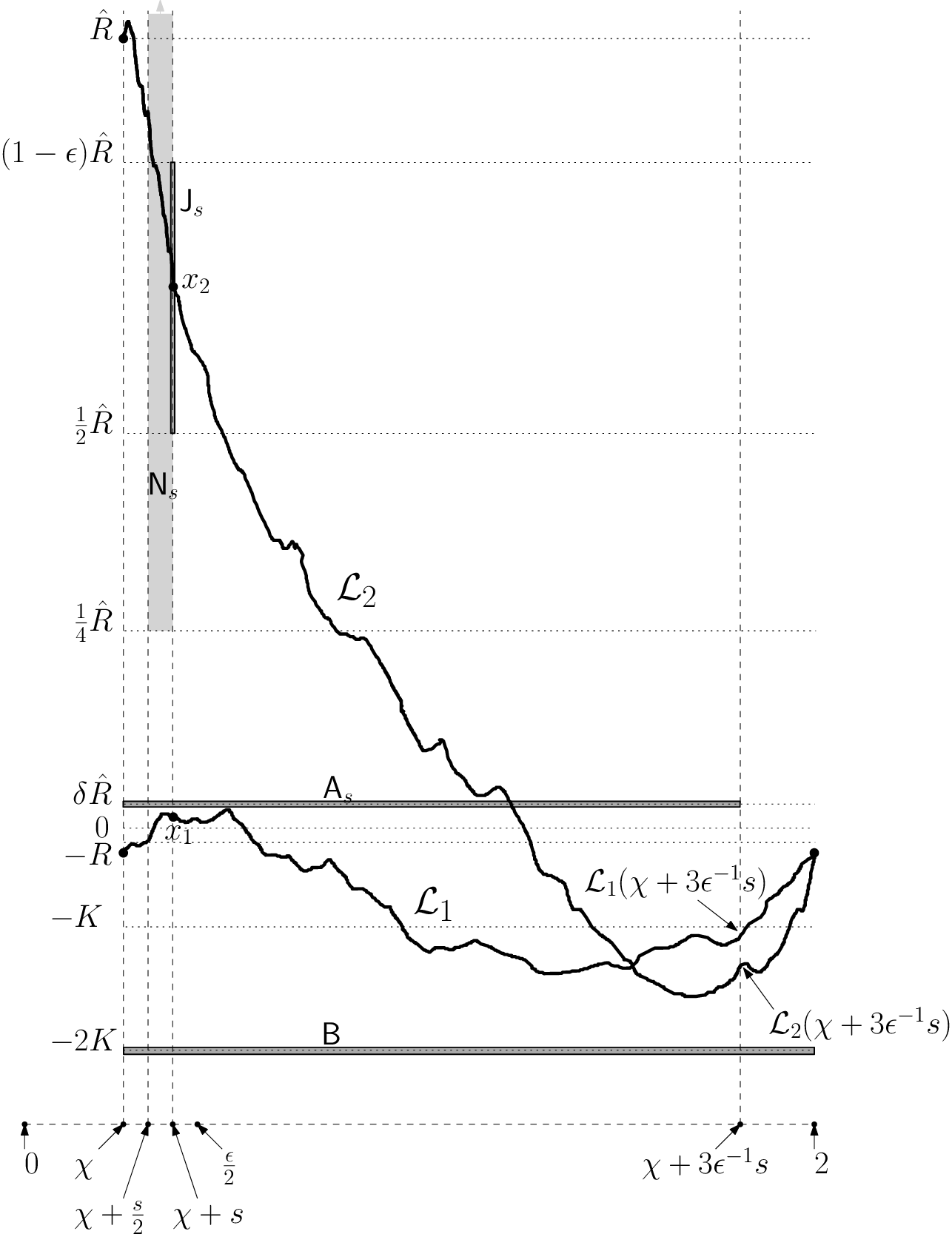, width=14cm}
\caption{Overview of the event $\mathsf{H}_s$ and the notation from the proof of Lemma~\ref{l.Alem}. The grey regions correspond to particular events. For instance, the vertical light-grey box with base between $\chi+s/2$ and $\chi+s$ and starting at height $\Rstar/4$ corresponds to the event $\mathsf{N}_s$ that $\mathcal{L}_2$ crosses this box. Likewise for the dark grey (with black outline) line segment corresponding to $\mathsf{J}_s$. The horizontal dark grey (with black outline) line at height $\delta \Rstar$ corresponds to the event $\mathsf{A}_s$ that $\mathcal{L}_1$ does not cross above this line (notice that this line segment ends at $\chi+3\e^{-1}s$), and the horizontal dark-grey (with black outline) line at height $-2K$ corresponds to the event $\mathsf{B}$ that neither $\mathcal{L}_1$ or $\mathcal{L}_2$ cross below this level.}\label{f.Rfigure}
\end{figure}

\begin{lemma}\label{l.jeqh}
For all $t\geq 1$, $\chi\in [0,\tfrac{1}{2}]$, and $s\in (0,\tfrac{\e}{2})$, the event $\mathsf{J}_s \cap \mathsf{A} \cap \mathsf{NU} \cap \mathsf{B}$ is a subset of the event $\mathsf{H}_s \cap \mathsf{A} \cap \mathsf{NU} \cap \mathsf{B}$.
\end{lemma}
\begin{proof}
That $\mathsf{J}_s \cap \mathsf{NU}$ occurs ensures that $\mathsf{N}_s$ does.  To see this, note that $\mathsf{J}_s$ implies that $ \frac{1}{2}\Rstar \leq \mathcal{L}_2( \chi+s)$ and setting $s_2=\chi+s$, $\mathsf{NU}$ implies that for all $s_1\in [\chi,\chi+s]$ (which certainly contains $[\chi+s/2,\chi+s]$),
$$\mathcal{L}_2(s_1)\geq \mathcal{L}_2(s_2)-\frac{\Rstar}{20} \geq \Rstar\left(\frac{1}{2}-\frac{1}{20}\right)\geq \Rstar \frac{1}{4},$$
as necessary for the occurrence of $\mathsf{N}_s$. Note also that the occurrence of~$\mathsf{A}$ ensures that of $\mathsf{A}_s$.
\end{proof}

\begin{lemma}\label{l.z}
For all $t\geq 1$ and $\chi\in [0,\tfrac{1}{2}]$, the normalizing constant $Z > 0$ appearing in~(\ref{e.zwnew}) satisfies
$$
 Z \geq  \exp \left\{ - 8 \Rstar   e^{  t^{1/3} \big( \frac{1}{2} + \frac{3}{2}\delta \big) \Rstar } \right\}.
$$
\end{lemma}

Note that Lemma \ref{l.Alem} is the rigorous rendering of the equation (\ref{e.oversimple}) which appeared in our overview of Proposition \ref{p.fiveone}'s proof. The events $\mathsf{B}$ and $\mathsf{N}_s$ join $\mathsf{A}$ (or, for technical reasons, $\mathsf{A}_s$) and $\mathsf{J}_s$ in the intersection defining $\mathsf{H}_s$ in order to permit a certain comparison of energy levels that will be carried out in the proof of Lemma~\ref{l.Alem}. The event $\mathsf{B}$ is typical, as Lemma~\ref{l.b} shows, while $\mathsf{N}_s$ typically occurs when $\mathsf{A}$ does, as may be inferred directly from Lemma~\ref{l.abovelemmaone}; thus, the change made to the heuristic~(\ref{e.oversimple}) to arrive at the rigorous Lemma~\ref{l.Alem} does not alter the basic meaning of these statements, which is that $\mathsf{A} \cap \mathsf{J}_s$ is a very improbable event when $s$ is small.

\subsubsection{Completing the proof of Proposition \ref{p.fiveone}}
Despite our not explicitly restating this, everything which follows holds for all $t\geq 1$ and $\chi\in [0,\tfrac{1}{2}]$.
%From Lemma \ref{l.Alem} we learn that for all $s\in (0,\tfrac{\e}{2})$
%\begin{equation}\label{e.infogained}
%\PP\big(\mathsf{A} \cap \mathsf{H}_s\big)\leq \PP\big(\mathsf{H}_s\big)  \leq \exp\Big\{-\frac{\e \delta}{40s}\Rstar^2\Big\}.
%\end{equation}

We wish to show (\ref{e.pbigacs}). With the aim of reaching a contradiction, let us assume the opposite:
\begin{equation}\label{e.contr}
\PP(\mathsf{A})>1-\mu=:\nu\,.
\end{equation}
We will first show that, under this assumption, it becomes likely that $\mathcal{L}_2$ undergoes a drastic drop in height in a very short period of time. This conclusion is reached in (\ref{e.rvertnew}). The probability of such a rapid plunge for Brownian bridge may be bounded above, and since that bound is much smaller than the normalizing constant, we reach a contradictory conclusion in (\ref{e.contradictionary}), hence proving that (\ref{e.contr}) must be false.

\medskip
Lemma \ref{l.abovelemmaone}, along with the assumed bound from Section \ref{fiveonedefs} that $\Rstar \geq \sqrt{1600 \log\big(16\cdot \tfrac{3}{\nu}\big)}$, implies that
\begin{equation}\label{e.anuminus}
\PP\big(\mathsf{A} \cap \mathsf{NU}\big) \geq \frac{2\nu}{3} \, ,
\end{equation}
%Recall the event $\mathsf{B}$ defined in (\ref{e.events}).
while Lemma \ref{l.b}, allied with the assumed bound from Section \ref{fiveonedefs} that $K^2>16-4\log\big(\tfrac{\nu}{12}\big)$, implies that
\begin{equation}\label{e.ewwnew}
\PP(\mathsf{B}) \geq 1-\frac{\nu}{3}.
\end{equation}
Hence, by combining (\ref{e.anuminus}) and (\ref{e.ewwnew}) we find
\begin{equation}\label{e.anu}
\PP\big(\mathsf{A} \cap \mathsf{NU} \cap \mathsf{B}\big) \geq \frac{\nu}{3} \, .
\end{equation}
%Recall the event $\mathsf{J}_s$ defined in (\ref{e.events}) for $s\in [0,\frac{\e}{2}]$.

For $s\in [0,\tfrac{\e}{2}]$,
\begin{equation}\label{e.jta}
\PP\Big(\mathsf{J}_s \, \big\vert \, \mathsf{A} \cap \mathsf{NU}  \cap \mathsf{B} \Big) = \frac{\PP\big(\mathsf{J}_s \cap \mathsf{A} \cap \mathsf{NU}  \cap \mathsf{B}\big)}{\PP\big(\mathsf{A} \cap \mathsf{NU}  \cap \mathsf{B}\big)} \leq \frac{3}{\nu}\,\cdot\, \PP\big(\mathsf{H}_s\cap \mathsf{A}\cap \mathsf{NU} \cap \mathsf{B}\big) \leq \frac{3}{\nu}\,\cdot\, \exp \Big\{ -\frac{\e \delta}{20 s} \Rstar^2  \Big\} \, .
\end{equation}
The equality is by Bayes' rule, the first inequality uses~(\ref{e.anu}) and Lemma \ref{l.jeqh} while the second inequality uses Lemma~\ref{l.Alem}.

Applying~(\ref{e.jta}), we find that
\begin{equation}\label{e.ltwo}
\PP \Big( \tfrac{1}{2}\Rstar \leq \mathcal{L}_2(\chi+s) \leq (1-\e) \Rstar \, \big\vert \, \mathsf{A} \cap \mathsf{NU} \cap \mathsf{B} \Big) \leq \frac{3}{\nu}\,\cdot\,  \exp \Big\{ -\tfrac{\e \delta}{20 s_0} \Rstar^2  \Big\}
\end{equation}
for all $s\in [0,s_0]$.

We will use (\ref{e.ltwo}) and using the assumptions of Section \ref{fiveonedefs} (which implies $\e<1/4$) to show that $\mathcal{L}_2$ must make a drastic drop from level $\tfrac{3}{4}\Rstar$ down to level $\tfrac{1}{2}\Rstar$ in a very small subinterval $I$ of $[\chi,\chi+s_0]$.
If $\mathcal{L}_2(\chi+s_0)\leq \tfrac{1}{2}\Rstar$ then we define an interval $I\subset [\chi,\chi + s_0]$ as follows: Consider the set of intervals $(a,b) \subset [\chi,\chi + s_0]$ such that $\mathcal{L}_2(a)=\tfrac{3}{4}\Rstar$, $\mathcal{L}_2(b) = \tfrac{1}{2}\Rstar$ and $\mathcal{L}_2(u) \in [\tfrac{1}{2}\Rstar,\tfrac{3}{4}\Rstar]$ for all $u\in (a,b)$. There is necessarily at least one such interval (since $\mathcal{L}_2$ is almost surely continuous, $\mathcal{L}_2(\chi) =\Rstar$ and $\mathcal{L}_2(\chi+s_0)\leq \tfrac{1}{2}\Rstar$) and the set of all such intervals is finite. Let $I$ be the rightmost interval (i.e. the interval corresponding to the maximal values of $a$ and $b$). On the other hand, if $\mathcal{L}_2(\chi+s_0)> \tfrac{1}{2}\Rstar$ then define $I=\emptyset$. This interval is illustrated in Figure \ref{f.interval}.

\begin{figure}
\centering\epsfig{file=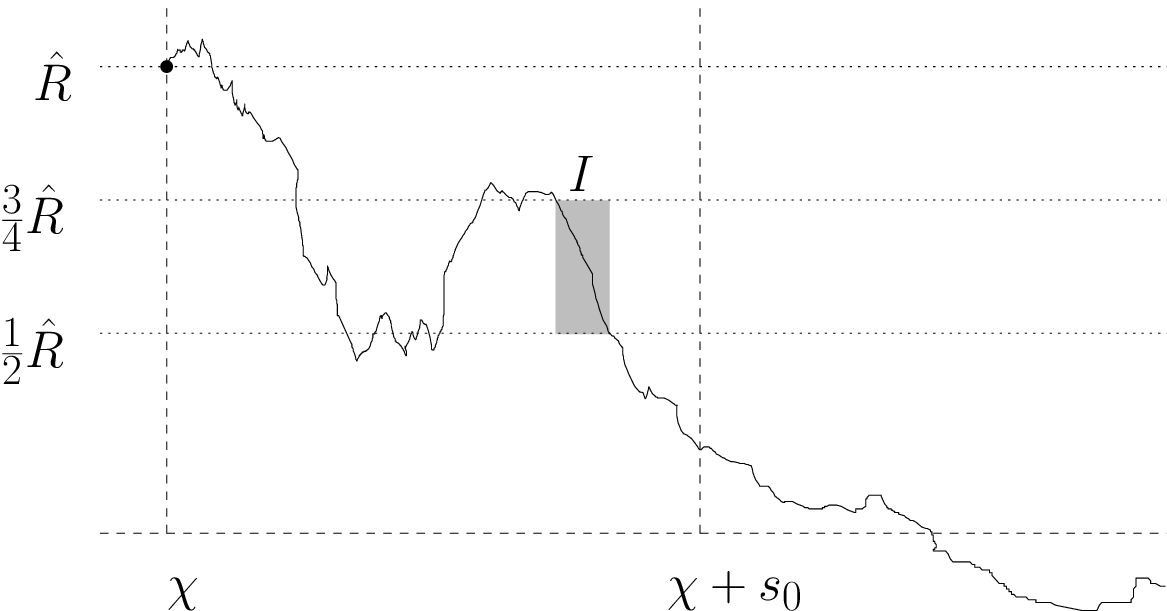, width=12cm}
\caption{Illustration of the interval $I$ during which $\mathcal{L}_2$ makes a drastic drop from level $\tfrac{3}{4}\Rstar$ down to level $\tfrac{1}{2}\Rstar$.}\label{f.interval}
\end{figure}

Letting $|I|$ be the length of this interval (with the convention that if $I=\emptyset$ then $|I|=0$), it follows that (recalling the notation $r_0$ from Section \ref{fiveonedefs}),
\begin{equation}\label{e.rvert}
\PP \Big( \big\vert I \big\vert \geq r_0 \, \big\vert \, \mathsf{A} \cap \mathsf{NU}  \cap \mathsf{B} \Big) \leq    \frac{4}{\nu}\,\cdot\, \exp \Big\{ -\frac{\e \delta}{20 s_0} \Rstar^2  \Big\}< \frac{1}{4} \, .
\end{equation}
The second inequality follows from the assumptions of Section \ref{fiveonedefs}. For the first inequality, note that if, for some $\eta > 0$,  $\PP\big( \vert I \vert \geq \eta \, \big\vert \, \mathsf{A} \cap \mathsf{NU} \cap \mathsf{B}\big) \geq \eta$, then there exists $s\in [0,s_0]$ such that
\begin{equation}\label{eq.pigeon}
\PP \Big( \mathcal{L}_2(\chi+s) \in \big[ \tfrac{1}{2}\Rstar, \tfrac{3}{4}\Rstar \big] \, \big\vert \, \mathsf{A} \cap \mathsf{NU}  \cap \mathsf{B} \Big) \geq \frac{\eta^2}{s_0}.
\end{equation}
This inequality follows from a sort of pigeon-hole principle. With probability at least $\eta$ there is an interval with $|I|\geq \eta$. A uniformly chosen $s\in [0,s_0]$ will fall into such an interval with probability at least $\eta/s_0$. Multiplying these two sets of probabilities gives the bound.

Note that ~(\ref{e.ltwo}) shows that (\ref{eq.pigeon}) is impossible if $\eta > 0$ satisfies
$\frac{\eta^2}{s_0} > \tfrac{3}{\nu}\,\cdot\,  \exp \Big\{ -\frac{\e \delta}{20 s_0} \Rstar^2  \Big\}$. The claim above follows from fixing $\eta$ such that $\frac{\eta^2}{s_0} =\tfrac{4}{\nu}\,\cdot\,  \exp \Big\{ -\frac{\e \delta}{20 s_0} \Rstar^2  \Big\}$, or in other words choosing $\eta=r_0$.
% with
%\begin{equation}\label{e.szero}
%r_0 :=  \sqrt{3 s_0} \exp \left\{ -\tfrac{\e \delta}{80 s_0} \Rstar^2 \right\}.
%\end{equation}
%
%As we have assumed $\Rstar$ to be sufficiently large, it follows that the right-hand side of (\ref{e.rvert}) is at most $1/4$.

We now claim that
\begin{equation}\label{c.ler}
\PP \Big(\mathsf{C}_{s_0} \, \Big\vert \, \mathsf{A} \cap \mathsf{NU}  \cap \mathsf{B} \Big) \leq    \exp \Big\{ -  \Rstar e^{t^{1/3} \frac{1}{2}\Rstar} \Big\}  \,.
\end{equation}
To prove this claim, observe that the occurrence of $\mathsf{C}_{s_0} \cap \mathsf{A} \cap \mathsf{NU}$ ensures that of
$$
\Big\{\mathcal{L}_2(\chi+s) \geq \big( 1 - \e - \tfrac{1}{20} \big) \Rstar \geq \tfrac{9}{10} \Rstar\, \textrm{  and  }\, \mathcal{L}_1(\chi+s) \leq \delta \Rstar \leq  \tfrac{1}{10} \Rstar\, \textrm{  for all }s \in [0,s_0]\Big\}
$$
(Note that here, we have used the assumptions $\e < 1/20$ and $\delta < 1/10$ which follow from those of Section \ref{fiveonedefs}. Note also that the absence of upcrossings on $[\chi,\chi+s_0]$ for $\mathcal{L}_2$ of any height at least $\Rstar/20$ is being used to find the lower bound on $\mathcal{L}_2(\chi+s)$.)
Thus, on the event $\mathsf{C}_{s_0} \cap \mathsf{A} \cap \mathsf{NU}$, the Boltzmann weight in the Radon-Nikodym derivative~(\ref{e.znew}) is at most
$$
\exp \Big\{ - s_0 e^{t^{1/3} \frac{4 }{5}\Rstar} \Big\} \,.
$$
We may thus conclude from the above consideration along with Lemma~\ref{l.z} and (\ref{e.anu}), that
\begin{eqnarray*}
 & & \hskip-.25in  \PP \Big( \mathsf{C}_{s_0} \, \big\vert \, \mathsf{A} \cap \mathsf{NU}  \cap \mathsf{B}
 \Big) \leq \frac{\PP\big(\mathsf{C}_{s_0}\cap \mathsf{A}\cap \mathsf{NU}\big)}{\PP\big(\mathsf{A} \cap \mathsf{NU}  \cap \mathsf{B}\big)}\\
 & \leq &  \frac{3}{\nu}\,\cdot\, \exp \Big\{  8 \Rstar e^{t^{1/3} \big( \frac{1}{2} + \frac{3}{2}\delta \big) \Rstar} \Big\} \,\cdot\,\exp \left\{  - s_0 e^{t^{1/3} \frac{4 }{5}\Rstar} \right\}
   \leq   \exp \left\{  -  \Rstar e^{t^{1/3} \frac{1}{2} \Rstar} \right\} \, .
\end{eqnarray*}
The final inequality (which yields the claim of (\ref{c.ler})) relies on the assumptions of Section \ref{fiveonedefs}.

By (\ref{e.jta}), (\ref{c.ler}), and the equality of events $\big\{\mathcal{L}_2(\chi + s_0) \geq \tfrac{1}{2}\Rstar\big\} =  \mathsf{J}_{s_0} \cup  \mathsf{C}_{s_0}$, we also find that
\begin{eqnarray}\label{e.rnsa}
\PP\big(\vert I\vert=0\, \big\vert \, \mathsf{A} \cap \mathsf{NU}  \cap \mathsf{B}\big) &=& \PP \Big( \mathcal{L}_2(\chi + s_0) > \tfrac{1}{2}\Rstar \, \big\vert \, \mathsf{A} \cap \mathsf{NU}  \cap \mathsf{B} \Big)\\
\nonumber &\leq & \frac{3}{\nu}\,\cdot\,   \exp \Big\{ -\tfrac{\e \delta}{20 s_0} \Rstar^2  \Big\} +  \exp \Big\{  -  \Rstar e^{t^{1/3} \frac{1}{2}\Rstar} \Big\}  < \frac{1}{2} \,,
\end{eqnarray}
where we have used an assumption from Section \ref{fiveonedefs} for the last inequality.

Combining (\ref{e.rvert}) and (\ref{e.rnsa}), we conclude
\begin{equation}\label{e.rvertnew}
\PP \Big(\big\{ 0 < \big\vert I \big\vert \leq  r_0  \big\} \, \big\vert \, \mathsf{A} \cap \mathsf{NU}  \cap \mathsf{B} \Big) > \frac{1}{4} \,.
 \end{equation}
This shows that, under the assumption (\ref{e.contr}), it is a typical state of affairs that, under the measure~$\PP$ conditioned on $\mathsf{A} \cap \mathsf{NU}  \cap \mathsf{B}$, there is an interval of $[\chi,\chi+s_0]$ whose duration is at most the tiny $r_0$ on which $\mathcal{L}_2$ changes in value by  $\tfrac{1}{4}\Rstar$.

We will now argue that the probability of such a transition is necessarily less than $1/4$ and hence derive a contradiction to the assumption (\ref{e.contr}).

Let us reconsider the probability on the left-hand side of (\ref{e.rvertnew}). Observe that, by~(\ref{e.anu}),
\begin{equation}\label{e.aboveineqal}
\PP \Big( 0 < \big\vert I \big\vert \leq  r_0  \, \big\vert \, \mathsf{A} \cap \mathsf{NU}  \cap \mathsf{B} \Big) \leq \frac{3}{\nu}\,\cdot\,\PP\Big( \big\{0 < \big\vert I \big\vert \leq r_0  \big\} \cap \mathsf{A} \cap \mathsf{NU}  \cap \mathsf{B} \Big).
\end{equation}
If $\big\{0 < \big\vert I \big\vert \leq r_0 \big\} \cap \mathsf{A} \cap \mathsf{NU}  \cap \mathsf{B}$ occurs then so does
$$
\mathsf{E} = \Big\{\exists \, t_1,t_2 \in [\chi,\chi+s_0]: 0 \leq t_2 - t_1 \leq r_0 \textrm{ and }\mathcal{L}_2(t_2) - \mathcal{L}_2(t_1) \leq -\tfrac{1}{4}\Rstar\Big\}\,.
$$
Thus,
\begin{equation}\label{e.aboveinequaityuyu}
(\ref{e.aboveineqal}) \leq  \frac{3}{\nu}\,\cdot\,\PP(\mathsf{E} \cap \mathsf{B}) \leq \frac{3}{\nu}\,\cdot\, Z^{-1} \,\PfreeShort\big(\mathsf{E} \cap \{\mathcal{L}_2(\chi+s_0)\geq -2K\}\big),
\end{equation}
where $\PfreeShort$ and $Z$ are defined in (\ref{e.zwnew}). Here, the latter inequality is due to the definition of $\PP$ in terms a Radon-Nikodym derivative and the fact that the Boltzmann factor $W(\mathcal{L}_1,\mathcal{L}_2)$ is always bounded above by one. We may further show that
\begin{eqnarray}\label{e.pfreeshortecap}
& & \hskip-.25in\nonumber \PfreeShort\big(\mathsf{E} \cap \{\mathcal{L}_2(\chi+s_0)\geq -2K\}\big) \\
& \leq & \int_{-\Rstar-2K}^{\infty} \frac{1}{\PP(N_{s_0}\leq y)} 16 \pi^{-1/2} s_0 r_0^{-1/2} \cdot 4\Rstar^{-1} e^{-\frac{(\Rstar/4)^2}{16r_0}} \frac{\dd \PfreeShort}{\dd y}\big(\mathcal{L}_2(\chi+s_0)
 \leq y\big)\dd y \nonumber \\
\nonumber&\leq &\frac{1}{\PP\big(N_{s_0}\leq -(1+2\delta)\Rstar\big)}
64 \pi^{-1/2} s_0 r_0^{-1/2} \Rstar^{-1} e^{-\frac{\Rstar^2}{256r_0}}
  \int_{-\Rstar-2K}^{\infty}  \frac{\dd \PfreeShort}{\dd y}\big(\mathcal{L}_2(\chi+s_0) \leq y\big)\dd y\\
&\leq &\frac{1}{\PP\big(N_{s_0} \leq -(1+2\delta)\Rstar\big)}
64 \pi^{-1/2} s_0 r_0^{-1/2} \Rstar^{-1} e^{-\frac{\Rstar^2}{256r_0}} .
\end{eqnarray}
In the above, $\tfrac{\dd \PfreeShort}{\dd y}\big(\mathcal{L}_2(\chi+s_0) \leq y\big)$ is the density of $\mathcal{L}_2(\chi+s_0)$ at $y$. The first inequality is due to Lemma \ref{c.bb} since, under $\PfreeShort$, $\mathcal{L}_2$ is Brownian bridge.
The second follows from $K \leq \delta \Rstar$.
%The lower bound on the integration on the right-hand side of the first line comes from the fact that $\big\{\mathcal{L}_2(\chi+s_0)\geq -2K\big\}$ is satisfied.
%The inequality between the first and second line follows from the fact that for $y'<y$, $\PP\big(N_{s_0} \leq y'\big)\leq \PP\big(N_{s_0}\leq y\big)$ and the fact that that $-(1+2\delta)\Rstar \leq -\Rstar-2K \leq y$ for all $y$ in the integration interval (this follows from the assumptions of Section \ref{fiveonedefs}). The inequality between the second and third line follows by the fact that the total integral in the second line is bounded by one.

Using the bound on the normalizing constant established in Lemma~\ref{l.z} and bound on the lower tail of the Gaussian in Lemma~\ref{l.normallb}, the inequalities (\ref{e.aboveinequaityuyu}) and (\ref{e.pfreeshortecap}) imply
\begin{eqnarray*}%\label{e.contradictionary}
&&\hskip-.25in\PP \Big( 0 < \big\vert I \big\vert \leq  r_0  \, \Big\vert \, \mathsf{A} \cap \mathsf{NU}  \cap \mathsf{B} \Big)\\
\nonumber &\leq& \frac{3}{\nu}\cdot  (2 \pi)^{1/2}\,\cdot\, \frac{\big((1+2\delta)\Rstar\big)^2 s_0^{-1} + 1}{(1+2\delta)\Rstar s_0^{-1/2}}\, \cdot\, \exp\bigg\{\frac{\big((1+2\delta)\Rstar\big)^2}{2s_0}\bigg\} \\
\nonumber && \qquad\qquad \cdot\, \,  64 \pi^{-1/2} s_0\, r_0^{-1} \Rstar^{-1} \exp\left\{-\frac{\Rstar^2}{256 r_0}\right\}\, \cdot\, \exp \Big\{ 8 \Rstar  e^{t^{1/3} \big( \frac{1}{2} + \frac{3}{2}\delta \big) \Rstar} \Big\}.\\
  %\\
%  & \leq & 8 (2 \pi)^{1/2} (1 + 2\delta) e^{t^{1/3}R_n^*/4} e^{\tfrac{(1 + 2\delta)^2}{2 s_0}} s_0 r_0^{-1}
%   e^{-\tfrac{(R_n^*)^2}{16 r_0}} \exp \left\{ 8 R_n^* e^{t^{1/3}\big( \tfrac{1}{2} + 3\delta/2 \big) R_n^*} \right\}
% \, .
%  2  \exp \left\{  R^*_n e^{t^{1/3} \big( \tfrac{1}{2} + 2\delta \big) R^*_n} \right\} 4 e^{2  (R_n^*)^2 } 2  (3 s_0)^{-1/2} \exp \left\{ \tfrac{\e \delta}{80 s_0} (R_n^*)^2 \right\} \\
%   & & \qquad \qquad \qquad \exp \left\{- \tfrac{1}{16} (R_n^*)^2  (3 s_0)^{-1/2} \exp \left\{ \tfrac{\e \delta}{80 s_0} (R_n^*)^2 \right\} \right\} \, .
\end{eqnarray*}
By the assumptions of Section \ref{fiveonedefs}, the latter expression is bounded above by $1/4$, so that
\begin{equation}\label{e.contradictionary}
\PP \Big( 0 < \big\vert I \big\vert \leq  r_0  \, \Big\vert \, \mathsf{A} \cap \mathsf{NU}  \cap \mathsf{B} \Big) < \frac{1}{4}.
\end{equation}
This final inequality contradicts (\ref{e.rvertnew}) and hence implies that the assumption (\ref{e.contr}) was false. As this conclusion is the desired result, we have completed the proof of Proposition \ref{p.fiveone}.

\subsection{Proof of lemmas from Section \ref{s.fiveone}}\label{s.prooflems}
We now record the proofs of the lemmas used in Section~\ref{s.fiveone}. Recall the notation of that section.

\begin{proof}[Proof of Lemma \ref{l.abovelemmaone}]
First note that
\begin{equation}\label{e.abin}
\PP\big( \mathsf{A} \cap \mathsf{NU}^c \big)\leq \PP\big(\mathsf{NU}^c \big)\leq \PP\bigg(\bigcup_{k\in\{0,\ldots,15\}} \mathsf{E}_k\bigg),
\end{equation}
where the events $\mathsf{E}_k$ are defined as follows. For each $k\in \{0,\ldots, 15\}$, set $\eta_k= \tfrac{1}{4} + \tfrac{k}{20}$ and define
$$
\alpha_k = \inf \big\{ s \in [\chi,2]:  \mathcal{L}_2(s) =  \eta_k \Rstar \big\}.
$$
Since $\mathcal{L}_2(\chi) = \Rstar$ and $\mathcal{L}_2(2) = -R$, the infimum is over a non-empty set, so that the $\alpha_k$ are well-defined. Define
$$
\mathsf{E}_k = \bigg\{\mathcal{L}_2(u) - \eta_k\Rstar \geq \frac{\Rstar}{40}\, \textrm{ for some }u\in [\alpha_k,2]\bigg\}.
$$
Clearly, $\mathsf{NU}^c \subseteq \mathsf{E}_0\cup \mathsf{E}_1\cup \cdots \cup \mathsf{E}_{15}$ so that we obtain~(\ref{e.abin}).

The following claim, along with the union bound on $\PP\big(E_0\cup E_1\cup \cdots \cup E_{15}\big)$, proves the lemma.
\begin{claim}\label{cl.ek}
For all $k\in \{0,\ldots, 15\}$,
$$
\PP\big(\mathsf{E}_k\big) \leq \exp\Big\{ - \tfrac{1}{1600} \Rstar^2\Big\}.
$$
\end{claim}
We prove this claim by utilizing the strong Gibbs property (Lemma~\ref{stronggibbslemma}) and the monotonicity results (Lemmas~\ref{monotonicity1} and \ref{monotonicity2}).
Observe that
$$
\PP\big(\mathsf{E}_k\big) = \EE\bigg[\EE\Big[\mathbf{1}_{\mathsf{E}_k} \,\big\vert\, \Fext\big(\{2\},(\alpha_k,2)\big)\Big]\bigg],
$$
and since $[\alpha_k,2]$ is a $\{2\}$-stopping domain (i.e. a stopping domain for $\mathcal{L}_2$), $\PP$-almost surely
$$
\EE\Big[\mathbf{1}_{\mathsf{E}_k} \,\big\vert\, \Fext\big(\{2\},(\alpha_k,2)\big)\Big] = \PH{2}{2}{(\alpha_k,2)}{\eta_k}{-R}{\mathcal{L}_1}{-\infty}{\Ham_t}\big(\mathsf{E}_k\big).
$$
Since $-R<0$, the monotonicity results, Lemmas~\ref{monotonicity1} and \ref{monotonicity2}, and the increasing nature of the
event~$\mathsf{E}_k$ imply that
$$
\PH{2}{2}{(\alpha_k,2)}{\eta_k}{-R}{\mathcal{L}_1}{-\infty}{\Ham_t}\big(\mathsf{E}_k\big)\leq \PH{2}{2}{(\alpha_k,2)}{\eta_k}{0}{+\infty}{-\infty}{\Ham_t}\big(\mathsf{E}_k\big).
$$
The right-hand side represents the expectation for a Brownian bridge on $[\alpha_k,2]$ starting at height $\eta_k$ and ending at height 0. The event $E_{k}$ is contained in the event that the Brownian bridge rises by at least $\tfrac{\Rstar}{40}$ from a line which interpolates its endpoint (this is because the interpolating line has negative slope). Since the duration of time $2 - \alpha_k$ is at most two, Lemma~\ref{l.bridgesup} implies that this $\tfrac{\Rstar}{40}$ rise occurs with probability at most $e^{-  \Rstar^2/1600}$. This, in turn, implies Claim~\ref{cl.ek}, so that the proof of Lemma~\ref{l.abovelemmaone} is complete.
\end{proof}

\medskip

\begin{proof}[Proof of Lemma \ref{l.b}]
By Lemma \ref{monotonicity2}, we may couple the measure $\PP=\PH{1}{2}{(\chi,2)}{(-R,\Rstar)}{(-R,-R)}{-K}{-\infty}{\Ham_t}$ on the curves $\mathcal{L}_1,\mathcal{L}_2:[\chi,2]\to \R$ with the measure
$\PP':=\PH{1}{2}{(\chi,2)}{(-\frac{5}{4}K,-\frac{6}{4}K)}{(-\frac{5}{4}K,-\frac{6}{4}K)}{-K}{-\infty}{\Ham_t}$ so that $\mathcal{L}_i(u)\geq \mathcal{L}'_i(u)$ for $i\in \{1,2\}$ and $u\in [\chi,2]$. Since the event $\mathsf{B}^c$ increases as the $\mathcal{L}_i$ decrease, this monotonicity implies that  $\PP\big( \mathsf{B}^c \big) \leq \PP'\big( \mathsf{B}^c \big)$.
Thus, it suffices to prove that
\begin{equation}\label{e.sufficesbc}
\PP'\big( \mathsf{B}^c \big)\leq 4 e^{4 - K^2/4}.
\end{equation}
To show this, recall that the measure $\PP'$ is defined relative to the law $\PfreeShort':=\Pfree{1}{2}{(\chi,2)}{(-\frac{5}{4}K,-\frac{6}{4}K)}{(-\frac{5}{4}K,-\frac{6}{4}K)}$ via the Radon-Nikodym derivative $(Z')^{-1} W'(\mathcal{L}_1,\mathcal{L}_2)$ where
\begin{equation}\label{e.znewvar}
W'(\mathcal{L}_1,\mathcal{L}_2) = \exp \bigg\{ -  \int_{\chi}^{2} \Big(  e^{t^{1/3} \big(  B_1(x) + K  \big)}  +
 e^{t^{1/3} \big(  B_2(x) - B_1(x) \big)} \Big)   \dd x \bigg\} \, ,
\end{equation}
and $Z'$ is the expectation of $W'(\mathcal{L}_1,\mathcal{L}_2)$ with respect to the measure $\PfreeShort'$.

Consider the situation when $\mathcal{L}_1$ and $\mathcal{L}_2$ do not differ from their endpoint values on the interval $[\chi,2]$ by more than $K/16$. When this event occurs, $W'(\mathcal{L}_1,\mathcal{L}_2) \geq \exp \left\{ - 4e^{-t^{1/3} K/8} \right\} \geq e^{-4}$. The $\PfreeShort'$ probability of this is, by Lemma~\ref{l.bridgesup} and
$2 - \chi \leq 2$, at least $\big( 1 - 2 e^{-  K^2/256} \big)^2\geq 1/2$, since from the assumptions of Section \ref{fiveonedefs} we have that $K^2 \geq 256 \log \Big( \tfrac{2 \sqrt{2}}{\sqrt{2} - 1} \Big)$. Thus,
$$Z' \geq e^{-4}/2.$$

Using the same sort  of reasoning as in the previous paragraph we find that
$$
\PP'\bigg( \inf_{s \in [\chi, 2]}  \mathcal{L}_i(s) \leq - 2 K   \bigg) \leq 2 e^{4 - K^2/4} \, .
$$
The proof is then completed by a simple union bound over $i\in \{1,2\}$.
\end{proof}

\medskip

We now reach the proof of the key Lemma \ref{l.Alem}. It is in this proof that we see one of the central ideas of the proof of Proposition~\ref{p.fiveone}, an idea explained in outline in Section \ref{secoverview} until (\ref{e.oversimple}): if $\mathsf{J}_s$ occurs, there will typically be huge interaction between  $\mathcal{L}_1$ and $\mathcal{L}_2$   during $[\chi,\chi+s]$; the kinetic cost of alleviating this interaction should be borne roughly evenly between the two curves, but this is not the case when $\mathsf{H}_s \subseteq \mathsf{A} \cap \mathsf{J}_s$ occurs, because, in this case, it is $\mathcal{L}_2$ which does all the travelling.

The approach we use in proving Lemma \ref{l.Alem} bears resemblance to that used in Theorem \ref{t.univonethird}(3). Herein we work purely with conditional expectations (and do not describe the argument in terms of resampling) whereas in the proof of Theorem \ref{t.univonethird}(3), for mostly pedagogical purposes, we choose to illustrate things from  the resampling perspective which underlies much of the intuition around using the $\Ham_t$-Brownian Gibbs property.

\begin{proof}[Proof of Lemma \ref{l.Alem}]
We may decompose the curves $\mathcal{L}_1$ and $\mathcal{L}_2$ on the interval $[\chi,\chi+ 3\e^{-1}s]$ as follows (see Figure \ref{f.Rfigure} for an illustration). The curve $\mathcal{L}_1$ is determined by knowledge of the values $\mathcal{L}_1(\chi),\mathcal{L}_1(\chi+s)$ and $\mathcal{L}_1(\chi+3\e^{-1}s)$ along with the curves $\mathcal{L}_{11}:[\chi,\chi+s]\to \R$ and $\mathcal{L}_{12}:[\chi+s,\chi+3\e^{-1}s]\to \R$ where $\mathcal{L}_{11}$ is the difference between $\mathcal{L}_1$ on $[\chi,\chi+s]$ and the function which linearly interpolates $\mathcal{L}_1(\chi)$ and $\mathcal{L}_1(\chi+s)$, and where $\mathcal{L}_{12}$ is the difference between $\mathcal{L}_1$ on $[\chi+s,\chi+3\e^{-1}s]$ and the function which linearly interpolates $\mathcal{L}_1(\chi+s)$ and $\mathcal{L}_1(\chi+3\e^{-1}s)$. We may likewise decompose $\mathcal{L}_2$ where the analogs of $\mathcal{L}_{11}$ and $\mathcal{L}_{12}$ are now denoted respectively by $\mathcal{L}_{21}$ and $\mathcal{L}_{22}$. Define $\mathcal{F}$ to be
the sigma-field
generated by the random variables
$\mathcal{L}_1(\chi),\mathcal{L}_1(\chi+3\e^{-1}s),\mathcal{L}_2(\chi),\mathcal{L}_2(\chi+3\e^{-1}s)$ and $\mathcal{L}_{11},\mathcal{L}_{12},\mathcal{L}_{21},\mathcal{L}_{22}$ (i.e. by everything except $\mathcal{L}_1(\chi+s)$ and $\mathcal{L}_2(\chi+s)$).

For $i=1$ and $2$, given the random variables $\mathcal{L}_i(\chi),\mathcal{L}_i(\chi+3\e^{-1}s)$ and $\mathcal{L}_{i1},\mathcal{L}_{i2}$, for $x_i\in \R$ we may define $\mathcal{L}^{x_i}_i$ as the reconstruction of $\mathcal{L}_i(\cdot)$ given $\mathcal{L}_i(\chi+s)=x_i$. Specifically, we have
\begin{equation}\label{e.reconstruct}
\mathcal{L}^{x_i}_{i}(u) =
\begin{cases}
\mathcal{L}_{i1}(u) + \mathcal{L}_i(\chi) \frac{\chi+s-u}{s} + x_i\, \frac{u-\chi}{s} & u\in [\chi,\chi+s],\\
\mathcal{L}_{i2}(u) +  x_i\, \frac{\chi+3\e^{-1} s-u}{(3\e^{-1}-1)s} + \mathcal{L}_i(\chi+3\e^{-1}s)\frac{u-(\chi+s)}{(3\e^{-1}-1)s} & u\in [\chi + s,\chi+3\e^{-1}s].
\end{cases}
\end{equation}
Define the $\mathcal{F}$-measurable set
\begin{equation}\label{e.seqn}
S=S\Big(\mathcal{L}_1(\chi),\mathcal{L}_1(\chi+3\e^{-1}s),\mathcal{L}_2(\chi),\mathcal{L}_2(\chi+3\e^{-1}s),\mathcal{L}_{11},\mathcal{L}_{12},\mathcal{L}_{21},\mathcal{L}_{22}\Big)
\end{equation}
to be the set of $(x_1,x_2)\in \R^2$ such that the event $\mathsf{H}_s$ is satisfied with $\mathcal{L}_1$ and $\mathcal{L}_2$ replaced by $\mathcal{L}^{x_1}_{1}$ and $\mathcal{L}^{x_2}_{2}$ in the definition of the constituent events of $\mathsf{H}_s$.

To explain these definitions: the curves  $\mathcal{L}_1$ and $\mathcal{L}_2$ on the interval $[\chi,\chi+ 3\e^{-1}s]$ are sampled under $\PP$, but then their values at $\chi + s$ are forgotten. After these values are forgotten, the two curves can be reconstructed as in~(\ref{e.reconstruct}) by specifying the two curves' values $x_1$ and $x_2$. The forgetful observer knows the set $S$;  the use of a given pair $(x_1,x_2)$ for reconstruction will produce a scenario which realizes $\mathsf{H}_s$ precisely when $(x_1,x_2) \in S$.

The reader may now wish to consult Figure~\ref{figjacobian} from the viewpoint of the observer after the instant of forgetting, for whom the values of $x_1$ and $x_2$ correspond to locations of the black and the white bead.
To each choice of pair $(x_1,x_2) \in S$ realizing $\mathsf{H}_s$, there is an alternative $\tau(x_1,x_2)\notin S$ which is far more probable under the law $\PP$ given the observer's information. Moreover,
the map $\tau:\R^2 \to \R^2$, which we shortly introduce, is measure-preserving and injective, so that the collection of alternatives is far more probable than is $\mathsf{H}_s$; thus, $\mathsf{H}_s$ is unlikely.

To make these notions rigorous, we use conditional expectations and write
\begin{equation}\label{e.condexppres}
\PP(\mathsf{H}_s) = \EE\bigg[ \, \EE\Big[ \mathbf{1}_{\mathsf{H}_s} \big\vert \mathcal{F}\Big]\bigg].
\end{equation}
Furthermore, $\PP$-almost surely
\begin{equation}\label{e.condexpFe}
\EE\Big[ \mathbf{1}_{\mathsf{H}_s} \big\vert \mathcal{F}\Big] = \int_{\R^2} \mathbf{1}_{(x_1,x_2)\in S} \,\cdot \, f(x_1,x_2) \, \,  \dd x_1 \dd x_2,
\end{equation}
where $f(x_1,x_2)$ is the density (with respect to Lebesgue measure) of the pair $\big(\mathcal{L}_1(\chi+s), \mathcal{L}_2(\chi+s)\big)$ given the random variables which generate $\mathcal{F}$; ($f$ is the density of the forgetful observer's conditional law for this pair). By the L\'{e}vy-Ciesielski construction of a Brownian bridge (cf. \cite[Lemma 2.8]{CH}), we can explicitly compute this density, finding it to be
\begin{eqnarray*}
f(x_1,x_2)\! &=&\! p\big(-R,x_1;s\big)\, p\big(x_1,\mathcal{L}_1(\chi+3\e^{-1}s);(3\e^{-1}-1)s\big)\,p\big(\Rstar,x_2;s\big)\, p\big(x_2,\mathcal{L}_2(\chi+3\e^{-1}s);(3\e^{-1}-1)s\big)\\
&& \cdot\, \tilde{Z}^{-1}  \exp\bigg\{ - \int_{\chi}^{\chi+3\e^{-1}s} e^{t^{1/3}\big(\mathcal{L}^{x_2}_2(u)- \mathcal{L}^{x_1}_1(u)\big)}\dd u\bigg\}\,
\exp\bigg\{ - \int_{\chi}^{\chi+3\e^{-1}s} e^{t^{1/3}\big(\mathcal{L}^{x_1}_1(u)+K\big)}\dd u\bigg\}.
\end{eqnarray*}
Here, $p\big(a,b;s\big)= (2\pi s)^{-1/2} \exp\big\{-\tfrac{(b-a)^2}{2s}\big\}$ denotes the transition probability of
Brownian motion from $a$ to $b$ in time $s$; $\mathcal{L}^{x_1}_1(\cdot)$ and $\mathcal{L}^{x_2}_2(\cdot)$ are reconstructed according to the prescription of (\ref{e.reconstruct}), and $\tilde{Z}$ is a normalizing constant necessary to make the total integral of $f(x_1,x_2)$ over $\R^2$ equal to one (its exact value will be inconsequential).

Let $\tau:\R^2\to \R^2$ denote the transformation $\big(x_1,x_2\big)\mapsto \big(x_1+2\delta \Rstar,x_2 + \delta \Rstar\big)$.
\begin{figure}
\centering\epsfig{file=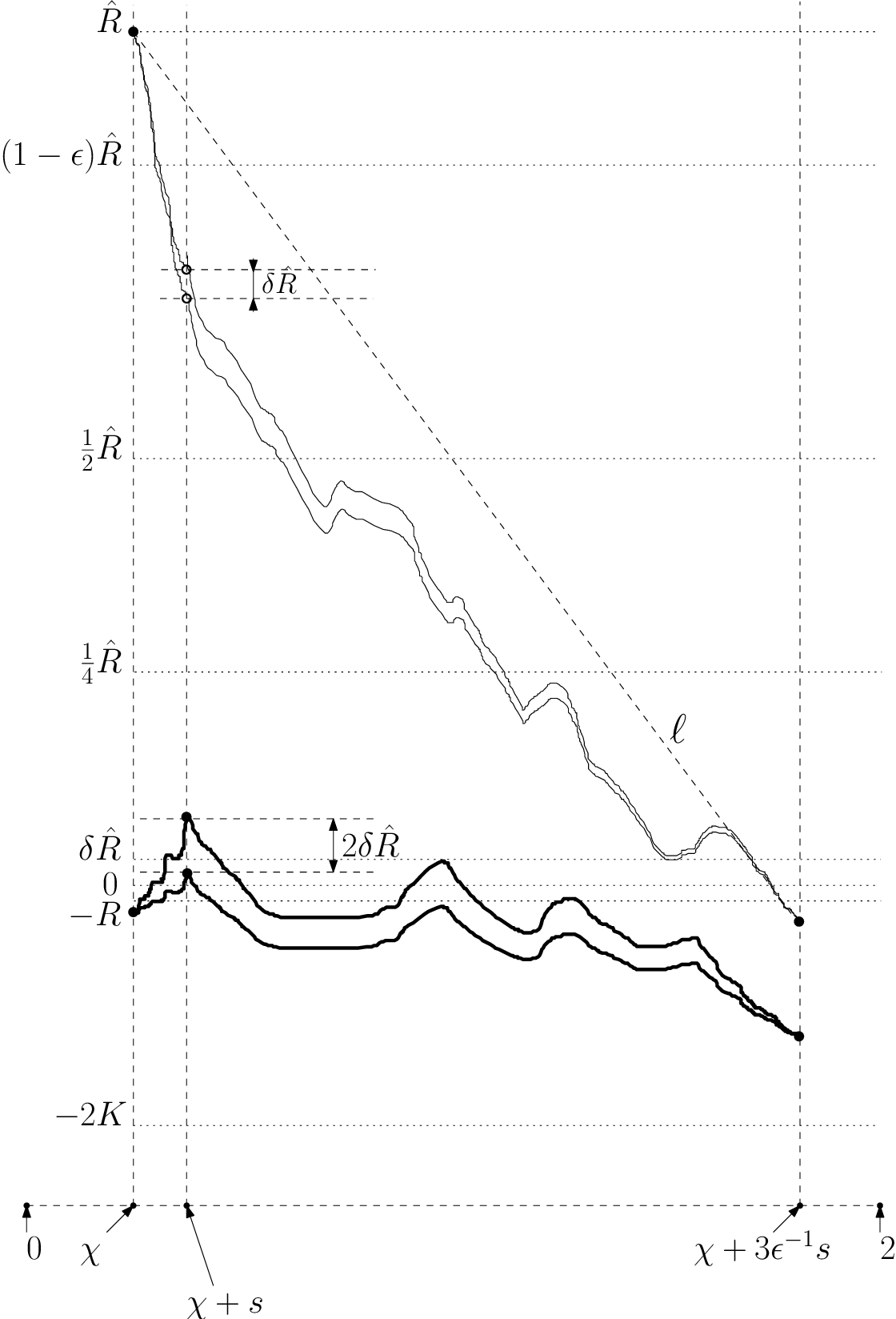, width=10cm}
\caption{The curves $\mathcal{L}_1$ and $\mathcal{L}_2$ are formed on the interval $[\chi,\chi + 3 \e^{-1}s]$ but then their values at $\chi + s$ are forgotten. The two thick curves show two possible values of $\mathcal{L}_1$ after this forgetting; the two thin curves, two possible instances of $\mathcal{L}_2$. Think of the thick curve as being hooked onto a black bead which may move upwards or downwards on the vertical line running through the point $(\chi + s,0)$. The thick curve is tethered at its endpoints and it dilates as the black bead moves. Likewise for the thin curve and the white bead.  If $\mathcal{L}_1(\chi + s) \leq \delta \Rstar$ and  $\mathcal{L}_2(\chi + s) \leq (1 - \e) \Rstar$, (a case illustrated by the lower pair of thick and thin curves), then it is energetically preferable to push both beads up in the manner indicated. Indeed, there is advantage for both kinetic and for potential energy, as we now explain. For kinetic energy: since the white bead lies at distance of order $\e \Rstar$
from the line segment denoted in the sketch by $\ell$, its raising by $\delta \Rstar$ causes a kinetic gain of $e^{O(\e \delta) \Rstar^2/s}$; raising the black bead by $2 \delta \Rstar$ entails a kinetic cost of   $e^{O(\delta^2) \Rstar^2/s}$; but since $\delta \ll \e$, the gain beats the cost. For potential energy: because the black bead is raised the further, the beads' rise causes a decrease in the interaction between
$\mathcal{L}_1$ and $\mathcal{L}_2$;  this decrease more than compensates for the increase in interaction between $\mathcal{L}_1$ and the constant boundary condition at level $-K$,  crudely speaking because the elements of the pair $(\mathcal{L}_1,\mathcal{L}_2)$ are at greater distance than are those of $(-K,\mathcal{L}_1)$.}   \label{figjacobian}
\end{figure}

\begin{claim}\label{c.U}
For $(x_1,x_2)\in S$ (as defined in (\ref{e.seqn})),
\begin{equation}\label{e.taux}
\frac{f\big(\tau(x_1,x_2)\big)}{f(x_1,x_2)} \geq \exp \Big\{ - \frac{16 \delta^2}{s} \Rstar^2 + \frac{\e \delta}{4s} \Rstar^2 \Big\}
\,\cdot\, (\Gamma_s)^{2/5}
\end{equation}
where $\Gamma_s$ is as in Section \ref{fiveonedefs}.
%where, as earlier, $J := \exp \left\{ s e^{t^{1/3}\tfrac{\Rstar}{8}} \right\}$.
\end{claim}

Before justifying this claim, let us observe how it completes the proof of Lemma \ref{l.Alem}. By the assumption (Section \ref{fiveonedefs}) that $\e\geq 80\delta$ we find that
the right-hand side of (\ref{e.taux}) is at least $\exp\Big\{\frac{\e \delta}{20 s}\Rstar^2\Big\} (\Gamma_s)^{2/5}$ from which it follows that
$$
\frac{f(x_1,x_2)}{f\big(\tau(x_1,x_2)\big)} \leq \exp\Big\{-\frac{\e \delta}{20 s}\Rstar^2\Big\} \,\cdot\,(\Gamma_s)^{-2/5};
$$
which is to say, the observer after forgetting judges $(x_1,x_2)$ to be much less probable than $\tau(x_1,x_2)$.

Now observe that
\begin{eqnarray*}
  \int_{\R^2} \mathbf{1}_{(x_1,x_2)\in S} f(x_1,x_2) \dd x_1 \dd x_2 & = & \int_{\R^2} \mathbf{1}_{(x_1,x_2)\in S}  f\big(\tau(x_1,x_2)\big) \frac{f(x_1,x_2)}{f\big(\tau(x_1,x_2)\big)} \dd x_1 \dd x_2 \\
   & \leq &  \exp \Big\{- \frac{\e \delta}{20s} \Rstar^2 \Big\}\,\cdot\,
 (\Gamma_s)^{-2/5} \int_{\R^2} \mathbf{1}_{(x_1,x_2)\in  S}  f\big(\tau(x_1,x_2)\big) \dd x_1 \dd x_2 \\
   & = &  \exp \Big\{- \frac{\e \delta}{20s} \Rstar^2 \Big\} \,\cdot\,
 (\Gamma_s)^{-2/5} \int_{\R^2} \mathbf{1}_{(x_1,x_2)\in \tau(S)}   f(x_1,x_2)  \dd x_1 \dd x_2 \\
  & \leq &  \exp \Big\{- \frac{\e \delta}{20s} \Rstar^2 \Big\}\,\cdot\,
 (\Gamma_s)^{-2/5} \int_{\R^2}  f(x_1,x_2)  \dd x_1 \dd x_2 \\
 & =&  \exp \Big\{- \frac{\e \delta}{20s} \Rstar^2 \Big\}\,\cdot\,  (\Gamma_s)^{-2/5} \, .
\end{eqnarray*}
The second equality is due to the Jacobian of the transform $\tau$ being identically one while the final one follows since $f$, being a probability density, has integral equal to one.

In light of this bound and equation (\ref{e.condexpFe}), we find that $\PP$-almost surely
$$
\EE\Big[ \mathbf{1}_{\mathsf{H}_s} \big\vert \mathcal{F}\Big] \leq  \exp \Big\{- \frac{\e \delta}{20s} \Rstar^2 \Big\} \,\cdot\, (\Gamma_s)^{-2/5};
$$
further applying (\ref{e.condexppres}), we conclude likewise that
$$
\PP(\mathsf{H}_s) \leq \exp \Big\{- \frac{\e \delta}{20s} \Rstar^2 \Big\} \,\cdot\, (\Gamma_s)^{-2/5}.
$$
Since $\Gamma_s\geq 1$ we arrive at the bound desired to prove Lemma~\ref{l.Alem}.

All that remains, therefore, is to prove Claim \ref{c.U}. We do this by viewing the ratio $\tfrac{f(\tau(x_1,x_2))}{f(x_1,x_2)}$ as a product of four terms, and providing an energetic lower bound on each one of these. The first two terms are {\it kinetic} and the second two {\it potential} in the sense that they correspond to the contributions coming from the Brownian transition probabilities and the Boltzmann weights.

\noindent {\bf The kinetic cost of $\mathcal{L}_1$}. The contribution to $\tfrac{f\big(\tau(x_1,x_2)\big)}{f(x_1,x_2)}$ is
$$
\frac{p\big(-R, x_1 + 2\delta \Rstar; s\big)}{p\big(-R, x_1; s\big)} \,\cdot\, \frac{p\big(x_1+2\delta \Rstar,\mathcal{L}_1(\chi+3\e^{-1}s);(3\e^{-1}-1)s\big)}{p\big(x_1,\mathcal{L}_1(\chi+3\e^{-1}s);(3\e^{-1}-1)s\big)}.
$$
This expression may be rewritten in the form
\begin{equation}\label{e.expressionwrite}
\exp\bigg\{ - \frac{\big(x_1+2\delta \Rstar + R\big)^2}{2s} - \frac{\big(x_1+2\delta \Rstar - \mathcal{L}_1(\chi+3 \e^{-1} s)\big)^2}{2(3\e^{-1}-1)s} + \frac{\big(x_1+R\big)^2}{2s} + \frac{\big(x_1-\mathcal{L}_1(\chi+ 3 \e^{-1} s)\big)^2}{2(3\e^{-1}-1)s}\bigg\}.
\end{equation}
When $(x_1,x_2)\in S$, we have that: $x_1 \leq \delta \Rstar$, since $\mathsf{H}_s \subseteq \mathsf{A}_s$; and
$\mathcal{L}_1(\chi+3\e^{-1}s)\geq -2K \geq -2\delta \Rstar$, since  $\mathsf{H}_s \subseteq \mathsf{B}$ and by the assumptions of Section \ref{fiveonedefs}. Also by these assumptions, we have that $3\e^{-1}\geq 60$ and $R\leq \delta \Rstar$; by disregarding the positive terms in the exponentials and using these bounds, we find that the above quantity is at least (and this is a rather brutal bound)
$$
(\ref{e.expressionwrite}) \geq \exp\left\{ -\big(8+\tfrac{25}{118}\big) \tfrac{\delta^2}{s} \Rstar^2\right\} \geq  \exp\left\{ -16 \tfrac{\delta^2}{s} \Rstar^2\right\}.
$$

\noindent{\bf The kinetic cost of $\mathcal{L}_2$}.
The contribution to $\tfrac{f\big(\tau(x_1,x_2)\big)}{f(x_1,x_2)}$ is
$$
\frac{p\big(\Rstar, x_2 + \delta \Rstar; s\big)}{p\big(\Rstar, x_2; s\big)} \,\cdot\, \frac{p\big(x_2+\delta \Rstar,\mathcal{L}_2(\chi+3\e^{-1}s\big);(3\e^{-1}-1)s)}{p\big(x_2,\mathcal{L}_2(\chi+3\e^{-1}s);(3\e^{-1}-1)s\big)}.
$$
This expression is equal to the ratio of the density at $x_2+\delta \Rstar$ and at $x_2$ of $B(\chi+s)$, where $B:[\chi,\chi+3\e^{-1}s]\to \R$ is Brownian bridge from $B(\chi) = \Rstar$ to $B(\chi+3\e^{-1}s)= \mathcal{L}_2(\chi+3\e^{-1}s)$. Let $L(u):[\chi,\chi+3\e^{-1}s]\to \R$ denote the linear interpolation of the Brownian bridge endpoints (so that $L(\chi)=\Rstar$ and $L(\chi+3\e^{-1}s)= \mathcal{L}_2(\chi+3\e^{-1}s)$). Then the law of $B(\chi+s)$ is that of a Gaussian random variable with mean $L(\chi+s)$ and variance $\kappa:= \frac{(3\e^{-1}-1)s}{3\e^{-1}}$. As such, the above contribution can be rewritten as
\begin{eqnarray}
&&\hskip-.25in\exp \bigg\{  - \frac{\big( L(\chi + s) - x_2 - \delta \Rstar \big)^2}{2\kappa} \bigg\}\,\cdot\,
     \exp \bigg\{   \frac{\big( L(\chi + s) - x_2  \big)^2}{2\kappa} \bigg\} \label{e.abovecontri} \\
&=&
    \exp \bigg\{   \frac{2 ( L(\chi + s) - x_2 )\delta \Rstar - \delta^2 \Rstar^2  }{2\kappa} \bigg\} \, . \nonumber
\end{eqnarray}
When $(x_1,x_2)\in S$, we know that $x_2\leq (1-\e)\Rstar$, since $\mathsf{H}_s \subseteq \mathsf{J}_s$.  We also know that:  $\mathcal{L}_2(\chi) = \Rstar$; and, when  $(x_1,x_2) \in S$ then, since $\mathsf{H}_s \subseteq \mathsf{B}$, we have that $\mathcal{L}_2(\chi+3\e^{-1}s) \geq -2K$, which is at most $-2\delta \Rstar \geq -\Rstar/2$ by the assumptions of Section \ref{fiveonedefs}.  It follows that $(x_1,x_2)\in S$ implies that $L(\chi+s) \geq (1-\e/2)\Rstar$.
Therefore, $L(\chi+s) - x_2 \geq \e \Rstar/2$. Since $\e>2\delta$, we find that the above expression is at least
$$
(\ref{e.abovecontri})\geq \exp\Big\{ \frac{\e \delta}{4 s} \Rstar^2\Big\}.
$$

\noindent{\bf The potential cost attached to $(\mathcal{L}_1,\mathcal{L}_2)$}.
The contribution to $\tfrac{f\big(\tau(x_1,x_2)\big)}{f(x_1,x_2)}$ is
\begin{equation}\label{e.expbiggint}
\exp\bigg\{-\int_{\chi}^{\chi+3\e^{-1}s} e^{t^{1/3}\big(\mathcal{L}^{x_2+\delta \Rstar}_2(u)-\mathcal{L}^{x_1+2\delta\Rstar}_1(u)\big)}\dd u\bigg\}\,\cdot \, \exp\bigg\{\int_{\chi}^{\chi+3\e^{-1}s} e^{t^{1/3}\big(\mathcal{L}^{x_2}_2(u)-\mathcal{L}^{x_1}_1(u)\big)}\dd u\bigg\},
\end{equation}
where we recall the notation $\mathcal{L}^{x_i}_i$ from (\ref{e.reconstruct}).
Observe that the pointwise inequality $\mathcal{L}^{x_2}_2(u)-\mathcal{L}^{x_1}_1(u) \geq\mathcal{L}^{x_2+\delta\Rstar}_2(u)-\mathcal{L}^{x_1+2\delta\Rstar}_1(u)$ for $u\in[\chi,\chi+3\e^{-1}s]\supseteq [\chi + s/2,\chi+s]$ implies that this contribution is at least
$$
(\ref{e.expbiggint}) \geq \exp\bigg\{-\int_{\chi+s/2}^{\chi+s} e^{t^{1/3}\big(\mathcal{L}^{x_2+\delta \Rstar}_2(u)-\mathcal{L}^{x_1+2\delta\Rstar}_1(u)\big)}\dd u\bigg\}\,\cdot \, \exp\bigg\{\int_{\chi+s/2}^{\chi+s} e^{t^{1/3}\big(\mathcal{L}^{x_2}_2(u)-\mathcal{L}^{x_1}_1(u)\big)}\dd u\bigg\}.
$$
When $(x_1,x_2)\in S$ we have that $\mathcal{L}^{x_2}_2(u)\geq \tfrac{\Rstar}{4}$ (since $\mathsf{H}_s \subseteq \mathsf{N}_s$) and that $\mathcal{L}^{x_1}_1(u)\leq \delta \Rstar$ (since $\mathsf{H}_s \subseteq \mathsf{A}_s$) for all $u\in [\chi+s/2,\chi+s]$. Using these facts and the assumption (Section \ref{fiveonedefs}) that $\delta <1/16$, we readily find that
$$
(\ref{e.expbiggint}) \geq (\Gamma_s)^{1/2}.
$$

\noindent{\bf The potential cost attached to $(-K,\mathcal{L}_1)$}.
The present contribution to $\tfrac{f\big(\tau(x_1,x_2)\big)}{f(x_1,x_2)}$ is
\begin{equation}\label{e.expbiggnow}
\exp\bigg\{-\int_{\chi}^{\chi+3\e^{-1}s} e^{t^{1/3}\big(\mathcal{L}^{x_1+2\delta\Rstar}_1(u)+K\big)}\dd u\bigg\}\,\cdot \, \exp\bigg\{\int_{\chi}^{\chi+3\e^{-1}s} e^{t^{1/3}\big(\mathcal{L}^{x_1}_1(u)+K\big)}\dd u\bigg\}.
\end{equation}
Since we seek a lower bound, we may disregard the second exponential term (as it is at least one). Observe that for $x_1$ such that $(x_1,x_2)\in S$, $\mathcal{L}^{x_1}_1(u)\leq \delta \Rstar$, and thus $\mathcal{L}^{x_1+2\delta\Rstar}_1(u)\leq 3\delta \Rstar$, for $u\in [\chi,\chi+3\e^{-1}s]$. Since $K\leq \delta \Rstar$ (Section \ref{fiveonedefs}), the above contribution is at least
$$
(\ref{e.expbiggnow}) \geq \exp\Big\{-3\e^{-1}s e^{t^{1/3} 4\delta \Rstar}\Big\},
$$
and further using the assumptions in Section \ref{fiveonedefs}, this quantity is bounded by
$$
(\ref{e.expbiggnow})  \geq (\Gamma_s)^{-1/10}.
$$
\medskip
By combining these four contributions, whose product equals $\tfrac{f(\tau(x_1,x_2))}{f(x_1,x_2)}$, we arrive at the desired lower bound and complete the proof of Claim \ref{c.U}. As explained earlier, having shown Claim \ref{c.U}, the proof of Lemma \ref{l.Alem} is readily completed.
\end{proof}
\begin{rem}
 As we discussed in Figure~\ref{figjacobian}, the alternative $\tau(x_1,x_2)$ is preferable to $(x_1,x_2) \in S$ both for potential and for kinetic energy. Since the gains cover the costs for each type of energy, the proof of Lemma~\ref{l.Alem} is carried out without the question of whether it is kinetic or potential energy whose levels are the greater being addressed.
\end{rem}

\begin{proof}[Proof of Lemma \ref{l.z}]
Recall the definition of the normalizing constant $Z$ and Boltzmann weight $W(\mathcal{L}_1,\mathcal{L}_2)$ from (\ref{e.zwnew}). Moreover, recall that
$$
Z = \PfreeExpShort\big[W(B_1,B_2)\big]
$$
where the measure $\PfreeShort$ (whose expectation is $Z$ and which is also defined in (\ref{e.zwnew})) is that of a pair of Brownian bridges $B_1,B_2:[\chi,2]\to \R$ with $B_1(\chi)=B_1(2)=B_2(2)=-R$ and $B_2(\chi)=\Rstar$.

For $h\in (0,1/2)$ (to be specified a little later), let us define events
$$
\mathsf{E}^h_1 = \Big\{B_2(\chi+h)\leq -R\Big\},\quad \mathsf{E}^h_2 = \Big\{\sup_{u\in[\chi, \chi + h]} B_2(u) \leq (1 + \delta) \Rstar\Big\},\quad \mathsf{E}^h_3 = \Big\{\sup_{u\in[\chi+h ,  2]} B_2(u) \leq  - R + \delta \Rstar\Big\}$$
and their intersection
$$\mathsf{E}^h = \mathsf{E}^h_1\cap \mathsf{E}^h_2\cap \mathsf{E}^h_3.$$

We claim the following inequality holds:
\begin{equation}\label{c.mathsfeh}
\PfreeShort(\mathsf{E}^h) \geq  \frac{1}{4} (\pi h)^{-1/2}\, \frac{\Rstar + R}{ 2(\Rstar + R)^2 h^{-1} +1}\, e^{- (\Rstar + R)^2 h^{-1}} \,.
\end{equation}

This claim follows from straightforward considerations. Under $\PfreeShort$, $B_2(\chi+h)$ is distributed as a normal random variable of mean $\tfrac{( 2 - \chi - h)\Rstar - h R}{2 - \chi} \leq \Rstar$ and variance $\tfrac{h( 2 - \chi - h)}{ 2 - \chi} \geq \tfrac{h}{2}$ (this lower bound follows since $2 \geq 2 - \chi \geq 3/2$ and $h \leq 1/2$). Thus, using the lower tail bound provided by Lemma~\ref{l.normallb}, we find that
$$\PfreeShort(\mathsf{E}^h_1) \geq (\pi h)^{-1/2} \,\frac{\Rstar + R}{ 2(\Rstar + R)^2 h^{-1} +1}\, e^{- (\Rstar + R)^2 h^{-1}}\,.$$

By Lemma~\ref{l.bridgesup} (and the monotone coupling for Brownian bridges with different endpoints -- a special case of Lemma~\ref{monotonicity2}),
\begin{eqnarray*}
 \PfreeShort\Big((\mathsf{E}^h_2)^c \big\vert \mathsf{E}^h_1\Big) &=&  \PfreeShort \bigg(  \sup_{s\in[\chi, \chi + h]} B_2(s) \geq (1 + \delta) \Rstar  \,\Big\vert \, B_2(\chi + h) \leq - R  \bigg)
 \leq e^{-2 h^{-1} \delta^2 \Rstar^2}\leq \frac{1}{2},\\
 \PfreeShort\Big((\mathsf{E}^h_3)^c \big\vert \mathsf{E}^h_1\Big) &=&  \PfreeShort \bigg(  \sup_{s\in[\chi + h, 2]} B_2(s) \geq - R + \delta \Rstar  \,  \Big\vert \, B_2(\chi + h) \leq - R  \bigg)
 \leq e^{-\delta^2\Rstar^2}\leq \frac{1}{2} \, .
\end{eqnarray*}
The bounds by $\tfrac{1}{2}$ follow from the assumption of Section \ref{fiveonedefs} that $\Rstar \geq \delta^{-1} \sqrt{\log 2}$ as well as $h\in(0,1/2)$. Thus, we have
\begin{eqnarray*}
\PfreeShort\Big(\mathsf{E}^h\Big) &=& \PfreeShort\Big(\mathsf{E}^h_1\cap \mathsf{E}^h_2\cap \mathsf{E}^h_3\Big) = \PfreeShort\Big(\mathsf{E}^h_1\Big)\,\cdot\, \PfreeShort\Big(\mathsf{E}^h_2\cap \mathsf{E}^h_3 \,\big\vert\, \mathsf{E}^h_1\Big)\\
 &\geq& \PfreeShort\Big(\mathsf{E}^h_1\Big) \,\cdot\,  \PfreeShort\Big(\mathsf{E}^h_2 \,\big\vert\, \mathsf{E}^h_1\Big)\,\cdot\,\PfreeShort\Big(\mathsf{E}^h_3 \,\big\vert\, \mathsf{E}^h_1\Big)\\
 & \geq& \frac{1}{4} (\pi h)^{-1/2}\, \frac{\Rstar + R}{ 2(\Rstar + R)^2 h^{-1} +1} \,e^{- (\Rstar + R)^2 h^{-1}} \, .
\end{eqnarray*}
The bound $\PfreeShort\Big(\mathsf{E}^h_2\cap \mathsf{E}^h_3 \,\big\vert\, \mathsf{E}^h_1\Big) \geq  \PfreeShort\Big(\mathsf{E}^h_2 \,\big\vert \, \mathsf{E}^h_1\Big)\,\cdot\,\PfreeShort\Big(\mathsf{E}^h_3 \,\big\vert\, \mathsf{E}^h_1\Big)$ is a consequence of the positive associativity of the events $\mathsf{E}^h_2$ and $\mathsf{E}^h_3$ conditioned on $\mathsf{E}^h_1$. This completes the proof of \ref{c.mathsfeh}.

We now define an event
$$\mathsf{E}= \bigg\{\sup_{u\in[\chi, 2]} \big\vert B_1(u) + R \big\vert \leq \delta \Rstar\bigg\}.$$
By Lemma~\ref{l.bridgesup} and the assumption of Section \ref{fiveonedefs} that $\Rstar \geq \delta^{-1} \sqrt{\log 4}$,
$$\PfreeShort(\mathsf{E}) \geq 1 - 2e^{- \delta^2 \Rstar^2} \geq \frac{1}{2}.$$

On the event $\mathsf{E}^h \cap \mathsf{E}$, we may bound the Boltzmann weight
$$
W(B_1,B_2) \geq \exp \Big\{ - h e^{t^{1/3} \big( (1 + \delta) \Rstar - (-R - \delta \Rstar) \big)} \, - 2 e^{t^{1/3} 2\delta \Rstar}  \, - 2 e^{t^{1/3} (- R + \delta \Rstar+K)} \Big\} \, .
$$
The first term in the exponential arises from the interaction of $(B_1,B_2)$ on the interval $u\in [\chi,\chi+h]$ and the fact that on the event $\mathsf{E}^h \cap \mathsf{E}$, $B_1(u) \geq -R-\delta \Rstar$ and $B_2(u)\leq (1+\delta)\Rstar$.
It is the interaction of $(B_1,B_2)$ on $u\in [\chi+h,2]$ that
is responsible for the second term: indeed, this interval of values of $u$ has length at most two, and such $u$ satisfy $B_1(u) \geq - \delta \Rstar - R$  when $\mathsf{E}$ occurs and  $B_2(u) \leq  \delta \Rstar - R$ when $\mathsf{E}^h_3$ does. The third term in the exponential arises from the interaction of $(-K,B_1)$ on the interval $u\in [\chi,2]$ as well as the fact that $B_1(u)\leq -R+\delta \Rstar$.% and $K\geq R$.

By the assumptions of Section \ref{fiveonedefs}, we may bound $(1+\delta)\Rstar -(-R-\delta \Rstar) \leq (1+3\delta)\Rstar$ and  $- R + \delta \Rstar+K \geq 2\delta \Rstar$, while the events $\mathsf{E}^h$ and $\mathsf{E}$ are independent under $\PfreeShort$; using these inputs, we find that the normalizing constant $Z$ satisfies the lower bound
$$
 Z \geq  \frac{1}{8} (\pi h)^{-1/2} \frac{\Rstar + R}{ 2(\Rstar + R)^2h^{-1} +1} \exp\Big\{- (1 + \delta)^2 \Rstar^2 h^{-1} - h e^{t^{1/3} (1 + 3\delta) \Rstar }  - 4 e^{t^{1/3} 2\delta \Rstar} \Big\} \,.
$$
We obtain different lower bounds on $Z$ as $h$ varies. With the aim of obtaining a bound among this collection which is roughly optimal, we choose $h\in (0,1/2)$ so that $(1 + \delta)^2 \Rstar^2 h^{-1} = h e^{t^{1/3} (1 + 3\delta) \Rstar }$. This is achieved by
$$
 h = (1 + \delta) \Rstar  e^{-t^{1/3} (\frac{1}{2} + \frac{3}{2}\delta) \Rstar} \, .
$$
Plugging in this value of $h$ and making use of the assumptions of Section \ref{fiveonedefs} on $R, \Rstar$ and $\delta$, we find that $Z$ satisfies the lower bound we sought in order to complete the proof of Lemma \ref{l.z}.%
\end{proof}

\section{Appendix}\label{s.apend}

\subsection{Proof of Proposition \ref{ACQprop}(2)}\label{s.tail}

This proof follows the approach of the proof of \cite[Corollary 14]{CQ} (though that work deals with different initial data) and hence we only review the main points, and do not provide careful estimates of integrals (which can be performed quite straightforwardly). We also do not review the theory of Fredholm determinants or Hilbert-Schmidt operators (see \cite{ACQ,CQ} for background).

From Theorem \ref{t.acqtheorem} and the fact that $\det(I)=1$, we have
\begin{equation}\label{e.khts}
\PP\big(\HKPZFPlinetnw{t}(0)\geq s\big) \geq -\frac{1}{2\pi i}\int_{\subset} \frac{d\mu}{\mu} e^{-\mu} \Big[\det\big(I- K_{\mu}\big)_{L^2(2^{1/3}s,\infty)}-\det(I)_{L^2(2^{1/3}s,\infty)}\Big].
\end{equation}

Let $U$ act on functions by means of $(Uf)(x) = (x^4+1)^{-1/2} f(x)$. Then, for   $B_{\mu}=U^{-1}K_{\mu} U$, we have that $\det\big(I- K_{\mu}\big)_{L^2(2^{1/3}s,\infty)} = \det\big(I- B_{\mu}\big)_{L^2(2^{1/3}s,\infty)}$.
For $B=B_1B_2$ with $B_1$ and $B_2$ Hilbert-Schmidt,
\begin{equation}\label{e.pluginto}
\big\vert \det(I+B) - \det(I)\big\vert \leq ||B_1||_2\, ||B_2||_2\, e^{||B_1||_2\, ||B_2||_2 +1},
\end{equation}
where the determinants and the norms are with respect to $L^2(2^{1/3}s,\infty)$. We can factor $B_{\mu}$ into Hilbert-Schmidt operators $B_1:L^2(\R)\to L^2(2^{1/3}s,\infty)$ and $B_2:L^2(2^{1/3}s,\infty)\to L^2(\R)$ as
\begin{eqnarray*}
B_1(x,r) &=& \Ai(x+r) \, (x^4+1)^{-1/2}\, (r^4+1)^{-1/2},\\
B_2(r,y) &=& \frac{\mu}{\mu - e^{-2^{-1/3}t^{1/3} r}}\, \Ai(y+r)\, (y^4+1)^{1/2}\, (r^4+1)^{1/2}\,.
\end{eqnarray*}
By using the standard bounds on the tail behavior of the Airy function, we find the following bounds on the Hilbert-Schmidt norms:
$$
||B_1||_2 \leq c_1, \qquad ||B_2||_2 \leq c_2 |\mu| s^b \big(e^{-c_3 t^{1/3} s} + e^{-c_4 s^{3/2}}\big),
$$
for some suitable positive constants $b$,$c_1$,$c_2$,$c_3$ and $c_4$.
Plugging these estimates into (\ref{e.pluginto}), and taking the resulting integral (\ref{e.khts}) in $\mu$, we arrive at the desired bound. \qed

\subsection{Monotone couplings and the $\Ham$-Brownian Gibbs property}\label{monotoneproofs}

We now prove Lemma \ref{monotonicity1}. The proof of Lemma \ref{monotonicity2} follows from essentially the same argument and we omit it.

\begin{proof}[Proof of Lemma \ref{monotonicity1}]
%The approach used in proving these two lemmas is identical so we will demonstrate it in the case of Lemma \ref{monotonicity} and briefly explain how it applies equally well to Lemma \ref{lemmonotonetwo}.
We follow the approach of \cite{CH} and find that we can generalize the non-intersection condition present in that paper to convex Hamiltonians $\Ham$. In order to demonstrate this coupling for the conditioned Brownian line ensembles $\mathcal{Q}^{1}$ and $\mathcal{Q}^2$, we exhibit an analogous coupling of conditioned simple symmetric random walk trajectories and then appeal to the fact that these laws converge to the conditioned Brownian ensembles. The result for the random walks follows from a coupling of Markov chains which converges to the constrained path measures and hence demonstrates the existence of the limiting coupling. The convexity of $\Ham$ appears in an essential way as it ensures that, at moments when the two path measures may become out of order, the Markov chain sampling of the path space may be coupled so as to prevent this from occurring.

Consider $\vec{x}^n=(x_1^n\ldots, x_k^n)$ and $\vec{y}^n=(y_1^n\ldots, y_k^n)$ such that $x_{j}^n, y_{j}^n \in n^{-1}(2\Z)$ for all $1\leq j\leq k$ and $n>1$ and such that $x_{j}^n \to x_j$ as well as $y_j^n\to y_j$. Likewise, let $(a^n,b^n)\to (a,b)$ in such a way that $b-a \in n^{-2} 2\Z$. Write $\Wfree{n}{k}{(a^n,b^n)}{\vec{x}^n}{\vec{y}^n}$ for the law of $k$ independent random walks
$X^n_j:[a^n,b^n] \to n^{-1} \Z$, $1 \leq j \leq k$, that satisfy $X^n_j(a^n) = x_j^n$ and $X^n_j(b^n) = y_j^n$ and take steps of size $\pm n^{-1}$ in discrete time increments of size $n^{-2}$. To simplify things we will actually let $X^n_j$ linearly interpolate between its values on $n^{-2} \Z$.
Define $\WH{n}{k}{(a^n,b^n)}{\vec{x}^n}{\vec{y}^n}{f}{g}{H}$ to be the law on the set of discrete paths $X^n_j$ specified above; this law has a Radon-Nikodym derivative with respect to $\Wfree{n}{k}{(a^n,b^n)}{\vec{x}^n}{\vec{y}^n}$ given by $\bolt{1}{k}{(a^n,b^n)}{\vec{x}^n}{\vec{y}^n}{f}{g}{H}$ (Definition \ref{maindefHBGP}) up to normalization. For shorthand now define $\WHShort{n}{i}{H}= \WH{n}{k}{(a^n,b^n)}{\vec{x}^n}{\vec{y}^n}{f^i}{g^i}{H}$ for $i=1,2$ and write $\walk{n}{i}{j}$ for the random discrete paths specified by these measures.

We will now demonstrate that if $(f^1,g^1)\geq (f^2,g^2)$ then there is a coupling measure $\WHShortCouple{n}{H}$ on which both $\walkfull{n}{1}=\{\walk{n}{1}{j}\}_{j=1}^{k}$ and $\walkfull{n}{2}=\{\walk{n}{2}{j}\}_{j=1}^{k}$ are defined with respective marginal distributions $\WHShort{n}{1}{H}$ and $\WHShort{n}{2}{H}$, and under which almost surely $\walk{n}{1}{j}(s)\geq \walk{n}{2}{j}(s)$ for all $1\leq j\leq k$ and all $s\in (a,b)$. As $n \to \infty$, the invariance principle for Brownian bridge as well as the continuity of $\Ham$ shows that this coupling measure converges to the desired coupling of Brownian bridges necessary to complete our proof.

Thus it remains to demonstrate the existence of $\WHShortCouple{n}{H}$. In order to do this, we introduce a continuous-time Markov chain dynamic on the random walk bridges $\walkfull{n}{1}$ and $\walkfull{n}{2}$. Let us write the collection of random walk bridges at time $t$ as $(\walkfull{n}{1})_t$ and $(\walkfull{n}{2})_t$. The time-$0$ configurations of $(\walkfull{n}{1})_0$ and $(\walkfull{n}{2})_t$ are both chosen to be the lowest possible trajectory of random walk bridges having the given endpoints. Hence it is clear that the time-$0$ values of these trajectories are ordered $(\walk{n}{1}{j}(s))_0 \geq (\walk{n}{2}{j}(s))_0$ for all  $s\in [a^n,b^n]$ and $1\leq j\leq k$ (in fact equality holds). The dynamics (evolving in continuous time $t$) of the Markov chain are as follows: for each $s\in n^{-2}\Z \cap [a^n,b^n]$, each $j\in \{1,\ldots, k\}$ and each $m\in \{+,-\}$ ($+$ stands for up flip and $-$ for down flip), there are independent exponential clocks which ring at rate one. When the clock labeled
$(s,j,m)$ rings, one attempts to
change $\walk{n}{i}{j}(s)$ to $\walk{n}{i}{j}(s)+2m n^{-1}$ for both $i=1,2$. Call the proposed updated walk $\walkfullupdated{n}{i}$ and extend this walk off $n^{-2}\Z$ via linear interpolation. This proposed change is accepted according to a Metropolis rule. Recalling Definition~\ref{maindefHBGP}, let
\begin{equation*}
R^{i} = \frac{\bolt{1}{k}{(a^n,b^n)}{\vec{x}^n}{\vec{y}^n}{f^i}{g^i}{H}(\walkfullupdated{n}{i})}{ \bolt{1}{k}{(a^n,b^n)}{\vec{x}^n}{\vec{y}^n}{f^i}{g^i}{H}(\walkfull{n}{i})}.
\end{equation*}
Then the proposed change is made if $R^{i}\geq U$ where $U=U^{n,s,m}$ are independent random variables distributed according to the uniform distribution on $[0,1]$. This is to say that an update which increases the Boltzmann weight is always accepted, whereas an update which decreases the Boltzmann weight is accepted with probability given by the ratio of the new weight to the old. The purpose of the collection of $U$-random variables is to couple the acceptance events between the $i=1$ and $i=2$ systems.

The dynamics having now been specified, we must show that it preserves the desired ordering. One quickly sees that there are only two types of situation in which a proposed move could possibly lead to a violation of the ordering. One case is when the clock $(s,j,m=+)$ rings and for some $d\in n^{-2}\Z$, $\walk{n}{i}{j}(s-n^{-1})-n^{-2}=\walk{n}{i}{j}(s)=\walk{n}{i}{j}(s+n^{-1})-n^{-2}=d$. If this happens, then line $j$ in systems $1$ and $2$ shares the same height at time $s$ and, for each of these lines, there is a $\diagdown\diagup$ subpath about position $s$. The other case is when the clock $(s,j,m=-)$ rings and for some $d\in n^{-2}\Z$, $\walk{n}{i}{j}(s-n^{-1})+n^{-2}=\walk{n}{i}{j}(s)=\walk{n}{i}{j}(s+n^{-1})+n^{-2}=d$. Here, there is a $\diagup\diagdown$ subpath for the $j$-th line at position $s$, coinciding for each $i=1,2$. We will focus just on the first case, as the second case is dealt with by an analogous argument.

We argue now that, in the first case, $R^{1}\geq R^{2}$ for this proposed move, which implies that, if the move is accepted for $i=2$, then it is also accepted for $i=1$ (since both are compared to the same $U$); thus,  monotonicity is preserved. (In the second case, $R^{1}\leq R^{2}$, and this likewise preserves the monotonicity.) To verify this claim, observe that from Definition \ref{maindefHBGP}
\begin{equation}\label{e.ri}
R^{i} = \frac{\exp\bigg\{ -\int_{u-n^{-1}}^{u+n^{-1}} \Big( \Ham(\walkupdated{n}{i}{j}(u')-\walk{n}{i}{j-1}(u')) - \Ham(\walk{n}{i}{j+1}(u')-\walkupdated{n}{i}{j}(u'))\Big)\dd u' \bigg\}}{\exp\bigg\{ -\int_{u-n^{-1}}^{u+n^{-1}} \Big( \Ham(\walk{n}{i}{j}(u')-\walk{n}{i}{j-1}(u')) - \Ham(\walk{n}{i}{j+1}(u')-\walk{n}{i}{j}(u'))\Big)\dd u' \bigg\}},
\end{equation}
where the convention is that $\walk{n}{i}{0}(\cdot )=f^i(\cdot)$ and $\walk{n}{i}{k+1}(\cdot)=g^{i}(\cdot)$. The reason that $\walk{n}{i}{j-1}(\cdot)$ and $\walk{n}{i}{j+1}(\cdot)$ appear in the numerator is that $\walkupdated{n}{i}{\ell}(\cdot)=\walk{n}{i}{\ell}(\cdot)$ for all $\ell\neq j$.

Due to the fact that $\walk{n}{1}{j-1}(u)\geq \walk{n}{2}{j-1}(u)$ and likewise $\walk{n}{1}{j+1}(u)\geq \walk{n}{2}{j+1}(u)$, by rearranging terms, we may reduce verifying that $R^1/R^2 \geq 1$ to a double application of the following claim: for any $\delta>0$ and $d^1\geq d^2$,
\begin{equation}\label{r.convexd}
\Ham(d^1-\delta)-\Ham(d^1)  \geq \Ham(d^2-\delta)-\Ham(d^2).
\end{equation}
This, however, follows immediately from the convexity of $\Ham$. This completes the argument that the dynamics in $t$ preserves the monotonicity.

By the above argument, we now know that the dynamics preserves ordering, so that, for all $t\geq 0$, $\big(\walk{n}{1}{j}(s)\big)_t \geq \big(\walk{n}{2}{j}(s)\big)_t$ for all $s\in [a^n,b^n]$ and $1\leq j\leq k$.  The second key fact is that the marginal distributions of these time dynamics converge to the invariant measure for this Markov chain, which is given by the measures we have denoted by $\WHShort{n}{1}{H}$ and $\WHShort{n}{2}{H}$. This fact follows since we have a finite state Markov chain which is irreducible with an obvious invariant measure. Combining these two facts implies the existence of the coupling measure $\WHShortCouple{n}{H}$ as desired to complete the proof.
\end{proof}

%The proof of Lemma \ref{lemmonotonetwo} also relies on the same Markov chain which preserves ordering. The ordering of the initial data is a direct consequence of the ordering of the starting and ending points, and hence the same argument carries over.

\subsection{General theory of killing and conditioning}\label{genthsec}

In Section \ref{QTLHgibbs} we state and prove that the diffusion $X^N(s)$, regarded as a line ensemble, has the $\Ham$-Brownian Gibbs property. In order to do this, we appeal to some general theory which
is based on Section VIII.3 of \cite{RevuzYor} and which we now recount. To distinguish the general probability measures and expectations from those of Definition~\ref{maindefHBGP}, we presently use subscripts rather than superscripts. We use $t$ here to denote time for a diffusion (hence a different usage than in the rest of the paper). We also use a slightly different sigma-field notation here from the rest of the paper.

Given a positive Borel function $V:\R^N\to [0,\infty)$ and a infinitesimal generator $L$ for a diffusion $\big\{ X_t: t\geq 0 \big\}$, one defines a new sub-Markov process $X^{(V)}_{t}$ via its semi-group $P_t^{(V)}$ or transition kernel $P_t^{(V)}(x,dy)$, given by
\begin{eqnarray*}
P_t^{(V)} f (x) &=& \E^L_{0,x}\left[f(X_t)\exp\left\{ - \int_0^{t} V(X_s) ds\right\}\right],\\
P_t^{(V)}(x,dy) &=& \E^L_{0,x}\left[\exp\left\{ - \int_0^{t} V(X_s) ds\right\} \Big\vert X_t=y \right] P_t(x,dy).
\end{eqnarray*}
Here, $f$ must be in the domain of $L$, $\E^L_{0,x}$ is the expectation of the process $\{X_s\}_{0\leq s\leq t}$ and $P_t(x,dy)$ is the transition kernel associated to $X_t$. This process $X^{(V)}_{t}$ has the interpretation of being the Markov process $X_t$ killed at rate $V(X_t)$. Then the Feynman-Kac formula shows that its infinitesimal generator is
\begin{equation*}
L^{(V)}= L - V,
\end{equation*}
and is defined on a suitable dense domain of functions $D(L^{(V)})$ in $C_0(\R^N)$.

In order to turn this sub-Markov process into a Markov process, one performs a {\it Doob $h$-transform} (sometimes called a {\it ground state transform}).
A function $h$ is called {\it $L^{(V)}$-harmonic} if $h\in D(L^{(V)})$ and $L^{(V)}h=0$. If there is a unique $L^{(V)}$-harmonic function $h$, then the {\it Doob $h$-transform} of the sub-Markov process generated by $L^{(V)}$ is defined in terms of its infinitesimal generator
\begin{equation*}
L^{(V,h)} = h^{-1} L^{(V)} h
\end{equation*}
on a suitable dense domain of functions $D(L^{(V,h)})$ in $C_0(\R^N)$. (For justification that this is, indeed, Markov see Proposition 3.9, Section VIII of \cite{RevuzYor}.)

Let $\Omega$ be $C(\R^+,\R^N)$, the space of continuous trajectories in $\R^N$, and let $\mathcal{F}_t$ be the filtration generated by the trajectories up to time $t$ and $\mathcal{F}_t^0=\sigma(\{x([0,t]):x(0)=x^0\})$. Consider the canonical processes associated to the diffusions with generator $L$ and $L^{(V,h)}$ started at time $0$ at some position $x^0\in \R^N$. Let $Q_{0,x^0;t}$ and $Q_{0,x^0;t}^{(V,h)}$ denote the respective restriction of the probability measures of these processes to their natural filtration $\mathcal{F}_t^0\subset \mathcal{F}_t$. Then we may calculate the Radon-Nikodym derivative of $Q_{0,x^0;t}^{(V,h)}$ with respect to $Q_{0,x^0;t}$ on $\mathcal{F}_t^0$ as
\begin{equation*}
\frac{Q_{0,x^0;t}^{(V,h)}}{Q_{0,x^0;t}}\Big(\{x(s)\}_{0\leq s\leq t}\Big)= \frac{h(x(t))}{h(x(0))} \exp\left\{ - \int_0^{t} V(x(s)) ds\right\}.
\end{equation*}
We are using $x(\cdot)$ here to represent a sample path.
Note that as $\mathcal{F}_t^0$ is the natural filtration for processes starting at $x^0$ (and hence supported on trajectories starting at $x^0$), there is no need here to specify that $x(0)=x^0$.

For $(a,b)\subset [0,t]$, we define sigma-fields $\mathcal{F}_{t;(a,b)}^{0}=\sigma(x(s): s\in (a,b))$ and $\mathcal{F}_{t;(a,b)^c}^{0}=\sigma(x(s): s\in [0,t]\setminus(a,b))$, i.e. the respective natural filtrations generated by the processes on $[0,t]$ respectively inside and outside of the set of time $(a,b)$. It then follows from the semi-group property and the description above of the Radon-Nikodym derivative that
\begin{equation}\label{RNFKeqn}
Q_{0,x^0;t}^{(V,h)}\Big[\mathsf{E} \,\big\vert\, \mathcal{F}_{t;(a,b)^c}^{0}\Big] = Q_{0,x^0;t}^{(V,h)}\Big[\mathsf{E} \,\big\vert\, \sigma(x(a),x(b))\Big] = \frac{1}{Z_{(a,b),x(a),x(b)}}\E_{(a,b),x(a),x(b)}\left[{\bf 1}_{\mathsf{E}} \exp\left\{ - \int_0^{t} V(x(s)) ds\right\}\right].
\end{equation}
Here, $\mathsf{E}$ is any event in $\mathcal{F}_{t;(a,b)}^{0}$, ${\bf 1}_{\mathsf{E}}$ is the indicator function of that event, $\E_{(a,b),x(a),x(b)}$ is the expectation with respect to the bridge process formed by conditioning $x(\cdot)$ to go from $x(a)$ at time $a$ to $x(b)$ at time $b$ (where we assume that such an $L$-bridge process exists -- as they certainly do in our case of $L=\tfrac{1}{2}\Delta$) and $Z_{(a,b),x(a),x(b)}$ is the normalizing constant needed to make the measure integrate to one:
\begin{equation*}
Z_{(a,b),x(a),x(b)} = \E_{(a,b),x(a),x(b)}\left[\exp\left\{ - \int_0^{t} V(x(s)) ds\right\}\right].
\end{equation*}
Equation (\ref{RNFKeqn}) represents a Gibbs property: it says that the law of the process $X_{s}^{(V,h)}$ inside the time interval $(a,b)$ is dependent only on its values at times $a$ and $b$. Moreover, this law is given explicitly by reweighting the measure on $L$-bridges from $(a,X^{(V,h)}_a)$ to $(b,X^{(V,h)}_{b})$ according to the (properly normalized) killing potential $V$.

%\note{revise the Sigma field notation here}

Finally, let us observe one refinement of this Gibbs property. Consider an infinitesimal generator $L$ which factors into one-dimensional generators acting on each coordinate of $\R^N$ separately: $L = \sum_{i=1}^{N} \tilde{L}_i$ where $\tilde{L}_i= {\rm Id}_1 \otimes\cdots {\rm Id}_{i-1}\otimes L_i \otimes {\rm Id}_{i+1}\otimes \cdots \otimes {\rm Id}_{N}$. Here ${\rm Id}_j$ is the identity on the $j$-th coordinate and $L_i$ acts on functions $f:\R\to \R$. One example of such an operator is $\Delta/2$ on $\R^N$, which factors into a sum of one-dimensional half-Laplacians on each coordinate. Also, consider $V(x) = \sum_{i=1}^{N-1} \Ham(x_{i+1}-x_{i})$ for a positive Borel function $\Ham:\R\to [0,\infty)$. Assume that $\Ham$ is such that the above discussion applies to $V$.
For $1\leq k_1\leq k_2\leq N$, let $K=(k_1,\ldots, k_2)$ and
define sigma-fields $\mathcal{F}_{t;K\times (a,b)}^{0}=\sigma(\{x_i(s)\}: (i,s)\in K\times (a,b))$ and $\mathcal{F}_{t;(K\times (a,b))^c}^{0}=\sigma(\{x_i(s)\}: (i,s)\in  \{1,\ldots, N\}\times [0,t] \setminus K\times (a,b))$, i.e. the respective natural filtrations generated by the processes on $[0,t]$ respectively inside and outside of the set of paths with index $i\in K$ on the time interval $s\in (a,b)$. The Gibbs property of equation (\ref{RNFKeqn}) can be refined as
\begin{eqnarray}\label{RNFKeqnSpec}
&&\hskip-.25in Q_{0,x^0;t}^{(V,h)}\Big[\mathsf{E} \,\big\vert\, \mathcal{F}_{t;(K\times (a,b))^c}^{0}\Big] = Q_{0,x^0;t}^{(V,h)}\Big[\mathsf{E} \big\vert \sigma(x(a),x(b),x_{k_1-1},x_{k_2+1})\Big]\\
&=& \frac{1}{{Z_{k_1,k_2,(a,b),x(a),x(b),x_{k_1-1},x_{k_2+1}}}} \E_{k_1,k_2,(a,b),x(a),x(b)}\left[{\bf 1}_{\mathsf{E}} W_{k_1,k_2,(a,b),x(a),x(b),x_{k_1-1},x_{k_2+1}}\right]
\end{eqnarray}
where
\begin{eqnarray*}
&&\hskip -.25in W_{k_1,k_2,(a,b),x(a),x(b),x_{k_1-1},x_{k_2+1}}(x_{k_1},\ldots, x_{k_2})\\
&=&\exp\bigg\{- \int_a^{b} \Ham(x_{k_1}(s)-f(s)) ds - \sum_{i=k_1}^{k_2-1} \int_a^{b} \Ham(x_{i+1}(s)-x_i(s)) ds - \int_a^{b} \Ham(g(s)-x_{k_2}(s)) ds\bigg\}.
\end{eqnarray*}
Here, $\mathsf{E}$ is any $\mathcal{F}_{t;K\times (a,b)}^{0}$-measurable event, $x(a)=(x_{k_1}(a),\ldots, x_{k_2}(a))$, $x(b)=(x_{k_1}(b),\ldots, x_{k_2}(b))$, and
$\E_{k_1,k_2,(a,b),x(a),x(b)}$ is the expectation with respect to the bridge process under which each coordinate function (for $k \in \{ k_1,\ldots, k_2 \}$) evolves independently as a diffusion with generator $L_k$ conditioned to be a bridge from $x_k(a)$ at time $a$ to $x_k(b)$ at time $b$. The normalizing constant $Z_{k_1,k_2,(a,b),x(a),x(b),x_{k_1-1},x_{k_2+1}}$ is chosen so that the measure integrates to one: its value is
\begin{equation*}
 \E_{k_1,k_2,(a,b),x(a),x(b)}\big[W_{k_1,k_2,(a,b),x(a),x(b),x_{k_1-1},x_{k_2+1}}(x_{k_1},\ldots, x_{k_2})\big].
\end{equation*}
Equation (\ref{RNFKeqnSpec}) says that the law of coordinates $k_1,\ldots, k_2$ of the process $X_{s}^{(V,h)}$ inside the time interval $(a,b)$ is only dependent on its values $x(a)$ and $x(b)$ at respective times $a$ and $b$ as well as the adjacently indexed curves $x_{k_1-1}$ and $x_{k_2+1}$ on the interval $(a,b)$. Moreover, this law is given explicitly by reweighting the measure on $L$-bridges from $(a,x(a))$ to $(b,x(b))$ according to the (properly normalized) killing potential $V$.

\clearpage
%\addcontentsline{toc}{section}{Index}
\printnotation
%\printindex

\clearpage
% \addcontentsline{toc}{chapter}{References}


\begin{thebibliography}{alpha}
\bibitem{ADvM}
M.~Adler, J.~Del\'{e}pine, P.~van Moerbeke.
\newblock Dyson's nonintersecting Brownian motions with a few outliers.
\newblock {\it Comm. Pure Appl. Math.}, {\bf 62}:334--395 (2009).

\bibitem{AKQ2}
T.~Alberts, K.~Khanin, J.~Quastel.
\newblock Intermediate disorder regime for $1+1$ dimensional directed polymers.
\newblock {\it Ann. Probab.}, {\bf 42}:1212--1256 (2014).

\bibitem{AKQ3}
T.~Alberts, K.~Khanin, J.~Quastel.
\newblock The continuum directed random polymer.
\newblock {\it J. Stat. Phys.}, {\bf 154}:305--326 (2014).

\bibitem{ACQ}
G.~Amir, I.~Corwin, J.~Quastel.
\newblock Probability distribution of the free energy of the continuum directed random polymer in $1+1$ dimensions.
\newblock {\em Comm. Pure Appl. Math.}, ,{\bf 64}:466--537 (2011).

\bibitem{BFP}
J.~Baik, P.~L.~Ferrari, S.~P\'{e}ch\'{e}.
\newblock Convergence of the two-point function of the stationary TASEP.
\newblock In {\it Singular Phenomena and Scaling in Mathematical Models} 91--110 (2014).

\bibitem{BQS}
M.~Bal\'{a}zs, J.~Quastel, T.~Sepp\"{a}l\"{a}inen.
\newblock Scaling exponent for the Hopf-Cole solution of KPZ/stochastic Burgers.
\newblock {\it J. Amer. Math. Soc.}, {\bf 24}:683--708 (2011).

\bibitem{Bar}
Y.~Baryshnikov.
\newblock GUEs and queues.
\newblock {\it Probab. Th. Rel. Fields}, {\bf 119}: 256--274 (2001).

\bibitem{BOCon}
F.~Baudoin, N. O'Connell.
\newblock Exponential functionals of Brownian motion and class one Whittaker functions.
\newblock {\it Ann. Inst. H. Poincare B}, {\bf 47}:1096--1120 (2011).


%\bibitem{BAC}
%G.~ Ben Arous, I.~Corwin.
%\newblock Current fluctuations for TASEP: A proof of the Pr\"{a}hofer-Spohn conjecture.
%\newblock {\em Ann. Probab.}, {\bf 39}:104-138 (2011).

\bibitem{BC}
L.~Bertini and N~Cancrini.
\newblock  The Stochastic Heat Equation: Feynman-Kac Formula and Intermittence.
\newblock {\em J. Stat. Phys.} {\bf 78}:1377--1401 (1995).

\bibitem{BG}
L.~Bertini and G.~Giacomin.
\newblock  Stochastic Burgers and KPZ equations from particle systems.
\newblock {\em Comm. Math. Phys.} {\bf 183}:571--607 (1997).

\bibitem{Bill}
P.~Billingsley.
\newblock {\it Convergence of probability measures.}
\newblock Wiley, New York, 1968.

\bibitem{BorCor}
A.~Borodin, I.~Corwin.
\newblock Macdonald processes.
\newblock {\it Probab. Theory Rel. Fields}, {\bf 158}:225--400 (2014).

\bibitem{BCF}
A.~Borodin, I.~Corwin, P.~L.~Ferrari.
\newblock Free energy fluctuations for directed polymers in random media in $1+1$ dimension.
\newblock {\it Comm. Pure Appl. Math.}, {\bf 67}:1129--1214 (2014).

\bibitem{BCFV}
A.~Borodin, I.~Corwin, P.~L.~Ferrari, B.~Vet\H{o}.
\newblock Height fluctuations for the stationary KPZ equation.
\newblock {\it Math. Phys. Analy. Geo.}, {\bf 18}:20 (2015).

\bibitem{BCS}
A.~Borodin, I.~Corwin, T.~Sasamoto.
\newblock From duality to determinants for $q$-TASEP and ASEP.
\newblock {\it Ann. Probab.}, {\bf 42}:2314--2382 (2014).


%\bibitem{BorFer}
%A.~Borodin, P.~L.~Ferrari.
%\newblock Anisotropic growth of random surfaces in 2+1 dimensions.
%\newblock {\it Commun. Math. Phys.}, to appear. arXiv:0804.3035.

\bibitem{BFS}
A.~Borodin, P.~L.~Ferrari, T.~Sasamoto.
\newblock Transition between Airy$_1$ and Airy$_2$ processes and TASEP fluctuations.
\newblock {\it Comm. Pure Appl. Math.}, {\bf 61}:1603--1629 (2008).

\bibitem{BJ}
Ph.~Bougerol, Th.~Jeulin.
\newblock Paths in Weyl chambers and random matrices.
\newblock {\it Probab. Th. Rel. Fields}, {\bf 124} 517--543 (2002).

\bibitem{CDR}
P.~Calabrese, P.~Le Doussal,A.~Rosso.
\newblock Free-energy distribution of the directed polymer at high temperature.
\newblock {\it Euro. Phys. Lett.}, {\bf 90}:20002 (2010).

\bibitem{Davar}
D.~Conus, M.~Joseph, D.~Khoshnevisan.
\newblock On the chaotic character of the stochastic heat equation, before onset of intermittency.
\newblock {\it Ann. Probab.}, {\bf 41}:2225--2260 (2013).

\bibitem{CorwinReview}
I.~Corwin.
\newblock The Kardar-Parisi-Zhang equation and universality class.
\newblock{\it Rand. Mat. Theo. Appl.}, {\bf 1}:1130001 (2012).

\bibitem{CH}
I.~Corwin, A.~Hammond.
\newblock Brownian Gibbs property for Airy line ensembles.
\newblock {\em Invent. Math.}, to appear. arXiv:1108.2291.

\bibitem{CLW}
I.~Corwin, Z.~Liu, D.~Wong.
\newblock Fluctuations for TASEP and LPP with general initial data.
\newblock {\it Ann. Appl. Probab.}, {\bf 26}:2030--2082 (2016).

\bibitem{CQ}
I.~Corwin, J.~Quastel.
\newblock Crossover distributions at the edge of the rarefaction fan.
\newblock {\it Ann. Probab.}, {\bf 41}:1243--1314 (2013).

\bibitem{CQ2}
I.~Corwin, J.~Quastel, D.~Remenik.
\newblock Renormalization fixed point of the KPZ universality class.
\newblock {\it J. Stat. Phys.}, {\bf 160}:815--834 (2015).

\bibitem{COSZ}
I.~Corwin, N.~O'Connell, T.~Sepp\"{a}l\"{a}inen, N.~Zygouras.
\newblock Tropical combinatorics and Whittaker functions.
\newblock {\it Duke. Math. J.}, {\bf 163}:513--563 (2014).

\bibitem{Dot}
V.~Dotsenko.
\newblock Bethe ansatz derivation of the Tracy-Widom distribution for one-dimensional directed polymers.
\newblock {\it Euro. Phys. Lett.}, {\bf 90}:20003 (2010).

\bibitem{FerNej}
P.~L.~Ferrari, P.~Nejjar.
\newblock Anomalous shock fluctuations in TASEP and last passage percolation models.
\newblock {\it Probab. Theo. Rel. Fields}, {\bf 161}:61--109 (2015).

\bibitem{FNS}
D.~Forster, D.R.~Nelson, M.J.~Stephen.
\newblock Large-distance and long-time properties of a randomly stirred fluid.
\newblock {\it Phys. Rev. A}, {\bf 16}:732--749 (1977).


\bibitem{GTW}
J.~Gravner, C.~Tracy, H.~Widom.
\newblock Limit theorems for height fluctuations in a class of discrete space and time growth models.
\newblock {\it J. Stat. Phys.} {\bf 102}: 1085--1132 (2001).

\bibitem{Gro}
P.~Groeneboom.
\newblock The tail of the maximum of Brownian motion minus a parabola.
\newblock {\it Elect. Commun. Probab.}, {\bf 16}:458--466 (2011).

\bibitem{Hairer}
M.~Hairer.
\newblock Solving the KPZ equation.
\newblock {\it Ann. Math}, {\bf 178}:559--664 (2013).

\bibitem{KJtransversal}
K.~Johannson.
\newblock Transversal fluctuations for increasing subsequences on the plane.
\newblock {\it Prob. Theory Related Fields}, {\bf 116}:445--456 (2000).

\bibitem{KPZ}
K.~Kardar, G.~Parisi, Y.Z.~Zhang.
\newblock  Dynamic scaling of growing interfaces.
\newblock {\em Phys. Rev. Lett.},  {\bf 56}:889--892 (1986).


\bibitem{KS}
I.~Karatzas, S.~Shreve.
\newblock {\it Brownian motion and stochastic calculus}.
\newblock Volume 113 of Graduate Texts in Mathematics.
\newblock Springer, 1988.



\bibitem{Katori}
M.~Katori.
\newblock Survival probability of mutually killing Brownian motions and the O'Connell process.
\newblock {\it J. Stat. Phys.}, {\bf 147}:206--223 (2012).

\bibitem{Katori2}
M.~Katori.
\newblock O'Connell's process as a vicious Brownian motion.
\newblock {\it Phys. Rev. E.}, {\bf 84}:061144 (2011).

%\bibitem{McKean}
%H.P.~McKean.
%\newblock {\it Stochastic Integrals.}
%\newblock Adademic Press, New York, 1969.

\bibitem{MorenoFlores}
G.~Moreno~Flores.
\newblock On the (strict) positivity of solutions of the stochastic heat equation.
\newblock {\it Ann. Probab.}, {\bf 42}:1635--1643 (2014).


\bibitem{MRQ}
G.~Moreno~Flores, J.~Quastel, D.~Remenik.
\newblock Unpublished manuscript.


\bibitem{Greg}
G.~Moreno~Flores, T.~Sepp\"{a}l\"{a}inen, B.~Valko.
\newblock Fluctuation exponents for directed polymers in the intermediate disorder regime.
\newblock {\it Elect. J. Probab.}, {\bf 19}:89 (2014).

\bibitem{M}
C.~M\"{u}ller.
\newblock On the support of solutions to the heat equation with noise.
\newblock {\it Stochastics}, {\bf 37}:225--246 (1991).

\bibitem{MN}
C.~M\"{u}ller, D.~Nualart.
\newblock Regularity of the density for the stochastic heat equation.
\newblock {\it Elect. J. Probab.}, {\bf 74}:2248--2258 (2008).


\bibitem{Nica}
M.~Nica.
\newblock Intermediate disorder limits for multi-layer semi-discrete directed polymers.
\newblock arXiv:1609.00298.

\bibitem{OCon}
N.~O'Connell.
\newblock Directed polymers and the quantum Toda lattice.
\newblock {\it Ann. Probab.} {\bf 40}:437--458, (2012).

\bibitem{OCon2}
N.~O'Connell.
\newblock Whittaker functions and related stochastic processes.
\newblock Proceedings of Fall 2010 MSRI semester on `Random matrices, interacting particle systems and integrable systems'. arXiv:1201.4849.

\bibitem{OConWar}
N.~O'Connell, J.~Warren.
\newblock A multi-layer extension of the stochastic heat equation.
\newblock {\it Commun. Math. Phys.}, {\bf 341}:1--33 (2016).

\bibitem{OConYor}
N.~O'Connell, M.~Yor.
\newblock Brownian analogues of Burke's theorem.
\newblock {\it Stoch. Process. Appl.}, {\bf 96}:285--304 (2001).

\bibitem{OConYor2}
N.~O'Connell, M.~Yor.
\newblock A representation for non-colliding random walks.
\newblock {\it Elect. Comm. Probab.}, {\bf 7} (2002).

\bibitem{QR}
J.~Quastel, D.~Remenik.
\newblock Local Brownian property of the narrow wedge solution of the KPZ equation.
\newblock {\it Elect. Commun. Probab.}, {\bf 16}: 712--719 (2011).

\bibitem{RemQuasReview}
J.~Quastel, D.~Remenik.
\newblock Airy processes and variational problems.
\newblock Topics in Percolative and Disordered Systems, 121--171 (2014).


\bibitem{RevuzYor}
D.~Revuz, M.~Yor.
\newblock {\it Continuous martingales and Brownian motion}.
\newblock Springer, 2005.


\bibitem{SaSp}
T.~Sasamoto, H.~Spohn.
\newblock One-dimensional KPZ equation: an exact solution and its universality.
\newblock {\em Phys. Rev. Lett.}, {\bf 104}:23 (2010).

\bibitem{SaSpReview}
T.~Sasamoto, H.~Spohn.
\newblock The $1+1$-dimensional Kardar-Parisi-Zhang equation and its universality class.
\newblock {\em J. Stat. Mech.}, P11013 (2010).


\bibitem{S}
T.~Sepp\"{a}l\"{a}inen.
\newblock Scaling for a one-dimensional directed polymer with boundary conditions.
\newblock {\it Ann. Probab.}, {\bf 40}:19--73 (2012).

\bibitem{SeppValko}
T.~Sepp\"{a}l\"{a}inen, B.~Valko.
\newblock Bounds for scaling exponents for a 1+1 dimensional directed polymer in a Brownian environment.
\newblock {\it ALEA}, {\bf 7}:451--476 (2010).

\bibitem{SpohnLineEnsemble}
H.~Spohn.
\newblock KPZ equation in one dimension and line ensembles.
\newblock {\it In Proceedings of STATPHYS22} 847--857. Springer India, 2005.
\end{thebibliography}
\end{document}